\documentclass[a4paper,11pt]{article}
\usepackage[latin1]{inputenc}
\usepackage[T1]{fontenc}
\usepackage[normalem]{ulem}
\usepackage[english]{babel}
\usepackage{verbatim}
\usepackage{graphicx}
\usepackage{enumerate}
\usepackage{amsmath,amssymb,amsfonts,amsthm}
\usepackage{rotating}
\usepackage{float}
\usepackage{array}

\newtheorem{theorem}{Theorem}[subsection]

\newtheorem{lemma}[theorem]{Lemma}
\newtheorem{prop}[theorem]{Proposition}
\newtheorem{cor}[theorem]{Corollary}

\theoremstyle{definition}
\newtheorem{defn}[theorem]{Definition}
\newtheorem{rem}[theorem]{Remark}
\newtheorem{ex}[theorem]{Example}

\input xy
\xyoption{all}
\newcommand{\sbt}{\,\begin{picture}(-1,1)(0.5,-1)\circle*{1.8}\end{picture}\hspace{.05cm}}

\DeclareMathOperator{\pr}{pr}

\DeclareMathOperator{\id}{id}
\DeclareMathOperator{\im}{Im}

\DeclareMathOperator{\tr}{Tr}

\DeclareMathOperator{\coker}{coker}

\DeclareMathOperator{\map}{Map}

\DeclareMathOperator{\Top}{Top}

\DeclareMathOperator*{\colim}{colim}
\DeclareMathOperator*{\hocolim}{hocolim}

\DeclareMathOperator{\THH}{THH}
\DeclareMathOperator{\THR}{THR}

\DeclareMathOperator{\sym}{Sym}
\DeclareMathOperator{\hofib}{hofib}
\DeclareMathOperator{\conn}{Conn}
\DeclareMathOperator{\modu}{Mod}
\DeclareMathOperator{\res}{res}
\DeclareMathOperator{\KR}{KR}
\newdir^{ (}{{}*!/-5pt/@^{(}}

\begin{document}

\begin{titlepage}
\phantom{d}
\noindent
PhD thesis\newline
University of Copenhagen --- Department of Mathematical Sciences --- 2012

\vskip 20pt

\noindent\rule{\textwidth}{2pt}
\newline
\phantom{d}
\newline
\noindent{{\LARGE{\textbf{\textsc{Stable Real K-theory and Real}}}}}
\newline
\newline
\noindent{{\LARGE{\textbf{\textsc{Topological Hochschild Homology}}}}}\vspace{5pt}
\newline
\noindent\rule{\textwidth}{2pt}

\vskip 60pt
\noindent
{\Large {\textsc{\textbf{Emanuele Dotto}}}}\\
\vskip 0pt
\vspace{7cm}

\begin{flushleft}
{\small
\begin{tabular}{l l l}
Submitted: & 30th October 2012\\

& \\

Academic adviser: & Ib Madsen \\
                                  & University of Copenhagen\\
& \\
Assessment committee:  
& Marcel Bökstedt \\
& University of Aarhus\\
&\\
& Bj\o rn Jan Dundas\\
& University of Bergen\\
&\\
& Nathalie Wahl (chair)\\
& University of Copenhagen\\
\end{tabular}
}
\end{flushleft}

\end{titlepage}

\thispagestyle{empty}

\begin{flushleft}
{\small
Emanuele Dotto \\
Department of Mathematical Sciences \\
University of Copenhagen \\
Universitetsparken 5 \\
DK-2100 K\o benhavn \O \\
Denmark \\
\vspace{0.5 em}
dotto@math.ku.dk \\
http://www.math.ku.dk/\textasciitilde dotto/
}
\end{flushleft}

\vspace{5em}

\vspace{\stretch{0.85}}

{\small
\noindent\copyright{} Emanuele Dotto, 2012
\vspace{0.5 em}

\noindent ISBN 978-87-7078-997-4
}

\pagebreak

\pagenumbering{roman}
\setcounter{page}{1}


\begin{center}
\Large{\textbf{Abstract}}
\end{center}

\noindent The classical trace map is a highly non-trivial map from
algebraic $K$-theory to topological Hochschild homology (or topological cyclic homology) introduced by Bökstedt, Hsiang and Madsen. It led to
many computations of algebraic $K$-theory of rings.
Hesselholt and Madsen recently introduced a $\mathbb{Z}/2$-equivariant version of
Waldhausen $S_{\sbt}$-construction for categories with duality. The output is a certain spectrum with involution, called the real $K$-theory spectrum, and associated bigraded groups $\KR_{p,q}(C,D)$ analogous to Atiyah's real (topological) $K$-groups. This thesis develops a theory of topological Hochschild homology for categories with duality, and a $\mathbb{Z}/2$-equivariant trace map from real $K$-theory to it. The main result of the thesis is the theorem that stable $\KR$ of the category of projective modules over a split square zero extension of a ring is equivalent to the real topological Hochschild homology of the ring with appropriate coefficients. This is the real version of the celebrated Dundas-McCarthy theorem for ordinary $K$-theory.

\vspace{2.3cm}

\begin{center}
\Large{\textbf{Resum\' e}}
\end{center}

\noindent Den klassiske sporafbildning, som blev indf\o rt af Bökstedt, Hsiang og
Madsen, er en særdeles ikke-triviel afbildning
fra algebraisk $K$-teori til topologisk Hochschild-homologi (eller til
topologisk cyklisk homologi).
Den har f\o rt til mange udregninger af algebraisk $K$-teori for ringe.
Hesselholt og Madsen har nyeligt
indf\o rt en $\mathbb{Z}/2$-ækvivariant version af Waldhausens $S_{\sbt}$-konstruktion for
kategorier med dualitet.
Udkommet af dette er et særligt spektrum med en involution, kaldet det
reelle $K$-teori-spektrum, samt
til-h\o rende bigraduerede grupper $\KR_{p,q}(C,D)$, svarende til Atiyahs reelle
(topologiske) $K$-grupper.
I denne afhandling udvikles en teori for topologisk Hochschild-homologi
for kategorier med dualitet, til hvilken
der er en $\mathbb{Z}/2$-ækvivariant sporafbildning fra reel $K$-teori. Afhandlingens
hovedresultat er, at stabil $\KR$ af kategorien af
projektive moduler over en split kvadrat-$0$ udvidelse af en ring er
ækvivalent med den reelle topologiske
Hochschild-homologi af ringen med passende koefficienter. Dette er en reel
version af den kendte Dundas-McCarthy-sætning
for sædvanlig $K$-teori.

\newpage

\begin{center}
\Large{\textbf{Acknowledgements}}
\end{center}

\noindent I want to thank Ib Madsen for being a present adviser and an inspiring Mentor, and for always guiding me in the right direction with an energetic and friendly push.

\noindent I thank David Ayala and Oscar Randall-Williams for teaching me all sorts of math tricks.

\noindent I finally thank Alexander, Cyril, David, Eduard and Massimiliano for giving me all the distractions I needed  from work, and for always making me feel home.
\vspace{2cm}

\begin{flushright}\it{Copenhagen, 30th October 2012}\\
\it{Emanuele Dotto}\end{flushright}

\pagebreak

\tableofcontents

\newpage

\pagenumbering{arabic}
\setcounter{page}{1}


\addcontentsline{toc}{section}{Introduction}
\section*{Introduction}

Hesselholt and Madsen in \cite{IbLars} define a $\mathbb{Z}/2$-equivariant version of Waldhausen $S_{\sbt}$-construction: the $S_{\sbt}^{2,1}$-construction. This induces the real $K$-theory functor $\KR$, from the category of exact categories with duality to a certain category of spectra with involution, called real spectra. When specialized to rings with antistructures (e.g. in the sense of Wall \cite{wall}), the infinite loop space of the $\KR$ spectrum is a $\mathbb{Z}/2$-space, whose fixed points is the group completion for the classifying space of the category of non-degenerate bilinear forms on the ring with antistructure.

This thesis is part of a bigger project aimed to develop trace methods for computations in real $K$-theory.
The classical trace map is a natural weak map $\tr\colon K\longrightarrow TC$ from $K$-theory to topological cyclic homology, introduced in \cite{BHM} by Bökstedt, Hsiang and Madsen. In \cite{dmthm}, \cite{dundasrelKandTC} and \cite{Dundasbook}, Dundas, Goodwillie and McCarthy  show that the trace map defines a homotopy cartesian square
\[\xymatrix{K(B)\ar[d]\ar[r]^-{\tr}&TC(B)\ar[d]\\
K(A)\ar[r]_-{\tr}&TC(A)}\]
for every map of $\mathcal{S}$-algebras $B\rightarrow A$ (or symmetric ring spectra) which induces a surjection with nilpotent kernel in $\pi_0$.
This theorem leads to several hard computations in $K$-theory, for example for perfect fields \cite{wittvectors}, or truncated polynomial algebras \cite{teena}.

This thesis contains a first step toward a version of this theorem for real $K$-theory.
By approximating $\mathcal{S}$-algebras by simplicial ring and via Goodwillie calculus one can reduce the theorem above to analytical properties of $K$ and $TC$ \cite{dmthm}, and to showing that the connectivities of two maps
\[\tr\colon \widetilde{K}(A\ltimes M)\longrightarrow \THH(A;M(S^1)) \ \ \ \ \res\colon \widetilde{TC}(A\ltimes M)\longrightarrow \THH(A;M(S^1))\]
grow fast enough with the connectivity of the simplicial $A$-bimodule $M$, cf. \cite{stableDM}, \cite{stabletcisthh}.

The main result of this thesis is a real version of the first map together with its corollary
\[\tr\colon \widetilde{\KR}^S(A\ltimes M)\stackrel{\simeq}{\longrightarrow} \THR(A;M(S^{1,1}))\]
We examine the real topological Hochschild homology functor $\THR$ (of \cite{IbLars}) defined on exact categories with duality and with values in $\mathbb{Z}/2$-equivariant spectra. The underlying non-equivariant spectrum of $\THR(C,D)$ is the classical topological Hochschild homology spectrum $\THH(C)$. Second, we produce an equivariant trace map
\[\tr\colon \KR(C,D)\longrightarrow\THR(C,D)\]
Following the "classical" results from \cite{ringfctrs} we study the behavior of $\THR$ on products, prove equivariant Morita equivalence, establish deloopings via the $S_{\sbt}^{2,1}$-construction and other real versions of the classical results from \cite{ringfctrs}.

The main results relating $\KR$ and $\THR$ are Theorem \ref{relKR}, Theorem \ref{stabTHR} and Theorem \ref{comparisontrbeta}. They can be summarized as follows: Given a bimodule $M$ over a ring $A$ and a split antistructure over $A\ltimes M$, the trace map induces a $\mathbb{Z}/2$-equivalence
\[\colim\limits_m\Omega^{2m,m}\widetilde{\KR}(A\ltimes M(S^{2m,m}))\longrightarrow \THR(A;M(S^{1,1}))\]
if $2\in A$ is invertible. I hope in later works to be able to remove the condition that $2\in A$ is invertible.

\newpage

\section{Review of $K$-theory and real $K$-theory}


\subsection{The $S_{\cdot}$-construction}

The $S_{\sbt}$-construction is a functor from the category of exact categories to simplicial categories introduced by Waldhausen in \cite{waldhausen}. We recall its construction and refer to \cite{waldhausen} for the details.
\begin{defn} An \textbf{additive category} is a category $C$ enriched in the symmetric monoidal category of abelian groups, with a zero object and finite products.
\end{defn}

\begin{defn}
A sequence $c\stackrel{i}{\longrightarrow}d\stackrel{p}{\longrightarrow}e$ of an additive category $C$ is \textbf{split-exact} if there is a sequence $e\stackrel{s}{\longrightarrow}d\stackrel{r}{\longrightarrow}c$ in $C$ such that
\[p\circ s=\id_e, r\circ i=\id_c\mbox{ and }s\circ p+i\circ r=\id_d\]
\end{defn}

\begin{defn}
An \textbf{exact category} is an additive category $C$ together with a family $\mathcal{E}$ of sequences
\[c\stackrel{i}{\longrightarrow}d\stackrel{p}{\longrightarrow}e\]
in $C$ satisfying the axioms below.
The morphisms appearing as the first map of a sequence of $\mathcal{E}$ are called \textbf{admissible monomorphisms}; those appearing as the second map of a sequence of $\mathcal{E}$ are the \textbf{admissible epimorphisms}, and the elements of $\mathcal{E}$ are the \textbf{exact sequences}.
\begin{enumerate}
\item For every sequence $c\stackrel{i}{\longrightarrow}d\stackrel{p}{\longrightarrow}e$ in $\mathcal{E}$, the map $i$ is a kernel for $p$ and the map $p$ is a cokernel for $i$.
\item $\mathcal{E}$ is closed under taking isomorphic sequences.
\item The class of admissible monomorphisms is closed under composition, and so is the class of admissible epimorphisms.
\item The pushout of an admissible monomorphism along any morphism exists
and is an admissible monomorphism; the pullback of an admissible epimorphism
along any morphism exists and is an admissible epimorphism.
\item $\mathcal{E}$ contains all the split-exact sequences of $C$.
\end{enumerate}
We say that $(C,\mathcal{E})$ is \textbf{split-exact} if every sequence in $\mathcal{E}$ is a split-exact sequence. We will often suppress $\mathcal{E}$ from the notation.

A functor between exact categories is called an \textbf{exact functor} if it is additive (i.e. it is $Ab$-enriched) and maps exact sequences to exact sequences.
\end{defn}

Exact categories form a category with exact functors as morphisms.

\begin{ex}\label{ringexact}
The main cases we will be interested in are the projective modules: let $A$ be a ring and $\mathcal{P}_A$ be the category with objects finitely generated projective right modules over $A$, and morphisms module maps. The enrichement in the category of abelian groups is given by pointwise sum of module homomorphisms. This has a canonical exact category structure, by declaring a sequence
\[P\stackrel{i}{\longrightarrow}Q\stackrel{p}{\longrightarrow}R\]
to be exact if $i$ is injective, $p$ is surjective, and $\ker p=\im i$. Since we are considering finitely generated projective modules, every exact sequence is split-exact.

This defines a lax functor $\mathcal{P}$ from rings to exact categories, associating to a ring $A$ the category $\mathcal{P}_A$ with the exact structure defined above, and to a ring map $f\colon A\longrightarrow B$ the exact functor $f_\ast\colon\mathcal{P}_A\longrightarrow \mathcal{P}_B$ sending a module $P$ to $P\otimes_{A}B$, where $B$ is the left $A$-module with
\[a\cdot b=f(a)\cdot b\]
In \ref{functorialPA} we will discuss a strictly functorial model for $\mathcal{P}$, but we take the lax construction as a definition.
\end{ex}

For every natural number $n$, let $[n]$ be the category associated to the poset $\{0,\dots,n\}$. Given two natural numbers $n$ and $m$, we denote $Cat([m],[n])$ the category of functors from $[m]$ to $[n]$ with natural transformations as morphisms. For every $0\leq l\leq m$, the functor $d_l\colon Cat([m],[n])\longrightarrow Cat([m-1],[n])$ is precomposition by the unique injective map $\delta^l\colon [m-1]\longrightarrow [m]$ whose image does not contain $l$; the functor $s_l\colon Cat([m],[n])\longrightarrow Cat([m+1],[n])$ denotes precomposition by the unique surjective map $\sigma^l\colon [m+1]\longrightarrow [m]$ that sends $l$ and $l+1$ to the same element.

\begin{defn}
The category $S_nC$ has objects the set of functors \[X\colon Cat([1],[n])\longrightarrow C\] satisfying the two following conditions
\begin{enumerate}
\item $X(\omega)=0$ unless $\omega\colon[1]\longrightarrow [n]$ is injective.
\item For all $\theta\colon [2]\longrightarrow [n]$, the sequence
\[X(d_2\theta)\longrightarrow X(d_1\theta)\longrightarrow X(d_0\theta)\]
is exact.
\end{enumerate}
The morphisms of $S_nC$ are natural transformations of functors. For any category $C$ we let $iC$ denote the category with the same objects but with morphisms the isomorphisms of $C$.
\end{defn}

\begin{ex}
An object in $S_3C$ is a diagram of the form
\[\xymatrix{X_{01} \ \ar@{>->}[r]&X_{02} \ \ar@{>->}[r]\ar@{->>}[d]&X_{03}\ar@{->>}[d]\\
&X_{12} \ \ar@{>->}[r]&X_{13}\ar@{->>}[d]\\
& & X_{23}
}\]
where the horizontal maps are all admissible monomorphisms and the vertical ones admissible epimorphisms.
\end{ex}

The maps $\delta^l$ and $\sigma^l$ defined above induce functors $d_l\colon S_nC\longrightarrow S_{n-1}C$ and $s_l\colon S_nC\longrightarrow S_{n+1}C$ by
\[d_lX(\omega)=X(\delta^l\omega) \mbox{  and  }s_lX(\omega)=X(\sigma^l\omega)\]
making $S_{\sbt} C$ into a simplicial category.
Restricting the morphisms to isomorphisms we get a simplicial category $iS_{\sbt} C$, and hence a bi-simplicial set
\[[n],[k]\longmapsto \mathcal{N}_kiS_nC\]
where $\mathcal{N}_k$ is the $k$-the nerve.
\begin{defn}
The \textbf{algebraic $K$-theory space} of an exact category $C$ is defined as the loop space
\[K(C)=\Omega|[n]\mapsto iS_nC|=\Omega|[n],[k]\mapsto \mathcal{N}_kiS_nC|\]
The \textbf{algebraic $K$-theory of a ring $A$} is defined by
\[K(A)=K(\mathcal{P}_A)\]
with $\mathcal{P}_A$ from example \ref{ringexact} above.
\end{defn}

The space $K(C)$ is an infinite loop space. Indeed, $S_0C$ is the one point category, and $S_1C=C$. Therefore the adjoints of the canonical maps $\mathcal{N}_kiS_1C\times \Delta^1\longrightarrow |\mathcal{N}_kiS_{\sbt} C|$ induce a simplicial map $\mathcal{N}_kiC\longrightarrow \Omega|\mathcal{N}_kiS_{\sbt} C|$ that realized gives
\[|iC|\longrightarrow \Omega|iS_{\sbt} C|\]
Moreover, an exact structure on $C$ induces an exact structure on $S_nC$ by declaring a sequence of diagrams to be exact if pointwise is an exact sequence of $C$. This allows to iterate the $S_{\sbt}$-construction.

\begin{theorem}[{\cite[1.5.3]{waldhausen}}]
The map $|iS_{\sbt} C|\longrightarrow \Omega|iS_{\sbt} S_{\sbt} C|$ is a homotopy equivalence.
\end{theorem}

In contrast the map $|iC|\longrightarrow \Omega|iS_{\sbt} C|$ is usually not an equivalence.

\begin{defn}
The \textbf{algebraic $K$-theory (spectrum)} of the exact category $C$ is the spectrum defined by the sequence of spaces
\[\uuline{K}(C)=\{n\mapsto |iS_{\sbt}^{(n)}C|\}\]
with structure maps adjoint to $|iS_{\sbt}^{(n)}C|\longrightarrow \Omega|iS_{\sbt}^{(n+1)} C|$.

The \textbf{algebraic $K$-theory (spectrum) of a ring $A$} is defined by
\[\uuline{K}(A)=\uuline{K}(\mathcal{P}_A)\]
\end{defn}
By the theorem above $\uuline{K}(C)$ is positively fibrant, in the sense that it is an $\Omega$-spectrum above degree zero. Its infinite loop space is $\Omega|S_{\sbt} C|$.

An exact functor $F\colon C\longrightarrow C'$ induces a simplicial functor $F_{\sbt}\colon S_{\sbt} C\longrightarrow S_{\sbt} C'$ by sending $X\colon Cat([1],[n])\longrightarrow C$ to $F\circ X$. Therefore $K$-theory defines a functor from the category of exact categories and exact functors to spectra. Precomposing this with the functor $\mathcal{P}$ of the example above, we obtain a functor
\[\uuline{K}\colon \mbox{Rings}\longrightarrow \mbox{Spectra}\]


\subsection{Categories with duality}\label{duality}

There are two generalizations of the $S_{\sbt}$-construction for categories "with duality", which will admit some extra structure. We recall the basic definitions from \cite{IbLars}.
\begin{defn} A \textbf{category with duality} is a category $C$ together with a functor $D\colon C^{op}\longrightarrow C$ and a natural isomorphism $\eta\colon \id \stackrel{\cong}{\Rightarrow}D^2$ such that for every object $c$ of $C$ the composite
\[D(c)\stackrel{\eta_{D(c)}}{\longrightarrow}DDD(c)\stackrel{D(\eta_c)}{\longrightarrow}D(c)\]
is the identity.
We say that $C$ is a \textbf{category with strict duality} if $D^2=\id$ and $\eta=\id$. 
\end{defn}
We will sometimes omit $\eta$, or even $D$ from the notation, and refer to $C$ as "a category with duality".

\begin{defn}\label{eqcatdual}
A \textbf{morphism $(C,D,\eta)\longrightarrow (C',D',\eta')$ of categories with duality} is a functor $F\colon C\longrightarrow C'$ together with a natural isomorphism $\xi\colon FD\stackrel{\cong}{\Rightarrow}D'F$ such that
\[\xymatrix{F(c)\ar[r]^-{F(\eta_c)}\ar[d]_{\eta'_{F(c)}}& FDD(c)\ar[d]^{\xi_{D(c)}}\\
D'D'F(c)\ar[r]_-{D'(\xi_c)} &D'FD(c)
}\]
commutes for all object $c$ of $C$.

Composition of $(F,\xi)\colon C\longrightarrow C'$ and $(G,\zeta)\colon C'\longrightarrow C''$ is defined by the pair $(G\circ F,\zeta_{F}\circ G(\xi))$.

A \textbf{morphism $(C,D,\id)\longrightarrow (C',D',\id)$ of categories with strict duality} is a functor $F\colon C\longrightarrow C'$ such that $FD=D'F$.

A \textbf{natural transformation of morphisms of categories with duality} $U\colon (F,\xi)\Rightarrow (G,\zeta)$ is a natural transformation $U\colon F\Rightarrow G$ such that
\[\xymatrix{D'G(c)\ar[r]^{D'(U_c)}&D'F(c)\\
GD(c)\ar[u]_{\zeta_c}&FD(c)\ar[l]^{U_{D(c)}}\ar[u]^{\xi_c}
}\]
commutes.

An \textbf{equivalence of categories with duality} is a morphism of categories with duality $(F,\xi)\colon C\longrightarrow C'$ such that there exist a morphism $(G,\zeta)\colon C'\longrightarrow C$ and natural isomorphisms of categories with duality $(F,\xi)\circ(G,\zeta)\Rightarrow\id$ and $(G,\zeta)\circ(F,\xi)\Rightarrow\id$.
\end{defn}

\begin{prop}[Pullback of duality structures under an equivalence of categories]\label{dualitiesandequivalences}
Let $(C,D,\eta)$ be a category with duality, $F\colon C'\longrightarrow C$ and $G\colon C\longrightarrow C'$ functors, and $\mu\colon \id\Rightarrow G\circ F$ and $\epsilon\colon F\circ G\Rightarrow\id$ natural isomorphism making $F$ and $G$ into an adjoint pair.
Then the functor $D'=GDF$ together with the natural isomorphism $\eta'\colon \id\Rightarrow (D')^2$ defined by
\[\xymatrix{ c\ar[r]^-{\eta_c'}\ar[d]_{\mu_c}&GDFGDF(c)\\
GF(c)\ar[r]_-{G(\eta_{F(c)})}&GDDF(c)\ar[u]_{GD(\epsilon_{DF(c)})}
}\]
is a duality on $C'$. Moreover the pair $(F,\epsilon_{DF})\colon (C',D')\longrightarrow (C,D)$ is an equivalence of categories with duality, with inverse $(G,GD(\epsilon))$.
\end{prop}

\begin{proof}
In order to show that $(D',\eta')$ is a duality, one needs to show that $D'(\eta_c')\circ\eta'_{D(c)}=\id_{D'(c)}$. A computation of the left hand side shows
\[G(D(F(\mu_c))\circ D(\epsilon_{F(c)})\circ\epsilon_{DF(c)})\circ\mu_{GDF(c)}=G(\epsilon_{DF(c)})\circ\mu_{GDF(c)}=\id_{D'c}\]
where the equalities follow from standard properties of an adjunction.

To show that $(F,\epsilon_{DF})$ is a map of categories with duality, one needs to show that
\[\epsilon_{DFGDF(c)}\circ F(GD(\epsilon_{DF(c)})\circ G(\eta_{F(c)})\circ\mu_c)=D(\epsilon_{DF(c)})\circ\eta_{F(c)}\]
This follows by naturality of $\epsilon$ and the adjunction properties. Similarly $(G,GD(\epsilon))$ is a map of categories with duality. The natural isomorphisms $\eta$ and $\epsilon$ are equivariant just by naturality.
\end{proof}

\begin{rem}
Let $F\colon C'\longrightarrow C$ be a full and faithful functor, injective on objects and essentially surjective. Then there exist a functor $G\colon C\longrightarrow C'$ that satisfies the hypothesis of the proposition above.
It is defined by choosing, for every object $c\in C$, an isomorphism $\epsilon_c\colon F(c')\longrightarrow c$ for some $c'\in C'$, but such that if $c$ is in the image of $F$, we take the identity $F(c')=c$ for $c'$ the unique object mapped to $c$ by $F$. Define on objects $G(c)=c'$. For a morphism $f\colon c\longrightarrow d$, define $G(f)$ to be the unique morphism $c'\longrightarrow d'$ such that
\[\xymatrix{F(c')\ar[r]\ar[d]_{\epsilon_c}&F(d')\ar[d]^{\epsilon_d}\\
c\ar[r]_f&d
}\]
commutes, that is $G(f)=F^{-1}(\epsilon_{d}^{-1}\circ f\circ \epsilon_c)$. Then by definition $G\circ F=\id$ and the isomorphisms $\epsilon_c$ define a natural isomorphism $\epsilon\colon F\circ G\Rightarrow\id$ 
It is easy to check that $\epsilon$ and $\mu=\id$ form an adjunction. In particular since $\epsilon_{F(c)}=\id$.
\end{rem}

\begin{prop}\label{equivalenceequivariant}
Suppose that $F\colon C'\longrightarrow C$ is an equivalence of categories, and that there are duality structures $(C',D',\eta')$, $(C,D,\eta)$ and a natural isomorphism $\xi\colon FD'\Rightarrow DF$ such that $(F,\xi)$ is a morphism of categories with duality. Then $(F,\xi)$ is an equivalence of categories with duality.
\end{prop}

\begin{proof}
We define a functor $G\colon C\longrightarrow C'$ in the same way as in the last proof.
Define a natural isomorphism $\chi\colon GD\Rightarrow D'G$ as follows. For an object $c\in C$, by definition $\epsilon_c\colon FG(c)\cong c$ and $\epsilon_{Dc}\colon FG(Dc)\cong Dc$. Consider the composite
\[FG(Dc)\stackrel{\epsilon_{Dc}}{\longrightarrow}D(c)\stackrel{D(\epsilon_c)}{\longrightarrow}DF(G(c))\stackrel{\xi_{G(c)}^{-1}}{\longrightarrow}FD'(G(c))\]
This is a morphism in $C$ from $F(GDc)$ to $F(D'Gc)$, and since $F$ is an equivalence of categories it has a unique preimage $\chi_c \colon GDc\longrightarrow D'Gc$, that is define
\[\chi_c=F^{-1}(\xi_{G(c)}^{-1}\circ D(\epsilon_c)\circ \epsilon_{Dc})\]
For $(G,\chi)$ to define a morphism of categories with duality, we need to show that 
\[\xymatrix{G(c)\ar[r]^-{G(\eta_{c})}\ar[d]_{\eta_{G(c)}'}& GDD(c)\ar[d]^{\chi_{D(c)}}\\
D'D'G(c)\ar[r]_-{D'(\chi_c)} &D'GD(c)
}\]
commutes. Let us compute $D'(\chi_c)$. Quite in general, 
\[D'F^{-1}(b\colon F(c')\longrightarrow F(d'))=F^{-1}(FD'd'\stackrel{\xi_{d'}}{\rightarrow}DF(d')\stackrel{Db}{\longrightarrow}DF(c')\stackrel{\xi_{c'}^{-1}}{\longrightarrow} FD'c')\]
Indeed, this is true if and only if
\[FD'F^{-1}(b)=\xi_{c'}^{-1}D(b)\xi_{d'}\]
which follows by naturality of $\xi$ for the map $F^{-1}(b)$. Therefore
\[\begin{array}{ll}D'(\chi_c)\!\circ\!\eta_{G(c)}'\!\!&=F^{-1}(\xi_{GDc}^{-1}\!\circ\!D(\xi_{G(c)}^{-1}\!\circ\!D(\epsilon_c)\!\circ\!\epsilon_{Dc})\!\circ\!\xi_{DGc}\!\circ\!F(\eta_{G(c)}'))\\
&=F^{-1}(\xi_{GDc}^{-1}\!\circ\!D(\epsilon_{Dc})\!\circ\!D^2(\epsilon_c)\!\circ\!D(\xi_{G(c)})^{-1}\!\circ\!\xi_{DGc}\!\circ\!F(\eta_{G(c)}'))
\end{array}\]
Since $(F,\xi)$ is a morphism of categories with duality, $D(\xi_{G(c)})^{-1}\circ\xi_{DGc}=\eta_{FGc}\circ F(\eta_{Gc}')^{-1}$ and therefore we get
\[D'(\chi_c)\circ\eta_{G(c)}=F^{-1}(\xi_{GDc}^{-1}\circ D(\epsilon_{Dc})\circ D^2(\epsilon_c)\circ \eta_{FGc})\]
By naturality of $\eta$, this is
\[F^{-1}(\xi_{GDc}^{-1}\circ D(\epsilon_{Dc})\circ \eta_{c}\circ \epsilon_c)\]

We now compute the other composite $\chi_{Dc}\circ G(\eta_{c})$. By definition, $\chi_{Dc}$ is given by
\[\chi_{Dc}=F^{-1}(\xi_{G(Dc)}^{-1}\circ D(\epsilon_{Dc})\circ \epsilon_{D^2c})\]
It remains to compute $G(\eta_{c})$, which by definition of $G$ is
\[G(\eta_{c})=F^{-1}(\epsilon_{D^2c}^{-1}\circ\eta_{c}\circ \epsilon_c)\]
Thus the composite gives
\[\chi_{Dc}\circ G(\eta_{c})=F^{-1}(\xi_{G(Dc)}^{-1}\circ D(\epsilon_{Dc})\circ\eta_{c}\circ \epsilon_c)\]

We now need to show that the natural transformation $\epsilon\colon (FG,\xi\circ \chi)\Rightarrow(\id,\id)$ given by $\epsilon_c=\phi_c$ is equivariant, that is
\[\xymatrix{Dc\ar[r]^{D(\epsilon_c)}&DFGc\\
Dc\ar[u]_{\id}&FGD(c)\ar[l]^{\epsilon_{Dc}}\ar[u]_{\xi_{Gc}\circ F(\chi_c)}
}\]
By definition, the right vertical map is
\[\xi_{Gc}\circ \xi_{G(c)}^{-1}\circ D(\epsilon_c)\circ \epsilon_{Dc}=D(\epsilon_c)\circ \epsilon_{Dc}\]
\end{proof}


\subsection{Antistructures and dualities on $\mathcal{P}_A$}\label{ringduality}

Here we recall the classification of duality structures on $\mathcal{P}_A$ from \cite{IbLars}, and we refer to it for the proofs.
Given a (right) $A\otimes A$-module $L$, we denote $L_t$ the right $A$-module structure on $L$ defined by $l\cdot a=l\cdot a\otimes 1$. Similarly, $L_s$ is the one given by $l\cdot a=l\cdot 1\otimes a$. 
\begin{defn}
An \textbf{antistructure over a ring $A$} is a triple $(A,L,\alpha)$, where $L$ is a (right) $A\otimes A$-module and $\alpha\colon L\longrightarrow L$ is an additive map such that
\begin{enumerate}
\item $\alpha(l\cdot a\otimes b)=\alpha(l)\cdot b\otimes a$
\item $\alpha^2=\id$
\item $L_t$ is a finitely generated projective $A$-module
\item The map $A\longrightarrow \hom_A(L_t,L_s)$ sending $a$ to $\alpha(-)\cdot a\otimes 1$ is bijective.
\end{enumerate}
\end{defn}

\begin{ex}\label{wallanti}
A Wall antistructure is a triple $(A,w,\epsilon)$ of a ring $A$, a unit $\epsilon\in A^{\!\times}$ and a ring isomorphism $w\colon A^{op}\longrightarrow A$ such that $w(\epsilon)=\epsilon^{-1}$ and $w^2(a)=\epsilon a \epsilon^{-1}$ (cf. \cite{wall}). It defines an antistructure $(A,A,\alpha)$ in our sense, where the $A\otimes A$-module structure on $A$ is
\[l\cdot a\otimes a'=w(a')la\]
and the map is $\alpha=w(-)\cdot \epsilon$. Notice that a Wall antistructure $(A,w,\epsilon)$ is exactly the data of an additive duality on the $Ab$-enriched category with one object $A$. Indeed, $w$ provides the additive functor $D=w\colon A^{op}\longrightarrow A$. The condition $w^2(a)=\epsilon a \epsilon^{-1}$ is naturality for the natural isomorphism
$\eta=\epsilon\colon\id\Rightarrow w^2$. Finally, the condition $w(\epsilon)=\epsilon^{-1}$ is the condition $D(\eta_\ast)\circ \eta_{D\ast}=\id$.
When $\epsilon$ is in the center the duality is strict, and we call the pair $(A,\alpha)$ an anti-involution.
\end{ex}

Given a projective $A$-module $P$, the abelian group
\[D_L(P)=\hom_A(P,L_s)\]
is an $A$-module via
\[(\lambda\cdot a)(p)=\lambda(p)\cdot a\otimes 1 \ , \ \ \lambda\in D_L(P)\]
Moreover, $D_L(P)$ is finitely generated projective, and thus this defines a functor
\[D_L\colon \mathcal{P}_A^{op}\longrightarrow \mathcal{P}_A\]
where a map $\phi\colon P\longrightarrow Q$ is sent to $\phi^\ast=(-)\circ\phi\colon D_L(Q)\longrightarrow D_L(P)$.

\begin{prop}[\cite{IbLars}] The functor $D_L\colon \mathcal{P}_A^{op}\longrightarrow \mathcal{P}_A$ together with the natural transformation $\eta^{\alpha}\colon \id\Rightarrow D_{L}^2$ defined by
\[\eta^{\alpha}_P(p)(\lambda\colon P\longrightarrow L_s)=\alpha(\lambda(p))\]
is a duality on the category $\mathcal{P}_A$.
\end{prop}

\begin{rem}\label{classdualpa} Conversely, every duality structure $(D,\eta)$ on $\mathcal{P}_A$ comes from an antistructure. Indeed the right module $D(A)$ admits an extra structure, namely it is an $A\otimes A$-module, where the module $D(A)_s$ is given by
\[\lambda\cdot a=D(l_a)(\lambda)\]
where $l_a\colon A\longrightarrow A$ is left multiplication by $a\in A$, which is a right-module map. In this case the triple $(A,D(A),\alpha)$ is an antistructure, where $\alpha\in\hom_A(D(A)_t,D(A)_s)$ the value at $1$ of the composite
\[A\stackrel{\eta_A}{\longrightarrow}DD(A)\stackrel{\Phi_{D(A)_t}}{\longrightarrow }\hom_A(D(A)_t,D(A)_s)\]
and $\Phi_P\colon D(P)\longrightarrow \hom_A(P,D(A)_s)$ is defined by
\[\Phi_P(\lambda)(p)(a)=\lambda(p\cdot a)\]
Then $(D,\eta)$ is naturally isomorphic to the duality induced by $(A,D(A),\alpha)$.
\end{rem}

\begin{defn}\label{mapofantistruct} A \textbf{map of antistructures} $(B,K,\beta)\longrightarrow(A,L,\alpha)$ is a pair $(f,F)$ of a ring map $f\colon B\longrightarrow A$ and a $B\otimes B$-module map $F\colon K\longrightarrow L$ with $B\otimes B$-module structure on $L$ via $f$, such that
\begin{enumerate}[ i)]
\item $F\circ\beta=\alpha\circ F$
\item The map
\[\widetilde{F}\colon K_t\otimes_{B}A\longrightarrow L_t \ , \ \ \widetilde{F}(l\otimes a)=F(l)\cdot a\otimes 1\]
is an isomorphism.
\end{enumerate}
\end{defn}

\begin{ex}
A map of Wall antistructures $(B,v,\mu)\longrightarrow (A,w,\epsilon)$ is a ring map $f\colon B\longrightarrow A$ such that $f\circ\beta=\alpha\circ f$ and $f(\mu)=\epsilon$. It induces a map of antistructures $(f,f)\colon (B,B,\beta(-)\cdot\mu)\longrightarrow (A,A,\alpha(-)\cdot\epsilon)$
\end{ex}

\begin{prop}[\cite{IbLars}]
A map of antistructures $(f,F)\colon (B,K,\beta)\longrightarrow(A,L,\alpha)$ induces a map of duality structures $(\mathcal{P}_B,D_K)\longrightarrow (\mathcal{P}_A,D_L)$ with underlying functor $-\otimes_BA\colon\mathcal{P}_B\longrightarrow \mathcal{P}_A$ and natural isomorphism
\[\xi_P\colon \hom_{B}(P,K_s)\otimes_BA\longrightarrow \hom_{A}(P\otimes_BA,L_s)\]
given by
\[\xi_P(\lambda\otimes a)(p\otimes a')=F(\lambda(p))\cdot a\otimes a'\]
\end{prop}

\begin{rem}\label{functorialPA}
This construction gives an association
\[\mathcal{P}_{-}\colon \mbox{Antistructures}\longrightarrow \mbox{Categories with duality}\]
which is very unfortunately not strictly functorial, since given $f\colon B\longrightarrow A$ and $g\colon C\longrightarrow B$, the functor $(-\otimes_CB)\otimes_B A$ is canonically naturally isomorphic to $-\otimes_C A$, but not equal. This will generally not be a problem, but for extending $\KR$-theory to simplicial rings in §\ref{stablekrsec} we will need a functorial model for $\mathcal{P}_A$.
One can replace $\mathcal{P}_A$ by the equivalent category $\mathcal{P}_{A}'$ with objects pairs $(m,p\colon A^m\longrightarrow A^m)$ with $p$ a projection $(p^2=p)$. A morphism $(m,p)\longrightarrow (l,q)$ is a module map between the images $f\colon \im(p)\longrightarrow \im(q)$. The functor
\[\im\colon \mathcal{P}_{A}'\longrightarrow \mathcal{P}_{A}\] that sends $(m,p)$ to the image of $p$ is an equivalence of categories (the inverse is given by choosing complements so that the sum is free). Thus by \ref{dualitiesandequivalences}, given an antistructure $(A,L,\alpha)$ there is an essentially unique duality on $\mathcal{P}_{A}'$ so that the functor $\im$ is an equivalence of categories with duality. For clarity, we give an explicit construction of this duality $D$. Since we want $\im$ to be a morphism of categories with duality, we need the projection of $D(m,p)$ to have image $\hom_A(\im(p),L_s)$. Remember that $L_t$ is a finitely generated projective $A$-module. Fix a complement $H$ and an isomorphism $L_t\oplus H\cong A^l$, and denote $\ker(p)^\ast=\hom_A(\ker(p),L_s)$. This determines an isomorphism 
\[\phi_p\colon A^{lm}\stackrel{\cong}{\longrightarrow}\hom_A(\im(p),L_s)\oplus\ker(p)^\ast\oplus H^m\] defined by the composite
\[\begin{array}{ll}A^{lm}\cong (L_t\oplus H)^m\cong\hom_A(A^m,L_s)\oplus H^m\cong\\
\cong \hom_A(\im(p)\oplus\ker(p),L_s)\oplus H^m\cong \hom_A(\im(p),L_s)\oplus\ker(p)^\ast\oplus H^m\end{array}\]
Then define the dual of $(m,p)$ as the pair $(lm,p^\ast)$, where $p^{\ast}$ corresponds to the projection
\[\hom_A(\im(p),L_s)\oplus\ker(p)^\ast\oplus H^m\longrightarrow\hom_A(\im(p),L_s)\oplus\ker(p)^\ast\oplus H^m\]
onto the $\hom_A(\im(p),L_s)$-summand under the isomorphism $\phi_p$ above.
This gives a duality $D$ on $\mathcal{P}_{A}'$, and the natural isomorphism
\[\xi_{(m,p)}\colon \im D(m,p)\stackrel{\phi^{-1}_p}{\longrightarrow}\hom_A(\im(p),L_s)\]
makes $(\im,\xi)\colon\mathcal{P}_A\longrightarrow \mathcal{P}_{A}'$ into an equivalence of categories with duality.
\end{rem}


\subsection{Real simplicial sets and real nerves}\label{secdefscriptd}

The classifying space of a category with duality should carry a $\mathbb{Z}/2$-action, coming from some extra structure on the nerve. We review this structure in some more generality.

\begin{defn}\label{defrealob}
A \textbf{real object in a category $C$} is a simplicial object $X\colon\Delta^{op}\longrightarrow C$ together with involutions
$\omega_k\colon X[k]\longrightarrow X[k]$ compatible with the simplicial structure in the following way.
\[\begin{array}{ll}d_l\omega_k=\omega_{k-1}d_{k-l} \\
s_l\omega_k=\omega_{k+1}s_{k-l} & 0\leq l\leq k\\
\end{array}\]
A \textbf{morphism of real objects} of $C$ is a natural transformation of functors that commutes with the involutions in every degree.
\end{defn}
We will mostly be interested in real sets and real spaces, that is real objects in the categories of sets and spaces.

\begin{ex} Let $C$ be a category with strict duality $D\colon C^{op}\longrightarrow C$. The map $\omega_k\colon \mathcal{N}_k C\longrightarrow\mathcal{N}_k C$ defined by
\[\omega_k(c_0\stackrel{f_1}{\longrightarrow}c_1\stackrel{f_2}{\longrightarrow}\dots\stackrel{f_k}{\longrightarrow}c_k)=D(c_k)\stackrel{D(f_k)}{\longrightarrow}D(c_{k-1})\stackrel{D(f_{k-1})}{\longrightarrow}\dots\stackrel{D(f_1)}{\longrightarrow}D(c_0)\]
gives the nerve $\mathcal{N}_{\sbt} C$ the structure of a real set. This real set is called the real nerve of $C$. We will consider categories with non-strict duality below.
\end{ex}

The geometric realization of a real space admits a $\mathbb{Z}/2$-action induced by the maps $X[k]\times\Delta^k\longrightarrow X[k]\times \Delta^k$ defined by
\[(x,(t_0,\dots,t_k))\longmapsto (\omega_k(x),(t_k,\dots,t_0))\]
Using Segal's subdivision, we can recover the fixed points space of this action as the realization of a simplicial space as follows.
Let $sd_e\colon\Delta^{op}\longrightarrow\Delta^{op}$ be the functor defined on objects by $sd_e([n])=[2n+1]=[n]\coprod[n]$ and on morphisms by sending $[n]\stackrel{\sigma}{\longrightarrow}[m]$ to $\sigma\coprod\overline{\sigma}\colon[n]\coprod[n]\longrightarrow[m]\coprod[m]$ where
\[\overline{\sigma}(s)=m-\sigma(n-s)\]
Precomposition by $sd_e$ induces a functor
\[sd_e\colon \mbox{Real spaces}\longrightarrow \mbox{Simplicial $\mathbb{Z}/2$-spaces}\]
from the category of real spaces to the category of simplicial objects in $\mathbb{Z}/2$-spaces.
Explicitly $(sd_e X)[k]=X[2k+1]$ and faces and degeneracies $(sd_ed)_l$ and $(sd_es)_l$ are given in degree $k$ by
\[\begin{array}{ll}(sd_ed)_l=d_l\circ d_{2k+1-l}\\(sd_es)_l=s_l\circ s_{2k+1-l}\end{array}\]
The simplicial involution $sd_e\omega\colon sd_e X\longrightarrow sd_eX$, given in degree $k$ by $\omega_{2k+1}$, induces a $\mathbb{Z}/2$-action on the realization $|sd_e X|$. Since the action is simplicial, taking levelwise fixed points defines a simplicial space $(sd_e X)^{\mathbb{Z}/2}$ and it's realization is homeomorphic to $|sd_e X|^{\mathbb{Z}/2}$.
\begin{prop}[\cite{IbLars}]\label{subdivision}
The maps $(sd_e X)[k]\times \Delta^{k}\longrightarrow X[2k+1]\times\Delta^{2k+1}$ defined by
\[(x,(t_0,\dots,t_k))\longmapsto (x,(t_0,\dots,t_k,t_k,\dots,t_0))\]
induce a $\mathbb{Z}/2$-homeomorphism $|sd_eX|\cong |X|$. In particular on the fixed points $|X|^{\mathbb{Z}/2}\cong |(sd_e X)^{\mathbb{Z}/2}|$.
\end{prop}
For a deeper and more conceptual approach to real space and their realization one can consult \cite{IbLars}.

If $C$ has a non-strict duality, the nerve of $C$ does not carry a real structure. To fix this, there is a functor that associates to a category with duality $(C,D,\eta)$ a category with strict duality $(\mathcal{D}(C,D,\eta),D)=(\mathcal{D}C,D)$. The objects of $\mathcal{D}C$ are triples $(c,d,\phi)$ with
$\phi\colon d\stackrel{\cong}{\longrightarrow} D(c)$ an isomorphism in $C$. The morphisms from $(c,d,\phi)$ to $(c',d',\phi')$ are pairs of maps $(a\colon c\rightarrow c',b\colon d'\rightarrow d)$ in $C$ such that
\[\xymatrix{d\ar[r]^-{\phi}_-{\cong}& D(c)\\
d'\ar[u]^{b}\ar[r]^-{\cong}_-{\phi'} & D(c')\ar[u]_{D(a\colon c\longrightarrow c')}
}\]
commutes.
The functor $D\colon (\mathcal{D}C)^{op}\longrightarrow \mathcal{D}C$ that send an object $(c,d,\phi)$ to
\[(d,c,c\stackrel{\eta_{c}}{\longrightarrow} DD(c)\stackrel{D(\phi)}{\longrightarrow}D(d))\]
and a morphism $(a\colon c\longrightarrow c',b\colon d'\longrightarrow d)$ to
\[D(a,b)=(b,a)\]
is a strict duality on $\mathcal{D}C$.

A morphism of categories with duality $(F,\xi)\colon (C,D,\eta)\longrightarrow (C',D',\eta')$ induces a functor $\mathcal{D}F\colon \mathcal{D}C\longrightarrow  \mathcal{D}C'$ that commutes strictly with the dualities, by sending $(c,d,\phi)$ to
\[\mathcal{D}F(c,d,\phi)=(F(c),F(d),F(d)\stackrel{F(\phi)}{\longrightarrow}FD(c)\stackrel{\xi_c}{\longrightarrow}D'(F(c)))\]
and a morphism $(a,b)$ to $(F(a),F(b))$.
This procedure gives a functor
\[\mathcal{D}\colon \mbox{Categories with duality}\longrightarrow \mbox{Categories with strict duality}\]
Moreover we did not change $C$ too much, in the following sense.
\begin{prop}[\cite{IbLars}]\label{canddc}
The functor $V\colon C\longrightarrow \mathcal{D}C$ that sends an object $c$ to $(c,Dc,\id_{Dc})$ and a morphism $a\colon c\longrightarrow d$ to $(a,Da)$, together with the natural isomorphism
\[\xi_c\colon VDc=(Dc,D^2c,\id_{D^2c})\stackrel{(\id_{Dc},\eta_c)}{\longrightarrow} (Dc,c,\eta_c)=DVc\]
is an equivalence of categories with duality. Its inverse is the projection functor $U\colon \mathcal{D}C\longrightarrow C$ that sends an object $(c,d,\phi)$ to $c$ and a morphism $(a,b)$ to $a$.
\end{prop}

Given $(C,D,\eta)$ we want to study the homotopy type of $|\mathcal{D}C|^{\mathbb{Z}/2}$.
The category $\sym C$ has objects $(c,f)$ with $f\colon c\longrightarrow D(c)$ a map in $C$ such that
\[\xymatrix{c\ar[r]^-{\eta_c}\ar[dr]_{f} & DD(c)\ar[d]^-{D(f)}\\
& D(c)
}\]
commutes.
Morphisms from $(c,f)$ to $(d,g)$ are commutative squares
\[\xymatrix{c\ar[r]^-{f}\ar[d]_-{a} & D(c)\\
d\ar[r]_{g} & D(d)\ar[u]_-{D(a)}
}\]

\begin{prop}[\cite{IbLars}]\label{symcsymdc} If $(C,D)$ is a category with strict duality, the nerve of $\sym C$ is isomorphic as a simplicial set to the fixed points of the subdivision $sd_e(\mathcal{N}_{\sbt} C)^{\mathbb{Z}/2}$. If $(C,D,\eta)$ is any category with duality, there is a natural equivalence of categories $\sym C\simeq \sym\mathcal{D}C$. 

In particular there is a natural homotopy equivalence on classifying spaces
\[|\sym C|\simeq |\sym\mathcal{D}C|\cong|sd_e(\mathcal{N}_{\sbt}\mathcal{D}C)^{\mathbb{Z}/2}|\cong|\mathcal{D}C|^{\mathbb{Z}/2}\]
\end{prop}

\begin{rem}
In the case of the category with duality $(\mathcal{P}_A,D_L,\eta^\alpha)$ associated to an antistructure $(A,L,\alpha)$, an object of $\sym\mathcal{P}_A$ is a pair $(P,f)$ of a module $P$ and a module map
\[f\colon P\longrightarrow\hom_A(P,L_s)\]
such that $f=f^\ast\circ\eta^\alpha$.
Taking the adjoint, this is the same as a map
\[\widetilde{f}\colon P\otimes_A P\longrightarrow L_s\]
such that $\widetilde{f}(xa,y)=\widetilde{f}(x,y)\cdot a\otimes 1$, $\widetilde{f}(x,ya)=\widetilde{f}(x,y)\cdot 1\otimes a$ and
\[\alpha \widetilde{f}(x,y)=\widetilde{f}(y,x)\]
Therefore $\sym i\mathcal{P}_A$ is the category of non-degenerated $\alpha$-bilinear forms on $A$.
\end{rem}

If $(F,\xi)\colon C\longrightarrow C'$ is a morphism of categories with duality, there is an induced functor $\sym(F,\xi)\colon\sym C\longrightarrow \sym C'$ that sends $f\colon c\longrightarrow D(c)$ to
\[F(c)\stackrel{F(f)}{\longrightarrow}FD(c)\stackrel{\xi_c}{\longrightarrow}DF(c)\]
Under the equivalence of the proposition above, $|\sym(F,\xi)|$ corresponds to the restriction of the equivariant map $|\mathcal{D}F|$ to the fixed points space. This proves the following corollary from \cite{IbLars}.

\begin{cor}\label{fixedpointstrick}
Let $(F,\xi)\colon C\longrightarrow C'$ be a morphism of categories with duality. If both $|F|\colon |C|\longrightarrow |C'|$
and $|\sym(F,\xi)|\colon|\sym C|\longrightarrow |\sym C'|$ are weak equivalences, the equivariant map $|\mathcal{D}F|\colon |\mathcal{D}C|\longrightarrow |\mathcal{D}C'|$ is a $\mathbb{Z}/2$-equivalence.

In particular, this is the case when $(F,\xi)$ is an equivalence of categories with duality.
\end{cor}

\begin{rem}\label{strictifytwice} One could wonder what happens if we apply $\mathcal{D}$ to a category $C$ that already had a strict duality. In this case the functor $V\colon C\longrightarrow\mathcal{D}C$ defined above commutes strictly with the dualities, and therefore it induces an equivariant map on classifying spaces. This map is a non-equivariant equivalence since $V$ is an equivalence of categories. Moreover, the functor $\sym(V,\id)\colon \sym C\longrightarrow\sym\mathcal{D}C$ is the equivalence of categories proposition \ref{symcsymdc}, and therefore it induces an equivalence on fixed points $|C|^{\mathbb{Z}/2}\longrightarrow|\mathcal{D}C|^{\mathbb{Z}/2}$.
\end{rem}

\begin{rem}
The careful reader might have noticed that a natural transformation of morphism of categories with duality as defined in \ref{eqcatdual} has to be an isomorphism. This condition it required if we want $U\colon (F,\xi)\Rightarrow (G,\zeta)$ to induce a natural transformation $\sym(F,\xi)\Rightarrow \sym(G,\zeta)$. However we can bypass the $\sym$ construction and see what condition is needed on a natural transformation $U\colon F\Rightarrow G$ to induce a simplicial $\mathbb{Z}/2$-homotopy between the subdivisions of the nerves of $\mathcal{D}(F,\xi)$ and $\mathcal{D}(G,\zeta)$. One finds that it is enough to ask for the weaker condition that
\[F(Df)\circ \xi^{-1}_d\circ D'(U_d)\circ\zeta_d\circ U_{Dd}= F(Df)\]
and
\[\xi^{-1}_c\circ D'(U_c)\circ\zeta_c\circ U_{Dc}\circ F(Df)=F(Df)\]
for every morphism $f\colon c\longrightarrow d$ in $C$.
There is a similar condition for $U$ to induce an equivariant homotopy in $\THR$ (cf. \ref{Gnattransf}).
\end{rem}


\subsection{The $S^{1,1}_{\cdot}$-construction and $S^{2,1}_\cdot$-construction}\label{S21andS11constr}

When an exact category has a compatible duality, $S_{\sbt} C$ becomes a category with duality called $S^{1,1}_{\sbt} C$. This allows to define an involution on the $K$-theory spectrum. However, if one wants deloopings in both the directions of trivial and non-trivial $\mathbb{Z}/2$-representations a more involved construction than $S^{1,1}_{\sbt} C$ is needed. This is $S^{2,1}_{\sbt} C$. The theory for these two construction can be found in \cite{IbLars}. Here we review the definitions and a few main statements from \cite{IbLars}.

\begin{defn} An \textbf{exact category with duality} is an exact category $(C,\mathcal{E})$ together with the structure of a category with duality $(C,D,\eta)$ such that
\begin{enumerate}
\item $D\colon C^{op}\longrightarrow C$ is an additive functor, that is $D\colon C(c,d)\longrightarrow C(Dd,Dc)$
is a group homomorphism
\item for every sequence $c\rightarrow d\rightarrow e\in\mathcal{E}$, its image
\[D(c\rightarrow d\rightarrow e)=D(e)\rightarrow D(d)\rightarrow D(c)\]
also belongs to $\mathcal{E}$.
\end{enumerate}
We will often denote this data by $(C,\mathcal{E},D)$ or just $C$.
A \textbf{morphism of exact categories with duality} is a morphism of categories with duality whose underlying functor is exact.
\end{defn}

\begin{ex}
The category $\mathcal{P}_A$ with the exact structure of \ref{ringexact} and the duality induced by an antistructure $(A,L,\alpha)$ is an exact category with duality. This defines a (lax) functor
\[\mathcal{P}\colon \mbox{Antistructures} \longrightarrow \mbox{Exact categories with duality}\]
\end{ex}

The categories $[n]$ have a canonical strict duality $\omega_n\colon[n]^{op}\longrightarrow [n]$ defined by
\[\omega_n(i)=n-i\]
This induces a strict duality on the functor category $Cat([n],[m])$ by sending $\sigma\colon[n]\longrightarrow[m]$ to its conjugate
\[\overline{\sigma}=\omega_m\circ\sigma\circ\omega_n\]
If $(C,D,\eta)$ is a category with duality, this defines in its turn a duality on the category of functors $X\colon Cat([n],[m])\longrightarrow C$ by sending $X$ to the conjugate
\[\xymatrix{Cat([n],[m])\ar[d]_{(\overline{-})^{op}}\ar[r]^-{DX} & C\\
Cat([n],[m])^{op}\ar[r]_-{X^{op}}&C^{op}\ar[u]_{D}
}\]
and where $\eta$ is the natural isomorphism
\[\eta_\sigma\colon X_\sigma\longrightarrow D^2(X_\sigma)=D^{2}(X)_\sigma\]

\begin{prop}
If $(C,\mathcal{E},D,\eta)$ is an exact category with duality, the duality above restricts to a duality $(D\colon S_nC^{op}\longrightarrow S_nC,\eta)$ on $S_nC$. Moreover the face and degeneracy functors of $S_{\sbt} C$ satisfy
\[\begin{array}{ll}d_lD(X)=D(d_{n-l}X)\\s_lD(X)=D(s_{n-l}X)\end{array}\]
for all $0\leq l\leq n$.\\
If the duality on $C$ is strict, so is the one on $S_nC$.
\end{prop}

\begin{defn}
The simplicial category $S_{\sbt} C$ together with the dualities $D_n\colon S_nC^{op}\longrightarrow S_nC$ is called \textbf{the $S^{1,1}_{\sbt}$-construction} of $C$, denoted $S^{1,1}_{\sbt} C$.
\end{defn}

\begin{defn}\label{defrealcat}
A \textbf{real category} is a simplicial category $C_{\sbt}$ with structures of categories with duality $(C_n,D_n,\eta_n)$ in every simplicial degree $n$, satisfying
\[\begin{array}{ll}d_lD_n=D_{n-1}d_{n-l}\\s_lD_n=D_{n+1}s_{n-l}\end{array}\]
for all $0\leq l\leq n$.

A \textbf{morphism of real categories} is a simplicial functor $F_{\sbt}\colon C_{\sbt}\longrightarrow C_{\sbt}'$ together with the structure of morphisms of categories with duality \[(F_n,\xi_n)\colon C_n\longrightarrow C_n'\] for all $n\geq 0$.
\end{defn}

This terminology might be misleading, since a real category as defined here is not a real object in the category of categories. Nevertheless, the levelwise nerve of the simplicial category $\mathcal{D}(C_{\sbt})$ defines a real object in the category of real sets. This gives a functor from real categories to real real sets.
In particular, $[n]\mapsto |\mathcal{N}_{\sbt} \mathcal{D}(C_n)|$ with the involutions induced by the functors $D_n$ on the classifying space is a real space. This gives a functor from real categories to real spaces. Taking realization, this induces a functor
\[|-|\colon \mbox{Real categories}\longrightarrow \mbox{$\mathbb{Z}/2$-Spaces$_\ast$}\]
If $(F,\xi)\colon (C,D,\eta,\mathcal{E})\longrightarrow (C',D',\eta',\mathcal{E}')$ is a morphism of exact categories with duality, the functor $F_n\colon S^{1,1}_n C\longrightarrow S^{1,1}_n C'$, defined as earlier by composing a diagram with $F$, admits a canonical structure of morphism of categories with duality, and it preserves the simplicial structure. This construction gives a functor
\[|\mathcal{D}S_{\sbt}^{1,1}(-)|\colon \mbox{Exact categories with duality}\longrightarrow \mbox{$\mathbb{Z}/2$-Spaces$_\ast$}\]

\begin{rem}\label{DfirstorlaterS11}
Let $(C,D,\eta,\mathcal{E})$ be an exact category with duality. The category $\mathcal{D}C$ inherits a structure of exact category with strict duality by declaring a sequence
\[\xymatrix{d \ar[d]^{\cong}& d'\ar[l]_{g}\ar[d]^{\cong} & d''\ar[l]_{g'}\ar[d]^{\cong}\\
D(c) & D(c')\ar[l]_{D(f)}& D(c'')\ar[l]_{D(f')}
}\]
to be exact if $c\stackrel{f}{\longrightarrow}c'\stackrel{f'}{\longrightarrow}c''$ is. This forces $d''\stackrel{g'}{\longrightarrow}d'\stackrel{g}{\longrightarrow}d$ to be exact as well, since exactness is preserved by the duality and under isomorphism of sequences. The equivalence of categories $C\longrightarrow \mathcal{D}C$ of \ref{canddc} is clearly an exact functor. This extends $\mathcal{D}$ to a functor
\[\mathcal{D}\colon \mbox{Exact categories with duality}\longrightarrow \mbox{Exact categories with strict duality}\]
Moreover, there is a canonical natural isomorphism of real categories with strict duality $S^{1,1}_{\sbt}\mathcal{D}C\cong \mathcal{D}S^{1,1}_{\sbt} C$. This is defined by a simple reindexing, as follows. An object of $\mathcal{D}S^{1,1}_{\sbt} C$ is a triple $(X,Y,\phi)$ with $\phi\colon Y\longrightarrow D(X)$ a natural isomorphism in $S^{1,1}_{\sbt} C$, given by a family of isomorphisms $\phi_\sigma\colon  Y_\sigma\longrightarrow D(X_{\overline{\sigma}})$. This is mapped by the isomorphism above to the diagram of $\mathcal{D}C$ defined at a $\sigma$ by the triple $(X_\sigma,Y_{\overline{\sigma}},\phi_{\overline{\sigma}})$.
\end{rem}

Suppose that the duality on $C$ is strict.
The $1$-simplicies are $S^{1,1}_1C=C$ with the same duality structure. Since the action on the realization of a real space flips the simplex coordinates, the maps $\mathcal{N}_kS_1C\times \Delta^1\longrightarrow |\mathcal{N}_kS_{\sbt} C|$ are equivariant for $\Delta^1=I$ on which $\mathbb{Z}/2$ reflects around $1/2$. We denote $S^{1,1}$ the circle with involution induced by complex conjugation, and for $Y$ a $\mathbb{Z}/2$-space, we denote $\Omega^{1,1}Y=\map_\ast(S^{1,1},Y)$ the based loop space with conjugation action. The adjoints of the maps above give an equivariant map
\[|C|\longrightarrow \Omega^{1,1}|S^{1,1}_{\sbt} C|\]
The duality structure on $S^{1,1}_nC$ is compatible with the exact category structure defined pointwise for $S_nC$, making $S^{1,1}_nC$ into an exact category with strict duality. This allows to iterate the $S^{1,1}_{\sbt}$-construction.
We would like to collect these iterated $S^{1,1}_{\sbt}$-constructions and the maps
\[|iC|\longrightarrow\Omega^{1,1}|iS^{1,1}_{\sbt}C|\]
above into some kind of $\mathbb{Z}/2$-spectrum (see \ref{realspec} below). In order to do this in a convenient category of $\mathbb{Z}/2$-spectra, we want a construction that deloops also with respect to the trivial action on the circle. For this, following \cite{IbLars}, we define another functor from exact categories with dualities to real categories, denoted $S^{2,1}_{\sbt}$. This should be somehow thought of as a $S^{1,1}_{\sbt}$-construction and a $S_{\sbt}$-construction combined.

\begin{defn}
Let $(C,\mathcal{E})$ be an exact category. A sequence \[a\stackrel{i}{\longrightarrow}b\stackrel{m}{\longrightarrow}c\stackrel{p}{\longrightarrow}d\] in $C$ is called \textbf{($4$-term) exact} if $i$ is an admissible monomorphism, $p$ is an admissible epimorphism and $m$ factors as
\[\xymatrix{b\ar[dr]_{m}\ar[r]^{f}&e\ar[d]^{g}\\
& c}\]
where $f$ is a cokernel for $i$ and $g$ is a kernel for $p$.
\end{defn}

\begin{defn} Let $(C,D,\eta,\mathcal{E})$ be an exact category with duality. The category $S^{2,1}_nC$ has objects the functors $X\colon Cat([2],[n])\longrightarrow C$ satisfying the two following conditions
\begin{enumerate}
\item $X(\theta)=0$ unless $\theta\colon[2]\longrightarrow [n]$ is injective.
\item For each $\psi\colon [3]\longrightarrow [n]$, the sequence
\[X(d_3\psi)\longrightarrow X(d_2\psi)\longrightarrow X(d_1\psi)\longrightarrow X(d_0\psi)\]
is exact.
\end{enumerate}
Morphisms of $S^{2,1}_nC$ are natural transformations of functors. Conjugation of functors by the dualities on $Cat([2],[n])$ and $C$ restricts to a duality $(D_n\colon S^{2,1}_nC^{op}\longrightarrow S^{2,1}_nC,\eta)$ on $S^{2,1}_nC$.

The subcategory of $S^{2,1}_nC$ with all objects and natural isomorphisms as morphisms is denoted $iS^{2,1}_nC$.
\end{defn}

\begin{ex}
An object of $S^{2,1}_5C$ is a diagram of the form
\begin{center}
\includegraphics[scale=0.7]{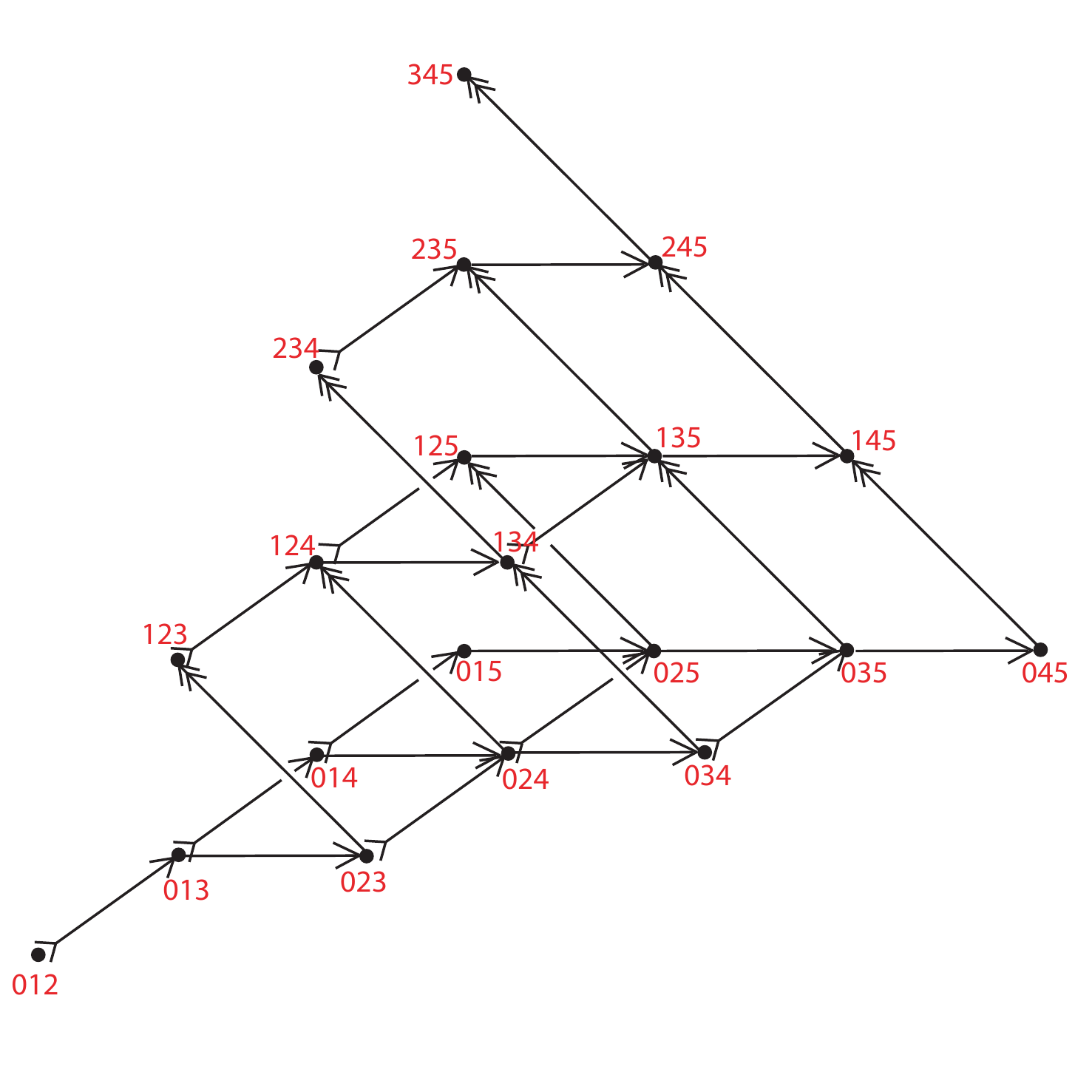}
\end{center}
where the tailed maps are admissible monomorphisms and the ones with double tip admissible epimorphisms. Notice that $S^{1,1}_4C$ includes in $S^{2,1}_5C$ both as the top and the bottom face of the tetrahedron. Quite in general $S^{1,1}_{k}C$ embeds in $S^{2,1}_{k+1}C$ in two ways. These two inclusion are swapped by the action, and therefore are not equivariant. For an equivariant map from $S^{1,1}_{\sbt} C$ to $S^{2,1}_{\sbt} C$ see \ref{picIk}.
\end{ex}
Just as for $S_{\sbt} C$, the maps $\delta^l\colon [n-1]\longrightarrow [n]$ and $\sigma^l\colon [n+1]\longrightarrow [n]$ make $S^{2,1}_{\sbt} C$ into a simplicial category.

\begin{prop}
The face and degeneracy functors of $S^{2,1}_{\sbt} C$ satisfy
\[\begin{array}{ll}d_lD(X)=D(d_{n-l}X)\\s_lD(X)=D(s_{n-l}X)\end{array}\]
for all $0\leq l\leq n$,
defining a structure of real category $(S^{2,1}_{\sbt} C,D,\eta)$.
\end{prop}

\begin{defn}
The real category $S^{2,1}_{\sbt} C$ is the \textbf{$S^{2,1}_{\sbt}$-construction of $C$}.
\end{defn}

As for $S^{1,1}_{\sbt}$, composing diagrams by a morphism of exact categories with duality $C\longrightarrow C'$ gives a morphism of real categories $S^{2,1}_{\sbt} C\longrightarrow S^{2,1}_{\sbt} C'$. Thus the $S^{2,1}_{\sbt}$-construction gives a functor
\[S^{2,1}_{\sbt}\colon \mbox{Exact categories with duality}\longrightarrow \mbox{Real categories}\]

\begin{rem}\label{DfirstorlaterS21}
Exactly like for \ref{DfirstorlaterS11} above, there is a canonical natural isomorphism of categories with strict duality $\mathcal{D}S^{2,1}_{\sbt} C\cong S^{2,1}_{\sbt}\mathcal{D}C$ defined by reindexing the data.
\end{rem}

\begin{defn}
The \textbf{real $K$-theory} of an exact category with duality $(C,D,\eta,\mathcal{E})$ is the $\mathbb{Z}/2$-space
\[\KR(C)=\Omega^{2,1}|iS^{2,1}_{\sbt} \mathcal{D}C|\]
The \textbf{real $K$-theory} of a ring with antistructure $(A,L,\alpha)$ is the $\mathbb{Z}/2$-space
\[\KR(A)=\KR(\mathcal{P}_A)\]
\end{defn}

Let us see how the $S^{2,1}_{\sbt}$-construction deloops. Suppose now that the duality on $C$ is strict. We denote $S^{2,1}$ the one point compactification of $\mathbb{C}$ with $\mathbb{Z}/2$-action induced from complex conjugation. For $Y$ a $\mathbb{Z}/2$-space, we denote $\Omega^{2,1}Y$ the based loop space $\map_\ast(S^{2,1},Y)$ with conjugation action. The projection maps $\mathcal{N}_kS^{2,1}_2C\times\Delta^2\longrightarrow|\mathcal{N}_kS^{2,1}_{\sbt} C|$ are equivariant for $\Delta^2$ equipped with the action that reverses the order of the coordinates. Since $S^{2,1}_0C=S^{2,1}_1C=0$ and $S^{2,1}_2C=C$, they induce a real map
\[\mathcal{N}_kC\longrightarrow \Omega^{2,1}|\mathcal{N}_kS^{2,1}_{\sbt} C|\]
that realized gives an equivariant map $|C|\longrightarrow \Omega^{2,1}|S^{2,1}_{\sbt} C|$ and $|iC|\longrightarrow \Omega^{2,1}|iS^{2,1}_{\sbt} C|$.
As for the other constructions, $S^{2,1}_n C$ is an exact category with duality by defining exact sequences as being sequences of diagrams that are pointwise exact. Thus the $S^{2,1}_{\sbt}$-construction can be iterated.

\begin{theorem}[\cite{IbLars}]\label{sumdiags}
For $(C,D,\mathcal{E})$ an exact category with strict duality, the map
\[|iS^{2,1}_{\sbt} C|\longrightarrow \Omega^{2,1}|iS^{2,1}_{\sbt} S^{2,1}_{\sbt} C|\]
is a $\mathbb{Z}/2$-homotopy equivalence. Moreover the map
\[|iC|\longrightarrow \Omega^{2,1}|iS^{2,1}_{\sbt} C|\]
is an equivariant group completion, i.e. it is a homology equivalence after inverting $\pi_0|iC|$, and the restriction
\[|\sym iC|\simeq|iC|^{\mathbb{Z}/2}\longrightarrow (\Omega^{2,1}|iS^{2,1}_{\sbt} C|)^{\mathbb{Z}/2}\]
is a homology equivalence after inverting $\pi_0|iC|^{\mathbb{Z}/2}$.
\end{theorem}

We can assemble the iterated $S^{2,1}_{\sbt}$-constructions into a convenient notion of spectrum with $\mathbb{Z}/2$-action. We recall this definition from \cite{IbLars} (see also \cite{shima}). We denote $S^{2n,n}$ the $n$-fold smash product of the sphere $S^{2,1}$ with itself with diagonal action, and $\Omega^{2n,n}=\map_\ast (S^{2n,n},-)$ the pointed mapping space with conjugation action.

\begin{defn}\label{realspec}
A \textbf{real spectrum} $X$ is the data of a pointed $\mathbb{Z}/2$-space $X_m$ for every integer $m\geq 0$, and maps of pointed $\mathbb{Z}/2$-spaces
\[\sigma_{m,n}\colon X_m\wedge S^{2n,n}\longrightarrow X_{m+n}\]
such that $\sigma_{m,0}\colon X_m\wedge S^{0}\longrightarrow X_{m}$ is the canonical isomorphism, and
\[\xymatrix{(X_m\wedge S^{2n,n})\wedge S^{2p,p}\ar[d]^\cong\ar[rr]^-{\sigma_{m,n}\wedge\id}&&X_{m+n}\wedge S^{2p,p}\ar[d]^{\sigma_{m+n,p}}\\
X_m\wedge S^{2(n+p),n+p}\ar[rr]_-{\sigma_{m,n+p}}&&X_{m+n+p}
}\]
commutes.\newline
If the adjoint of the structure maps
\[\widetilde{\sigma}_{m,n}\colon X_m\longrightarrow \Omega^{2n,n}X_{m+n}\]
are weak $\mathbb{Z}/2$-equivalences, we say that $X$ is a \textbf{real $\Omega$-spectrum}.\newline
The \textbf{infinite real loop space of $X$} is the $\mathbb{Z}/2$-space
\[\Omega^{2\infty,\infty}X=\colim_n\Omega^{2n,n}X_{n}\]
where the colimit is taken over the adjoint of the structure maps.\newline
A $\mathbb{Z}/2$-space $Z$ is said to be \textbf{an infinite real loop space} if there is a real $\Omega$-spectrum $X$ with $X_0=Z$.\newline
A \textbf{map of real spectra $f\colon X\longrightarrow Y$} is a family of $\mathbb{Z}/2$-maps $f_m\colon X_m\longrightarrow Y_m$ such that
$f_{m+n}\circ\sigma_{m,n}=\sigma_{m,n}\circ(f_m\wedge\id)$. We say that $f$ is a \textbf{levelwise $\mathbb{Z}/2$-equivalence} if every $f_m$ is a weak $\mathbb{Z}/2$-equivalence of $\mathbb{Z}/2$-spaces. We say that $f$ is a \textbf{stable $\mathbb{Z}/2$-equivalence} if its induced map
\[\Omega^{2\infty,\infty}f\colon\Omega^{2\infty,\infty}X\longrightarrow \Omega^{2\infty,\infty}Y\]
is a weak $\mathbb{Z}/2$-equivalence of $\mathbb{Z}/2$-spaces.
\end{defn}

For an exact category with strict duality $C$, the family $|i(S^{2,1}_{\sbt})^{(m)}C|$ together with the structure maps 
\[|i(S^{2,1}_{\sbt})^{(m)}C|\wedge S^{2n,n}\longrightarrow |i(S^{2,1}_{\sbt})^{(m+n)}C|\]
defined by iterating the maps $|i(S^{2,1}_{\sbt})^{(m)}C|\wedge S^{2,1}\longrightarrow |i(S^{2,1}_{\sbt})^{(m+1)}C|$ defined above, is a real spectrum.

\begin{defn}
The \textbf{real algebraic $K$-theory} of an exact category with strict duality $(C,D,\mathcal{E})$ is the real spectrum
\[\uuline{\KR}(C)=\{|i(S^{2,1}_{\sbt})^{(m)}\mathcal{D}C|\}_{m\in\mathbb{N}}\]
\end{defn}
By \ref{sumdiags}, if $C$ is split-exact the spectrum $\uuline{\KR}(C)$ is "positively fibrant", i.e. the adjoint of the structure maps are $\mathbb{Z}/2$-equivalence above degree $1$, and the first map $|iC|\longrightarrow |iS^{2,1}_{\sbt}C|$ is an equivariant group completion. In this case the $K$-theory space $K(C)$ is a real infinite loop space. All together, we defined a functor
\[\uuline{\KR}\colon \mbox{Exact categories with duality}\longrightarrow \mbox{Real spectra}\]
that precomposed with $\mathcal{D}\mathcal{P}_{(-)}$ gives a (lax) functor
\[\KR\colon \mbox{Antistructures}\longrightarrow \mbox{Positively fibrant real spectra}\]

It is possible to define the more refined notion of a "real symmetric spectrum". However, in order to obtain such a structure on $\KR(C)$, one needs to replace $(S^{2,1}_{\sbt})^{(m)}$ with a homeomorphic construction, denoted $S_{\sbt}^{2m,m}$. This is done in great details in \cite{IbLars}.


\newpage

\section{Real $K$-theory of split square zero extensions of antistructures}\label{splitKR}

\subsection{Introduction}\label{secintrokranti}

Let $M$ be a bimodule over a ring $A$. Define $A\ltimes M$ to be the ring with underlying abelian group $A\oplus M$, and multiplication
\[(a,m)\cdot(a',m')=(a\cdot a',a\cdot m'+m\cdot a')\]
If $f\colon B\longrightarrow A$ is a split-surjective ring map with square zero kernel, there is an isomorphism over $A$ between $B$ and $A\ltimes \ker f$. Using this identification, one can describe the homotopy fiber of $K(B)\longrightarrow K(A)$ as the homotopy fiber of a map $K(A\ltimes M)\longrightarrow K(A)$ induced by the projection $A\ltimes M\longrightarrow A$. In \cite{stableDM}, the authors define a model for this homotopy fiber that can be easily mapped to $\THH(A;M)$. Here we generalize this for $\KR$ when $A$ and $B$ are equipped with antistructures and both $f$ and the section are maps of antistructures.

In \cite[{§4}]{stableDM}, the authors define a simplicial category $S_{\sbt}(A;M)$ with objects $S_{\sbt}\mathcal{P}_A$ and with the endomorphisms of $X$ given by the abelian group of natural transformations of diagrams of right $A$-modules $\hom_A(X,X\otimes_A M)$; for short
\[S_{\sbt}(A;M)=\coprod_{X\in S_{\sbt}\mathcal{P}_A}\hom_A(X,X\otimes_A M)\]
Composition is given by objectwise addition of module maps.
Then they build a homotopy commutative diagram
\[\xymatrix{S_{\sbt}(A;M)\ar[r]^{\Psi}_{\simeq}\ar[d]& iS_{\sbt}\mathcal{P}_{A\ltimes M}\ar[d]^{-\otimes_{A\ltimes M}A}\\
ObS_{\sbt}\mathcal{P}_A\ar[r]_\simeq& iS_{\sbt}\mathcal{P}_{A}
}\]
where the horizontal maps are weak equivalences.
The bottom map is the inclusion of the discrete category of objects $ObS_{\sbt}\mathcal{P}_A$. The left vertical map is the projection that remembers only the object of $S_{\sbt}\mathcal{P}_A$. It splits by the section sending the identity of an object $X\in ObS_{\sbt}\mathcal{P}_A$ to the zero natural transformation of $\hom_A(X,X\otimes_A M)$. In simplicial degree $1$, $\Psi$ sends an $A$-module $P$ to the abelian group $F(P)=P\oplus(P\otimes_AM)$ with $A\ltimes M$-module structure
\[(p,p'\otimes m')\cdot(a,m)=(p\cdot a,p'\otimes (m'\cdot a)+p\otimes m)\]
This module is canonically isomorphic to $P\otimes_{A}(A\ltimes M)$, via the map that sends $(p,p'\otimes m)$ to $p\otimes (1,0)+p'\otimes (0,m)$.
A module map $f\colon P\longrightarrow P\otimes_A M$ is sent to
\[\Psi_1(f)=\left(\begin{array}{cc}1&0\\f&1\end{array}\right)\colon(p,p'\otimes m')\longmapsto(p,p'\otimes m+f(p))\]
The functor $\Psi$ on higher simplicies is defined pointwise from $\Psi_1$.
\begin{rem} It might look surprising that we just have identities in the diagonal for $\Psi(f)$. This is similar to the fact that $ObS_{\sbt}C\longrightarrow iS_{\sbt}C$ induces a weak equivalence on realizations.
\end{rem}

The diagram above shows that the homotopy fiber $\widetilde{K}(A\ltimes M)$ of the map $K(A\ltimes M)\longrightarrow K(A)$ is weakly equivalent to the loop space of the homotopy fiber of the realization of $S_{\sbt}(A;M)\longrightarrow ObS_{\sbt}\mathcal{P}_A$. Then the authors compute the homotopy fiber of this last map, showing that it is weakly equivalent to the realization of the bisimplicial set $\bigvee_{X\in S_{\sbt}\mathcal{P}_A}\hom_A(X,X\otimes_A M(S_{\sbt}^1))$
where $M(S_{\sbt}^1)$ is the Bar construction on $M$.
This gives a weak equivalence
\[\widetilde{K}(A\ltimes M)\simeq\Omega|\bigvee_{X\in S_{\sbt}\mathcal{P}_A}\hom_A(X,X\otimes_A M(S_{\sbt}^1))|\]

Here, we generalize this model to the equivariant situation. An "$M$-twisting" $J$ defines both a lifting of an antistructure $(A,L,\alpha)$ to $A\ltimes M$ and natural duality maps
\[\widehat{J}\colon \hom_A(P,Q\otimes M)\longrightarrow \hom_A(D_LQ,(D_LP)\otimes_A M)\]
for all pairs of modules $P,Q\in\mathcal{P}_A$.
From this data, we give a model for $iS^{2,1}_{\sbt}\mathcal{D}\mathcal{P}_{A\ltimes M}$ by defining a simplicial category
\[S^{2,1}_{\sbt}(A;M)=\coprod_{\varphi\in ObS^{2,1}_{\sbt}\mathcal{D}\mathcal{P}_A}\hom_A(\varphi,\varphi\otimes_A M)\]
It has objects $ObS^{2,1}_{\sbt}\mathcal{D}\mathcal{P}_{A}$ and only endomorphisms. The endomorphisms of an object $\varphi=(X,Y,\phi\colon Y\longrightarrow D_L(X))$ is the abelian group $\hom_A(\varphi,\varphi\otimes_A M)$ of pairs of morphisms of diagrams $f\colon X\longrightarrow X\otimes_A M$ and $g\colon Y\longrightarrow  Y\otimes_A M$ such that
\[\xymatrix{Y_\theta\ar[d]_{\phi_\theta}\ar[r]^-{g_\theta}&Y_\theta\otimes_AM\ar[d]^{\phi_\theta\otimes M}\\
D_L(X_\theta)\ar[r]_-{\widehat{J}(f_\theta)}&D_L(X_\theta)\otimes_A M 
}\]
commutes for all $\theta\in Cat([2],[k])$. Composition is pointwise addition. This simplicial category has a duality that sends a pair $(\{f_\theta\},\{g_\theta\})$ to the pair $(\{g_{\overline{\theta}}\},\{f_{\overline{\theta}}\})$.
Theorem \ref{mainKR} is the main result of §\ref{splitKR}, where we define an equivariant version of the equivalence $\Psi$, fitting into a diagram analogous to the above diagram from \cite{stableDM}:
\begin{equation}\label{diagmainKR}
\xymatrix{S^{2,1}_{\sbt}(A;M)\ar[r]^-{\Psi}_-{\simeq}\ar[d]& iS^{2,1}_{\sbt}\mathcal{D}\mathcal{P}_{A\ltimes M}\ar[d]^{\mathcal{D}(-\otimes_{A\ltimes M}A)}\\
ObS^{2,1}_{\sbt}\mathcal{D}\mathcal{P}_{A}\ar[r]_-{\simeq}& iS^{2,1}_{\sbt}\mathcal{D}\mathcal{P}_{A}
}
\end{equation}
that commutes up to equivariant homotopy. The horizontal maps are both $\mathbb{Z}/2$-equivalences. The two vertical maps are induced from the projection $A\ltimes M\longrightarrow A$, and the bottom map is the inclusion of the objects as a discrete category.

By this theorem, the homotopy fiber $\widetilde{\KR}(A\ltimes M)$ of the projection map $\KR(A\ltimes M)\longrightarrow \KR(A)$ is $\mathbb{Z}/2$-equivalent to the equivariant loop space $\Omega^{2,1}=\map_\ast(\mathbb{C}^+,-)$ of the homotopy fiber of the left vertical map of diagram \eqref{diagmainKR}. Similarly to the non-equivariant case, we describe this fiber equivariantly as
\[\widetilde{\KR}(A\ltimes M)\simeq\Omega^{2,1}|\bigvee_{\varphi\in ObS^{2,1}_{\sbt}\mathcal{D}\mathcal{P}_A}\hom_A(\varphi,\varphi\otimes_A M(S_{\sbt}^{1,1}))|\]
Here $M(S_{\sbt}^{1,1})$ is the Bar construction on $M$ with levelwise involution given by sending $(m_1,\dots,m_q)$ to $(m_q,\dots,m_1)$.


\subsection{Split antistructures and $M$-twistings}\label{squarezeroext}

Given an antistructure $(A,L,\alpha)$ and an $A$-bimodule $M$ one may ask for antistructures $(A\ltimes M,L^M,\alpha^M)$ such that the projection $p\colon A\ltimes M\longrightarrow M$ is the underlying ring map of a map of antistructures \[(p,F)\colon (A\ltimes M,L^M,\alpha^M)\longrightarrow(A,L,\alpha)\] For general $M$ there might not be any such antistructure. However, we show in appendix \ref{classanti} that lifts of $(A,L,\alpha)$ to $A\ltimes M$ exist if and only if there is an $M$-twisting of $(A,L,\alpha)$ as defined below, and that the set of $M$-twistings classify the possible antistructures on $A\ltimes M$ that lift the original antistructure on $A$ and for which the zero section $A\longrightarrow A\ltimes M$ defines a map of antistructures.

Recall that $L$ is an $A\otimes A$-module and that $L_t$ and $L_s$ are the $A$-modules $(L,A\otimes 1)$ and $(L,1\otimes A)$. We consider $L_t\otimes_A M$ as an $A\otimes A$-module by
\[(l\otimes m)\cdot (a\otimes a')=(l\cdot 1\otimes a')\otimes m\cdot a\]

\begin{defn}\label{defmaptwist}
An \textbf{$M$-twisting of an antistructure $(A,L,\alpha)$} is an additive involution $J\colon L_t\otimes_A M\longrightarrow L_t\otimes_A M$ such that
\[J((l\otimes m)\cdot (a\otimes a'))=J(l\otimes m)\cdot (a'\otimes a)\]
If $M$ and $N$ are two $A$-bimodules, and $J$ and $J'$ are respectively an $M$-twisting an $N$-twisting of $(A,L,\alpha)$, a \textbf{map of twistings }from $J$ to $J'$ is a map of bimodules $f\colon M\longrightarrow N$ such that
\[J'\circ (L_t\otimes f)=(L_t\otimes f)\circ J\]
\end{defn}

\begin{rem}
One should think of an $M$-twisting as a bimodule over $(A,L,\alpha)$, just as a split square zero extension of rings is a bimodule over the base ring. For a category $C$ with finite limits, Quillen defines in \cite{quillenmod} the category of bimodules over an object $X\in C$ as the category of abelian group objects in the over category $C/X$. We show in \ref{quillenmodstory} that the category of bimodules over a antistructure $(A,L,\alpha)$ is equivalent to the category of twistings defined in \ref{defmaptwist} above.
\end{rem}

Given an $M$-twisting $J$, we construct an antistructure $(A\ltimes M,L^J,\alpha^J)$ lifting $(A,L,\alpha)$. First remember  for every $A$-module $P$ that $F(P)=P\oplus(P\otimes_AM)$ is a $A\ltimes M$-module via
\[(p,p'\otimes m)\cdot (a,n)=(p\cdot a,p'\otimes ma+p\otimes n)\]
For $P=L_t$ this defines an $A$-module $F(L_t)$. We use the $M$-twisting to define to define the second $A\ltimes M$-module $F(L_t)_s$ as follows
\[(l,l'\otimes m)\cdot 1\otimes(a,n)=(l\cdot 1\otimes a,l'\otimes m\cdot 1\otimes a+J(\alpha(l)\otimes n))\]
The two $A\ltimes M$-module structures on $F(L_t)$ commute to define a $(A\ltimes M)\otimes (A\ltimes M)$-module structure on $F(L_t)$. We denote this $(A\ltimes M)\otimes (A\ltimes M)$-module by $L^J$.

There is an involution $\alpha^J$ on $L^J=F(L_t)$ given by
\[\alpha^J(l,l'\otimes m)=(\alpha(l),J(l'\otimes m))\]

We show below in Proposition \ref{mtwistingantistructure} that $(A\ltimes M,L^J,\alpha^J)$ is indeed an antistructure. The obvious maps
\[p\colon(A\ltimes M,L^J,\alpha^J)\longrightarrow (A,L,\alpha) \ , \ \ s\colon (A,L,\alpha)\longrightarrow (A\ltimes M,L^J,\alpha^J)\]
are maps of antistructures ($p$ is projection on the first summand and $s$ the zero section, both on the ring level and on the modules). We have $p\circ s=\id$, and we say that $(A\ltimes M,L^J,\alpha^J)$ is a split antistructure over $(A,L,\alpha)$.

\begin{ex}\label{antiinvolutioncase}
Suppose that $\alpha\colon A^{op}\longrightarrow A$ and $\overline{\alpha}\colon A\ltimes M^{op}\longrightarrow A\ltimes M$ are anti-involutions and the projection $p\colon A\ltimes M\longrightarrow M$ is a map of anti-involutions. It is not hard to see that the map $\overline{\alpha}$ must be of the form
\[\overline{\alpha}=\left(\begin{array}{cc}\alpha&0\\
r&j\end{array}\right)\]
for additive maps $j\colon M\longrightarrow M$ and $r\colon A\longrightarrow M$ satisfying
\begin{enumerate}
\item $j^2=\id$
\item $j(a\cdot m)=j(m)\cdot\alpha(a)$
\item $j(m\cdot a)=\alpha(a)\cdot j(m)$
\item $r(a\cdot b)=\alpha(b)\cdot r(a)+r(b)\cdot\alpha(a)$
\item $r\circ\alpha=-j\circ r$
\end{enumerate}
Moreover the zero section $s\colon A\longrightarrow A\ltimes M$ is equivariant if and only if $r=0$. In this case,
the corresponding antistructure $(A\ltimes M,A\ltimes M,\overline{\alpha})$ is associated to the $M$-twisting $J$ given by $j$ under the canonical isomorphism $L_t\otimes M=A\otimes M\cong M$.

The case where $r\neq 0$ is not covered by this set of examples, but it is discussed in \ref{appendixdualities}.
\end{ex}

The following lemma is a central ingredient not only for proving that $(A\ltimes M,L^{J},\alpha^{J})$ is indeed an antistructure, but also for the next section.

\begin{lemma}\label{lemmaisoequivariant}
Let $J$ be an $M$-twisting of $(A,L,\alpha)$. Then for any $P\in\mathcal{P}_A$ the map \[\xi_P\colon F(\hom_A(P,L_s))\longrightarrow \hom_{A\ltimes M}(F(P),L^{J}_s)\]
sending $(\lambda,\delta\otimes n)$ to the module map
\[\xi(\lambda,\delta\otimes n)(p,p'\otimes m)=(\lambda(p),\delta(p)\otimes n+J(\alpha(\lambda(p'))\otimes m))\]
is an isomorphism of $A\ltimes M$-modules.
\end{lemma}

\begin{proof}
We define an inverse for $\xi_P$. For $\chi\in \hom(F(P),L^{J}_s)$, denote $\chi_1$ and $\chi_2$ the compositions with the projections onto $L$ and $L_t\otimes_A M$, respectively.
Since $\chi$ is a module map
\[\begin{array}{lll}
\chi_1(p\cdot a, p'\!\otimes\!m\cdot a+p\!\otimes\!n)\!\!&=\chi_1(p,p'\!\otimes\!m)\cdot 1\!\otimes\!a\\
\chi_2(p\cdot a, p'\!\otimes\!m\cdot a+p\!\otimes\!n)\!\!&=\chi_2(p,p'\!\otimes\!m)\cdot 1\!\otimes\! a+J(\alpha(\chi_1(p,p'\!\otimes\!m))\!\otimes\!n)
\end{array}\]
Setting $a=0$ we get that $\chi_1(0,p\otimes n)=0$ and that
\[\chi_2(0,p\otimes n)=J(\alpha(\chi_1(p,0))\otimes n)\]
is determined by $\chi_1$. Setting $m,n$ and $p$ to zero, one sees that $\chi_1(-,0)$ and $\chi_2(-,0)$ are elements of $\hom_A(P,L_s)$ and $\hom_A(P,(L_t\otimes_A M)_s)$, respectively. For $P\in \mathcal{P}_A$, let 
\[\mu\colon \hom_A(P,L_s)\otimes_A M\longrightarrow\hom_A(P,(L_t\otimes_A M)_s)\]
be the canonical map which sends $\lambda\otimes m$ to $p\mapsto \lambda(p)\otimes m$. It is an isomorphism since $P$ is finitely generated projective.
The inverse $\xi_{P}^{-1}$ sends $\chi$ to
\[(\chi_1(-,0),\mu^{-1}(\chi_2(-,0)))\]
\end{proof}

\begin{prop}\label{mtwistingantistructure}
The triple $(A\ltimes M,L^{J},\alpha^{J})$ is a split antistructure lifting $(A,L,\alpha)$. A map of twistings $f\colon M\longrightarrow N$ from $J$ to $J'$ induces a map of antistructures
\[(A\ltimes f,L_t\oplus L_t\otimes_A f)\colon(A\ltimes M,L^{J},\alpha^{J})\longrightarrow (A\ltimes N,L^{J'},\alpha^{J'})\]
\end{prop}

\begin{proof}
We must show that the map $f_{\alpha^{J}}\colon A\ltimes M\longrightarrow\hom_{A\ltimes M}(L^{J}_t,L^{J}_s)$ sending $(a,m)$ to $\alpha^J(-)\cdot (a,m)\otimes 1$ is an isomorphism.
Setting $P=L_t$ in the lemma \ref{lemmaisoequivariant} above,  shows that $f_{\alpha^{J}}$ is an isomorphism if and only if the map
\[\kappa\colon A\ltimes M\longrightarrow F(\hom_A(L_t,L_s))\]
that takes $(a,m)$ to $(\alpha\cdot a\otimes 1,\alpha(-)\otimes m)$ is an isomorphism.
Under the isomorphism $A\ltimes M\cong A\oplus (A\otimes_A M)=F(A)$, this is the map
\[f_\alpha\oplus (f_\alpha\otimes M)\colon F(A)\longrightarrow F(\hom(L_t,L_s))\]
which is an isomorphism with inverse $f^{-1}_\alpha\oplus f^{-1}_\alpha\otimes M$.

To see that $(A\ltimes M,L^{J},\alpha^{J})$ lifts $(A,L,\alpha)$, we need to show that the projections $p\colon A\ltimes M\longrightarrow A$ and $p\colon L^J\longrightarrow L$ onto the first summands define a map of antistructures $(p,p)$. The projection clearly commutes with the dualities $\alpha$ and $\alpha^J$, and the map
\[\widetilde{p}\colon L^{J}\otimes_{A\ltimes M}A\longrightarrow L\]
is the canonical isomorphism
\[L^{J}\otimes_{A\ltimes M}A=(L_t\oplus L_t\otimes_A M)\otimes_{A\ltimes M}A\cong (L_t\otimes_A (A\ltimes M))\otimes_{A\ltimes M}A\cong L\]
The zero sections $s\colon A\longrightarrow A\ltimes M$ and $s\colon L\longrightarrow L^J$ also define a map of antistructures, since the map
\[\widetilde{s}\colon L_t\otimes_{A}(A\ltimes M)\longrightarrow L^J=L_t\oplus L_t\otimes_A M\]
is the canonical isomorphism.

For the map of twistings, we must show that the map
\[ L^{J}\otimes_{A\ltimes M}A\ltimes N\longrightarrow  L^{J'}_t\]
that sends $(l,l'\otimes m)\otimes (a,n)$ to
\[(l,l'\otimes f(m))\cdot (a,n)\otimes 1\]
is an isomorphism.
The inverse is 
\[(l,l'\otimes n)\longmapsto (l,0)\otimes (1,0)+(l',0)\otimes(0,n)\]
\end{proof}

Given an $M$-twisting $J$ of $(A,L,\alpha)$ and $P,Q\in\mathcal{P}_A$, define

\[\widehat{J}\colon \hom_A(P,Q\otimes_A M)\longrightarrow \hom_A(D_L(Q),D_L(P)\otimes_A M)\]

\begin{flushleft}as follows. For $f\colon P\longrightarrow Q\otimes_A M$ and $\lambda\colon Q\longrightarrow L_s$ let it be \end{flushleft}
\[\widehat{J}(f)(\lambda)=\mu^{-1}(P\stackrel{f}{\longrightarrow}Q\otimes_A M\stackrel{\lambda\otimes\id}{\longrightarrow}L_s\otimes_A M\stackrel{\alpha\otimes\id}{\longrightarrow}L_t\otimes_A M\stackrel{J}{\longrightarrow}L_t\otimes_A M)\]

\begin{flushleft}where $\mu\colon D_L(P)\otimes_A M\longrightarrow \hom_A(P,(L_t\otimes_A M)_s)$  is the canonical isomorphism.
The following properties of $\widehat{J}$ are easy consequences of the definitions.\end{flushleft}

\begin{prop}\label{naturalityJK}
The map \[\widehat{J}\colon \hom_A(P,Q\otimes_A M)\longrightarrow \hom_A(D_L(Q),D_L(P)\otimes_A M)\] is natural in both variables. Moreover, the diagram
\[\xymatrix{P\ar[d]_f\ar[r]^-{\eta^\alpha}&D^{2}_A(P)\ar[d]^{\widehat{J}^{2}(f)}\\
Q\otimes_A M\ar[r]_-{\eta^{\alpha}\otimes M}&D^{2}_A(Q)\otimes_A M}\]
commutes.
\end{prop}


\subsection{The models for $\KR(A\ltimes M)$ and $\widetilde{\KR}(A\ltimes M)$}\label{secmainthmkr}

We now have the tools to give a precise definition of diagram \eqref{diagmainKR}. Fix an $M$-twisting $J$ of $(A,L,\alpha)$, and recall that the category 
\[S^{2,1}_{\sbt}(A;M)=\coprod_{\varphi\in ObS^{2,1}_{\sbt}\mathcal{D}\mathcal{P}_A}\hom_A(\varphi,\varphi\otimes_A M)\]
has objects $Ob S^{2,1}_{\sbt}\mathcal{D}\mathcal{P}_A$, that is triples $\varphi=(X,Y,\phi)$ with $X$ and $Y$ diagrams in $S^{2,1}_{\sbt}\mathcal{P}_A$, and $\phi$ a natural isomorphism $\{\phi_\theta\colon Y_\theta\longrightarrow D_L(X_\theta)\}$. The endomorphisms of $\varphi$ is the abelian group $\hom_A(\varphi,\varphi\otimes_A M)$, defined as the set of pairs of natural transformations of diagrams $f\colon X\longrightarrow X\otimes_A M$ and $g\colon Y\longrightarrow  Y\otimes_A M$ such that
\[\xymatrix{Y_\theta\ar[d]_{\phi_\theta}\ar[r]^-{g_\theta}&Y_\theta\otimes_AM\ar[d]^{\phi_\theta\otimes M}\\
D_L(X_\theta)\ar[r]_-{\widehat{J}(f_\theta)}&D_L(X_\theta)\otimes_A M 
}\]
commutes for all $\theta\in Cat([2],[k])$. Here $\widehat{J}$ is the natural map defined in \ref{naturalityJK}. Composition in $S^{2,1}_{\sbt}(A;M)$ is pointwise sum, and in particular every morphism is invertible. The duality is, on objects, the same as the duality on $S^{2,1}_k\mathcal{D}\mathcal{P}_A$ induced by $(A,L,\alpha)$. On morphisms it sends a pair $(\{f_\theta\},\{g_\theta\})$ to $(\{g_{\overline{\theta}}\},\{f_{\overline{\theta}}\})$ as an endomorphism
\[\xymatrix{X_{\overline{\theta}}\ar[d]_{\eta}\ar[r]^-{f_{\overline{\theta}}}&X_{\overline{\theta}}\otimes_AM\ar[d]^{\eta\otimes M}\\
D_{L}^{2}(X_{\overline{\theta}})\ar[d]_{D_L(\phi_{\overline{\theta}})}\ar@{-->}[r]^-{J^2(f)}&D_{L}^{2}(X_{\overline{\theta}})\otimes_A M\ar[d]^{D_L(\phi_{\overline{\theta}})\otimes M}\\
D_L(Y_{\overline{\theta}})\ar[r]_-{J(g_{\overline{\theta}})}&D_L(Y_{\overline{\theta}})\otimes_A M 
}\]
of $D_L(\varphi)=(Y,X,D_L(\phi)\circ\eta)$.
Both squares commute by the properties of $\widehat{J}$ showed in \ref{naturalityJK}.

There is a duality preserving functor 
\[\Psi\colon S^{2,1}_{\sbt}(A;M)\longrightarrow iS^{2,1}_{\sbt}\mathcal{D}\mathcal{P}_{A\ltimes M}\]
where the right hand-side is equipped with the duality induced by the antistructure $(A\ltimes M,L^J,\alpha^J)$.
It sends an object $\varphi$ to the diagram defined at $\theta$ by
\[\xymatrix{F(Y_\theta)\ar[rr]^-{\phi_\theta\oplus(\phi_\theta\otimes M)}&&F(D_L(X_\theta))\ar[r]^-{\xi}&D_{L^J}(F(X_\theta))}\]
with $\xi$ defined in \ref{lemmaisoequivariant}. We recall that $F(P)=P\oplus P\otimes_A M\cong P\otimes_AA\ltimes M$ for every $A$-module $P$.
An endomorphism $(f,g)$ of $\varphi$ is sent to the endomorphism of $\Psi(\varphi)$ in $iS^{2,1}_{\sbt}\mathcal{D}\mathcal{P}_{A\ltimes M}$ defined by the pair of natural maps
\[\left(\left(\begin{smallmatrix}
1&0\\f_\theta&1
\end{smallmatrix}\right)\colon F(X_\theta)\longrightarrow F(X_\theta),\left(\begin{smallmatrix}
1&0\\g_\theta&1
\end{smallmatrix}\right)\colon F(Y_\theta)\longrightarrow F(Y_\theta)\right)\]

To finish the definition of diagram \eqref{diagmainKR} we need to define the vertical maps. The left one is the functor $S^{2,1}_{\sbt}(A;M)\longrightarrow Ob S^{2,1}_{\sbt}\mathcal{D}\mathcal{P}_A$ that projects onto the objects. The right one is the functor \[\mathcal{D}(-\otimes_{A\ltimes M} A)\colon iS^{2,1}_{\sbt} \mathcal{D}\mathcal{P}_{A\ltimes M}\longrightarrow iS^{2,1}_{\sbt} \mathcal{D}\mathcal{P}_{A}\] induced by the map of antistructures $(A\ltimes M,L^J,\alpha^J)\longrightarrow (A,L,\alpha)$ from \ref{mtwistingantistructure}. This is the main result of §\ref{splitKR}.

\begin{theorem}\label{mainKR} Let $J$ be an $M$-twisting of an antistructure $(A,L,\alpha)$. Then 
there is a diagram of duality preserving functors
\[\xymatrix{S^{2,1}_{\sbt}(A;M)\ar[r]^-{\Psi}_-{\simeq}\ar[d]& iS^{2,1}_{\sbt}\mathcal{D}\mathcal{P}_{A\ltimes M}\ar[d]^{\mathcal{D}(-\otimes_{A\ltimes M}A)}\\
ObS^{2,1}_{\sbt}\mathcal{D}\mathcal{P}_{A}\ar[r]_-{\simeq}& iS^{2,1}_{\sbt}\mathcal{D}\mathcal{P}_{A}
}\]
that commutes up to $\mathbb{Z}/2$-homotopy, where the two horizontal maps are $\mathbb{Z}/2$-equivalences.
\end{theorem}

The proof of this theorem is presented in the next four sections. We finish this section by proving a crucial corollary. Recall that we denote $\widetilde{\KR}(A\ltimes M)$ the homotopy fiber of the projection $\KR(A\ltimes M)\longrightarrow \KR(A)$ over the basepoint. Since the basepoint is a fixed point, $\widetilde{\KR}(A\ltimes M)$ is a $\mathbb{Z}/2$-space, equivariantly homeomorphic to 
\[\widetilde{\KR}(A\ltimes M)\cong \Omega^{2,1}\hofib(|iS^{2,1}_{\sbt}\mathcal{D}\mathcal{P}_{A\ltimes M}|\longrightarrow |iS^{2,1}_{\sbt}\mathcal{D}\mathcal{P}_A|)\]
A model for this homotopy fiber is $\Omega^{2,1}$ of the realization of the bisimplicial set
\[\widetilde{S}^{2,1}_{\sbt}(A;M(S_{\sbt}^{1,1}))=\bigvee_{\varphi\in ObS^{2,1}_{\sbt}\mathcal{D}\mathcal{P}_A}\hom_A(\varphi,\varphi\otimes_A M(S_{\sbt}^{1,1}))\]
where we recall that $M(S_{\sbt}^{1,1})$ is the Bar construction on $M$ with the involution that reverses the order of the components. The bisimplicial set $\widetilde{S}^{2,1}_{\sbt}(A;M(S_{\sbt}^{1,1}))$ has a degreewise involution defined as follows. An element of $\widetilde{S}^{2,1}_k(A;M(S_{q}^{1,1}))$ consists of a $\varphi\in Ob S^{2,1}_{\sbt}\mathcal{D}\mathcal{P}_{A}$ labeling the wedge component, and of a pair of $q$-tuples of natural transformations
\[(\{f_{\theta}^1\},\dots, \{f_{\theta}^q\}\colon X\longrightarrow X\otimes_A M, \{g_{\theta}^1\},\dots, \{g_{\theta}^q\}\colon Y\longrightarrow Y\otimes_A M)\]
such that each pair $(f^i,g^i)$ belongs to $\hom_A(\varphi,\varphi\otimes_AM)$.
The involution sends this pair to the pair
\[( \{g_{\overline{\theta}}^q\},\dots, \{g_{\overline{\theta}}^1\}\colon Y\longrightarrow Y\otimes_A M,\{f_{\overline{\theta}}^q\},\dots, \{f_{\overline{\theta}}^1\}\colon X\longrightarrow X\otimes_A M)\]
in the $D_L(\varphi)$ wedge component.

\begin{cor}\label{relKR}
Let $J$ be an $M$-twisting of an antistructure $(A,L,\alpha)$. Then there is a weak $\mathbb{Z}/2$-equivalence
\[\widetilde{\KR}(A\ltimes M)\simeq \Omega^{2,1}|\bigvee_{\varphi\in ObS^{2,1}_{\sbt}\mathcal{D}\mathcal{P}_A}\hom_A(\varphi,\varphi\otimes_A M(S_{\sbt}^{1,1}))|\]
\end{cor}

\begin{proof} We use an equivariant version of the argument for \cite[3.4]{stableDM}.
We define $\KR(A;M(S^{1,1}_{\sbt}))$ as the realization of the real category
\[\KR(A;M(S^{1,1}_{\sbt}))=\Omega^{2,1}|S^{2,1}_{\sbt}(A;M)|\]
The homotopy fiber of the map $\KR(A;M(S^{1,1}_{\sbt}))\longrightarrow \KR(A)$ is denoted $\widetilde{\KR}(A;M(S^{1,1}_{\sbt}))$ and it is equivariantly homeomorphic to
\[\widetilde{\KR}(A;M(S^{1,1}_{\sbt}))\cong \Omega^{2,1}\hofib(| S^{2,1}_{\sbt}(A;M)|\longrightarrow |ObS^{2,1}_{\sbt}\mathcal{D}\mathcal{P}_A|)\]
With this notation, taking $\Omega^{2,1}$ of the realization of the diagram of \ref{mainKR} gives a $\mathbb{Z}/2$-homotopy commutative diagram
\[\xymatrix{\widetilde{\KR}(A; M(S^{1,1}_{\sbt}))\ar[d]\ar[r]^-{\widetilde{\Psi}}_-{\simeq} & \widetilde{\KR}(A\ltimes M)\ar[d]\\
\KR(A;M(S^{1,1}_{\sbt}))\ar[r]^-{\Psi}_-{\simeq}\ar[d]& \KR(A\ltimes M)\ar[d]\\
\Omega^{2,1}|ObS^{2,1}_{\sbt}\mathcal{D}\mathcal{P}_A|\ar[r]_-{\simeq}& \KR(A)
}\]
Therefore it is enough to show that there is an equivalence
\[\hofib(|S^{2,1}_{\sbt}(A;M)|\longrightarrow|ObS^{2,1}_{\sbt}\mathcal{D}\mathcal{P}_A|)\simeq|\!\!\!\!\!\bigvee\limits_{\varphi\in ObS^{2,1}_{\sbt}\mathcal{D}\mathcal{P}_A}\!\!\!\!\!\hom_A(\varphi,\varphi\otimes_A M(S^{1,1}_{\sbt}))|\]
We denote this homotopy fiber by
\[F=\hofib(|S^{2,1}_{\sbt}(A;M)|\longrightarrow |ObS^{2,1}_{\sbt}\mathcal{D}\mathcal{P}_A|)\]
Since composition in $S^{2,1}_{\sbt}(A;M)$ is addition, the nerve $\mathcal{N}_{\sbt} S^{2,1}_{\sbt}(A;M)$ is given by the disjoint union
\[\mathcal{N}_{\sbt} S^{2,1}_{\sbt}(A;M)=\coprod_{\varphi\in ObS^{2,1}_{\sbt}\mathcal{D}\mathcal{P}_A}\hom_A(\varphi,\varphi\otimes_A M(S^{1,1}_{\sbt}))\]
with involution defined in exactly the same way as for the wedge.
Restricting on the homotopy fiber the projection map from the disjoint union to the wedge we obtain an equivariant map
\[\alpha\colon F\longrightarrow |\bigvee_{\varphi\in ObS^{2,1}_{\sbt}\mathcal{D}\mathcal{P}_A}\hom_A(\varphi,\varphi\otimes_A M(S^{1,1}_{\sbt}))|\]

We are going to build, for any given positive integer $p$, a model for $\alpha$ which is roughly $2p$-connected as a map of spaces, and roughly $p$-connected on the fixed points. This shows that it $\alpha$ is arbitrarily highly connected both non-equivariantly and on the fixed points, and therefore it is a $\mathbb{Z}/2$-equivalence. This will finish the proof. We construct this models for $\alpha$ by iterating the $S^{2,1}_{\sbt}$-construction. Define for all natural number $p$ a multisimplicial category with duality 
\[S^{2p,p}_{\sbt}(A;M)=\coprod_{\vartheta\in Ob(S^{2,1}_{\sbt})^{(p)}\mathcal{D}\mathcal{P}_A}\hom_A(\vartheta,\vartheta\otimes_A M)\]
in a completely similar way as we did in the previous sections. Since all the maps we defined were pointwise, theorem \ref{mainKR} extends word by word to define a diagram of $\mathbb{Z}/2$-spaces
\[\xymatrix{|S^{2p,p}_{\sbt}(A;M)|\ar[r]^-{\Psi^{(p)}}_-{\simeq}\ar[d]& |i(S^{2,1}_{\sbt})^{(p)}\mathcal{D}\mathcal{P}_{A\ltimes M}|\ar[d]\\
|Ob(S^{2,1}_{\sbt})^{(p)}\mathcal{D}\mathcal{P}_{A}|\ar[r]^-{\simeq}&|i(S^{2,1}_{\sbt})^{(p)}\mathcal{D}\mathcal{P}_{A}|
}\]

Moreover there is a commutative square
\[\xymatrix{|S^{2p,p}_{\sbt}(A;M)|\ar[d]_{\Psi^{(p)}}\ar[r]&\Omega^{2,1}|S^{2(p+1),(p+1)}_{\sbt}(A;M)|\ar[d]_{\Psi^{(p+1)}}\\
|i(S^{2,1}_{\sbt})^{(p)}\mathcal{D}\mathcal{P}_{A\ltimes M}|\ar[r]^-{\simeq}&\Omega^{2,1}|i(S^{2,1}_{\sbt})^{(p+1)}\mathcal{D}\mathcal{P}_{A\ltimes M})|
}\]
where the two vertical maps are equivariant equivalences, and the lower horizontal is one as well by \ref{sumdiags}. Therefore the realization of the projection map $S^{2,1}_{\sbt}(A;M)\longrightarrow ObS^{2,1}_{\sbt}\mathcal{D}\mathcal{P}_{A}$ deloops equivariantly as
\[\xymatrix{|S^{2,1}_{\sbt}(A;M)|\ar[r]^-{\simeq}\ar[d]& (\Omega^{2,1})^{(p-1)}|S^{2p,p}_{\sbt}(A;M)|\ar[d]\\
|ObS^{2,1}_{\sbt}\mathcal{D}\mathcal{P}_{A}|\ar@/^1pc/[u]^{s}\ar[r]^-{\simeq} &(\Omega^{2,1})^{(p-1)}|Ob(S^{2,1}_{\sbt})^{(p)}\mathcal{D}\mathcal{P}_{A}|\ar@/_1pc/[u]_{s^{(p)}}
}\]
and both vertical maps have equivariant splittings, that also commute with the horizontal maps.
This shows that the homotopy fiber
\[F^{(p)}=(\hofib|S^{2p,p}_{\sbt}(A;M)|\longrightarrow |Ob(S^{2,1}_{\sbt})^{(p)}\mathcal{D}\mathcal{P}_A|)\]
deloops $F$ equivariantly by
\[F\stackrel{\simeq}{\longrightarrow}(\Omega^{2,1})^{(p-1)}F^{(p)}\]
Also, by homotopy invariance of the homotopy cofiber applied to the sections, there is an equivalence
\[\xymatrix{|\bigvee\limits_{\varphi\in ObS^{2,1}_{\sbt}\mathcal{D}\mathcal{P}_A}\hom_A(\varphi,\varphi\otimes_A M(S^{1,1}_{\sbt}))| \ar[d]^{\simeq}\\
(\Omega^{2,1})^{(p-1)}|\bigvee\limits_{\vartheta\in Ob(S^{2,1}_{\sbt})^{(p)}\mathcal{D}\mathcal{P}_A}\hom_A(\vartheta,\vartheta\otimes_A M((S^{1,1}_{\sbt})^{\wedge p}))|}\]
where $M((S^{1,1}_{\sbt})^{\wedge p})$ is the Bar construction $M(S^{1,1}_{\sbt})$ iterated degreewise $p$ times.
Similarly as before, the projection from disjoint union to wedge induces a map $\alpha^{(p)}$ fitting in a commutative diagram
\[\xymatrix{F\ar[d]^{\simeq}\ar[rr]^-{\alpha}&&|\bigvee\limits_{ ObS^{2,1}_{\sbt}\mathcal{D}\mathcal{P}_A}\hom_A(\varphi,\varphi\otimes_A M)|\ar[d]^{\simeq}\\
(\Omega^{2,1})^{(p-1)}F^{(p)}\ar[rr]_-{(\Omega^{2,1})^{(p-1)}\alpha^{(p)}}&&(\Omega^{2,1})^{(p-1)}|\!\!\!\!\!\!\!\!\!\!\bigvee\limits_{ Ob(S^{2,1}_{\sbt})^{(p)}\mathcal{D}\mathcal{P}_A}\!\!\!\!\!\!\!\!\!\!\hom_A(\vartheta,\vartheta\otimes_A\!\!M((S^{1,1}_{\sbt})^{\wedge p}))|
}\]
Therefore, it is enough to show that the connectivity of $(\Omega^{2,1})^{(p-1)}\alpha^{(p)}$ tends to infinity with $p$.

In order to prove this, consider the homotopy cocartesian square of $\mathbb{Z}/2$-spaces
\[\xymatrix{|S^{2p,p}_{\sbt}(A;M)|\ar[r]\ar[d]&|\bigvee\limits_{\vartheta\in Ob(S^{2,1}_{\sbt})^{(p)}\mathcal{D}\mathcal{P}_A}\hom_A(\vartheta,\vartheta\otimes_A M((S^{1,1}_{\sbt})^{\wedge p}))|\ar[d]\\
|(ObS^{2,1}_{\sbt})^{(p)}\mathcal{D}\mathcal{P}_A|\ar[r]&\ast
}\]
where the top horizontal map is the projection from the disjoint union to the wedge.
By this we mean that both the square and its restriction to the fixed points are homotopy cocartesian squares of spaces. For the statement on fixed points, simply notice that the fixed points of a disjoint union is a disjoint union of fixed points, and similarly for the wedge. Therefore the restriction of the square to the fixed points is of the same form, with a projection from disjoint union to wedge as top row, and therefore homotopy cocartesian.

Notice that, for each space of the last diagram, every simplicial direction is $1$-reduced since $S^{2,1}_0C=S^{2,1}_1C=\ast$. Therefore each of the spaces is at least $2(p-1)$-connected non-equivariantly. By the Blakers-Massey theorem (see e.g.  \cite[2.3]{calculusII}) the diagram above is $(2\cdot 2(p-1)-1)$-homotopy cartesian, that is the map $\alpha^{(p)}$ is non equivariantly $(4p-3)$-connected.

We want to know the connectivity of the fixed points of each of the spaces in the diagram above to do another Blakers-Massey argument. The fixed points are the realization of the multisimplicial set obtained by subdividing every simplicial direction. Again since every simplicial direction is $1$-reduced, every subdivided simplicial direction is $0$-reduced. Therefore the fixed points of the realizations above are each at least $(p-1)$-connected. Again by Blakers-Massey, the map $\alpha^{(p)}$ is $(2p-3)$-connected on the fixed points.

Looping $\alpha^{(p)}$ down by $(\Omega^{2,1})^{(p-1)}$ we get that $(\Omega^{2,1})^{(p-1)}\alpha^{(p)}$, and $\alpha$, is non-equivariantly $(2p-1)$-connected, and $(p-2)$-connected on the fixed points. This behavior of connectivity with respect to equivariant loop spaces is a consequence of the equivariant Whitehead theorem. The connectivity of $(\Omega^{2,1})^{(p-1)}\alpha^{(p)}$ on the fixed points is the connectivity of $\alpha^{(p)}$ minus the dimension of the fixed points of $p-1$ smash copies of $S^{2,1}$, which is $p-1$. This general result of equivariant homotopy theory is proved later in \ref{eqmappingspace}. This tends to infinity with $p$.
\end{proof}


\subsection{The equivalence $ObS^{2,1}_{\cdot} C\longrightarrow iS^{2,1}_{\cdot} C$}\label{obintoiso}

We show that the realization of the real set defined by the objects of $S^{2,1}_{\sbt} C$ is $\mathbb{Z}/2$-equivalent to the simplicial category $iS^{2,1}_{\sbt} C$. The proof is a straightforward generalization of \cite[I-2.3.2]{Dundasbook} (see also \cite[1.4.1]{waldhausen}). In particular this shows that the bottom horizontal map of the diagram of \ref{mainKR} is an equivalence.
 
Suppose that $C$ is an exact category with strict duality, and consider the simplicial set of objects $ObS^{2,1}_{\sbt} C$ as a simplicial category with only identities.

\begin{prop}\label{inclobinisos21}
The inclusion $ObS^{2,1}_{\sbt} C\longrightarrow iS^{2,1}_{\sbt} C$ induces a $\mathbb{Z}/2$-equivalence on the realization.
\end{prop}

\begin{proof}
For a fixed $k$ we build a retraction 
\[r\colon \mathcal{N}_{2k+1}iS^{2,1}_{\sbt} C\longrightarrow \mathcal{N}_{2k+1}ObS^{2,1}_{\sbt} C\cong ObS^{2,1}_{\sbt} C\]
This is not going to be simplicial if we move $k$, but it is simplicial in the $S^{2,1}_{\sbt}$-direction.
An element of $\mathcal{N}_{2k+1}iS^{2,1}_{\sbt} C$ is a sequence of isomorphisms
\[\underline{\phi}=(X_0\stackrel{\phi_1}{\longrightarrow}X_1\stackrel{\phi_2}{\longrightarrow}\dots \stackrel{\phi_{2k+1}}{\longrightarrow}X_{2k+1})\]
We define the retraction $r$ by mapping this element to
\[X_{k+1}\stackrel{=}{\longrightarrow}X_{k+1}\stackrel{=}{\longrightarrow}\dots \stackrel{=}{\longrightarrow}X_{k+1}\]
Since we chose the middle object $X_{k+1}$, the map $r$ defines a map of real sets
\[r\colon \mathcal{N}_{2k+1}iS^{2,1}_{\sbt} C\longrightarrow \mathcal{N}_{2k+1}ObS^{2,1}_{\sbt} C\]
which is a retraction for the inclusion $\iota\colon \mathcal{N}_{2k+1}ObS^{2,1}_{\sbt} C\longrightarrow\mathcal{N}_{2k+1}iS^{2,1}_{\sbt} C$.
We find a simplicial homotopy \[H\colon \mathcal{N}_{2k+1}iS^{2,1}_{\sbt} C\times\Delta[1]\longrightarrow \mathcal{N}_{2k+1}iS^{2,1}_{\sbt} C\] commuting with the dualities, between $\iota\circ r$ and the identity. Look at $\mathcal{N}_{2k+1}iS^{2,1}_2 C=\mathcal{N}_{2k+1}iC$ as a category with duality, with natural transformations of diagrams as morphisms. Then $\iota_2\circ r_2$ extends to a functor $\mathcal{N}_{2k+1}iC\longrightarrow \mathcal{N}_{2k+1}iC$ in the obvious way, and there is a natural isomorphism $U\colon \id\Rightarrow\iota_2\circ r_2$ given at an object $\underline{\phi}$ by the diagram
\[\resizebox{1\hsize}{!}{\xymatrix{
X_{0}\ar[d]^{\phi_{k+1}\dots\phi_{1}}\ar[r]^{\phi_1}&  X_1\ar[d]^{\phi_{k+1}\dots\phi_{2}}\ar[r]^{\phi_2}&\dots \ar[r]^{\phi_{k}}&X_{k}\ar[d]^{\phi_{k+1}}\ar[r]^{\phi_{k+1}}&X_{k+1}\ar@{=}[d]\ar[r]^{\phi_{k+2}}&X_{k+2}\ar[d]^{\phi_{k+2}^{-1}}\ar[r]^{\phi_{k+3}}&\dots\ar[r]^{\phi_{2k+1}}&X_{2k+1}\ar[d]_{\phi_{k+2}^{-1}\dots\phi_{2k+1}^{-1}}\\
X_{k+1}\ar@{=}[r]&X_{k+1}\ar@{=}[r]&\dots\ar@{=}[r]&X_{k+1}\ar@{=}[r]&X_{k+1}\ar@{=}[r]&X_{k+1}\ar@{=}[r]&\dots\ar@{=}[r]&X_{k+1}
}}\]

Clearly it respects the duality, in the sense that $U_{D\underline{\phi}}\circ DU_{\underline{\phi}}=\id_{D\underline{\phi}}$.
As natural transformations do, it induces a functor $U\colon \mathcal{N}_{2k+1}iC\times[1]\longrightarrow \mathcal{N}_{2k+1}iC$.
In simplicial degree $p$, we define the homotopy 
\[H_p\colon \mathcal{N}_{2k+1}iS^{2,1}_p C\times\Delta[1]_p\longrightarrow \mathcal{N}_{2k+1}ObS^{2,1}_p C\] by sending a pair $(\underline{\phi},\sigma\colon [p]\longrightarrow [1])$ to the element
\[Cat([2],[p])\stackrel{\id\times ev_{1}}{\longrightarrow}Cat([2],[p])\times [p]\stackrel{\underline{\phi}\times \sigma}{\longrightarrow}\mathcal{N}_{2k+1}iC\times [1]\stackrel{U}{\longrightarrow}\mathcal{N}_{2k+1}iC\]
of $\mathcal{N}_{2k+1}iS^{2,1}_pC$. Here $ev_1\colon Cat([2],[n])\longrightarrow [n]$ is the evaluation functor that sends $\theta$ to $\theta(1)$.
For $\sigma=0$, the functor $H_p(\underline{\phi},0)$ sends $\theta\in Cat([2],[p])$ to
\[H_p(\underline{\phi},0)(\theta)=U(\underline{\phi}_{\theta},0)=(\iota_2\circ r_2)(\underline{\phi}_\theta)=(\iota\circ r)(\underline{\phi})_\theta\]
and for $\sigma=1$ to
\[H_p(\underline{\phi},1)(\theta)=U(\underline{\phi}_{\theta},1)=\underline{\phi}_\theta\]
The homotopy $H\colon \mathcal{N}_{2k+1}iS^{2,1}_{\sbt} C\times\Delta[1]\longrightarrow \mathcal{N}_{2k+1}iS^{2,1}_{\sbt} C$ is simplicial (always for $k$ fixed!), since for all $\kappa\colon [m]\longrightarrow [n]$ we have
\[\kappa(ev_1(\theta))=ev_1(\kappa\circ\theta)\]
for all $\theta\colon[2]\longrightarrow [m]$. Moreover it is equivariant, in the sense that
\[H(D\underline{\phi},\sigma)=DH(\underline{\phi},\sigma)\]
We show this carefully. The left hand-side evaluated at $\theta\colon[2]\longrightarrow [n]$ gives
\[H(D\underline{\phi},\sigma)(\theta)=U((D\underline{\phi})_{\theta},\sigma(ev_1(\theta)))\]
The right hand side gives
\[DH(\underline{\phi},\sigma)(\theta)=DU(\underline{\phi}_{\overline{\theta}},\sigma(ev_1(\overline{\theta})))=U(D\underline{\phi}_{\overline{\theta}},\sigma(n-ev_1(\overline{\theta})))\]
where the second equality follows from equivariancy of the natural transformation $U$. But now
\[n-ev_1(\overline{\theta})=n-\overline{\theta}(1)=n-(n-\theta(2-1))=\theta(1)\]

This shows that $(Sd\iota)_k\colon |ObS^{2,1}_{\sbt} C|\longrightarrow |\mathcal{N}_{2k+1}iS^{2,1}_{\sbt} C|$ is a levelwise equivariant homotopy equivalence. Therefore its realization is an equivariant weak homotopy equivalence
\[Sd\iota\colon |ObS^{2,1}_{\sbt} C|\longrightarrow |Sd\mathcal{N}_{\sbt}iS^{2,1}_{\sbt} C|\cong |\mathcal{N}_{\sbt}iS^{2,1}_{\sbt} C|\]

\end{proof}


\subsection{The factorization of $\Psi$ through $i\mathcal{D}T^{2,1}_{\cdot}(A;M)$}\label{model}

In order to prove that the map $\Psi$ of \ref{mainKR} is an equivalence, we factor it through an intermediate real category with strict duality $i\mathcal{D}T^{2,1}_{\sbt}(A;M)$. We prove in the next sections that each of the maps of the factorization is a $\mathbb{Z}/2$-equivalence.

The simplicial category $T^{2,1}_{\sbt}(A;M)$ has the same objects as $S^{2,1}_{\sbt}\mathcal{P}_A$. Morphisms from $X$ to $X'$ are pairs of maps of diagrams
\[(\phi\colon X\longrightarrow X',f\colon X\longrightarrow X'\otimes_A M)\]
with $\phi$ a morphism in $S^{2,1}_{\sbt}\mathcal{P}_A$ and $f$ a map of diagrams of right $A$-modules. Composition is defined by
\[(\psi,g)\circ (\phi,f)=(\psi\circ\phi,\psi\otimes M\circ f+g\circ\phi)\]
where composition, sum and tensor products of maps of diagrams are defined objectwise.
The identity of $X$ is the pair $(\id,0)$.

We remark that there is a canonical bijection between the set of morphisms in $T^{2,1}_{\sbt}(A;M)$ from $X$ to $X'$ and
\[\hom_{A\ltimes M}(X\otimes_A A\ltimes M,X'\otimes_A A\ltimes M)\]
Composition as defined above corresponds to standard composition of module maps under this bijection.

Now suppose that $J\colon L_t\otimes_A M\longrightarrow L_t\otimes_A M$ is an $M$-twisting of an antistructure $(A,L,\alpha)$.
For a map $f\colon X\longrightarrow X'\otimes_A M$, $\widehat{J}(f)\colon D_L(X')\longrightarrow D_L(X)\otimes_A M$ is the map of diagrams defined at a $\theta\colon [2]\longrightarrow [n]$ by 
\[\widehat{J}(f)_{\theta}=\widehat{J}(f_{\overline{\theta}}\colon X_{\overline{\theta}}\longrightarrow X'_{\overline{\theta}}\otimes_A M)\]
where we remember that $\overline{\theta}\in Cat([2],[k])$ is the dual given by $\overline{\theta}(i)=k-\theta(2-i)$.

We now define a duality structure $(D_J,\eta_J)$ on $T^{2,1}_{\sbt}(A;M)$. The functor $D_{J}\colon T^{2,1}_{\sbt}(A;M)^{op}\longrightarrow T^{2,1}_{\sbt}(A;M)$ is defined by $D_L$ on objects, and by
\[D_J(\phi,f)=(D_L(\phi),\widehat{J}(f))\]
on morphisms. The natural isomorphism $\eta_J\colon \id\Rightarrow D_{J}^2$ is defined at $X$ by the pair
\[\eta_J=(\eta\colon X\longrightarrow D_{L}^2(X),0\colon X\longrightarrow D_{L}^2(X)\otimes_AM)\]

\begin{prop}\label{defs21(A;M)}
$(T^{2,1}_{\sbt}(A;M),D_J,\eta_J)$ is a real category.
\end{prop}

\begin{proof}
Functoriality of $D_J$ follows by functoriality of $D_L$ and the two first properties of $\widehat{J}$ from \ref{naturalityJK}. The fact that $\eta_J$ is a natural transformation $\id\Rightarrow D_{J}^2$ follows from the third property of $\widehat{J}$ from \ref{naturalityJK}. Clearly $(\eta_J)_{D_J(-)}\circ D_J(\eta_J)=\id$ since a similar property is satisfied by $\eta$ and $\eta_J$ is zero in the second component.
\end{proof}

We want to identify $S^{2,1}(A;M)$ as a subcategory of $i\mathcal{D}T^{2,1}_{\sbt}(A;M)$. There is a natural isomorphism of categories with strict duality
\[S^{2,1}_{\sbt}(A;M)\cong\coprod_{\varphi\in Ob\mathcal{D}S^{2,1}_{\sbt}\mathcal{P}_A}\hom_A(\varphi,\varphi\otimes_A M)=\widehat{S}^{2,1}_{\sbt}(A;M)\] 
were now the $\mathcal{D}$ functor is taken after the $S^{2,1}_{\sbt}$-construction. The isomorphism is similar to the one from \ref{DfirstorlaterS21}.
An object of $i\mathcal{D}T^{2,1}_{\sbt}(A;M)$ is a triple $(X,Y,(\phi,h))$ with $X,Y\in S^{2,1}_{\sbt}\mathcal{P}_A$  and an isomorphism $(\phi,h)\colon Y\longrightarrow D_L(X)$ in $T^{2,1}_{\sbt}(A;M)$, that is a pair of natural transformations
\[(\phi\colon Y\stackrel{\cong}{\longrightarrow} D_L(X),h\colon Y\longrightarrow D_L(X)\otimes_A M)\]
where $\phi$ is an isomorphism.
The objects of $\widehat{S}^{2,1}_{\sbt}(A;M)$ are those with $h=0$. It is not however a full subcategory. Indeed, a morphism in $i\mathcal{D}T^{2,1}_{\sbt}(A;M)$ from $(X,Y,(\phi,0))$ to $(X',Y',(\phi',0))$ is a pair of isomorphisms $(a,f)\colon X\longrightarrow X'$ and $(b,g)\colon Y'\longrightarrow Y$ in $T^{2,1}_{\sbt}(A;M)$ satisfying certain properties coming from composition. The category $\widehat{S}^{2,1}_{\sbt}(A;M)$ has only endomorphisms, and the endomorphisms of $(X,Y,(\phi,0))$ coming from $\widehat{S}^{2,1}_{\sbt}(A;M)$ are exactly those where $a$ and $b$ are identities. This defines an inclusion
\[S^{2,1}_{\sbt}(A;M)\cong \widehat{S}^{2,1}_{\sbt}(A;M)\stackrel{\iota}{\longrightarrow} i\mathcal{D}T^{2,1}_{\sbt}(A;M)\]
We now define a morphism of categories with duality $(\overline{F},\overline{\xi})\colon T^{2,1}_{\sbt}(A;M)\longrightarrow S^{2,1}_{\sbt}\mathcal{P}_{A\ltimes M}$ fitting into a commutative diagram
\[\xymatrix{S^{2,1}_{\sbt}(A;M)\ar[d]_-{\cong}\ar[r]^-{\Psi}&
iS^{2,1}_{\sbt}\mathcal{D}\mathcal{P}_{A\ltimes M}\cong i\mathcal{D}S^{2,1}_{\sbt}\mathcal{P}_{A\ltimes M}\\
\widehat{S}^{2,1}_{\sbt}(A;M)\ar[r]^-{\iota}
& *[r]{i\mathcal{D}T^{2,1}_{\sbt}(A;M)}\ar@<-7ex>[u]_-{\mathcal{D}(\overline{F},\overline{\xi})}
}\]
giving the factorization of $\Psi$ we were after. The isomorphism in the top right corner is from \ref{DfirstorlaterS21}.

Let $\overline{F}\colon T^{2,1}_{\sbt}(A;M)\longrightarrow iS^{2,1}_{\sbt}\mathcal{P}_{A\ltimes M}$ be the functor which sends an object $X=\{X_\theta\}$ to the diagram
\[\overline{F}(X)_\theta=F(X_\theta)=X_\theta\oplus X_\theta\otimes_AM\]
It sends a morphism $(\phi,f)$ to the family of morphisms
\[\overline{F}(\phi,f)_\theta=\left(\begin{array}{cc}\phi_\theta&0\\f_\theta&\phi_\theta\otimes M\end{array}\right)\colon F(X_\theta)\longrightarrow F(X'_\theta)\] 
There is a natural isomorphism $\overline{\xi}\colon \overline{F}D_J\Rightarrow D_{L^J}\overline{F}$ defined at $\theta$ by the isomorphisms
\[\overline{\xi}_\theta=\xi_{X_{\overline{\theta}}}\colon FD_L(X_{\overline{\theta}})\longrightarrow D_{L^J}F(X_{\overline{\theta}})\]
of \ref{lemmaisoequivariant}. The pair $(\overline{F},\overline{\xi})$ defines a morphism of categories with duality, and it is easy to see that $\mathcal{D}(\overline{F},\overline{\xi})$ extends $\Psi$ making the diagram above commutative.


\subsection{The equivalence $T^{2,1}_{\cdot}(A;M)\longrightarrow S^{2,1}_{\cdot}\mathcal{P}_{A\ltimes M}$}\label{model1}

We prove that the pair $(\overline{F},\xi)\colon T^{2,1}_{\sbt}(A;M)\longrightarrow S^{2,1}_{\sbt}\mathcal{P}_{A\ltimes M}$ from the previous section is in every simplicial degree an equivalence of categories with duality in the sense of definition \ref{eqcatdual}. As a consequence, the functor $\mathcal{D}(\overline{F},\xi)$ is going to be levelwise an equivalence of categories with strict duality, and therefore it induces a $\mathbb{Z}/2$-homotopy equivalence on the realization.

We start by proving the result in simplicial degree $2$. In this case we denote $T^{2,1}_2(A;M)=\mathcal{P}(A;M)$. It has the same objects as $\mathcal{P}_A$ and the pairs of module maps $(\phi\colon P\longrightarrow Q,f\colon P\longrightarrow Q\otimes_A M)$ are the morphisms from $P$ to $Q$. We denote
\[(F,\xi)=(\overline{F}_1,\overline{\xi}_1)\colon \mathcal{P}(A;M)\longrightarrow \mathcal{P}_{A\ltimes M}\]

\begin{prop}\label{modelPAM}
The morphism of categories with duality \[(F,\xi)\colon (\mathcal{P}(A;M),D_J,\eta_J)\longrightarrow (\mathcal{P}_{A\ltimes M},D_{L^J},\eta^{\alpha^J})\]
is an equivalence of categories with duality.
\end{prop}

\begin{proof}
By \ref{equivalenceequivariant}, it is enough to show that $F$ is an equivalence of categories, and there is going to be automatically an inverse as a morphism of categories with duality. We prove equivalently that $F$ is fully faithful and essentially surjective.

It is faithful since if $F(\phi,f)=F(\psi,g)$, then there is an equality of matricies
\[\left(\begin{array}{cc}\phi&0\\f&\phi\otimes M\end{array}\right)=\left(\begin{array}{cc}\psi&0\\g&\psi\otimes M\end{array}\right)\]
and therefore clearly $\phi=\psi$ and $f=g$.

To show that $F$ is full, let $\chi\colon F(P)\longrightarrow F(Q)$ be a map of right modules. Let $\chi_1$ and $\chi_2$ denote respectively the $Q$ and $Q\otimes_A M$ component. Thus $\chi$ satisfies
\[\begin{array}{ll}\chi_1(p\cdot a,p'\otimes (m'\cdot a)+p\otimes m)=\chi_1(p,p'\otimes m')\cdot a\\
\chi_2(p\cdot a,p'\otimes (m'\cdot a)+p\otimes m)=\chi_2(p,p'\otimes m')\cdot a+\chi_1(p,p'\otimes m')\otimes m\end{array}\]
Setting appropriately the variables to zero, one can easily see that
\begin{enumerate}
\item $\chi_1(0,-)=0$
\item $\chi_1(-,0)\in\hom_A(P,Q)$
\item $\chi_2(0,p\otimes m)=\chi_1(p,0)\otimes m$
\item $\chi_2(-,0)\in\hom_A(P,Q\otimes_A M)$
\end{enumerate}
That is
\[\chi=\left(\begin{array}{cc}\chi_1(-,0)&0\\\chi_2(-,0)&\chi_1(-,0)\otimes M\end{array}\right)=F(\chi_1(-,0),\chi_2(-,0))\]

To show that $F$ is essentially surjective, we show equivalently that the isomorphic functor sending $P$ to $P\otimes_AA\ltimes M$ is. More precisely, we show that every projective $A\ltimes M$-module $K$ is (non canonically!) isomorphic to
\[H(K)=(K\otimes_{A\ltimes M}A)\otimes_AA\ltimes M\]
Here is the argument from \cite[I-2.5.5]{Dundasbook}. Denote $p^\ast (K\otimes_{A\ltimes M} A)$ the $A\ltimes M$-module structure on $K\otimes_{A\ltimes M} A$ induced by $p\colon A\ltimes M\longrightarrow A$, that is defined by
\[k\otimes a\cdot (a',m)=k\otimes (a\cdot a')\]
We see $M$ as an ideal of $A\ltimes M$ by $m\mapsto (m,0)$.
There are two short exact sequences of $A\ltimes M$-modules over $p^\ast (K\otimes_{A\ltimes M} A)$
\[\xymatrix{K\cdot M\ar[d]&H(K)\cdot M\ar[d]\\
K\ar[d]_{\pi'}\ar@{-->}[r]^{\epsilon} &H(K)\ar[d]_{\pi}\\
p^\ast (K\otimes_{A\ltimes M} A)\ar[r]_-{=}& p^\ast (K\otimes_{A\ltimes M} A)
}\]
Where $\pi'(k)=k\otimes 1$ and $\pi(k\otimes a'\otimes (a,m))=k\otimes (a'\cdot a)$. We denoted $K\cdot M$ the submodule of sums of elements of the form $k\cdot m$ for $k\in K$ and $m\in M$, and similarly for $H(K)\cdot M$. By projectivity of $K$, there is a lift $\epsilon$ of $\pi'$ along $\pi$.
If we show that the cokernel of $\epsilon$ satisfies $\coker\epsilon\subset\coker\epsilon\cdot M$ we would have
\[\coker\epsilon=\coker\epsilon\cdot M=\coker\epsilon\cdot M\cdot M=0\]
showing that $\epsilon$ is surjective. Decompose an element $k\otimes a'\otimes (a,m)\in H(K)$ as
\[k\otimes a'\otimes (a,m)=k\otimes a'\otimes (a,0)+k\otimes a'\otimes (0,m)\]
and notice that by commutativity of the diagram above
\[\pi\epsilon(k\cdot a'a)=\pi'(k\cdot a'a)=k\cdot a'a\otimes 1=k\otimes a'a=\pi(k\otimes a'\otimes (a,0))\]
Therefore there is an element $x\in \ker \pi=H(K)\cdot M$ such that
\[k\otimes a'\otimes (a,0)=\epsilon(k\cdot a'a)+x\]
Thus $k\otimes a'\otimes (a,m)$ decomposes as
\[k\otimes a'\otimes (a,m)=\epsilon(k\cdot a'a)+(x+k\otimes a'\otimes (0,m))\]
with $x+k\otimes a'\otimes (0,m)\in H(K)\cdot M$. This shows that
\[\coker\epsilon \subset \coker\epsilon\cdot M\]
and therefore $\epsilon$ is surjective. Since $H(K)$ is projective, one can choose a splitting $\nu$ for $\epsilon$. The same kind of argument shows that $\coker\nu \cdot M=\coker\nu$, and therefore that $\coker\nu=0$. This shows that $\epsilon$ is an isomorphism.
\end{proof}

Using this last result, we prove that $(\overline{F},\overline{\xi})$ is an equivalence in higher simplicial degrees as well.

\begin{prop}\label{noneqeq}
The map \[(\overline{F},\overline{\xi})\colon (T^{2,1}_{\sbt}(A;M),D_{J},\eta_J)\longrightarrow (S^{2,1}_{\sbt}\mathcal{P}_{A\ltimes M},D_{L^{J}},\eta^{\alpha^J})\] is an equivalence of categories with duality in every simplicial degree.
\end{prop}

\begin{proof}
Again by \ref{equivalenceequivariant}, it is enough to show that $\overline{F}$ is in each degree fully faithful and essentially surjective. The proof that $\overline{F}$ is fully faithful is completely analogous to the one for $F$.

The only point is to see that $\overline{F}$ is essentially surjective, and again we prove equivalently that $-\otimes_AA\ltimes M$ is. By \cite{IbLars}, every diagram $X\in S^{2,1}_k\mathcal{P}_{A\ltimes M}$ is isomorphic (non-canonically) to the diagram $Y$ given by
\[Y_\theta=\bigoplus_{\rho=(0^i1^j2^{k-i-j+1})\in r(\theta)}\ker(X_{i-1<i+j-1<i+j}\longrightarrow X_{i\leq i+j-1<i+j})\]
where $r(\theta)$ is the set of retractions for the map $\theta\colon [2]\longrightarrow [k]$. The maps of $Y$ are inclusions and projections of the direct summands. Therefore it is enough to find a diagram in $S^{2,1}_k\mathcal{P}_A$ whose image is isomorphic to $Y$.
Denote the kernel above corresponding to a retraction $\rho$ by $K_\rho$. Each of these $K_\rho$ is a finitely generated projective $A\ltimes M$-module, and by the previous proposition there exist some isomorphism 
\[\epsilon_\rho\colon K_\rho\stackrel{\cong}{\longrightarrow} K_\rho\otimes_{A\ltimes M}A\otimes_AA\ltimes M \]
Then the diagram $Y\otimes_{A\ltimes M}A\in S^{2,1}_k\mathcal{P}_A$ is sent by the functor to $Y\otimes_{A\ltimes M}A\otimes_AA\ltimes M $, and the isomorphisms $\epsilon_\rho$ provide a natural isomorphism
\[\xymatrix{Y_\theta=\bigoplus_{\rho\in r(\theta)}K_\rho\ar@{-->}[d]\ar[r]^-{\oplus\epsilon_\rho}&\bigoplus_{\rho\in r(\theta)}K_\rho\otimes_{A\ltimes M}A\otimes_AA\ltimes M\ar[d]^\cong\\
Y_\theta\otimes_{A\ltimes M}A\otimes_AA\ltimes M\ar@{=}[r]& (\bigoplus_{\rho\in r(\theta)}K_\rho)\otimes_{A\ltimes M}A\otimes_AA\ltimes }\]

\end{proof}


\subsection{The equivalence $S^{2,1}_{\cdot}(A;M)\longrightarrow i\mathcal{D}T^{2,1}_{\cdot}(A;M)$}\label{model2}

Recall that we denote
\[\widehat{S}^{2,1}_{\sbt}(A;M)=\coprod_{\varphi\in Ob\mathcal{D}S^{2,1}_{\sbt}\mathcal{P}_A}\hom_A(\varphi,\varphi\otimes_A M)\cong S^{2,1}_{\sbt}(A;M)\]
The following will finish the proof of \ref{mainKR}.

\begin{prop}\label{inclusion}
The inclusion functor
\[\widehat{S}^{2,1}_{\sbt}(A;M)\longrightarrow i\mathcal{D}T^{2,1}_{\sbt}(A;M)\]
induces a $\mathbb{Z}/2$-equivalence on realizations.
\end{prop}

We prove this by factoring the inclusion through the full subcategory $i\mathcal{D}T^{2,1}_{\sbt}(A;M)^0$ of $i\mathcal{D}T^{2,1}_{\sbt}(A;M)$ on the objects of $\widehat{S}^{2,1}_{\sbt}(A;M)$, i.e. the objects of the form $(X,Y,(\phi,0))$.

\begin{lemma}\label{inclfullsubcat} The inclusion $\iota\colon i\mathcal{D}T^{2,1}_{\sbt}(A;M)^0\longrightarrow i\mathcal{D}T^{2,1}_{\sbt}(A;M)$ induces a $\mathbb{Z}/2$-equivalence on geometric realizations.
\end{lemma}

\begin{proof} We denote $D=D_L$. The inclusion functor $\iota$ is an equivalence of categories. Indeed, it is fully faithful as inclusion of a full subcategory. It is essentially surjective for the following reason.
A morphism in $i\mathcal{D}T^{2,1}_{\sbt}(A;M)$ from $(X,Y,(\phi,h))$ to $(X',Y',(\phi',h'))$ is the data of isomorphisms $a\colon X\longrightarrow X'$ and $b\colon Y'\longrightarrow Y$ represented in the commutative diagram
\[\xymatrix{D(X)\otimes_A M&D(X')\otimes_A M\\
Y\ar[u]^{h}\ar[d]_{\phi}&Y'\ar[u]^{h'}\ar[d]_{\phi'}\ar[l]_{b}\\
D(X)&D(X')\ar[l]^{D(a)}
}\]
together with two maps $f\colon X\longrightarrow X'\otimes_A M$ and $g\colon Y'\longrightarrow Y\otimes_A M$ satisfying
\[(\phi\otimes M)\circ g+h\circ b=(D(a)\otimes M)\circ h'+\widehat{J}(f)\circ\phi'\]
as maps $Y'\longrightarrow D(X)\otimes_A M$.
We define an isomorphism 
\[\kappa_{(\phi,h)}\colon(\phi,0)\longrightarrow(\phi,h)\] for every object $(X,Y,(\phi,h))$ of $i\mathcal{D}T^{2,1}_{\sbt}(A;M)$, by the diagram
\[\xymatrix{D(X)\otimes_A M&D(X)\otimes_A M\\
Y\ar[u]^{0}\ar[d]_{\phi}&Y\ar[u]^{h}\ar[d]_{\phi}\ar[l]_{\id}\\
D(X)&D(X)\ar[l]^{D(\id)}
}\]
with the maps 
\[(\kappa_f=0,\kappa_g=(\phi^{-1}\otimes M)\circ h)\]
One can check that this is indeed a morphism in $i\mathcal{D}T^{2,1}_{\sbt}(A;M)$, and therefore $\iota$ is essentially surjective.

Since $\iota$ is an equivalence of categories, and it commutes strictly with the dualities,
 the pair $(\iota,\id)$ is an equivalence of categories with duality by \ref{equivalenceequivariant}. Notice though that even if the dualities are strict, the inverse
\[r\colon i\mathcal{D}T^{2,1}_{\sbt}(A;M)\longrightarrow i\mathcal{D}T^{2,1}_{\sbt}(A;M)^0\] does not need to commute strictly with the duality. Indeed in this case it does not, since the isomorphism $\kappa$ above does \textbf{not} satisfy $\kappa_{D(\phi,h)}=D\kappa_{(\phi,h)}$ in general. In order to obtain a strictly duality preserving inverse one needs to apply $\mathcal{D}$ (again), to get a commutative diagram of strictly duality preserving functors
\[\xymatrix{\mathcal{D}(i\mathcal{D}T^{2,1}_{\sbt}(A;M)^0)\ar@<1ex>[r]^-{\mathcal{D}(\iota)} &\mathcal{D}(i\mathcal{D}T^{2,1}_{\sbt}(A;M))\ar@<1ex>[l]^-{\mathcal{D}(r)}\\
i\mathcal{D}T^{2,1}_{\sbt}(A;M)^0\ar[u]^{\simeq}\ar[r]^-{\iota}&i\mathcal{D}T^{2,1}_{\sbt}(A;M)\ar[u]_{\simeq}
}\]
where now the top row induces mutually inverse $\mathbb{Z}/2$-homotopy equivalences on the realization. The vertical maps also induce $\mathbb{Z}/2$-equivalences on the realization by \ref{strictifytwice}, and so $\iota$ does as well.
\end{proof}

\begin{lemma}
The inclusion functor 
\[\widehat{S}^{2,1}_{\sbt}(A;M)\longrightarrow i\mathcal{D}T^{2,1}_{\sbt}(A;M)^0\] induces a $\mathbb{Z}/2$-equivalence on the realization.
\end{lemma}

\begin{proof} This proof has exactly the same structure as \ref{obintoiso} above, but with more data to keep track of.

We want to show that the inclusion $\iota$ induces a levelwise equivalence on the nerve, by moving the $S^{2,1}_{\sbt}$-simplicial direction. For $k$ fixed, an element in the nerve $\mathcal{N}_ki\mathcal{D}T^{2,1}_{\sbt}(A;M)^0$ is a commutative diagram $(\underline{\phi},\underline{f},\underline{g})$ in $S^{2,1}_{\sbt}\mathcal{P}_A$ of the form
\[\resizebox{1\hsize}{!}{\xymatrix{&Y_0\!\otimes_A\!\!M&\dots&Y_{k-2}\!\otimes_A\!\!M&Y_{k-1}\!\otimes_A\!\!M\\
Y_0\ar[d]_{\phi_0}&Y_1\ar[u]^{g^1}\ar[d]_{\phi_1}\ar[l]_{b_1}&\dots\ar[l]&Y_{k-1}\ar[d]_{\phi_{k-1}}\ar[l]\ar[u]^{g^{k-1}}&Y_k\ar[d]_{\phi_k}\ar[l]_{b_k}\ar[u]^{g^k}\\
D(X_0)&D(X_1)\ar[d]_{J(f^1)}\ar[l]^{D(a_1)}&\dots\ar[l]&D(X_{k-1})\ar[d]_{J(f^{k-1})}\ar[l]&D(X_k)\ar[d]_{J(f^{k})}\ar[l]^{D(a_k)}\\
& D(X_0)\!\otimes_A\!\!M &\dots &D(X_{k-2})\!\otimes_A\!\!M&D(X_{k-1})\!\otimes_A\!\!M
}}\]
with 
\[(\phi_{i-1}\otimes_A M)\circ g^i=\widehat{J}(f^i)\circ\phi_i\colon Y_i\longrightarrow D(X_{i-1})\otimes_A M\]
for all $i$.

The idea is that since all the horizontal maps are isomorphisms, we can retract all the data on one column. However, if we want to do this equivariantly we need to work on the subdivision, and retract to the middle column.
We define a map
\[\xymatrix{r\colon \mathcal{N}_{2k+1}i\mathcal{D}T^{2,1}_{\sbt}(A;M)^0\ar[r]& \mathcal{N}_{2k+1}\widehat{S}^{2,1}_{\sbt}(A;M)\ar@{=}[d]\\
&\coprod\limits_{\varphi\in Ob\mathcal{D}S^{2,1}_{\sbt}\mathcal{P}_A}\hom_A(\varphi,\varphi\otimes_A M)^{\oplus 2k+1}}\]
that takes $(\underline{\phi},\underline{f},\underline{g})$ to the element of $\hom_A(\phi_{k+1},\phi_{k+1}\otimes_A M)^{\oplus 2k+1}$ with components $i=1,\dots,k+1$ given by
\[\xymatrix{Y_{k+1}\ar[d]_{\phi_{k+1}}\ar[rrrrr]^-{((b_{i}\dots b_{k+1})^{-1}\otimes M) g^ib_{i+1}\dots b_{k+1}} &&&&&
Y_{k+1}\otimes_A M\ar[d]^-{\phi_{k+1}\otimes M}\\
D(X_{k+1})\ar[rrrrr]_-{J(((a_{k+1}\dots a_{i+1})\otimes M) f^{i}(a_{k+1}\dots a_{i})^{-1})}&&&&&
D(X_{k+1})\otimes_A M
}\]
and components $i=k+2,\dots,2k+1$
\[\xymatrix{Y_{k+1}\ar[d]_{\phi_{k+1}}\ar[rrrrr]^-{((b_{k+2}\dots b_{i-1})\otimes M) g^i(b_{k+2}\dots b_{i})^{-1}} &&&&&
Y_{k+1}\otimes_A M\ar[d]^-{\phi_{k+1}\otimes M}\\
D(X_{k+1})\ar[rrrrr]_-{J(((a_{i}\dots a_{k+2})^{-1}\otimes M) f^{i}a_{i-1}\dots a_{k+2})}&&&&&
D(X_{k+1})\otimes_A M
}\]

This is not simplicial in the $k$ direction. However, for $k$ fixed it is simplicial in the $S^{2,1}_{\sbt}$-direction. The composite $r\circ \iota$ is clearly the identity. In order to define a homotopy between $\iota\circ r$ and the identity, consider the equivariant natural isomorphism $U$ in the category of diagrams $\mathcal{N}_{2k+1}\mathcal{D}(\mathcal{P}_A)$ from
\[\xymatrix{
Q_0\ar[d]_{\phi_0}&Q_1\ar[d]_{\phi_1}\ar[l]_{b_1}&\dots\ar[l]&Q_{2k}\ar[d]_{\phi_{2k}}\ar[l]&Q_{2k+1}\ar[d]_{\phi_{2k+1}}\ar[l]_{b_{2k+1}}\\
D(P_0)&D(P_1)\ar[l]^{D(a_1)}&\dots\ar[l]&D(P_{2k})\ar[l]&D(P_{2k+1})\ar[l]^{D(a_{2k+1})}
}\]
to the constant diagram
\[\xymatrix{
Q_{k+1}\ar[d]_{\phi_{k+1}}&Q_{k+1}\ar[d]_{\phi_{k+1}}\ar@{=}[l]&\dots\ar@{=}[l]&Q_{k+1}\ar[d]_{\phi_{k+1}}\ar@{=}[l]&Q_{k+1}\ar[d]_{\phi_{k+1}}\ar@{=}[l]\\
D(P_{k+1})&D(P_{k+1})\ar@{=}[l]&\dots\ar@{=}[l]&D(P_{k+1})\ar@{=}[l]&D(P_{k+1})\ar@{=}[l]
}\]
In the $i$-th column it is given by the pair $(a_{k+1}\dots a_{i+1},(b_{i+1}\dots b_{k+1})^{-1})$ for $i\leq k+1$, and by $((a_i\dots a_{k+2})^{-1},b_{k+2}\dots b_i)$ for $ k+1<i$.
This induces a functor $U\colon \mathcal{N}_{2k+1}\mathcal{D}(\mathcal{P}_A)\times[1]\longrightarrow \mathcal{N}_{2k+1}\mathcal{D}(\mathcal{P}_A)$.
Now define a simplicial homotopy \[H\colon \mathcal{N}_{2k+1}i\mathcal{D}T^{2,1}_{\sbt}(A;M)^0\times \Delta[1]\longrightarrow \mathcal{N}_{2k+1}i\mathcal{D}T^{2,1}_{\sbt}(A;M)^0\] by sending $((\underline{\phi},\underline{f},\underline{g}),\sigma)$ in degree $n$ to the diagram
\[Cat([2],[n])\stackrel{\id\times ev_1}{\longrightarrow}Cat([2],[n])\times [n]\stackrel{\underline{\phi}\times \sigma}{\longrightarrow}\mathcal{N}_{2k+1}\mathcal{D}(\mathcal{P}_A)\times [1]\stackrel{U}{\longrightarrow}\mathcal{N}_{2k+1}\mathcal{D}(\mathcal{P}_A)\]
together with the maps 
\[\resizebox{1\hsize}{!}{$H_{g^i}(\theta,\sigma)\!=\!\left\{\begin{array}{lll}g^{i}_{\theta}& \sigma(\theta(1))\!=0\!\\
(((b_{i}\dots b_{k+1})^{-1}\otimes M) g^ib_{i+1}\dots b_{k+1})_{\theta}& \sigma(\theta(1))\!=\!1 \mbox{ and }i\!\leq\!k\!+\!1\\
((b_{k+2}\dots b_{i-1})\otimes M) g^i(b_{k+2}\dots b_{i})^{-1}_{\theta}& \sigma(\theta(1))\!=\!1 \mbox{ and }k\!+\!1\!<\!i
\end{array}\right.$}\]
and similarly for $H_{f^i}(\theta,\sigma)$. Here $ev_1\colon Cat([2],[n])\longrightarrow [n]$ sends $\theta$ to $\theta(1)$. The homotopy is simplicial since for all $\kappa\colon [m]\longrightarrow [n]$ we have
\[\kappa(ev_1(\theta))=ev_1(\kappa\circ\theta)\]
for all $\theta\colon[2]\longrightarrow [m]$. Moreover it is equivariant, in the sense that
\[H(D(\underline{\phi},\underline{f},\underline{g}),\sigma)=D(H(\underline{\phi},\underline{f},\underline{g},\sigma))\]
We show this carefully. Denote $Z=(\underline{\phi},\underline{f},\underline{g})$. The left hand side at the object $\theta\colon[2]\longrightarrow [n]$ is
\[H(D(Z),\sigma)_{\theta}=U(D(Z)_{\theta},\sigma(ev_1(\theta)))\]
The right hand-side is
\[D(H(Z,\sigma))_{\theta}=D(U(Z_{\overline{\theta}},\sigma(ev_1(\overline{\theta}))))=U(D(Z_{\overline{\theta}}),\sigma(n-ev_1(\overline{\theta})))\]
where the second equality follows from the equivariancy of the natural transformation $U$. But now
\[n-ev_1(\overline{\theta})=n-\overline{\theta}(1)=n-(n-\theta(2-1))=\theta(1)\]
\end{proof}


\newpage

\section{Review of $\THH$ and definition of $\THR$}

Given an abelian group $A$ and a simplicial set $X$ we denote $A[X]$ the simplicial abelian group which in degree $k$ is the $A$-linear combinations of the points in $X_k$
\[A[X]_k=A[X_k]=\bigoplus_{x\in X_k}A\cdot x\]
When $X$ is pointed, we denote $A(X)=A[X]/A\cdot\ast$.

We will be particularly interested in the case where $X$ is the $i$-th simplicial sphere
\[S^i=S^1\wedge\dots\wedge S^1 \ \ \ \ \ \ \ (i\mbox{ times})\]
Here $S^1$ is the simplicial circle $S^1=\Delta[1]/\partial\Delta[1]$.
We note that in this case $A(S^i)$ is the $i$-iterated Bar construction on $A$.


\subsection{Bimodules over categories enriched in symmetric spectra}\label{sectmodsym}

We denote $Sp^{\Sigma}$ the symmetric monoidal category of symmetric ring spectra. We refer to \cite{schwede} and \cite{HSS} for general facts about $Sp^{\Sigma}$, e.g. for the smash product. The input for $\THH$ is a category $\mathcal{C}$ enriched in $Sp^{\Sigma}$, together with a "bimodule over $\mathcal{C}$". This approach is completely analogous to $\THH$ for bimodules over ring functors as defined in \cite{ringfctrs}.

Given two objects $c$ and $d$ of a category $\mathcal{C}$ enriched in $Sp^\Sigma$, we denote $\mathcal{C}(c,d)$ the symmetric spectrum of morphisms from $c$ to $d$, and $\mathcal{C}(c,d)_i$ it's $i$-th space, for every natural number $i$. If $\mathcal{C}$ and $\mathcal{C}'$ are two categories enriched in $Sp^{\Sigma}$, we denote $\mathcal{C}\wedge \mathcal{C}'$ the category enriched in $Sp^{\Sigma}$ with objects $ObC\times Ob C'$ and morphism spectra
\[(\mathcal{C}\wedge \mathcal{C}')((c,c'),(d,d'))=\mathcal{C}(c,d)\wedge \mathcal{C}'(c',d')\]

\begin{defn} A \textbf{bimodule structure on $\mathcal{C}$} is an enriched functor $\mathcal{M}\colon\mathcal{C}^{op}\wedge\mathcal{C}\longrightarrow Sp^{\Sigma}$.
The \textbf{left action} map $l$ is defined by the diagram
\[\xymatrix{\mathcal{C}(d,e)\wedge\mathcal{M}(c,d)\ar[d]_{\cong}\ar[r]^l&\mathcal{M}(c,e)\\
\mathbb{S}\wedge\mathcal{C}(d,e)\wedge\mathcal{M}(c,d)\ar[d]_{\id_c\wedge\mathcal{C}(d,e)\wedge\mathcal{M}(c,d)}&Sp^{\Sigma}(\mathcal{M}(c,d),\mathcal{M}(c,e))\wedge\mathcal{M}(c,d)\ar[u]_{ev}\\
\mathcal{C}(c,c)\wedge\mathcal{C}(d,e)\wedge\mathcal{M}(c,d)\ar@{=}[r]&\hom_{\mathcal{C}^{op}\wedge\mathcal{C}}((c,d),(c,e))\wedge\mathcal{M}(c,d)\ar[u]_{\mathcal{M}\wedge\mathcal{M}(c,d)}
}\]
It is determined by the maps $l\colon \mathcal{C}(d,e)_i\wedge\mathcal{M}(c,d)_j\longrightarrow\mathcal{M}(c,e)_{i+j}$ defined by \[l(f\wedge m)=\mathcal{M}(\id_c\wedge f)(m)\]
and denoted $l(f\wedge m)=f\cdot m$.
The \textbf{right action} map
$r\colon \mathcal{M}(d,e)\wedge\mathcal{C}(c,d)\longrightarrow \mathcal{M}(c,e)$ is defined by
\[\xymatrix{\mathcal{M}(d,e)\wedge\mathcal{C}(c,d)\ar[d]_{\cong}\ar[r]^r&\mathcal{M}(c,e)\\
\mathcal{M}(d,e)\wedge\mathcal{C}(c,d)\wedge\mathbb{S}\ar[d]_{\mathcal{M}(d,e)\wedge\mathcal{C}(c,d)\wedge\id_e}&\mathcal{M}(d,e)\wedge Sp^{\Sigma}(\mathcal{M}(d,e),\mathcal{M}(c,e))\ar[u]_{ev\circ\gamma}\\
\mathcal{M}(d,e)\wedge\mathcal{C}(c,d)\wedge \mathcal{C}(e,e)\ar@{=}[r]&\mathcal{M}(d,e)\wedge\hom_{\mathcal{C}^{op}\wedge\mathcal{C}}((d,e),(c,e))\ar[u]_{\mathcal{M}(d,e)\wedge\mathcal{M}}
}\]
and we denote its components by
\[r(m\wedge g)=m\cdot g=\mathcal{M}(g\wedge\id_e)(m)\]
\end{defn}

Notice that since $\mathcal{M}$ preserve composition, the right and the left action satisfy the following compatibility condition
\[\xymatrix{ \mathcal{C}(d,f)\wedge\mathcal{M}(e,d)\wedge\mathcal{C}(c,e)\ar[r]^-{\mathcal{C}(d,f)\wedge r}\ar[d]_-{l\wedge\mathcal{C}(c,e)} &
\mathcal{C}(d,f)\wedge\mathcal{M}(c,d)\ar[d]^--{l}\\
 \mathcal{M}(e,f)\wedge\mathcal{C}(c,e)\ar[r]_-{r} &\mathcal{M}(c,f)
}\]

\begin{ex}
The very first example of a bimodule over a category $\mathcal{C}$ enriched in $Sp^{\Sigma}$ is the Hom-functor $\mathcal{C}^{op}\wedge\mathcal{C}\longrightarrow Sp^{\Sigma}$ which sends a pair of objects $(c,d)$ to the morphism spectrum $\mathcal{C}(c,d)$.
\end{ex}

We will mostly be concerned with the following class of linear examples.

\begin{ex}\label{homexeil} Let $C$ be a category enriched in abelian groups. Composing with the Eilenberg-MacLane functor $H\colon Ab\longrightarrow Sp^{\Sigma}$, we obtain an enrichment in $Sp^{\Sigma}$. More precisely, define a category $HC$ with same objects as $C$, and with morphisms spectra $HC(c,d)$ having $i$-th space $HC(c,d)_i$ the realization of
\[C(c,d)\otimes\mathbb{Z}(S^i)\cong C(c,d)(S^{i})\]
The symmetric action is given by permutation of the smash components of $S^{i}$.

A bimodule on $C$ is an enriched functor $M\colon C^{op}\otimes C\longrightarrow Ab$, where $C^{op}\otimes C$ has objects $ObC\times ObC$ and morphism abelian groups $C^{op}(c,d)\otimes C(c',d')$. It induces a bimodule $HM$ over $HC$ defined by the spectra $HM(c,d)$ with $i$-th space
\[M(c,d)\otimes\mathbb{Z}(S^i)\cong M(c,d)(S^{i})\]

If $A$ is a ring, and $M$ a bimodule over $A$, the functor $M\colon\mathcal{P}_{A}^{op}\otimes \mathcal{P}_{A}\longrightarrow Ab$ defined by the abelian group of right module homomorphisms
\[M(P,Q)=\hom_{A}(P,Q\otimes_A M)\]
is a bimodule over the category $\mathcal{P}_A$, and $HM$ is a bimodule over $H\mathcal{P}_A$.
\end{ex}

\begin{defn} Let $\mathcal{M}$ and $\mathcal{M}'$ be two bimodules on $\mathcal{C}$ and $\mathcal{C}'$ respectively. A \textbf{morphism of bimodules} is a pair $(F,\Phi)\colon (\mathcal{C},\mathcal{M})\longrightarrow (\mathcal{C}',\mathcal{M}')$ of an enriched functor $F\colon \mathcal{C}\longrightarrow \mathcal{C}'$ and an enriched natural transformation $\Phi\colon \mathcal{M}\Rightarrow \mathcal{M}'\circ(F^{op}\wedge F)$.
\end{defn}

\begin{rem} We remember (see e.g. \cite{kelly}) that an enriched natural transformation $U\colon G\Rightarrow G'$ between enriched functors $G,G'\colon \mathcal{C}\longrightarrow Sp^{\Sigma}$ consists of maps of spectra
\[U_{c}\colon\mathbb{S}\longrightarrow Sp^{\Sigma}(G(c),G'(c)),\]
one for every object $c\in \mathcal{C}$, satisfying a certain naturality condition. The map $U_{c}$ is determined by an element of the $0$-th space of the mapping spectrum $Sp^{\Sigma}(G(c),G'(c))$, that is, a map of spectra $U_c\colon G(c)\longrightarrow G'(c)$, and naturality for $U$ is commutativity of
\[\xymatrix{\mathcal{C}(c,d)\ar[d]_{G'}\ar[r]^-{G}&Sp^{\Sigma}(Gc,Gd)\ar[d]^{(-)\circ U_d}\\
Sp^{\Sigma}(G'c,G'd)\ar[r]_-{U_c\circ(-)}&Sp^{\Sigma}(Gc,G'd)
}\]
in $Sp^{\Sigma}$. We  refer to the maps $U_c\colon G(c)\longrightarrow G'(c)$ as an enriched natural transformation.
\end{rem}

In our specific case, $\Phi\colon \mathcal{M}\Rightarrow \mathcal{M}'\circ(F^{op}\wedge F)$ consists of maps of spectra $\Phi_{(c,d)}\colon \mathcal{M}(c,d)\longrightarrow \mathcal{M}'(Fc,Fd)$, and naturality of $\Phi$ is commutativity of the diagram
\[\xymatrix{\mathcal{C}(d,e)\wedge\mathcal{M}(c,d)\ar[d]_{F\wedge\Phi}\ar[r]^-{l}&\mathcal{M}(c,e)\ar[d]^{\Phi}\\
\mathcal{C}'(Fd,Fe)\wedge\mathcal{M}'(Fc,Fd)\ar[r]_-{l'}&\mathcal{M}'(Fc,Fe)
}\]
The situation for the right actions is completely similar.

Categories enriched in $Sp^{\Sigma}$ with bimodules together with morphism of bimodules define a category, and $\THH$ is going to be a functor from this category to spaces.


\subsection{The dihedral category of finite sets and injective maps}

The functors $\THH$ and $\THR$ are defined in terms of homotopy colimits of functors over a certain  dihedral category $I[-]$, first introduced by Bökstedt in \cite{bok}.
We recall the definition of a cyclic object in the sense of Connes, and of a dihedral object. We refer to \cite[§6.1-6.3]{loday} for further details.
\begin{defn}
A \textbf{cyclic object} in a category $\mathcal{D}$ is a simplicial object $E[-]\colon\Delta^{op}\longrightarrow \mathcal{D}$ together with maps $t_k\colon E[k]\longrightarrow E[k]$ with $t_k^{k+1}=\id$, and compatible with the simplicial structure in the following way.
\[\begin{array}{llll}d_lt_k=t_{k-1}d_{l-1} & 0<l\leq k\\
d_0t_k=d_k\\
s_lt_k=t_{k-1}s_{l-1} & 0<l\leq k\\
s_0t_k=t_{k+1}^2s_k
\end{array}\]

A \textbf{dihedral structure} on a cyclic object $E[-]$ is a family of maps $\omega_k\colon E[k]\longrightarrow E[k]$ satisfying $\omega_{k}^{2}=\id$ and compatible with the cyclic structure in the following way,
\[\begin{array}{lll}d_l\omega_k=\omega_{k-1}d_{k-l} & 0\leq l\leq k\\
s_l\omega_k=\omega_{k+1}s_{k-l} & 0\leq l\leq k\\
\omega_kt_k=t_{k}^{-1}\omega_k
\end{array}\]
\end{defn}
Notice that a dihedral object is a cyclic object with the additional structure of a real object as defined in \ref{defrealob} with a certain compatibility condition.
We will mostly be interested in cyclic and dihedral objects in the categories of spaces and categories. We refer to these respectively as cyclic/dihedral spaces and cyclic/dihedral categories.
The geometric realizations of a cyclic space and a cyclic category carry a natural $S^1$-action, see \cite[7.1.4]{loday}. In the dihedral case, this action is compatible with the $\mathbb{Z}/2$-action induced by the real structure, in the sense that it intertwines into a $O(2)=\mathbb{Z}/2\rtimes S^1$-action.

The index category $I$ has the natural numbers as objects, and the set of morphisms from $i$ to $j$ consists of all injective maps
\[\alpha\colon \{1,\dots, i\}\longrightarrow \{1,\dots, j\}\]
The empty set $0$ is an initial object of $I$. 

The category $I$ has a strict monoidal structure $+\colon I\times I\longrightarrow I$ sending a pair of objects $(i,j)$ to $i+j$ and maps $\alpha\colon i\longrightarrow i'$, $\beta\colon j\longrightarrow j'$ to
\[(\alpha+\beta)(s)=\left\{\begin{array}{ll}\alpha(s) & 1\leq s\leq i\\
\beta(s-i)+i' & i+1\leq s\leq j
\end{array}\right.\]

Also, every object $j$ has a canonical involution $\omega^j\colon j\longrightarrow j$ defined by
\[\omega^j(s)=j-s+1\]
This induces an endofunctor $\omega\colon I\longrightarrow I$ which is the identity on objects, and conjugation on morphisms. A morphism $\alpha\colon i\longrightarrow j$ is mapped to $\alpha^{\omega}=\omega^j\alpha\omega^i$. Notice that $(\alpha+\beta)^{\omega}=\beta^{\omega}+\alpha^\omega$

The monoidal structure allows us to define the cyclic bar construction of $I$. This is the simplicial category $I[-]$ given in degree $k$ by the $k+1$-fold cartesian product $I[k]=I^{k+1}$, with faces $d_i\colon I[k]\longrightarrow I[k-1]$ and degeneracy $s_j\colon I[k]\longrightarrow I[k+1]$ functors given on objects by
\[\begin{array}{lll}
d_l(i_0,\dots,i_k)=\left\{\begin{array}{ll}(i_0,\dots,i_{l-1},i_l+i_{l+1},\dots,i_k)\\
(i_k+i_0,i_1,\dots,i_{k-1})
\end{array}\right. &
\begin{array}{ll}0\leq l<k\\
i=k
\end{array}\\
\ \\
s_l(i_0,\dots,i_k)=(i_0,\dots,i_l,0,i_{l+1},\dots,i_k) & \ \ 0\leq l\leq k
\end{array}\]
and similarly on morphisms.
These functors satisfy the standard simplicial identities. In addition, the functors $t_k\colon I[k]\longrightarrow I[k]$, defined on objects by cyclic permutation
\[t_k(i_0,\dots,i_k)=(i_k,i_0,\dots,i_{k-1}),\]
give $I[-]$ the structure of a cyclic category.
The involution on $I$, induces functors $\omega_k\colon I[k]\longrightarrow I[k]$ defined on an object $\underline{i}=(i_0,\dots,i_k)$ by
\[\omega_k(i_0,\dots,i_k)=(i_0,i_k,i_{k-1},\dots,i_1)\]
and mapping a $(k+1)$-tuple of maps $(\alpha_0,\dots,\alpha_k)$ to
\[\omega_k(\alpha_0,\dots,\alpha_k)=(\alpha_{0}^{\omega},\alpha_{k}^{\omega},\alpha_{k-1}^{\omega},\dots,\alpha_{1}^{\omega})\]
The cyclic category $I[-]$ equipped with these extra maps defines a dihedral category.

It is going to be useful to describe the fixed points of the involution on $I[-]$. Let $G=\mathbb{Z}/2$. If $\mathcal{D}$ is a category with involution $\omega\colon \mathcal{D}\longrightarrow\mathcal{D}$, we denote by $\mathcal{D}^G$ it's fixed point category, that is the subcategory of $\mathcal{D}$ whose objects and morphisms are (strictly) fixed by the involution. 
In the case of $I$, its fixed point category $I^G$ has objects the natural numbers and morphisms the equivariant maps, $\alpha\colon i\longrightarrow j$ satisfying $\alpha^{\omega}=\alpha$.

\begin{prop}\label{fixedIk}
The functor $I^{G}\times I^k\times I^G\longrightarrow I[2k+1]^G$ that sends $(\alpha_0,\dots,\alpha_{k+1})$ to
\[(\alpha_0,\dots,\alpha_k,\alpha_{k+1},\alpha_{k}^{\omega},\alpha_{k-1}^{\omega},\dots,\alpha_{1}^{\omega})\]
is an isomorphism of categories, with inverse the projection onto the first $k+2$ components.
\end{prop}


\subsection{Topological Hochschild homology }\label{thhsec}

We recall the definition of Bökstedt's functor $\THH$. For details one can see \cite{ringfctrs} or \cite{Dundasbook}, and \cite{Ibtraces} for a nice survey.
Let $\mathcal{C}$ be a category enriched in $Sp^{\Sigma}$ and $\mathcal{M}\colon \mathcal{C}^{op}\wedge \mathcal{C}\longrightarrow Sp^{\Sigma}$ a bimodule over $\mathcal{C}$. Given an object $\underline{i}=(i_0,\dots,i_k)\in I[k]$, we denote by $V(\mathcal{C};\mathcal{M},\underline{i})$ the pointed space
\[V(\mathcal{C};\mathcal{M},\underline{i})=\bigvee \mathcal{M}(c_0,c_k)_{i_0}\wedge\mathcal{C}(c_1,c_0)_{i_1}\wedge\mathcal{C}(c_2,c_1)_{i_2}\wedge\dots\wedge\mathcal{C}(c_k,c_{k-1})_{i_k}\]
where the wedge runs over $(k+1)$-tuples of objects $c_0,\dots,c_k$ of $\mathcal{C}$. Inserting an identity of $\mathcal{C}$ after the $i_l$-factor, induces a map
\[\overline{s}_l\colon V(\mathcal{C};\mathcal{M},\underline{i})\longrightarrow V(\mathcal{C};\mathcal{M},s_l\underline{i})\]
for all $0\leq k$. Similarly, there are maps
\[\overline{d}_l\colon V(\mathcal{C};\mathcal{M},\underline{i})\longrightarrow V(\mathcal{C};\mathcal{M},d_l\underline{i})\]
defined as follows. For $1\leq l\leq k-1$, they are induced by composition in $\mathcal{C}$ of the $i_l$ and $i_{l+1}$ smash factor. The map $d_0$ is defined by the right action map of the bimodule on the $i_0$ and $i_1$ factors. The map $d_k$ is given by the left action map on the $i_k$ and $i_0$ factors. More precisely, it is the composition
\[\xymatrix{\bigvee \mathcal{M}(c_0,c_k)_{i_0}\wedge\mathcal{C}(c_1,c_0)_{i_1}\wedge\dots\wedge\mathcal{C}(c_k,c_{k-1})_{i_k}\ar[d]^\cong \\
\bigvee\mathcal{C}(c_k,c_{k-1})_{i_k}\wedge \mathcal{M}(c_0,c_k)_{i_0}\wedge\mathcal{C}(c_1,c_0)_{i_1}\wedge\dots\wedge\mathcal{C}(c_{k-1},c_{k-2})_{i_{k-1}}\ar[d]^{l\wedge \id\wedge\dots\wedge\id}\\
\bigvee \mathcal{M}(c_0,c_{k-1})_{i_k+i_0}\wedge\mathcal{C}(c_1,c_0)_{i_1}\wedge\dots\wedge\mathcal{C}(c_{k-1},c_{k-2})_{i_{k-1}}
}\]
of the canonical isomorphism permuting the smash factor, and of the $(i_k,i_0)$-component of the left module structure map.
When $\mathcal{M}=\mathcal{C}$ is the Hom-bimodule, cyclic permutation of the smash factor induces a map
\[\overline{t}_k\colon V(\mathcal{C};\mathcal{C},\underline{i})\longrightarrow V(\mathcal{C};\mathcal{C},t_k\underline{i})\]

For every natural number $k$, define a functor $\mathcal{G}_k(\mathcal{C};\mathcal{M})\colon I[k]\longrightarrow \Top_\ast$ by sending $\underline{i}=(i_0,\dots,i_k)$ to the pointed mapping space
\[\mathcal{G}_k(\mathcal{C};\mathcal{M})(\underline{i})=\map_{\ast}(S^{i_0}\wedge\dots\wedge S^{i_k},V(\mathcal{C};\mathcal{M},\underline{i}))\]
On morphisms, this functor is defined using the structure maps of the morphisms spectra. We refer to \cite[IV-1.2.1]{Dundasbook} for the details.
We denote the homotopy colimit of this functor by
\[\THH_k(\mathcal{C};\mathcal{M})=\hocolim_{I[k]} \mathcal{G}_k(\mathcal{C};\mathcal{M})\]
Composition with the maps $\overline{d}_l$,$\overline{s}_l$ and $\overline{t}$ above induce natural transformations
\[\begin{array}{lll}\overline{d}_l\colon \mathcal{G}_k(\mathcal{C};\mathcal{M})\longrightarrow \mathcal{G}_{k-1}(\mathcal{C};\mathcal{M})\circ d_l\\
\overline{s}_l\colon \mathcal{G}_k(\mathcal{C};\mathcal{M})\longrightarrow \mathcal{G}_{k+1}(\mathcal{C};\mathcal{M})\circ s_l\\
\overline{t}_k\colon \mathcal{G}_k(\mathcal{C};\mathcal{C})\longrightarrow \mathcal{G}_k(\mathcal{C};\mathcal{C})\circ t_k
\end{array}\]
These natural transformation followed by the natural maps on homotopy colimits induce maps
\[\begin{array}{lll}d_l\colon \THH_k(\mathcal{C};\mathcal{M})\longrightarrow \hocolim_{I[k]} \mathcal{G}_{k-1}(\mathcal{C};\mathcal{M})\circ d_l\longrightarrow\THH_{k-1}(\mathcal{C};\mathcal{M})\\
s_l\colon \THH_k(\mathcal{C};\mathcal{M}) \longrightarrow \hocolim_{I[k]} \mathcal{G}_{k+1}(\mathcal{C};\mathcal{M})\circ s_l\longrightarrow\THH_{k+1}(\mathcal{C};\mathcal{M})\\
t_k\colon\THH_k(\mathcal{C};\mathcal{C})\longrightarrow\hocolim_{I[k]} \mathcal{G}_k(\mathcal{C};\mathcal{C})\circ t_k\longrightarrow\THH_k(\mathcal{C};\mathcal{C})
\end{array}\]
The maps $d_l$ and $s_l$ make $[k]\mapsto \THH_k(\mathcal{C};\mathcal{M})$ into a simplicial space. When $\mathcal{M}=\mathcal{C}$ is the Hom-bimodule, the maps $t_k$ make $[k]\mapsto\THH_k(\mathcal{C};\mathcal{C})$ into a cyclic space. Therefore it's realization has a natural $S^1$-action.

\begin{defn}
Let $\mathcal{C}$ be a category enriched in symmetric spectra and $\mathcal{M}$ a bimodule over $\mathcal{C}$. The \textbf{topological Hochschild homology of $\mathcal{C}$ with coefficients in $\mathcal{M}$} is the topological space defined as the fat geometric realization
\[\THH(\mathcal{C};\mathcal{M})=|[k]\mapsto \THH_k(\mathcal{C};\mathcal{M})|\]
The \textbf{topological Hochschild homology of $\mathcal{C}$ (with coefficients in $\mathcal{C}$)} is the $S^1$-space defined as the geometric realization
\[\THH(\mathcal{C})=|[k]\mapsto \THH_k(\mathcal{C};\mathcal{C})|\]
\end{defn}

A map of bimodules $(F,\Phi)\colon (\mathcal{C},\mathcal{M})\longrightarrow (\mathcal{C}',\mathcal{M}')$ induces in the obvious way maps $V(\mathcal{C};\mathcal{M},\underline{i})\longrightarrow V(\mathcal{C}';\mathcal{M}',\underline{i})$, and hence maps $THH_k(\mathcal{C};\mathcal{M})\longrightarrow THH_k(\mathcal{C}';\mathcal{M}')$ compatible with the simplicial structure. All in all, a morphism $(F,\Phi)$ induces
\[(F,\Phi)_\ast\colon \THH(\mathcal{C};\mathcal{M})\longrightarrow \THH(\mathcal{C}';\mathcal{M}')\]
This map is $S^1$-equivariant when $\mathcal{M}=\mathcal{C}$ and $\mathcal{M}'=\mathcal{C}'$.


\subsection{Basic equivariant homotopy theory}\label{secghtpytheory}

Here is a quick review of basic notion of $\mathbb{Z}/2$-equivariant homotopy theory, mainly to fix some terminology. We denote $G=\mathbb{Z}/2$.

\begin{defn}
Let $\nu=(\nu_1,\nu_2)$ be a pair of integers, and $X$ a $G$-space. We say that \textbf{$X$ is $\nu$-connected} if it is $\nu_1$-connected as a non-equivariant space, and the fixed points space $X^G$ is $\nu_2$-connected.

A $G$-equivariant map $f\colon X\longrightarrow Y$ is $\nu$-connected if it is $\nu_1$-connected as a map of topological spaces, and its restriction to the fixed points $f^G\colon X^G\longrightarrow Y^G$ is $\nu_2$-connected; we remember that a map of spaces is $n$-connected if its homotopy fiber is $(n-1)$-connected, that is if the first $n-1$ homotopy groups are trivial.
\end{defn}

\begin{defn}
A \textbf{$G$-CW-complex} is a CW-complex $X$ together with a $G$-action via cellular maps such that $X^G$ is a subcomplex of $X$.

The \textbf{dimension} of a $G$-CW-complex $X$ is the pair of natural number
\[\dim_GX=(\dim X,\dim X^G)\]
given by the dimensions of $X$ and $X^G$ as CW-complexes.
\end{defn}

We denote $[X,Y]_G$ the set of based $G$-homotopy classes of based $G$-maps between $X$ and $Y$.

\begin{theorem}[$G$-Whitehead theorem,{\cite[2.7]{adams}},{\cite[3.4]{alaska}}] $ $\newline
\noindent Let $f\colon X\longrightarrow Y$ be a $\nu$-connected based  map of $G$-CW-complexes. Then for all $G$-CW-complex $K$, the induced map
\[f_\ast\colon [K,X]_G\longrightarrow[K,Y]_G\]
is a bijection if $\dim_GK<\nu$, and a surjection if $\dim_GK\leq\nu$.
\end{theorem}

We denote $\map_\ast(X,Y)^G$ the space of based $G$-maps with the compactly generated compact-open topology. This is equal to the fixed points of the mapping space $\map_\ast(X,Y)$ with conjugation action.

\begin{cor}\label{eqmappingspace}
Let $f\colon X\longrightarrow Y$ be a $\nu$-connected based  map of $G$-CW-complexes. Then for all $G$-CW-complex $K$, the induced map
\[(f_\ast)^G\colon \map_\ast(K,X)^G\longrightarrow\map_\ast(K,Y)^G \]
is $min\{\nu_1-\dim K,\nu_2-\dim K^G\}$-connected.
\end{cor}

\begin{proof}
Consider the natural bijection
\[\resizebox{1\hsize}{!}{$\pi_k(\map_\ast(K,X)^G)=[S^k,\map_\ast(K,X)^G]=[S^k,\map_\ast(K,X)]_G\cong [S^k\wedge K,X]_G$}\]
where the second equality holds because $S^k$ has trivial action, and $S^k\wedge K$ has diagonal action. Under this bijection, our map corresponds to
\[f_\ast\colon[S^k\wedge K,X]_G\longrightarrow[S^k\wedge K,Y]_G\]
of the $G$-Whitehead theorem,
which is a bijection for $\dim_G(S^k\wedge K)<\nu$ and a surjection for $\dim_G(S^k\wedge K)\leq\nu$. But now,
\[\dim_G(S^k\wedge K)=(k,k)+\dim_G(K)\]
and therefore $\pi_k(f_{\ast}^G$) is an isomorphism for 
\[k<min\{\nu-\dim_G(K)\}=min\{\nu_1-\dim K,\nu_2-\dim K^G\}\]
and surjective for $k\leq min\{\nu_1-\dim K,\nu_2-\dim K^G\}$.
\end{proof}

For $V$ a $G$-representation, we denote $S^V$ its one point compactification with induced $G$-action, and $\Omega^V X=\map_\ast(S^V,X)$ the mapping space with conjugation action. For a space $Y$, we denote $\conn(Y)$ the biggest integer $k$ such that $Y$ is $k$-connected.

\begin{theorem}[$G$-suspension theorem,{\cite[3.3]{adams}},{\cite[2.11]{shah}}]\label{Gsuspthm}
The map \[\mu\colon X\longrightarrow \Omega^V(S^V\wedge X)\] adjoint to the identity of $S^V\wedge X$ is 
\[\nu=(2\conn(X)+1,\min\{\conn(X),2\conn(X^G)+1\})\]
connected.
\end{theorem}

\begin{cor}\label{corgsusp}
Let $X$ and $Y$ be two $G$-CW-complexes, and $V$ a representation of $G$. Then the map
\[(-)\wedge S^V\colon\map_\ast(X,Y)^G\longrightarrow\map_\ast(X\wedge S^V,Y\wedge S^V)^G\]
that sends an equivariant map $f$ to $f\wedge S^V$ has connectivity
\[\resizebox{1\hsize}{!}{$\min\{2\conn(Y)-\dim(X)+1,\conn(Y)-\dim(X^G),
2\conn(Y^G)-\dim(X^G)+1\}$}\]
\end{cor}

\begin{proof}
Notice that there is a commutative diagram
\[\xymatrix{\map_\ast(X,Y)^G\ar[dr]_-{(-)\circ\mu}\ar[r]^-{(-)\wedge S^V}&\map_\ast(X\wedge S^V,Y\wedge S^V)^G\ar[d]^{\cong}\\
&\map_\ast(X,\Omega^V(Y\wedge S^V))^G
}\]
where the connectivity of $\mu$ is given by the $G$-suspension theorem. Computing the connectivity of $(-)\circ\mu$ using corollary \ref{eqmappingspace} gives the formula above.
\end{proof}

We finish this section by describing how to construct a $G=\mathbb{Z}/2$-action on wedges and products of $G$-spaces.

Suppose that $\{X_\lambda\}_{\lambda\in\Lambda}$ is a family of based spaces indexed over a finite set $\Lambda$ with involution $\omega\colon \Lambda\longrightarrow \Lambda$, and suppose that there are based maps $D_\lambda\colon X_{\lambda}\longrightarrow X_{\omega\lambda}$ satisfying $D_{\omega \lambda}\circ D_{\lambda}=\id_{X_{\lambda}}$. The wedge $\bigvee_{\lambda\in\Lambda} X_\lambda$ admits a $G$-action described as follows. Let $(x,\lambda)$ be an element of the wedge, indicating that $x$ belongs to the $\lambda$-component of the wedge. Then the action sends $(x,\lambda)$ to
\[D(x,\lambda)=(D_\lambda x,\omega\lambda)\]
Similarly, there is an action on the product $\prod_{\lambda\in \Lambda}X_\lambda$, sending a sequence $\underline{x}=\{x_\lambda\}$ to the sequence $D\underline{x}$ with $\lambda$-component
\[(D\underline{x})_\lambda=D_{\omega\lambda}x_{\omega\lambda}\in X_{\omega^2\lambda}=X_\lambda\]

The inclusion $\bigvee_{\lambda\in \Lambda}X_\lambda\longrightarrow\prod_{\lambda\in \Lambda}X_\lambda$ is equivariant with respect to these actions, and we want to know how connected this map is on the fixed points. Notice that if $\lambda\in\Lambda^G$ is a fixed point, the map $D_\lambda$ defines a $G$-action on $X_\lambda$. The fixed points of the wedge are then
\[(\bigvee_{\lambda\in \Lambda}X_\lambda)^G=\bigvee_{\lambda\in \Lambda^G}(X_\lambda)^G\]
The description of the fixed points of the product is slightly more involved. For every "free element" $\lambda\in\Lambda\backslash\Lambda^G$, chose one representative of its orbit. Denote the set of these representatives by $\Lambda_f$. Then the fixed points of the product are
\[(\prod_{\lambda\in \Lambda}X_\lambda)^G=(\prod_{\lambda\in \Lambda^G}(X_\lambda)^G)\times \prod_{\lambda\in \Lambda_f}X_\lambda\]
Notice that for different choices of representatives, the map $D_\lambda\colon X_{\lambda}\longrightarrow X_{\omega\lambda}$ is a homeomorphism.
Under these identification, the map $\bigvee_{\lambda\in \Lambda}X_\lambda\longrightarrow\prod_{\lambda\in \Lambda}X_\lambda$ on the fixed points is the inclusion in the first factor
\[\bigvee_{\lambda\in \Lambda^G}(X_\lambda)^G\longrightarrow (\prod_{\lambda\in \Lambda^G}(X_\lambda)^G)\times \prod_{\lambda\in \Lambda_f}X_\lambda\]
In particular, this gives the following description for the connectivity of the inclusion.

\begin{prop}\label{wedgesintoproducts}
Suppose that every space $X_\lambda$ is $n$-connected, and that for $\lambda\in\Lambda^G$ the fixed points space $X_{\lambda}^G$ is $m$-connected. Then the inclusion
\[\bigvee_{\lambda\in \Lambda}X_\lambda\longrightarrow\prod_{\lambda\in \Lambda}X_\lambda\]
is $(2n+1, \min\{n,2m+1\})$-connected
\end{prop}

\begin{proof}
By the Blakers-Massey's theorem, (see e.g. \cite[2.3]{calculusII}), the inclusion of wedges into products is non-equivariantly $(2n+1)$-connected. On the fixed points, we factor the inclusion as composition of the two inclusions
\[\bigvee_{\lambda\in \Lambda^G}(X_\lambda)^G\longrightarrow\prod_{\lambda\in \Lambda^G}(X_\lambda)^G \longrightarrow (\prod_{\lambda\in \Lambda^G}(X_\lambda)^G)\times \prod_{\lambda\in \Lambda_f}X_\lambda\]
The first map is $(2m+1)$-connected as inclusion of wedges into products of $m$-connected spaces again by Blakers-Massey. The second one is as connected as $\prod_{\lambda\in \Lambda_f}X_\lambda$ is, that is $n$.
\end{proof}

As a consequence, we have the following useful result. If $K$ is a $G$-space, and $\{X_\lambda\}$ is a family of spaces as above, one can consider the family of mapping spaces $\{\map_\ast(K,X_\lambda)\}$ with maps
\[\widehat{D}_\lambda\colon \map_\ast(K,X_\lambda)\longrightarrow \map_\ast(K,X_{\omega\lambda})\]
defined by conjugation, $\widehat{D}_\lambda(f)(x)=D_\lambda(f(Dk))$.

\begin{cor}\label{wedgesandmaps} The canonical homeomorphism
\[\prod_{\lambda\in \Lambda}\map_\ast(K,X_\lambda)\longrightarrow\map_\ast(K,\prod_{\lambda\in \Lambda}X_\lambda)\]
is $G$-equivariant, when the target is equipped with conjugation action.

If $K$ is a $G$-CW complex, the canonical map
\[\bigvee_{\lambda\in \Lambda}\map_\ast(K,X_\lambda)\longrightarrow\map_\ast(K,\bigvee_{\lambda\in \Lambda}X_\lambda)\]
is non-equivariantly $(2(\conn X_{\lambda}-\dim K)+1)$-connected, and  \[\min\{\conn X_{\lambda}-\dim K,2(\conn X_{\lambda}^G-\dim K^G)+1\}\]connected on fixed points.
\end{cor}

\begin{proof}
It is easy to check that the canonical isomorphism is equivariant. For the statement about the wedge, consider the commutative diagram
\[\xymatrix{\bigvee_{\lambda\in \Lambda}\map_\ast(K,X_\lambda)\ar[r]\ar[d]&\map_\ast(K,\bigvee_{\lambda\in \Lambda}X_\lambda)\ar[d]\\
\prod_{\lambda\in \Lambda}\map_\ast(K,X_\lambda)\ar[r]^{\cong}&\map_\ast(K,\prod_{\lambda\in \Lambda}X_\lambda)
}\]
The spaces $\map_\ast(K,X_\lambda)$ are non equivariantly $(\conn X_\lambda-\dim K)$-connected, and when $\lambda$ is fixed by the involution its fixed points are $(\conn X_{\lambda}^G-\dim K^G)$-connected
by \ref{eqmappingspace}. Thus the left vertical map is non-equivariantly
$(2(\conn X_{\lambda}-\dim K)+1)$-connected, and 
\[\min\{\conn X_{\lambda}-\dim K,2(\conn X_{\lambda}^G-\dim K^G)+1\}\]
connected on the fixed points by \ref{wedgesintoproducts}. The connectivity of the right vertical map is by \ref{eqmappingspace} the connectivity of the inclusion of wedges into product minus the dimension of $K$. This is by \ref{wedgesintoproducts}
\[(2\conn X_{\lambda}+1, \min\{\conn X_{\lambda},2\conn X_{\lambda}^G+1\})-\dim_G K\]
Thus the non equivariant connectivity of the map we are interested in is
\[2(\conn X_{\lambda}-\dim K)+1\]
and its connectivity on the fixed points is
\[\min\{\conn X_{\lambda}-\dim K,2(\conn X_{\lambda}^G-\dim K^G)+1\}\]

\end{proof}


\subsection{Real topological Hochschild homology}\label{defTHR}

The following is a version with coefficients of the definition of $\THR(\mathcal{C})$ of \cite{IbLars}. It uses the dihedral structure on $I[-]$ to obtain extra structure on $\THH$.

\begin{defn}\label{defdualityfromop}
Let $\mathcal{C}$ be a category enriched in $Sp^{\Sigma}$.
A \textbf{strict duality} on $\mathcal{C}$ is an enriched functor $D\colon\mathcal{C}^{op}\longrightarrow\mathcal{C}$ such that $D\circ D^{op}=\id_{\mathcal{C}}$.

Functoriality of $D$ means that for each pair $i,j$ of natural numbers the diagram of spaces commutes
\[\xymatrix{\mathcal{C}(c,d)_i\wedge\mathcal{C}(d,e)_j\ar[r]^-{D\wedge D}\ar[d]_{\gamma} &
\mathcal{C}(Dd,Dc)_i\wedge\mathcal{C}(De,Dd)_j\ar[ddd]^-{\circ}\\
 \mathcal{C}(d,e)_j\wedge\mathcal{C}(c,d)_i\ar[d]_{\circ} \\
\mathcal{C}(c,e)_{i+j}\ar[d]^-{\chi_{i,j}}\\
\mathcal{C}(c,e)_{i+j}\ar[r]^-D &  \mathcal{C}(De,Dc)_{i+j}
}\]
Here $\chi_{i,j}$ is the permutation of $\Sigma_{i+j}$ that interchanges the first $i$ with the last $j$ elements
\[\chi_{i,j}(s)=\left\{\begin{array}{ll}s+j&\mbox{ if }1\leq s\leq i\\
s-i&\mbox{ if }i+1\leq s\leq j\end{array}\right.\]
\end{defn}
From now on we will take all the dualities to be strict.

\begin{ex}
Suppose that a category $C$ enriched in abelian groups has an additive duality $D\colon C^{op}\longrightarrow C$. It induces a duality $D$ on $HC$ by the maps $C(c,d)\otimes\mathbb{Z}(S^i)\stackrel{D\otimes \id}{\longrightarrow} C(D(d),D(c))\otimes\mathbb{Z}(S^i)$ which sends $f\otimes x$ to $D(f)\otimes x$.
\end{ex}

We want to define an involution on $\THH(\mathcal{C};\mathcal{M})$, provided $\mathcal{M}$ has an involution as defined below. First we note that if $D\colon\mathcal{C}^{op}\longrightarrow\mathcal{C}$ is a duality on $\mathcal{C}$, there is an  endofunctor of $\mathcal{C}^{op}\wedge\mathcal{C}$ given by
\[D_{\gamma}\colon\mathcal{C}^{op}\wedge\mathcal{C}\stackrel{D\wedge D^{op}}{\longrightarrow}\mathcal{C}\wedge\mathcal{C}^{op}\stackrel{\gamma}{\longrightarrow}\mathcal{C}^{op}\wedge\mathcal{C}\]
where $\gamma$ is the natural twist isomorphism.

\begin{defn}
A \textbf{duality} on a bimodule $\mathcal{M}\colon\mathcal{C}^{op}\wedge\mathcal{C}\longrightarrow Sp^{\Sigma}$ is an enriched natural transformation $J\colon \mathcal{M}\Rightarrow \mathcal{M}\circ D_{\gamma}$ such that the composite
\[\mathcal{M}(c,d)\stackrel{J_{c,d}}{\longrightarrow}\mathcal{M}(Dd,Dc)\stackrel{J_{Dd,Dc}}{\longrightarrow}\mathcal{M}(DDc,DDd)=\mathcal{M}(c,d)\]
equals the identity.
\end{defn}

Notice that by naturality of $J$, the diagram
\[\xymatrix{\mathcal{C}(d,e)\wedge\mathcal{M}(c,d)\ar[dd]_{l}\ar[r]^-{D\wedge J} & \mathcal{C}(De,Dd)\wedge\mathcal{M}(Dd,Dc)\ar[d]^-{\gamma}
\\ & \mathcal{M}(Dd,Dc)\wedge\mathcal{C}(De,Dd)\ar[d]^-{r}\\
\mathcal{M}(c,e)\ar[r]^-{J} & \mathcal{M}(De,Dc)
}\]
commutes.

\begin{ex} In the case of the Hom-bimodule $\mathcal{M}=\mathcal{C}$, the duality itself defines a bimodule duality
\[\mathcal{C}(c,d)\stackrel{D}{\longrightarrow}\mathcal{C}(Dd,Dc)\]
\end{ex}
\begin{ex}\label{dualitymodabelian} Suppose that $M\colon C^{op}\otimes C\longrightarrow Ab$ is a bimodule over an $Ab$-enriched category with additive duality $(C,D)$. A duality on $M$ is an enriched natural transformation $J\colon M\Rightarrow M\circ D_\gamma$ such that $J\circ J=\id$. It induces a duality on the bimodule $HM\colon HC^{op}\wedge HC\longrightarrow Sp^{\Sigma}$ via the maps
\[M(c,d)\otimes \mathbb{Z}(S^i)\stackrel{J\otimes\id}{\longrightarrow} M(D(d),D(c))\otimes \mathbb{Z}(S^i)\]
\end{ex}

We now proceed in defining a $\mathbb{Z}/2$ action on $\THH$.
The dualities on $\mathcal{C}$ and $\mathcal{M}$ induce maps
\[\overline{\omega}_k\colon V(\mathcal{C};\mathcal{M},\underline{i})\longrightarrow V(\mathcal{C};\mathcal{M},\omega_k\underline{i})\]
for $\underline{i}\in I[k]$ by
\[\xymatrix{V(\mathcal{C};\mathcal{M},\underline{i})\ar@{=}[d]\\
\bigvee \mathcal{M}(c_0,c_k)_{i_0}\wedge\mathcal{C}(c_1,c_0)_{i_1}\wedge\dots\wedge\mathcal{C}(c_k,c_{k-1})_{i_k}\ar[d]^\cong \\
\bigvee\mathcal{M}(c_0,c_k)_{i_0}\wedge\mathcal{C}(c_k,c_{k-1})_{i_k}\wedge\dots\wedge\mathcal{C}(c_1,c_0)_{i_1}\ar[d]^{\omega_{i_0}\wedge \dots\wedge \omega_{i_1}}\\
\bigvee\mathcal{M}(c_0,c_k)_{i_0}\wedge\mathcal{C}(c_k,c_{k-1})_{i_k}\wedge\dots\wedge\mathcal{C}(c_1,c_0)_{i_1}\ar[d]^{J\wedge D\wedge\dots\wedge D}\\
\bigvee \mathcal{M}(Dc_{k},Dc_0)_{i_0}\wedge\mathcal{C}(Dc_{k-1},Dc_{k})_{i_k}\wedge\dots\wedge\mathcal{C}(Dc_{0},Dc_{1})_{i_{1}}\ar@{=}[d]\\
V(\mathcal{C};\mathcal{M},(i_0,i_k,\dots,i_1))
}\]
The first map is the natural isomorphism permuting the smash factors, and $\omega_l$ denotes the involution induced by the symmetric structure of the spectra of the permutation $\omega_l\in\Sigma_l$ that reverses the order of the elements.
Denoting $S^{\underline{i}}=S^{i_0}\wedge\dots\wedge{S}^{i_k}$, this induces a map
\[\map_\ast(S^{\underline{i}},V(\mathcal{C};\mathcal{M},\underline{i}))\longrightarrow \map_\ast(S^{\omega_k\underline{i}},V(\mathcal{C};\mathcal{M},\omega_k\underline{i}))\]
by sending a $f\in\map_\ast(S^{\underline{i}},V(\mathcal{C};\mathcal{M},\underline{i}))$ to
\[\xymatrix{S^{i_0}\wedge S^{i_k}\wedge\dots\wedge{S}^{i_1}\ar[d]^{\omega_{i_0}\wedge\dots\wedge\omega_{i_1}}\ar[r] & V(\mathcal{C};\mathcal{M},\omega_k\underline{i})\\
S^{i_0}\wedge S^{i_k}\wedge\dots\wedge{S}^{i_1}\ar[d]^{\cong}\\
S^{i_0}\wedge S^{i_1}\wedge\dots\wedge{S}^{i_k}\ar[r]^-{f}&V(\mathcal{C};\mathcal{M},\underline{i})\ar[uu]_{\overline{\omega}_k}
}\]
Here $\omega_l\colon S^l\longrightarrow S^l$ denotes the permutation that reverses the order of the $S^1$-smash factors.
This gives a natural transformation $\mathcal{G}_k(\mathcal{C};\mathcal{M})\longrightarrow \mathcal{G}_k(\mathcal{C};\mathcal{M})\circ\overline{\omega}_k$, and taking homotopy colimits a map
\[\omega_k\colon \THH_k(\mathcal{C};\mathcal{M})\longrightarrow\hocolim_{I[k]}\mathcal{G}_k(\mathcal{C};\mathcal{M})\circ\overline{\omega}_k\longrightarrow\THH_k(\mathcal{C};\mathcal{M})\]
These structure maps together with the simplicial maps described earlier define a real structure (see \ref{defrealob}) on the simplicial space $[k]\mapsto \THH_k(\mathcal{C};\mathcal{M})$, and a dihedral structure on $[k]\mapsto \THH_k(\mathcal{C};\mathcal{C})$. We denote the $k$-simplicies of this spaces together with their involution by $\THR_k(\mathcal{C};\mathcal{M})$ and $\THR_k(\mathcal{C};\mathcal{C})$ respectively.

Recall that the realization of a real space $X$ with levelwise involution $\omega$ has an involution defined by
\[[x_k,(t_0,\dots,t_k)]\longmapsto[\omega(x_k),(t_k,\dots,t_0)]\]
\begin{defn}
The \textbf{real topological Hochschild homology of $(\mathcal{C},D)$ with coefficients in $(\mathcal{M},J)$} is the $\mathbb{Z}/2$-space defined as the realization of the real space
\[\THR(\mathcal{C};\mathcal{M})=|[k]\mapsto\THR_k(\mathcal{C};\mathcal{M})|\]
The \textbf{real topological Hochschild homology of $(\mathcal{C},D)$} is the $\mathbb{Z}/2\rtimes S^1$-space
defined as the realization of the dihedral space
\[\THR(\mathcal{C})=|[k]\mapsto\THR_k(\mathcal{C};\mathcal{C})|\]
\end{defn}

We now restrict attention to a class of categories with duality and bimodules for which the homotopy colimit defining $\THR$ behaves nicely.
First we recall the non-equivariant condition.
\begin{defn}
Let $\mathcal{C}$ be a category enriched in $Sp^{\Sigma}$. We say that $\mathcal{C}$ \textbf{is $0$-connected} if
\begin{enumerate}[i)]
\item The spaces $\mathcal{C}(c,d)_i$ are $(i-1)$-connected,
\item There is a constant $\epsilon$ such that the structure maps $\mathcal{C}(c,d)_i\wedge S^j\longrightarrow \mathcal{C}(c,d)_{i+j}$ are $(2i+j-\epsilon)$-connected,
\end{enumerate}
A bimodule $\mathcal{M}$ over $\mathcal{C}$ is $0$-connected if the levels and the structure maps of its spectra satisfy the same properties.
\end{defn}

For any integer $i$ we denote $\lceil\frac{i}{2}\rceil$ the upper integral approximation of $\frac{i}{2}$.

\begin{defn}\label{defGconn}
Suppose that $(\mathcal{C},D)$ is a category enriched in $Sp^\Sigma$ with duality . We say that $\mathcal{C}$ \textbf{is $G$-connected} if it is $0$-connected, and
\begin{enumerate}[i)]
\item The fixed point spaces $(\mathcal{C}(c,Dc)_i)^G$ are $(\lceil\frac{i}{2}\rceil-1)$-connected, where $G$ acts via $D\circ\omega_i$,
\item There is a constant $\epsilon$ such that the restriction to the fixed points \[\mathcal{C}(c,Dc)_{i}^G\wedge (S^{j})^G\longrightarrow \mathcal{C}(c,Dc)_{i+j}^G\] of the structure maps are $(i+\lceil\frac{j}{2}\rceil-\epsilon)$-connected, with respect to the actions of $(D\circ\omega_i)\wedge\omega_j$ on the source space, and $D\circ(\omega_i\times\omega_j)$ on the target.
\end{enumerate}
A bimodule with duality $(\mathcal{M},J)$ over $(\mathcal{C},D)$ is $G$-connected if it is $0$-connected and it satisfies the analogous properties.
\end{defn}
Notice that we can assume that the constant $\epsilon$ is the same both for the non-equivariant condition and on the fixed points.
In particular, the class of examples we are mostly interested in satisfies this property by the following proposition.

\begin{prop}\label{linear0connected}
Let $(C,D)$ be a category enriched in abelian groups with duality, and let $(M,J)$ be a bimodule with duality over $(C,D)$. Then both $HC$ and $HM$ are $G$-connected.
\end{prop}

\begin{proof}
We check the conditions for the category $HC$. The proof for $HM$ is totally analogous. Given two objects $c,d\in C$, the space $C(c,d)\otimes\mathbb{Z}(S^i)$ is clearly $(i-1)$-connected (it is the realization of a $(i-1)$-reduced fibrant simplicial set). The fixed points space $(C(c,Dc)\otimes\mathbb{Z}(S^i))^G$ is $(\lceil\frac{i}{2}\rceil-1)$-connected by \ref{connconfspace} of the appendix, where
we prove more generally that for an abelian group with additive involution $A$, and a simplicial $G$-set $X$, the space $A(X)^G$ is
\[\min\{\conn(X),\conn(X^G)\}\]
connected.

The structure map $(C(c,d)\otimes\mathbb{Z}(S^i))\wedge S^j\longrightarrow C(c,d)\otimes\mathbb{Z}(S^{i+j})$ fits into a commutative diagram
\[\xymatrix{\Omega^j((C(c,d)\otimes\mathbb{Z}(S^i))\wedge S^j)\ar[r]&\Omega^j(C(c,d)\otimes\mathbb{Z}(S^{i+j}))\\
C(c,d)\otimes\mathbb{Z}(S^i)\ar[u]\ar[ur]_{\simeq}
}\]
where the vertical map is the adjoint of the identity, which is $(2(i-1)+1)$-connected by the suspension theorem. Thus the structure map is $(2i+j-1)$-connected.
For the restriction of the structure maps to the fixed points we refer to \ref{connectvitysmashwithx} in the appendix. There we prove in general that 
the map
\[A(S^i)\wedge |X|\longrightarrow A(S^i\wedge X)\]
induced by sending $(\sum m_l\cdot\sigma_l)\wedge x$ to $\sum m_l\cdot(\sigma_l\wedge x)$
is non-equivariantly $(2i+\conn X)$-connected, and $(i+\min\{\conn X^G,\conn X\})$-connected on the fixed points for every real set $X$ and abelian group with additive involution $A$.
\end{proof}


\subsection{$\THR$ of rings with antistructures}\label{THRrings}

Let $A$ be a ring, $M$ an $A$-bimodule, and let $M\colon \mathcal{P}_{A}^{op}\otimes\mathcal{P}_A\longrightarrow Ab$ be the bimodule over $\mathcal{P}_A$ which sends a pair of projective modules $(c,d)$ to
\[M(c,d)=\hom_A(c,d\otimes_A M)\]
Also recall that $M$ induces a bimodule $HM\colon H\mathcal{P}_{A}^{op}\wedge H\mathcal{P}_A\longrightarrow Sp^{\Sigma}$.

\begin{defn} The \textbf{topological Hochschild homology of $A$ with coefficients in $M$} is defined as
\[\THH(A;M)=\THH(H\mathcal{P}_A;HM)\]
\end{defn}

We want to define $\THR$ of an antistructure $(A,L,\alpha)$ with coefficients in an $A$-bimodule $M$, with a duality on the bimodule defined from an $M$-twisting $J\colon L_t\otimes_AM\longrightarrow L_t\otimes_AM$ (cf. \ref{defmaptwist}).
We defined $\THR$ of a category and a bimodule with strict dualities, but the dualities on $H\mathcal{P}_A$ and $HM$ are not strict. We know how to make the duality strict on $\mathcal{P}_A$ by applying the construction $\mathcal{D}\mathcal{P_A}$ of §\ref{duality}, but we do not know yet how to make a duality strict on the bimodule. We describe this in general.

Let $(C,D,\eta)$ be an $Ab$-enriched category with additive non-strict duality, and $M\colon C^{op}\otimes C\longrightarrow Ab$ a bimodule. Recall that $D_\gamma$ is the duality on $C^{op}\otimes C$ that sends $(c,d)$ to $(Dd,Dc)$.

\begin{defn}\label{defnonstrictdualmod}
A \textbf{non-strict duality} on $M$ is an enriched natural transformation $J\colon M\Rightarrow M\circ D_\gamma$ such that
\[\xymatrix{M(c,d)\ar[r]^-{J}\ar[dr]_-{M(\eta_{c}^{-1}\otimes\eta_d)}&M(Dd,Dc)\ar[d]^-{J}\\
&M(D^2c,D^2d)
}\]
commutes for all $c,d\in C$.
\end{defn}

\begin{ex}\label{Mtwistnonstrictdual}
Let $J$ be an $M$-twisting of $(A,L,\alpha)$. This induces a non-strict duality on the bimodule $M=\hom_A(-,-\otimes_AM)\colon \mathcal{P}_{A}^{op}\otimes\mathcal{P}_A\longrightarrow A$, defined as the natural transformation
\[J=\widehat{J}\colon \hom_A(P,Q\otimes_AM)\longrightarrow \hom_A(D_LQ,D_L(P)\otimes_AM)\]
of \ref{naturalityJK}.
\end{ex}

We recall that the objects of $\mathcal{D}C$ are triples $\varphi=(c,d,\phi)$ with $\phi\colon d\stackrel{\cong}{\longrightarrow} D(c)$. The morphisms are pairs of morphisms $(a\colon c\longrightarrow c',b\colon d'\longrightarrow d)$ in $C$ such that the square
\[\xymatrix{d\ar[d]_-{\phi}& d'\ar[d]^-{\phi'}\ar[l]_-{b}\\
Dc & Dc'\ar[l]^-{Da}
}\]
commutes. There is an $Ab$-enrichement of $\mathcal{D}C$ by seeing the morphism sets as subgroups of $\hom_C(c,c')\oplus\hom_C(d',d)$. The 
duality on $\mathcal{D}C$ is strict, and defined by \[D(c,d,\phi)=(d,c, c\stackrel{\eta_c}{\longrightarrow}D^2(c)\stackrel{D\phi}{\longrightarrow}Dd)\] on objects, and it sends $(a,b)$ to $(b,a)$ on morphisms.
A bimodule $M\colon C^{op}\otimes C\longrightarrow Ab$ induces a bimodule on $\mathcal{D}C$ by the composition
\[\mathcal{D}M\colon \mathcal{D}C^{op}\otimes \mathcal{D}C\stackrel{U^{op}\otimes U}{\longrightarrow}  C^{op}\otimes C\stackrel{M}{\longrightarrow} Ab\]
where $U\colon \mathcal{D}C\longrightarrow C$ is the equivalence of categories of \ref{canddc}. Explicitly, the value of $\mathcal{D}M$ at a pair $(\varphi=(c,d,\phi),\varphi'=(c',d',\phi))$ is
\[\mathcal{D}M(\varphi,\varphi')=M(c,c')\]
Let $J\colon M\Rightarrow M\circ D_\gamma$ be a non-strict duality on $(C,D,\eta)$. Define a natural transformation $\mathcal{D}J\colon \mathcal{D}M\Rightarrow \mathcal{D}M\circ D_\gamma$ by the composite
\[\xymatrix{*[l]{\mathcal{D}M(\varphi,\varphi')=M(c,c')}\ar@<-17ex>[d]_-{\mathcal{D}J}\ar[r]^-{J}&M(Dc',Dc)\ar[dl]^{M(\phi'\otimes \phi^{-1})}\\
*[l]{\mathcal{D}M(D\varphi',D\varphi)=M(d',d)}}\]
This is a strict duality on  $\mathcal{D}M$.

\begin{defn}
Let $(C,D,\eta)$ be an $Ab$-enriched category with additive non-strict duality, and $M\colon C^{op}\otimes C\longrightarrow Ab$ a bimodule with non-strict duality $J$. Its \textbf{real topological Hochschild homology} is
\[\THR(C;M):=\THR(H\mathcal{D}C;H\mathcal{D}M)\]
\end{defn}

We describe explicitly the bimodule and the duality in the ring case. Let $(A,L,\alpha)$ be an antistructure and $J\colon L_t\otimes_A M\longrightarrow L_t\otimes_A M$ an $M$-twisting. The induced bimodule $\mathcal{D}M$ sends a pair $\varphi=(P,Q,\phi),\varphi'=(P',Q',\phi')\in\mathcal{D}\mathcal{P}_A$ to the abelian group of module maps
\[\mathcal{D}M(\varphi,\varphi')=\hom_A(P,P'\otimes_AM)\]
with duality sending $f\colon P\longrightarrow P'\otimes_A M$ to
\[\mathcal{D}J(f)\colon Q'\stackrel{\phi'}{\longrightarrow}D_LP'\stackrel{\widehat{J}(f)}{\longrightarrow}D_L(P)\otimes_A M\stackrel{(\phi^{-1}\otimes M)}{\longrightarrow}Q\otimes_AM\]
We can display this information in a slightly different and more familiar way. Consider the abelian group $M^J(\varphi,\varphi')$ of pair of maps 
$(f\colon P\longrightarrow P'\otimes_A M,g\colon f\colon Q'\longrightarrow Q\otimes_A M)$ such that
\[\xymatrix{Q\otimes_A M\ar[d]_-{\phi\otimes M}& Q'\ar[l]_-{g}\ar[d]^-{\phi'}\\
D(P)\otimes_A M& DP'\ar[l]^-{\widehat{J}(f)}
}\]
commutes, viewed as a subgroup of $\hom_A(P,P'\otimes_A M)\oplus \hom_A(Q',Q\otimes_A M)$. Since both $\phi$ and $\phi'$ are isomorphism, the map $g$ is completely determined by $f$, and therefore $M^J(\varphi,\varphi')$ is isomorphic to $\mathcal{D}M(\varphi,\varphi')$. Under this identification the duality induced by $J$ sends $(f,g)$ to $(g,f)$.
Notice that for $\varphi=\varphi'$ there is an equality
\[M^J(\varphi,\varphi)=\hom_A(\varphi,\varphi\otimes_A M)\]
where the right hand side is the abelian group appearing in the description of $i\mathcal{D}\mathcal{P}_{A\ltimes M}$ of §\ref{secmainthmkr}.

\begin{defn}\label{defTHRforanti}
The \textbf{real topological Hochschild homology} of $(A,L,\alpha)$ with coefficients in the $M$-twisting $J$ is
\[\THR((A,L,\alpha),(M,J))=\THR(H\mathcal{D}\mathcal{P}_A;HM^J)\]
\end{defn}

Notice that the equivalence of categories $\mathcal{P}_A\longrightarrow \mathcal{D}\mathcal{P}_A$ of \ref{canddc} induces a (non-equivariant) equivalence $\THH(H\mathcal{P}_A;HM)\simeq \THH(H\mathcal{D}\mathcal{P}_A;HM^J)$ (see e.g. \cite[1.6.7]{ringfctrs} or \ref{THRpreservesGequiv} below). We will prove in \ref{scriptdandthr} that if the dualities on $C$ and $M$ are strict, the equivalence of categories $C\longrightarrow \mathcal{D}C$ induces a $G$-equivalence
\[\THR(HC;HM)\stackrel{\simeq}{\longrightarrow}\THR(H\mathcal{D}C;H\mathcal{D}M)\]

\begin{rem}
In the special case of an antistructure coming from an anti-involution $\alpha\colon A^{op}\longrightarrow A$, and an $M$-twisting given by a map $j\colon M\longrightarrow M$ (cf. example \ref{antiinvolutioncase}) there is another description for $\THR((A,A,\alpha),(M,j))$. One can consider the Eilenberg-MacLane spectrum $HA$ as a category enriched in  $Sp^\Sigma$ with only one object. $M$ defines a bimodule over $HA$ that sends the unique object to the Eilenberg-MacLane spectrum $HM$. It sends morphisms to the action on the bimodule. The anti-involution $\alpha$ defines a duality on $HA$, and the map $j$ a duality on the bimodule $HM$. Thus one can consider $\THR(HA,HM)$ instead. In §\ref{moritasec} we define an equivariant equivalence
\[\THR(HA,HM)\stackrel{\simeq}{\longrightarrow} \THR(H\mathcal{D}\mathcal{P}_A;HM^j)\]
Although the definition of $\THR(HA,HM)$ is simpler, it is easier to relate $\THR(H\mathcal{D}\mathcal{P}_A;HM^j)$ to $K$-theory since it already involves the category $\mathcal{P}_A$.
\end{rem}

\begin{rem}
For an $A$-bimodule $M$, one might want to classify all the non-strict dualities $\overline{J}$ on the bimodule
\[M=\hom_A(-,-\otimes_AM)\colon\mathcal{P}_{A}^{op}\otimes\mathcal{P}_A\longrightarrow Ab\]
We claim that they all come from an $M$-twisting $J$ as described above. Given a non-strict duality $\overline{J}$, define an $M$-twisting $J\colon L_t\otimes_AM\longrightarrow L_t\otimes_AM$ as follows. For a fixed element $l\otimes m$, consider the module map
\[f_{l\otimes m}\colon A_t\longrightarrow L_t\otimes_AM \ \ \ ; \ \ \ a\longmapsto l\otimes ma\]
Define $J$ as the composite
\[\xymatrix{ L_t\otimes_AM\ar[dr]_-{J}\ar[r]^-{f_{-\otimes -}}&\hom_A(A_t,L_t\otimes_AM)\ar[r]^-{\overline{J}}&\hom_A(D_LL_t,D_L(A_t)\otimes_AM)\ar[d]^-{ev_\alpha}\\
&L_t\otimes_AM&D_L(A_t)\otimes_A M\ar[l]^-{ev_1\otimes M}}\]
where we remind that $\alpha\in D_LL_t=\hom_A(L_t,L_s)$ is the involution of the antistructure, and $ev_x$ denotes evaluation at $x$. A computation shows that
\[\overline{J}(f)=\widehat{J}(f)\]
for all $f\in\hom_A(P, Q\otimes_A M)$, where $\widehat{J}$ is the natural transformation defined from $J$ in \ref{naturalityJK}.
\end{rem}

\newpage


\section{Properties of $\THR$}


\subsection{Functoriality of $\THR$ and homotopies}\label{functorialityTHR}

We will call an $Sp^\Sigma$-enriched category with duality an $Sp^{\Sigma}$-category with duality.
Let $(\mathcal{C},D)$ and $(\mathcal{C}',D')$ be $Sp^{\Sigma}$-categories with duality.

\begin{defn}
A \textbf{morphism of $Sp^{\Sigma}$-categories with duality} from $(\mathcal{C},D)$ to $(\mathcal{C}',D')$ is an enriched functor $F\colon \mathcal{C}\longrightarrow \mathcal{C}'$ such that on objects $FD=D'F$, and such that the diagram
\[\xymatrix{\mathcal{C}(c,d)\ar[d]_{F}\ar[r]^-{D}&\mathcal{C}(Dd,Dc)\ar[d]^F\\
\mathcal{C}'(Fc,Fd)\ar[r]_-{D'}&\mathcal{C}'(D'Fd,D'Fc)
}\]
commutes in $Sp^{\Sigma}$.
\end{defn}

Let $(\mathcal{M},J)$ be a bimodule with duality on $(\mathcal{C},D)$ and $(\mathcal{M}',J')$ be a bimodule with duality on $(\mathcal{C}',D')$, cf. §\ref{sectmodsym} and §\ref{defTHR}.

\begin{defn}
A \textbf{morphism of bimodules with duality} is a morphism of bimodules $(F,\Phi)\colon (\mathcal{C},\mathcal{M})\longrightarrow (\mathcal{C}',\mathcal{M}')$ such that $F$ is a morphism of $Sp^{\Sigma}$-categories with duality, and the natural transformation $\Phi\colon \mathcal{M}\Rightarrow \mathcal{M}'\circ(F^{op}\wedge F)$ is such that
\[\xymatrix{\mathcal{M}(c,d)\ar[r]^-{J}\ar[d]_{\Phi}&\mathcal{M}(Dd,Dc)\ar[d]^{\Phi}\\
\mathcal{M'}(Fc,Fd)\ar[r]_-{J'}&\mathcal{M}'(D'Fd,D'Fc)
}\]
commutes in $Sp^{\Sigma}$.
\end{defn}

We next explain how a morphism of bimodules with duality \[(F,\Phi)\colon (\mathcal{C},\mathcal{M})\longrightarrow (\mathcal{C}',\mathcal{M}')\] induces an equivariant map in real topological Hochschild homology, cf. §\ref{defTHR}. For every $\underline{i}\in I[k]$ we define
\[(F,\Phi)_\ast\colon V(\mathcal{C};\mathcal{M},\underline{i})\longrightarrow V(\mathcal{C}';\mathcal{M}',\underline{i})\]
to be the map that sends an element $m_0\wedge f_1\wedge\dots\wedge f_{k}$ in the $(c_0,\dots,c_{k})$ wedge component to
\[\Phi_{i_0}(m_0)\wedge F_{i_1}(f_1)\wedge\dots\wedge F_{i_k}(f_k)\]
in the $(F(c_0),\dots,F(c_{k}))$ wedge component,
where $F_{i_l}$ denotes the $i_l$-th level map $\mathcal{C}(c_{i_l},c_{i_{l-1}})_{i_l}\longrightarrow \mathcal{C}(Fc_{i_l},Fc_{i_{l-1}})_{i_l}$ of the map of spectra $F$.
Taking homotopy colimits, this induces a simplicial map on $\THR_{\sbt}$, and the compatibility of $F$ and $\Phi$ with the dualities insures that it is equivariant. All in all, $(F,\Phi)$ induces an equivariant map
\[(F,\Phi)_\ast\colon \THR(\mathcal{C};\mathcal{M})\longrightarrow \THR(\mathcal{C}';\mathcal{M}')\]
on realizations.
There is a category whose objects are pairs $(\mathcal{C},\mathcal{M})$ of an $Sp^\Sigma$-category with duality and a bimodule with duality, and whose morphisms are the morphisms of bimodules with duality as defined above. The $\THR$ construction defines a functor from this category to pointed $G$-spaces. 

\begin{defn}\label{Gnattransf}Let $(F,\Phi),(G,\Psi)\colon (\mathcal{C},\mathcal{M})\longrightarrow (\mathcal{C}',\mathcal{M}')$ be two morphisms of bimodules with duality. A \textbf{$G$-natural transformation} from $(F,\Phi)$ to $(G,\Psi)$ is an enriched natural transformation $\kappa\colon F\Rightarrow G$ such that for every object $c\in\mathcal{C}$, the element $\kappa_c\in\mathcal{C}'(F(c),G(c))_0$ defining the map of spectra $\kappa_c\colon\mathbb{S}\longrightarrow \mathcal{C}'(F(c),G(c))$ satisfies
\begin{enumerate}[i)]
\item For all $f\in\mathcal{C}(c,d)_i$  \[F(f)\circ D\kappa_{Dc}\circ \kappa_c=F(f)\] as elements of $\mathcal{C}(Fc,Fd)_i$,
\item For all $f\in\mathcal{C}(c,d)_i$
\[\kappa_d\circ D\kappa_{Dd}\circ G(f)=G(f)\] as elements of $\mathcal{C}(Gc,Gd)_i$
\item The diagram
\[\xymatrix{\mathcal{M}(c,d)\ar[d]_-{\Psi}\ar[rr]^-{\Phi}&&\mathcal{M}'(Fc,Fd)\ar[d]^-{\mathcal{M}'_0(\id_{Fc}\wedge\kappa_d)}\\
\mathcal{M}'(Gc,Gd)\ar[rr]_{\mathcal{M}'_0(\kappa_{c}\wedge\id_{Gd})}&& \mathcal{M}'(Fc,Gd)
}\]
commutes,
where the right vertical arrow is the map of spectra 
\[\mathcal{M}'_0(\id_{Fc}\wedge\kappa_d)\in Sp^{\Sigma}(\mathcal{M}'(Fc,Fd),\mathcal{M}'(Fc,Gd))_0\]
and similarly for the bottom horizontal map.
\end{enumerate}
\end{defn}

\begin{prop}\label{THRpreservesGequiv}
A $G$-natural transformation $\kappa$ between morphisms of bimodules with duality $(F,\Phi),(G,\Psi)\colon (\mathcal{C},\mathcal{M})\longrightarrow (\mathcal{C}',\mathcal{M}')$ induces a simplicial homotopy of equivariant maps
\[sd_e\THR_{\sbt}(\mathcal{C};\mathcal{M})\times\Delta[1]_{\sbt} \longrightarrow sd_e\THR_{\sbt}(\mathcal{C}';\mathcal{M}')\]
between the induced maps $(F,\Phi)_\ast$ and $(G,\Psi)_\ast$. In particular it induces a $G$-homotopy
\[\THR(\mathcal{C};\mathcal{M})\times I \longrightarrow \THR(\mathcal{C}';\mathcal{M}')\]
on realizations.
\end{prop}

\begin{proof}
Denote $G\circ \kappa_{c}$ the level maps of the map of spectra given by the composite
\[\xymatrix{\mathcal{C}(c,d)\ar@{-->}[d]_{G\circ \kappa_{c}}\ar[r]^-{G}&\mathcal{C}'(Gc,Gd)\wedge\mathbb{S}\ar[d]^{\id\wedge \kappa_{c}}\\
\mathcal{C}'(Fc,Gd)&\mathcal{C}'(Gc,Gd)\wedge \mathcal{C}'(Fc,Gc)\ar[l]^-{\circ}
}
\]
Given a $\sigma\in \Delta[1]_k$ we define for every $\underline{i}\in I[2k+1]$ a map
\[H_\sigma\colon V(\mathcal{C};\mathcal{M},\underline{i})\longrightarrow V(\mathcal{C}';\mathcal{M}',\underline{i})\]
These maps will provide a simplicial homotopy as claimed upon taking loops and homotopy colimits. Elements $\sigma\in\Delta[1]_k$ can be displayed as $\sigma=(0^b1^{k+1-b})$.\newline
If $b=k+1$, $H_\sigma$ sends $m_0\wedge f_1\wedge\dots\wedge f_{2k+1}$ in the wedge component $(c_0,\dots,c_{2k+1})$ to
\[\Phi(m_0)\wedge F(f_{1})\wedge \dots\wedge F(f_{2k+1})\]
in the wedge component $(F(c_0),\dots,F(c_{2k+1}))$.\newline
If $b=0$, $H_\sigma$ maps $m_0\wedge f_1\wedge\dots\wedge f_{2k+1}$ to
\[\Psi(m_0)\wedge G(f_{1})\wedge \dots\wedge G(f_{2k+1})\]
in the wedge component $(G(c_0),\dots,G(c_{2k+1}))$.\newline
If $0<b<k+1$, $H_\sigma(m_0\wedge f_1\wedge\dots\wedge f_{2k+1})$ belongs to the wedge component \[(F(c_0),\dots, F(c_{b_\sigma-1}),G(c_{b_\sigma}),\dots, G(c_{2k+1-b_\sigma}),F(c_{2k+2-b_\sigma}),\dots, F(c_{2k+1}))\] and it is given by
\[\begin{array}{lll}\Phi(m_0)\wedge F(f_1)\wedge\dots\wedge F(f_{b-1})\wedge\\
(D\kappa_{Dc_{b-1}}\circ G(f_{b}))\wedge G(f_{b+1})\wedge\dots
\dots\wedge G(f_{2k+1-b})\wedge\\
(G(f_{2k+2-b})\circ \kappa_{c_{2k+2-b}})\wedge F(f_{2k+3-b})\wedge\dots\wedge F(f_{2k+1})\end{array}\]
One checks that the induced map
\[H_\sigma\colon \THR_{2k+1}(\mathcal{C};\mathcal{M})\longrightarrow \THR_{2k+1}(\mathcal{C}';\mathcal{M}')\] is $G$-equivariant upon using the various compatibility conditions satisfied by $(F,\Phi)$, $(G,\Psi)$ and $\kappa$.
By inspection, the family $H_\sigma$, $\sigma\in\Delta[1]_{\sbt}$ induces a simplicial homotopy
\[H\colon sd_e\THR_{\sbt}(\mathcal{C};\mathcal{M})\times \Delta[1]\longrightarrow sd_e\THR_{\sbt}(\mathcal{C}';\mathcal{M}')\]
\end{proof}

\begin{ex}We show that $d_0H=Hd_0$ when $k=1$ and $\sigma=(01)$. The map $Hd_0$ sends $(m_0\wedge f_1\wedge f_2\wedge f_3,(01))$ to
\[H((f_3m_0f_1)\wedge f_2,1)=\Psi(f_3\cdot m_0\cdot f_1)\wedge G(f_2)\]
Its image by $d_0H$ is
\[(( G(f_{3})\circ \kappa_{c_{3}})\cdot \Phi(m_0)\cdot(D\kappa_{Dc_{0}}\circ G(f_{1})))\wedge G(f_2)\]
The first smash factor is equal to
\[G(f_{3})\cdot(\kappa_{c_{3}}\cdot \Phi(m_0))\cdot (D\kappa_{Dc_{0}}\circ G(f_{1}))\]
which by definition of the bimodule actions is
\[G(f_{3})\cdot(\mathcal{M}_{0}'(\id_{Fc_0}\wedge\kappa_{c_{3}})\circ \Phi)(m_0)\cdot (D\kappa_{Dc_{0}}\circ G(f_{1}))\]
By property iii) of $\kappa$ this is equal to
\[G(f_{3})\cdot(\Psi\circ\mathcal{M}_{0}'(\kappa_{c_0}\wedge\id_{Gc_{3}}))(m_0)\cdot (D\kappa_{Dc_{0}}\circ G(f_{1}))\]
that by definition of the bimodule action is
\[G(f_{3})\cdot\Psi(m_0)\cdot (\kappa_{c_0}\circ D\kappa_{Dc_{0}}\circ G(f_{1}))\]
By property ii) this is equal to
\[G(f_{3})\cdot\Psi(m_0)\cdot G(f_{1})\]
which by naturality of
 $\Psi\colon\mathcal{M}\Rightarrow \mathcal{M}'\circ (G^{op}\wedge G)$ is
\[\mathcal{M}'( G(f_{1})\wedge G(f_{3}))(\Psi(m_0))=\Psi(\mathcal{M}(f_{1}\wedge f_{3})(m_0))=\Psi(f_3\cdot m_0\cdot f_1)\]
\end{ex}

Let $(C,D)$ be an $Ab$-enriched category with additive strict duality, and $(M,J)$ a bimodule with strict duality. Recall the functor $V\colon C\longrightarrow \mathcal{D}C$ from \ref{canddc} that sends $c$ to the triple $(c,Dc,\id_{Dc})$. Since the duality on $C$ is strict it commutes strictly with the duality. The identity map defines a map of bimodules $ M\Rightarrow\mathcal{D}M\circ (V^{op}\otimes V)$ commuting strictly with the dualities, since $M(c,c')=\mathcal{D}M(V(c),V(c'))$.

\begin{prop}\label{scriptdandthr} The $G$-map
\[(V,\id)_\ast\colon\THR(HC;HM)\longrightarrow\THR(H\mathcal{D}C;H\mathcal{D}M)\]
is a natural $G$-homotopy equivalence.
\end{prop}

\begin{proof}
We define a $G$-homotopy inverse for the subdivided map
\[(V,\id)_\ast\colon sd_e\THR_{\sbt}(HC;HM)\longrightarrow sd_e\THR_{\sbt}(H\mathcal{D}C;H\mathcal{D}M)\]
For $\underline{i}\in I[2k+1]$ we display an element of $V(HC;HM,\underline{i})$ as a string
\[c_{2k+1}\stackrel{m}{\leftarrow}c_0\stackrel{f_1}{\leftarrow}c_1\stackrel{}{\leftarrow}\dots\stackrel{}{\leftarrow}c_{2k}\stackrel{f_{2k+1}}{\leftarrow}c_{2k+1}\]
With this notation, an element of $V(H\mathcal{D}C;H\mathcal{D}M,\underline{i})$ is a string
\[\xymatrix{d_{2k+1}\ar[d]^{\phi_{2k+1}}\ar[r]^{m'}&d_{0}\ar[d]^{\phi_{0}}\ar[r]^{b_1}&d_1\ar[d]^{\phi_{1}}\ar[r]&\dots\ar[r]&d_{2k}\ar[d]^{\phi_{2k}}\ar[r]^{b_{2k+1}}&d_{2k+1}\ar[d]^{\phi_{2k+1}}\\
Dc_{2k+1}\ar[r]_{Jm}&Dc_0\ar[r]_{Da_1}&Dc_1\ar[r]&\dots\ar[r]&Dc_{2k}\ar[r]_{Da_{2k+1}}&Dc_{2k+1}
}\]
where the vertical maps are isomorphisms. This string is sent to
\[\xymatrix{Dd_{2k+1}&c_{0}\ar[l]_-{m\cdot\phi_{2k+1}}& c_1\ar[l]_-{a_1}&\dots\ar[l]_-{a_2}&c_k\ar[l]_{a_k}\\
Dd_{2k+1}\ar[r]_-{Db_{2k+1}}&Dd_{2k}\ar[r]_-{Db_{2k}}&\dots\ar[r]&Dd_{k+2}\ar[r]_-{Db_{k+2}}&Dd_{k+1}\ar[u]_-{a_{k+1}\circ D\phi_{k+1}^{-1}}
}\]
in $V(HC;HM,\underline{i})$.
These maps respect the structure maps and the dualities, and therefore they induce a map of real sets
\[r\colon sd_e\THR_{\sbt}(H\mathcal{D}C;H\mathcal{D}M)\longrightarrow sd_e\THR_{\sbt}(HC;HM)\]
This is a retraction for $(V,\id)_\ast$, and the composite $(V,\id)_\ast\circ r$ sends the string of $V(H\mathcal{D}C;H\mathcal{D}M,\underline{i})$ above to
\[\xymatrix{d_{2k+1}\ar@{=}[d]\ar[rr]^-{J(m\cdot\phi_{2k+1})}&&Dc_{0}\ar@{=}[d]\ar[r]^-{Da_1}&Dc_1\ar@{=}[d]\ar[r]&\dots\ar[r]&d_{2k}\ar@{=}[d]\ar[r]^-{b_{2k+1}}&dc_{2k+1}\ar@{=}[d]\\
d_{2k+1}\ar[rr]^-{J(m\cdot\phi_{2k+1})}&&Dc_{0}\ar[r]^-{Da_1}&Dc_1\ar[r]&\dots\ar[r]&d_{2k}\ar[r]^-{b_{2k+1}}&dc_{2k+1}
}\]

We define a $G$-homotopy from the identity to this composite by means of the maps  $H_\sigma\colon V(HC;HM,\underline{i})\longrightarrow V(HC;HM,\underline{i})$ defined as follows. For $\sigma=(0^\alpha1^{k+1-\alpha})\in\Delta[1]_k$ define $H_\sigma$ to be the identity if $\alpha=0$, the composite $(V,\id)_\ast\circ r$ if $\alpha=k+1$, and otherwise it sends the string above to
\[\resizebox{1\hsize}{!}{\xymatrix{
d_{2k\!+\!1}\ar@{=}[d]\ar[rr]^-{J(m\cdot\phi_{2k+1})}&&Dc_{0}\ar@{=}[d]\ar[r]^-{Da_1}&\dots\ar[r]^-{Da_{\alpha\!-\!1}}&Dc_{\alpha\!-\!1}\ar@{=}[d]\ar[rrdd]^{b_\alpha\circ\phi_{\alpha\!-\!1}^{-1}}
\\
d_{2k\!+\!1}\ar[rr]_-{J(m\cdot\phi_{2k+1})}&&Dc_{0}\ar[r]_-{Da_1}&\dots\ar[r]_-{Da_{\alpha\!-\!1}}&Dc_{\alpha\!-\!1}\ar[rd]_{Da_\alpha}\\
&&&&&
Dc_{\alpha}\ar[d]_-{Da_{\alpha\!+\!1}}&d_{\alpha}\ar[d]^-{b_{\alpha\!+\!1}}\ar[l]^-{\phi_\alpha}
\\
&&&&&
\vdots\ar[d]_-{Da_{\overline{\alpha}\!-\!1}}&\vdots\ar[d]^-{b_{\overline{\alpha}\!-\!1}}
\\
&&&&&
Dc_{\overline{\alpha}\!-\!1}\ar[dl]_-{\phi_{\overline{\alpha}}^{-1}\circ Da_{\overline{\alpha}}}&d_{\overline{\alpha}\!-\!1}\ar[ddll]^{b_{\overline{\alpha}}}\ar[l]_-{\phi_{\overline{\alpha}\!-\!1}}
\\
d_{2k\!+\!1}\ar@{=}[d]&&d_{2k}\ar@{=}[d]\ar[ll]_-{b_{2k\!+\!1}}&\dots\ar[l]_-{b_{2k}}&d_{\overline{\alpha}}\ar@{=}[d]\ar[l]_-{b_{\overline{\alpha}\!+\!1}}\\
d_{2k\!+\!1}&&d_{2k}\ar[ll]^-{b_{2k\!+\!1}}&\dots\ar[l]_-{b_{2k}}&d_{\overline{\alpha}}\ar[l]^-{b_{\overline{\alpha}\!+\!1}}
}}\]
where $\overline{\alpha}=2k+2-\alpha$. The maps $H_\sigma$ give a $G$-homotopy
\[sd_e\THR_{\sbt}(H\mathcal{D}C;H\mathcal{D}M)\times\Delta[1]\longrightarrow sd_e\THR_{\sbt}(H\mathcal{D}C;H\mathcal{D}M)\]
from the identity to $(V,\id)_\ast\circ r$.
\end{proof}

\begin{cor}\label{inversenonstrict}
Let $(C,D)$ and $(C',D')$ be $Ab$-enriched categories with additive strict dualities, $(M,J)$ and $(M',J')$ bimodules with strict duality respectively over $C$ and $C'$, and \[(F,\Phi)\colon (C,M)\longrightarrow (C',M')\] a morphism of bimodules with duality.
If the functor $F$ is an equivalence of categories and the natural maps
\[\Phi\colon M(c,c')\longrightarrow M'(Fc,Fc')\]
are all isomorphisms, the induced map
\[(F,\Phi)_\ast\colon\THR(HC;HM)\longrightarrow\THR(HC';HM')\]
is a $G$-equivalence.
\end{cor}

\begin{proof}
Since $F$ is an equivalence of categories, by \ref{equivalenceequivariant} it is an equivalence of categories with duality, i.e. there is a morphism of categories with duality
\[(G,\xi)\colon (C',D')\longrightarrow (C,D)\]
and natural isomorphisms $\epsilon\colon FG\Rightarrow\id_{C'}$ and $\mu\colon \id_C\Rightarrow GF$ respecting the dualities in the sense of definition \ref{eqcatdual}. Notice that even though the dualities on $C$ and $C'$ are strict, the isomorphism $\xi\colon GD'\Rightarrow DG$ does not need to be the identity. Define a bimodule map $\Psi\colon M'\Rightarrow M\circ(G^{op}\otimes G)$ as the composite
\[\xymatrix{\Psi\colon M'(d,d')\ar[rr]^-{M'(\epsilon\otimes\epsilon^{-1})}_{\cong}&&M'(FGd,FGd')\ar[r]^-{\Phi^{-1}}&M(Gd,Gd')
}\]
This construction induces a morphism of bimodules with duality
\[(\mathcal{D}(G,\xi),\mathcal{D}\Psi)\colon (\mathcal{D}C',\mathcal{D}M')\longrightarrow (\mathcal{D}C,\mathcal{D}M)\]
The natural isomorphisms $\epsilon$ and $\mu$ induce $G$-natural isomorphisms between the compositions of $(\mathcal{D}(G,\xi),\mathcal{D}\Psi)$ and $(\mathcal{D}(F,\id),\mathcal{D}\Phi)$ and the identities. By \ref{THRpreservesGequiv} above, the map
\[(\mathcal{D}(F,\id),\mathcal{D}\Phi)_\ast\colon\THR(H\mathcal{D}C;H\mathcal{D}M)\longrightarrow \THR(H\mathcal{D}C';H\mathcal{D}M')\]
is a $G$-equivalence, with $G$-homotopy inverse induced by $(G,\Psi)$. Proposition \ref{scriptdandthr} provides a commutative diagram
\[\xymatrix{\THR(HC;HM)\ar[d]_-{\simeq}\ar[rr]^-{(F,\Phi)_\ast}&&\THR(HC';HM')\ar[d]^-{\simeq}\\
\THR(H\mathcal{D}C;H\mathcal{D}M)\ar[rr]^-{\simeq}_-{(\mathcal{D}(F,\id),\mathcal{D}\Phi)_\ast}&&\THR(H\mathcal{D}C';H\mathcal{D}M')
}\]
where the vertical maps are $G$-equivalences. Therefore $(F,\Phi)_\ast$ is an equivalence as well.
\end{proof}


\subsection{The $\mathbb{Z}/2$-fixed points of $\THR$}

Let $(\mathcal{C},D)$ be an $Sp^\Sigma$-category with duality, and $(\mathcal{M},J)$ a bimodule with duality over $(\mathcal{C},D)$. We saw in §\ref{duality} that the $G$-fixed points of the realization of a real space can be described as the realization of a simplicial space by means of subdivision. Thus, if we want to say something about $\THR(\mathcal{C};\mathcal{M})^{G}$ we better understand the spaces $\THR_{2k+1}(\mathcal{C};\mathcal{M})^{G}$. Recall the explicit description of the fixed points category $I[2k+1]^G\cong I^G\times I^k\times I^G$ from proposition \ref{fixedIk}. To simplify the notation let $\mathcal{G}=\mathcal{G}_{2k+1}(\mathcal{C};\mathcal{M})$, so that 
\[\THR_{2k+1}(\mathcal{C};\mathcal{M})=\hocolim_{I[2k+1]}\mathcal{G}\]
In §\ref{defTHR} we defined maps
\[\omega_{2k+1}\colon\mathcal{G}(\underline{i})\longrightarrow \mathcal{G}(\omega_{2k+1}\underline{i})\]
that induce the $G$-action on $\THR_{2k+1}(\mathcal{C};\mathcal{M})$.
If $\underline{i}\in I[2k+1]^G$ is a fixed point,
\[\omega_{2k+1}\colon\mathcal{G}(\underline{i})\longrightarrow \mathcal{G}(\omega_{2k+1}\underline{i})=\mathcal{G}(\underline{i})\]
defines an involution on $\mathcal{G}(\underline{i})$.
Let $(\mathcal{G}|_{I[2k+1]^G})^G$ be the functor $I[2k+1]^G\longrightarrow \Top_\ast$ that sends $\underline{i}$ to $\mathcal{G}(\underline{i})^G$.

\begin{prop}\label{fixedhocolim} There is a natural homeomorphism
\[\THR_{2k+1}(\mathcal{C};\mathcal{M})^{G}\cong \hocolim_{I[2k+1]^G}(\mathcal{G}|_{I[2k+1]^G})^G\]
\end{prop}

\begin{proof}
This is a general result based on properties of the homotopy colimit. Indeed, suppose that $J$ is a category with involution $\omega$, $X\colon J\longrightarrow \Top_\ast$ a functor and $\tau\colon X\longrightarrow X\circ\omega$ a natural transformation such that
$\tau_{\omega j}\tau_j=\id_{X(j)}$. This induces an action on $\hocolim_J X$
\[\hocolim_J X\stackrel{\tau}{\longrightarrow}\hocolim_J (X\circ\omega)\longrightarrow\hocolim_J X\]
We show that $(\hocolim_J X)^G\cong\hocolim_{J^G}(X|_{J^G})^G$ by explicit calculation. The space $\hocolim_J X$ is the realization of the simplicial space
\[[k]\longmapsto \bigvee_{\underline{\sigma}\in\mathcal{N}_kJ}X(\sigma_0)\]
where $\underline{\sigma}=(\sigma_0\rightarrow \sigma_1\rightarrow\dots\rightarrow \sigma_k)$. The action on $\hocolim_J X$ is the realization of the simplicial map given in degree $k$ by
\[(x\in X(\sigma_0),\underline{\sigma})\longmapsto (\tau_{\sigma_0}x\in X(\omega\sigma_0),\omega\sigma_0\rightarrow \omega\sigma_1\rightarrow\dots\rightarrow \omega\sigma_k)\]
Since the action is simplicial, fixed points functor and realization functor commute, so
\[(\hocolim_J X)^G\cong|(\bigvee_{\underline{\sigma}\in\mathcal{N}_{\sbt} J}X(\sigma_0))^G|=|\bigvee_{\underline{\sigma}\in(\mathcal{N}_{\sbt} J)^G}X(\sigma_0)^G|\]
Since $(\mathcal{N}_{\sbt} J)^G=\mathcal{N}_{\sbt} J^G$, and therefore
\[(\hocolim_J X)^G\cong|\bigvee_{\underline{\sigma}\in\mathcal{N}_{\sbt} J^G}X(\sigma_0)^G|=\hocolim_{J^G}(X|_{J^G})^G\]
\end{proof}

We next examine the space of equivariant maps
\[\mathcal{G}_{2k+1}(\mathcal{C};\mathcal{M})(\underline{i})^G=\map_{\ast}(S^{\underline{i}},V(\mathcal{C};\mathcal{M},\underline{i}))^G\]
for $\underline{i}\in I[2k+1]^G$,
where $S^{\underline{i}}=S^{i_0}\wedge\dots\wedge S^{i_{2k+1}}$ has $G$-action
\[\xymatrix{S^{\underline{i}}=S^{i_0}\wedge\dots\wedge S^{i_{2k+1}}\ar[r]^{\omega_{i_0}\wedge\dots\wedge\omega_{i_{2k+1}}}&S^{i_0}\wedge\dots\wedge S^{i_{2k+1}}\ar[d]^{\cong}\\
&S^{\underline{i}}=S^{\omega\underline{i}}=S^{i_0}\wedge S^{i_{2k+1}}\wedge\dots\wedge S^{i_1}
}\]
where the vertical homeomorphism reverses the order of the last $2k+1$ smash factors. The action on $V(\mathcal{C};\mathcal{M},\underline{i})$ is given by the maps $\overline{\omega}_k$ defined in §\ref{defTHR}.
We begin by describing  $(S^{\underline{i}})^G$ and $V(\mathcal{C};\mathcal{M},\underline{i})^G$ and refer to §\ref{secghtpytheory}. The fixed set of $\omega_j\colon S^j\longrightarrow S^j$ is $S^{\lceil\frac{j}{2}\rceil}$, and using that $\underline{i}\in I[2k+1]^G$ one gets
\[(S^{\underline{i}})^G\cong S^{\lceil\frac{i_0}{2}\rceil}\wedge S^{\lceil\frac{i_1+\dots+i_{2k+1}}{2}\rceil}\cong S^{\lceil\frac{i_0}{2}\rceil+\lceil\frac{|\underline{i}|-i_0}{2}\rceil}\]
Here $\lceil \ \rceil$ denotes the upper integral approximation and $|\underline{i}|$ is the sum of the components of $\underline{i}$.

We now describe $V(\mathcal{C};\mathcal{M},\underline{i})^G$, for $\underline{i}\in I[2k+1]^G$. The space $\mathcal{C}(c,Dc)_{i}$ has a $G$-action by the involution $D\circ\omega_i$, where $\omega_i\in\Sigma_i$ is as usual the permutations that reverses the order of the elements. Similarly $J\circ\omega_i$ gives a $G$-action on $\mathcal{M}(c,Dc)_{i}$.

\begin{prop}\label{fixedptdescr}
For $\underline{i}=(i_0,\dots,i_k,i_{k+1},i_k,\dots,i_1)\in I[2k+1]^G$ there is a natural homeomorphism
\[\resizebox{1\hsize}{!}{$V(\mathcal{C};\mathcal{M},\underline{i})^G\cong\bigvee \mathcal{M}(c_0,D(c_0))_{i_0}^G\wedge\mathcal{C}(c_1,c_0)_{i_1}\wedge\dots\wedge\mathcal{C}(c_k,c_{k-1})_{i_k}\wedge\mathcal{C}(D(c_k),c_{k})_{i_{k+1}}^G$}\]
where the wedge runs over the $(k+1)$-tuples $c_0,\dots,c_k$ of objects of $\mathcal{C}$.

In particular if $\mathcal{C}$ and $\mathcal{M}$ are $G$-connected, the connectivity of $V=V(\mathcal{C};\mathcal{M},\underline{i})$ is \[(\conn V,\conn V^G)=(|\underline{i}|-1,\lceil\frac{i_0}{2}\rceil+\lceil\frac{|\underline{i}|-i_0}{2}\rceil-1)\]
\end{prop}

\begin{proof}
By definition, $V(\mathcal{C};\mathcal{M},\underline{i})^G$ is the wedge over the set of $(2k+2)$-tuples of objects $(c_0,\dots,c_{2k+1})$ fixed by the action of the spaces
\[(\mathcal{M}(c_0,c_k)_{i_0}\wedge\mathcal{C}(c_1,c_0)_{i_1}\wedge\dots\wedge\mathcal{C}(c_{2k+1},c_{2k})_{i_{2k+1}})^G\]
Now, a sequence $(c_0,\dots,c_{2k+1})$ is fixed if and only if
\[(c_0,\dots,c_{2k+1})=(D(c_{2k+1}),\dots,D(c_{0}))\]
that is, if $(c_0,\dots,c_{2k+1})$ is of the form
\[(c_0,\dots,c_k,Dc_k,\dots,Dc_0)\]
Therefore $V(\mathcal{C};\mathcal{M},\underline{i})^G$ is the wedge over $(c_0,\dots,c_{k})$ of the fixed point spaces
\[(\mathcal{M}(c_0,D(c_0))_{i_0}\wedge\mathcal{C}(c_1,c_0)_{i_1}\wedge\dots\wedge\mathcal{C}(D(c_{0}),D(c_{1}))_{i_{2k+1}})^G\]
An element $m\wedge f_1\wedge\dots\wedge f_{2k+1}$ is in this space if and only if
\[m\wedge f_1\wedge\dots\wedge f_{2k+1}=J\omega_{i_0}m\wedge D\omega_{i_{2k+1}}f_{2k+1}\wedge\dots\wedge D\omega_{i_1}f_{1}\]
that is, if and only if it is of the form
\[m\wedge f_1\wedge\dots\wedge f_k\wedge f_{k+1}\wedge D\omega_{i_k}f_k\wedge\dots\wedge D\omega_{i_1}f_1\]
with $m=J\omega_{i_0}m$ and $f_{k+1}=D\omega_{i_{k+1}}f_{k+1}$. This means that the last $k$ smash factors are determined by the first $k$, and projecting them off one gets the homeomorphism of the statement.

When $\mathcal{C}$ and $\mathcal{M}$ are $G$-connected the space $V(\mathcal{C};\mathcal{M},\underline{i})$ is non-equivariantly $(|\underline{i}|-1)$-connected, as the connectivity of the smash product of two spaces is the sum of the connectivities plus one. By our description, the connectivity of $V(\mathcal{C};\mathcal{M},\underline{i})^G$ is
\[\lceil\frac{i_0}{2}\rceil+i_1+\dots +i_k+\lceil\frac{i_{k+1}}{2}\rceil-1=\lceil\frac{i_0}{2}\rceil+\lceil\frac{|\underline{i}|-i_0}{2}\rceil-1\]
\end{proof}


\subsection{Equivariant approximation lemma}\label{Gapproxsec}

In this section we prove two closely related results about the behavior of the homotopy colimit defining $\THR$. One shows that if the connectivity of a map of diagrams over $I[k]$ "tends to infinity", it induces an equivalence on the homotopy colimits. The other is an equivariant version of Bökstedt approximation lemma for $\THH$.

We define a partial order on the set of objects $ObI[k]=\mathbb{N}^{k+1}$ by declaring $\underline{i}\leq\underline{j}$ if the components satisfy $i_l\leq j_l$ for all $0\leq l\leq k$. The subset $ObI[k]^G\subset ObI[k]$ has the induced order.
Given a partially ordered set $J$, we say that a map  $\lambda\colon J\longrightarrow \mathbb{N}$ tends to infinity on $J$ if given any $n\in \mathbb{N}$ there is a $j_0\in J$ such that $\lambda(j)\geq n$ for all $j\geq j_0$.

\begin{prop}\label{connectgoestoinfty} Let $J$ be either $I[k]$ or $I[k]^G$, let $X,Y\colon J\longrightarrow \Top_\ast$ be two functors and $f\colon X\Rightarrow Y$ a natural transformation. Suppose that for all object $i\in J$ the map $f_i\colon X(i)\longrightarrow Y(i)$ is  $\lambda(i)$-connected, for a function $\lambda\colon J\longrightarrow \mathbb{N}$ that tends to infinity on $J$. Then the induced map
\[f\colon \hocolim_J X\stackrel{\simeq}{\longrightarrow} \hocolim_J Y\]
is a weak equivalence.
\end{prop}

\begin{prop}[G-approximation lemma]\label{eqapproxlemma} Let  $\mathcal{M}$ be a bimodule with duality over an $Sp^\Sigma$-category with duality $\mathcal{C}$, and suppose that both $\mathcal{C}$ and $\mathcal{M}$ are $G$-connected (cf. \ref{defGconn}). Then
for every natural number $n$, there is a $\underline{i}\in I[2k+1]^G$ such that the natural $G$-equivariant map
\[\mathcal{G}_{2k+1}(\mathcal{C};\mathcal{M})(\underline{i}')\longrightarrow \THR_{2k+1}(\mathcal{C};\mathcal{M})\]
is $(n,n)$-connected for all $\underline{i}\leq\underline{i}'\in I[2k+1]^G$.
\end{prop}

These two statements depend only on structural properties of $I[k]$ and $I[k]^G$ and on the connectivity properties of the functor $\mathcal{G}_{2k+1}(\mathcal{C};\mathcal{M})$. Indeed, \ref{connectgoestoinfty} and \ref{eqapproxlemma} follow from the general statements \ref{gapproxlemma1} and \ref{genapproxlemma} below, respectively.

\begin{defn}[cf. {\cite[1.3]{bok}}]
A \textbf{good index category} is a triple $(J,\overline{J},\mu)$
where
\begin{enumerate}[i)]
\item $J$ is a category and $\overline{J}\subset J$ is a full subcategory
\item $\mu=\{\mu_j\colon J\longrightarrow J \ |j\in \overline{J}\}$ is a family of functors, one for each object $j\in \overline{J}$
\item for every $j\in \overline{J}$, there is a natural transformation $U\colon \id\Rightarrow \mu_j$ such that \[\mu_j(U_i)=U_{\mu_j(i)}\colon \mu_j(i)\longrightarrow \mu_j\mu_j(i)\]
\end{enumerate}
\end{defn}

\begin{ex} Given any pointed category $J$, there is a structure of trivial good index category on $J$ given by the triple $(J,\ast,\id_J)$, where the identity of $J$ is the unique functor of the family. Notice though that for this good index category the statement \ref{gapproxlemma1} below is trivial.
\end{ex}

\begin{ex}\label{Ikgoodindexcat}
We want to define a useful structure of good index category on $I[k]$ and $I[k]^G$. We start with the $k=0$ case. We denote $I_e$ the full subcategory of $I$ with objects the even natural numbers.
For $j\in I_e$ let $\mu_j\colon I\longrightarrow I$ be the functor that sends an object $i$ to
\[\mu_j(i)=i+j=\frac{j}{2}+i+\frac{j}{2}\]
and a map $\alpha \colon i\longrightarrow i'$ to
\[\mu_j(\alpha)=\id_{j/2}+\alpha+\id_{j/2}\]
It commutes strictly with the involution on $I$. It should be thought of as an equivariant variation of the functor that adds $j$ using the monoidal structure on $I$. There is a natural transformation $U\colon \id\Rightarrow \mu_j$, defined at an object $i\in I$ by the middle inclusion \[U_i\colon i\longrightarrow \frac{j}{2}+i+\frac{j}{2}\] that sends $s$ to $s+j/2$.
 Each map $U_i$ is equivariant. The triple $(I,I_e,\mu)$ is then a good index category.
Now let $I^{G}_e$ be the full subcategory of $I^G$ on even objects. For $j\in I^{G}_e$, the functor $\mu_j$
 restricts to a functor $\mu^{G}_j\colon I^G\longrightarrow I^G$, and the natural transformation $U$ gives a natural transformation $\id\Rightarrow \mu_{j}^G$. The triple $(I^G,I^{G}_e,\mu^G)$ is also a good index category.

For higher $k$, define $\underline{j}\in I[k]$ to be even if all of its components are even natural numbers. Denote $I[k]_e$ and $I[k]^{G}_e$ respectively the full subcategories of $I[k]$ and $I[k]^G$ on the even objects. For $\underline{j}\in I[k]_e$, the product functor
\[\mu_{\underline{j}}=\mu_{j_0}\times\dots\times\mu_{j_{k}}\colon I[k]\longrightarrow I[k]\]
defines a good index category $(I[k],I[k]_e,\mu)$. Similarly for $\underline{j}\in I[k]^{G}_e$ the restrictions $\mu_{\underline{j}}^G$ give a good index category $(I[k]^G,I[k]^{G}_e,\mu^G)$.
\end{ex}

\begin{rem}
Our definition of good index category differs from the one one of Bökstedt of \cite[1.3]{bok}. In \cite{bok}, $I$ is equipped with the monoidal structure $+\colon I\times I\longrightarrow I$. This gives a good index category in our sense, defined by the triple $(J,J,\mu)$ with $\mu_j=j+(-)$. We need a more general definition because $+$ does not restrict to a functor $I^G\times I^G\longrightarrow I^G$ since the sum of equivariant maps is not equivariant.

The techniques used for our proofs are generalizations of the methods used in the approximation lemma proof of \cite[IV-2.2.3]{Dundasbook}. 
\end{rem}

Let $(J,\overline{J},\mu)$ be a good index category and $\lambda\colon ObJ\longrightarrow \mathbb{N}$ a function that tends to infinity in the following sense: given $n\in\mathbb{N}$ there is a $j_0\in \overline{J}$ such that $\lambda(j)\geq n$ for all $j\in\mu_{j_0}(J)$.

\begin{lemma}\label{gapproxlemma1} Let $(J,\overline{J},\mu)$ be a good index category, $X,Y\colon J\longrightarrow\Top_\ast$ two functors and $f\colon X\Rightarrow Y$ a natural transformation. Suppose that $f_j$ is at least $\lambda(j)$-connected, for a function $\lambda\colon ObJ\longrightarrow \mathbb{N}$ that tends to infinity. Then the induced map
\[f_\ast\colon \hocolim_J X\stackrel{\simeq}{\longrightarrow} \hocolim_J Y\]
is a weak equivalence.
\end{lemma}

\begin{proof}
Given any natural number $n$ we prove that $f_\ast$ is $n$-connected. Take a $j_0$ for $n$ as in the statement, and consider the diagram
\[\xymatrix{\hocolim_J X\ar[r]^-{f_\ast}&\hocolim_J Y\\
\hocolim_J X\circ \mu_{j_0}\ar[u]\ar[r]_{(f|_{\mu_{j_0}})_\ast}&\hocolim_J Y\circ \mu_{j_0}\ar[u]
}\]
Where the vertical maps are the canonical maps, and the bottom map is the restriction of $f$. We claim that the vertical maps are equivalences. Given that, it suffices to prove that $(f|_{\mu_{j_0}})_\ast$ is $n$-connected. For every object $j\in J$ the map
\[f_{\mu_{j_0}(j)}\colon X(\mu_{j_0}(j))\longrightarrow Y(\mu_{j_0}(j))\]
is $\lambda(\mu_{j_0}(j))\geq n$-connected. Since homotopy colimits preserve connectivity (cf. \cite[IV-4.13]{goersjardine}), $(f|_{\mu_{j_0}})_\ast$ is also $n$-connected.

In order to prove that the vertical maps are equivalences, it suffices to prove that the under categories $i/\mu_{j_0}$ are contractible for every $i\in J$ (see \cite[XI-9.2]{bouskan}). This follows from the natural transformation $U\colon \id\Rightarrow\mu_{j_0}$. Indeed, set $F:=\mu_{j_0}$ and consider the inclusion $\ast\longrightarrow i/F$ into the object $(i,U_i\colon i\longrightarrow F(i))$. In order to build a homotopy between the projection to $(i,U_i)$ and the identity, we exhibit a functor $G\colon i/F\longrightarrow i/F$ and natural transformations $\id\Rightarrow G$ and $\ast\Rightarrow G$, where $\ast$ is the constant functor at $(i,U_i)$. For an object $(i',f\colon i\longrightarrow F(i'))$ of $i/F$, define
\[G(i',f)=(F(i'),i\stackrel{f}{\longrightarrow} F(i')\stackrel{F(U_{i'})}{\longrightarrow}FF(i'))\]
A morphism $a\colon (i',f)\longrightarrow (i'',g)$ consisting of a map $a\colon i'\longrightarrow i''$ such that $F(a)\circ f=g$, is sent by $G$ to $F(a)\colon (F(i'),F(U_{i'})\circ f)\longrightarrow (F(i''),F(U_{i''})\circ g)$.
The natural transformation $\id\Rightarrow G$ is defined at an object $(i',f)$ by the map $U_{i'}\colon i'\longrightarrow F(i')$. Indeed the diagram
\[\xymatrix{i\ar[r]^{f}\ar[dr]_{G(f)}&F(i')\ar[d]^{F(U_{i'})}\\
& FF(i')
}\]
commutes by definition of $G$. Define the natural transformation $\ast\Rightarrow G$ at an object $(i',f\colon i\longrightarrow F(i'))$ by the morphism $(i,U_i)\longrightarrow G(i',f)$ given by the map $f\colon i\longrightarrow F(i')$. To see that it is a well defined morphism, we must show that
\[\xymatrix{i\ar[r]^{U_i}\ar[dr]_{G(f)}&F(i)\ar[d]^{F(f)}\\
& FF(i')
}\]
commutes, that is $F(f)\circ U_i=F(U_{i'})\circ f$. Since we assumed that $F(U_{i'})=U_{F(i')}$, this is naturality of $U$ for the map $f$.
\end{proof}

\begin{proof}[Proof of \ref{connectgoestoinfty}]
Let us first consider $J=I[k]$, and take the good index category $(I[k],I[k]_e,\mu)$ of the example \ref{Ikgoodindexcat} above. Notice that $\underline{i}$ belongs to $\mu_{\underline{j}}(I)$ exactly when $\underline{i}\geq \underline{j}$.
Therefore the connectivity function $\lambda\colon\mathbb{N}^{k+1}\longrightarrow \mathbb{N}$ for the natural transformation $f$ tends to infinity if and only if the condition of \ref{gapproxlemma1} is satisfied. Therefore $f_\ast$ is an equivalence by \ref{gapproxlemma1}.
Since $ObI[k]^G$ has the order induced from $I[k]^G$, the same argument applies to the good index category $(I[k]^G,I[k]^{G}_e,\mu^G)$.
\end{proof}

\begin{lemma}\label{genapproxlemma}
Let $(J,\overline{J},\mu)$ be a good index category such that $J$ has an initial object $0\in J$ with $\mu_j(0)=j$ for all $j\in \overline{J}$. Let $X\colon J\longrightarrow \Top_\ast$ be a functor such that for a fixed $j\in \overline{J}$ the map
\[X(j)=X(\mu_j(0))\longrightarrow X(\mu_j(i))\]
induced by the map $0\longrightarrow i$ is $n$-connected for all $i\in J$.

Then the canonical map
\[X(j)\longrightarrow \hocolim_J X\]
is $n$-connected.
\end{lemma}

\begin{proof}
The canonical map $X(j)\longrightarrow \hocolim_J X$ fits into a commutative diagram
\[\xymatrix{X(j)\ar@{=}[d]\ar[r]&\hocolim\limits_{J}X\\
X(\mu_{j}(0))\ar[r]&\hocolim\limits_{J} X\circ\mu_{j}\ar[u]^{\simeq}_{\ref{gapproxlemma1}}
}\]
The right vertical map is the canonical map, which is an equivalence by the proof of \ref{gapproxlemma1} above. Therefore it is enough to show that the bottom row
\[X(\mu_{j}(0))\longrightarrow \hocolim_{J} X\circ\mu_{j}\]
is $n$-connected. Looking at $X(\mu_{j}(0))$ as the constant diagram, this map factors as
\[X(\mu_{j}(0))\stackrel{\simeq}{\longrightarrow} \hocolim_{J} X(\mu_{j}(0))\longrightarrow \hocolim\limits_{J} X\circ\mu_{j}\]
Since $J$ has an initial object, its nerve is contractible, and therefore the first map is an equivalence.
The second map is $n$-connected since by assumption each $X(\mu_j(0))\longrightarrow X(\mu_j(i))$ is $n$-connected, and homotopy colimits preserve connectivity.
\end{proof}

\begin{proof}[Proof of \ref{eqapproxlemma}] For simplicity, write $\mathcal{G}_{2k+1}=\mathcal{G}_{2k+1}(\mathcal{C};\mathcal{M})$.
The two good index categories $(I[2k+1],I[2k+1]_e,\mu)$ and $(I[2k+1]^G,I[2k+1]^{G}_e,\mu^G)$ of \ref{Ikgoodindexcat} both have an initial object $\emptyset$ with the property of the proposition above.
By proposition \ref{fixedhocolim}, the statement on the fixed points follows if 
for every $n$ there is a $\underline{j}$ in $I[2k+1]^G$ such that
\[\mathcal{G}_{2k+1}(\underline{j})^G=\mathcal{G}_{2k+1}(\mu_{\underline{j}}(\emptyset))^G\longrightarrow \mathcal{G}_{2k+1}(\mu_{\underline{j}}(\underline{i}))^G\]
is $n$-connected for all $\underline{i}\in I[2k+1]^G$. We prove this in lemma \ref{annoyingone} below. 
Similarly, the non-equivariant statement of \ref{eqapproxlemma} follows from the last proposition if we can prove that for every $n$ there is a $\underline{j}$ in $I[2k+1]^G$ such that
\[\mathcal{G}_{2k+1}(\underline{j})=\mathcal{G}_{2k+1}(\mu_{\underline{j}}(\emptyset))\longrightarrow \mathcal{G}_{2k+1}(\mu_{\underline{j}}(\underline{i}))\]
is $n$-connected for all $\underline{i}\in I[2k+1]^G$. The proof of this is similar to the proof of \ref{annoyingone} below.
\end{proof}

The proof of lemma \ref{annoyingone} depends on the connectivity of the map \[\sigma_{\underline{i}\underline{j}}\colon V(\mathcal{C};\mathcal{M},\underline{j})\wedge S^{\underline{i}}\longrightarrow V(\mathcal{C};\mathcal{M},\underline{j}+\underline{i})\]
defined to be the composition
\[\xymatrix{V(\mathcal{C};\mathcal{M},\underline{j})\wedge S^{\underline{i}}\ar[d]^{\cong}\\
\bigvee (\mathcal{M}(c_0,c_{2k+1})_{j_0}\wedge\mathcal{C}(c_1,c_0)_{j_1}\wedge \dots\wedge\mathcal{C}(c_{2k+1},c_{2k})_{j_{2k+1}}\ar[d]^{\cong}\wedge S^{\underline{i}})\\
\bigvee \mathcal{M}(c_0,c_{2k+1})_{j_0}\wedge S^{i_0}\wedge\dots\wedge\mathcal{C}(c_{2k+1},c_{2k})_{j_{2k+1}}\wedge S^{i_{2k+1}}\ar[d]\\
\bigvee \mathcal{M}(c_0,c_{2k+1})_{j_0+i_0}\wedge\dots\wedge\mathcal{C}(c_{2k+1},c_{2k})_{j_{2k+1}+i_{2k+1}}\ar@{=}[d]\\
V(\mathcal{C};\mathcal{M},\underline{j}+\underline{i})
}\]
where the first two maps are canonical isomorphisms, the third map is induced by the structure maps of the spectra and where $\underline{j}+\underline{i}$ is defined pointwise. When $\underline{i}$ and $\underline{j}$ belong to $I[2k+1]^G$, the map $\sigma_{\underline{i}\underline{j}}$ is $G$-equivariant with respect to the diagonal action on the source, and the action on the target $V(\mathcal{C};\mathcal{M},\underline{j}+\underline{i})$ defined by the involution
\[\xymatrix{\bigvee \mathcal{M}(c_0,c_{2k+1})_{j_0+i_0}\wedge\dots\wedge\mathcal{C}(c_{2k+1},c_{2k})_{j_{2k+1}+i_{2k+1}}\ar[d]^{\bigvee(\omega_{j_0}\times\omega_{i_0})\wedge\dots\wedge(\omega_{j_{2k+1}}\times\omega_{i_{2k+1}})}\\
\bigvee\mathcal{M}(c_0,c_{2k+1})_{j_0+i_0}\wedge\dots\wedge\mathcal{C}(c_{2k+1},c_{2k})_{j_{2k+1}+i_{2k+1}}\ar[d]^{\bigvee J\wedge D\wedge\dots\wedge D}\\
\bigvee\mathcal{M}(c_0,c_{2k+1})_{j_0+i_0}\wedge\dots\wedge\mathcal{C}(c_{2k+1},c_{2k})_{j_{2k+1}+i_{2k+1}}\ar[d]^{\cong}\\
\bigvee\mathcal{M}(c_0,c_{2k+1})_{j_0+i_0}\wedge\mathcal{C}(c_{2k+1},c_{2k})_{j_{2k+1}+i_{2k+1}}\wedge\dots\wedge\mathcal{C}(c_{1},c_{0})_{j_{0}+i_{0}}
}\]
Notice that this action is different from the standard action on $V(\mathcal{C};\mathcal{M},\underline{j}+\underline{i})$ defined in §\ref{defTHR}, where we used the permutations $\omega_{j_l+i_l}$ instead of $\omega_{j_l}\times\omega_{i_l}$.
For a $\underline{i}\in I[k]$, we denote $\min (\underline{i})$ its smallest component. Recall that $\epsilon$ denotes the connectivity constant from the $G$-connectivity of $\mathcal{C}$ and $\mathcal{M}$.

\begin{lemma}\label{smashstructmaps} If $\mathcal{C}$ and $\mathcal{M}$ are $G$-connected,
the map \[\sigma_{\underline{i}\underline{j}}\colon V(\mathcal{C};\mathcal{M},\underline{j})\wedge S^{\underline{i}}\longrightarrow V(\mathcal{C};\mathcal{M},\underline{j}+\underline{i})\] is $(\min(\underline{j})+|\underline{j}|+|\underline{i}|-1-\epsilon)$-connected for all $\underline{i},\underline{j}\in I[2k+1]$.

When, $\underline{i},\underline{j}\in I[2k+1]^G$ are fixed by the action, its restriction to the fixed points is \[\min(\underline{j})+\lceil\frac{j_0}{2}\rceil+\lceil\frac{|\underline{j}|-j_0}{2}\rceil+\lceil\frac{i_0}{2}\rceil+\lceil\frac{|\underline{i}|-i_0}{2}\rceil-\epsilon-1\]
connected.
\end{lemma}

\begin{proof}
The first two maps in the composition defining $\sigma_{\underline{i}\underline{j}}$ form an equivariant homeomorphism. We are therefore only interested in the third map.

Since the wedge of $n$-connected maps is $n$-connected, it is enough to see how connected the smash of the structure maps is.
Recall that if $f\colon A\longrightarrow X$ and $g\colon B\longrightarrow Y$ are respectively $n$ and $m$ connected maps of spaces, the connectivity of their smash product $f\wedge g\colon A\wedge B\longrightarrow X\wedge Y$ is $\min\{n+\conn(B),m+\conn(X)\}$.
Since the $l$-th structure map appearing in the smash is $(2j_l+i_l-\epsilon)$-connected by assumption, the connectivity of their smash product is
\[\resizebox{1\hsize}{!}{$\begin{array}{ll}\min\limits_{l}\{(2j_l+i_l-\epsilon)+(\sum_{h\neq l}j_h+i_h)-1\}&=\min\limits_{l}\{j_l+\sum_{h}(j_h+i_h)\}-1-\epsilon\\
&=\min(\underline{j})+|\underline{j}|+|\underline{i}|-1-\epsilon\end{array}$}\]

Now suppose that $\underline{i},\underline{j}\in I[2k+1]^G$.
On the fixed points $\sigma_{\underline{i}\underline{j}}$ is, under the homeomorphism of \ref{fixedptdescr}, the map
\[\xymatrix{*\txt{$\bigvee (\mathcal{M}(c_0,D(c_0))_{j_0}\wedge S^{i_0})^G\wedge\mathcal{C}(c_1,c_0)_{j_1}\wedge S^{i_1}\wedge\dots$\\
$\dots\wedge\mathcal{C}(c_k,c_{k-1})_{j_k}\wedge{S^{i_k}}\wedge(\mathcal{C}(D(c_k),c_{k})_{j_{k+1}}\wedge S^{i_{k+1}})^G$}\ar[d]\\
*\txt{$\bigvee \mathcal{M}(c_0,D(c_0))_{j_0+i_0}^G\wedge\mathcal{C}(c_1,c_0)_{j_1+i_1}\wedge\dots$\\
$\dots\wedge\mathcal{C}(c_k,c_{k-1})_{j_k+i_k}\wedge\mathcal{C}(D(c_k),c_{k})_{j_{k+1}+i_{k+1}}^G$}
}\]
given by the wedge of the smash of the structure maps, where the first and the last are restricted to the fixed points. For any $l=0,\dots,k+1$ let $c_l$ be the connectivity of the $l$-th map appearing in the smash product above, plus the connectivity of the sources of the other $(k+1)$-maps, so that the connectivity of $\sigma_{\underline{i}\underline{j}}$ on the fixed points is the minimum of the $c_l$'s. We compute these numbers. 
Notice that since $\underline{i}\in I[2k+1]^G$,
\[|\underline{i}|=i_0+2(\sum_{1\leq l\leq k}i_l)+i_{k+1}\]
and similarly for $\underline{j}$.
By our connectivity assumptions
\[\begin{array}{ll}c_0&=(j_0+\lceil\frac{i_0}{2}\rceil-\epsilon)+\sum_{1\leq h\leq k}(j_h+i_h)+\lceil\frac{j_{k+1}}{2}\rceil+\lceil\frac{i_{k+1}}{2}\rceil-1\\
&=(j_0+\dots+j_k+\lceil\frac{j_{k+1}}{2}\rceil)+(\lceil\frac{i_{0}}{2}\rceil+i_1+\dots+i_k+\lceil\frac{i_{k+1}}{2}\rceil)-1-\epsilon
\end{array}\]
and similarly
\[\begin{array}{ll}c_{k+1}&=(j_{k+1}+\lceil\frac{i_{k+1}}{2}\rceil-\epsilon)+\lceil\frac{j_0}{2}\rceil+\lceil\frac{i_0}{2}\rceil+\sum_{1\leq h\leq k}(j_h+i_h)-1\\
&=(\lceil\frac{j_{0}}{2}\rceil+j_1+\dots+j_{k+1})+(\lceil\frac{i_{0}}{2}\rceil+i_1+\dots+i_k+\lceil\frac{i_{k+1}}{2}\rceil)-1-\epsilon
\end{array}\]
For $1\leq l\leq k$, we have that
\[\resizebox{1\hsize}{!}{$\begin{array}{ll}c_l&\!\!=\!(2j_l+i_l-\epsilon)+\lceil\frac{j_0}{2}\rceil+\lceil\frac{i_0}{2}\rceil+\sum_{1\leq h\neq l\leq k}(j_h+i_h)+\lceil\frac{j_{k+1}}{2}\rceil+\lceil\frac{i_{k+1}}{2}\rceil-1\\
&\!\!=\!(\lceil\frac{j_{0}}{2}\rceil\!+\!j_1\!+\!\dots\!+\!2j_l\!+\!\dots\!+\!j_{k}\!+\!\lceil\frac{j_{k+1}}{2}\rceil)+(\lceil\frac{i_{0}}{2}\rceil\!+\!i_1\!+\!\dots\!+\!i_k\!+\!\lceil\frac{i_{k+1}}{2}\rceil)\!-\!1\!-\!\epsilon\end{array}$}\]
The minimum of these numbers is at least
\[\resizebox{1\hsize}{!}{$\begin{array}{ll}\min(\underline{j})+
(\lceil\frac{j_{0}}{2}\rceil\!+\!j_1\!+\!\dots\!+\!j_{k}\!+\!\lceil\frac{j_{k+1}}{2}\rceil)+(\lceil\frac{i_{0}}{2}\rceil\!+\!i_1\!+\!\dots\!+\!i_k\!+\!\lceil\frac{i_{k+1}}{2}\rceil)\!-\!1\!-\!\epsilon\\
=\min(\underline{j})+\lceil\frac{j_0}{2}\rceil+\lceil\frac{|\underline{j}|-j_0}{2}\rceil+\lceil\frac{i_0}{2}\rceil+\lceil\frac{|\underline{i}|-i_0}{2}\rceil-\epsilon-1\end{array}$}\]
\end{proof}

\begin{lemma}\label{annoyingone}
For every natural number $n$, there is an even $\underline{j}\in I[2k+1]^G$ such that the map induced by $\emptyset\longrightarrow \underline{i}$
\[\mathcal{G}_{2k+1}(\underline{j})^G=\mathcal{G}_{2k+1}(\mu_{\underline{j}}(\emptyset))^G\longrightarrow \mathcal{G}_{2k+1}(\mu_{\underline{j}}(\underline{i}))^G\]
is $n$-connected for all $\underline{i}\in I[2k+1]^G$.
\end{lemma}

\begin{proof}
By definition of $\mathcal{G}_{2k+1}$ on morphisms (see e.g. \cite[IV-1.2.1]{Dundasbook}), the induced map $\mathcal{G}_{2k+1}(\underline{j})^G\longrightarrow \mathcal{G}_{2k+1}(\mu_{\underline{j}}(\underline{i}))^G$ sends $f\in \mathcal{G}_{2k+1}(\underline{j})^G$ to the composite
\[\resizebox{1\hsize}{!}{\xymatrix{
S^{\frac{\underline{j}}{2}\!+\!\underline{i}\!+\!\frac{\underline{j}}{2}}\ar[d]_-{(\id_{\frac{j_{0}}{2}}\!\!\!+\chi_{i_0,\frac{j_0}{2}}\!\!)\wedge\dots\wedge (\id_{\frac{j_{2k\!+\!1}}{2}}\!\!\!+\chi_{i_{2k\!+\!1},\frac{j_{2k\!+\!1}}{2}}\!\!)}\ar[r]& V(\mathcal{C};\!\mathcal{M},\!\underline{j}\!+\!\underline{i})\\
S^{\frac{\underline{j}}{2}\!+\!\underline{i}\!+\!\frac{\underline{j}}{2}}\ar[d]^{\cong}
&V(\mathcal{C};\!\mathcal{M},\!\underline{j}\!+\!\underline{i})\ar[u]_-{(\id_{\frac{j_{0}}{2}}\!\!\!+\chi_{\frac{j_{0}}{2},i_0}\!\!)\wedge\dots\wedge (\id_{\frac{j_{2k\!+\!1}}{2}}\!\!\!+\chi_{\frac{j_{2k\!+\!1}}{2},i_{2k\!+\!1}}\!\!\!\!)}\\
S^{\underline{j}}\!\wedge\!S^{\underline{i}}\ar[r]^-{f\wedge\id}&
V(\mathcal{C};\!\mathcal{M},\!\underline{j})\!\wedge\!S^{\underline{i}}\ar[u]_{\sigma_{\underline{i},\underline{j}}}\\
}}\]
The non-labeled homeomorphism in the diagram groups the last $i_l$-factors in every $(j_l+i_l)$-component. Precisely, it is the smash of the homeomorphisms 
\[S^{\frac{j_l}{2}+i_l+\frac{j_l}{2}}\cong \underbrace{S^{1}\wedge\dots\wedge S^{1}}_{j_l+i_l}=(\underbrace{S^{1}\wedge\dots\wedge S^{1}}_{j_l})\wedge(\underbrace{S^{1}\wedge\dots\wedge S^{1}}_{i_l})\cong S^{j_l}\wedge S^{i_l}\]
We recall from \ref{defdualityfromop} that the permutation $\chi_{p,q}\in \Sigma_{p+q}$ interchanges the first $p$ and the last $q$ elements.
The composition of the two vertical maps on the left is the equivariant homeomorphism $S^{\frac{\underline{j}}{2}+\underline{i}+\frac{\underline{j}}{2}}\cong S^{\underline{j}}\wedge S^{\underline{i}}$ induced by
\[\xymatrix{(x_1,\dots,x_{\frac{j_l}{2}},x_{\frac{j_l}{2}+1},\dots, x_{\frac{j_l}{2}+i_l},x_{\frac{j_l}{2}+i_l+1},\dots x_{j_l+i_l})\ar@{|->}[d]\\
(x_1,\dots,x_{\frac{j_l}{2}},x_{\frac{j_l}{2}+i_l+1},\dots, x_{j_l+i_l},x_{\frac{j_l}{2}+1},\dots, x_{\frac{j_l}{2}+i_l})}\]
Similarly, the map $(\id_{\frac{j_{0}}{2}}+\chi_{i_0,\frac{j_0}{2}})\wedge\dots\wedge (\id_{\frac{j_{2k+1}}{2}}+\chi_{i_{2k+1},\frac{j_{2k+1}}{2}})$ is an equivariant homeomorphism when the source is equipped with the action induced from the permutations $\omega_{j_l}\times\omega_{i_l}$, as defined just before \ref{smashstructmaps}, and the target with the standard action induced by the permutations $\omega_{j_l+i_l}$ of §\ref{defTHR}. Thus, we just need to show that the composite
\[\xymatrix{\mathcal{G}_{2k+1}(\underline{j})^G\ar@{-->}[drr]\ar[rr]^-{(-)\wedge S^{\underline{i}}}&&\map_\ast(S^{\underline{j}}\wedge S^{\underline{i}},V(\mathcal{C};\mathcal{M},\underline{j})\wedge S^{\underline{i}})^G\ar[d]^{(-)\circ\sigma_{\underline{i}\underline{j}}}\\
&&\map_\ast(S^{\underline{j}}\wedge S^{\underline{i}},V(\mathcal{C};\mathcal{M},\underline{j}+\underline{i}))^G}\]
sending $f$ to $\sigma_{\underline{i},\underline{j}}\circ(f\wedge\id)$ is highly connected. By corollary \ref{eqmappingspace} of the Whitehead theorem, the connectivity of $(-)\circ\sigma_{\underline{i},\underline{j}}$ is the connectivity of $\sigma_{\underline{i},\underline{j}}$ minus the dimension of $(S^{\underline{j}}\wedge S^{\underline{i}})^G$. By the  homeomorphism
\[(S^{\underline{j}}\wedge S^{\underline{i}})^G\cong S^{\lceil\frac{j_0}{2}\rceil+\lceil\frac{|\underline{j}|-j_0}{2}\rceil}\wedge S^{\lceil\frac{i_0}{2}\rceil+\lceil\frac{|\underline{i}|-i_0}{2}\rceil}\]
and \ref{smashstructmaps}, this difference is
\[\min(\underline{j})-1-\epsilon\]
Corollary \ref{corgsusp} of the $G$-suspension theorem gives us a way to compute the connectivity of $(-)\wedge S^{\underline{i}}$. Indeed, using that $V(\mathcal{C};\mathcal{M},\underline{j})$ is \[(|\underline{j}|-1,\lceil\frac{j_0}{2}\rceil+\lceil\frac{|\underline{j}|-j_0}{2}\rceil-1)\]
connected, one finds after using \ref{corgsusp} that $(-)\wedge S^{\underline{i}}$ is
\[\lfloor\frac{j_0}{2}\rfloor+j_1+\dots+j_k+\lfloor\frac{j_{k+1}}{2}\rfloor-1\]
connected. Therefore the dashed map above is at least 
\[\lfloor\frac{\min(\underline{j})}{2}\rfloor-\epsilon-1\]
connected, which can be chosen to be arbitrarily big with $\underline{j}$.
\end{proof}


\subsection{$\THR$ and stable equivalences}

We describe another class of maps of bimodules that induce $G$-equivalences on $\THR$.

\begin{defn}
A morphism of bimodules with duality $(F,\Phi)\colon (\mathcal{C},\mathcal{M})\longrightarrow (\mathcal{C}',\mathcal{M}')$ is a \textbf{stable $G$-equivalence} if the functor $F\colon \mathcal{C}\longrightarrow \mathcal{C}'$ is a bijection on objects, and there are functions $\delta,\delta^G\colon\mathbb{N}\longrightarrow\mathbb{N}$ such that
\begin{enumerate}[i)]
\item $\delta(i)-i$ and $\delta^G(i)-\lceil\frac{i}{2}\rceil$ tend to infinity with $i$,
\item For every objects $c,d$ of $\mathcal{C}$, the maps $F\colon\mathcal{C}(c,d)_i\longrightarrow\mathcal{C}'(Fc,Fd)_i$ and 
$\Phi\colon \mathcal{M}(c,d)_i\longrightarrow\mathcal{M}'(Fc,Fd)_i$
are at least $\delta(i)$-connected,
\item For every object $c$ of $\mathcal{C}$, the maps $F^G\colon\mathcal{C}(c,Dc)_{i}^G\longrightarrow\mathcal{C}'(Fc,FDc)_{i}^G$ and 
$\Phi^G\colon \mathcal{M}(c,Dc)_{i}^G\longrightarrow\mathcal{M}'(Fc,FDc)_{i}^G$
are at least $\delta^G(i)$-connected.
\end{enumerate}
\end{defn}

\begin{prop}\label{stableGeq}
Suppose that $\mathcal{C}$ and $\mathcal{M}$ are $G$-connected. A stable $G$-equivalence $(F,\Phi)\colon (\mathcal{C},\mathcal{M})\longrightarrow (\mathcal{C}',\mathcal{M}')$ induces a weak $G$-equivalence
\[(F,\Phi)_\ast\colon \THR(\mathcal{C};\mathcal{M})\longrightarrow \THR(\mathcal{C}';\mathcal{M}')\]
\end{prop}

\begin{proof}
Since $(F,\Phi)$ is a stable equivalence, the connectivity of the map
\[F_i\colon\mathcal{C}(c,d)_i\longrightarrow \mathcal{C}(Fc,Fd)_i\]
is a function $\delta(i)$ of $i$ such that $\delta(i)-i$ tends to infinity with $i$, and similarly for $\mathcal{M}$.
Recall that if $f\colon A\longrightarrow X$ and $g\colon B\longrightarrow Y$ are respectively $n$ and $m$ connected maps of spaces, the connectivity of their smash product $f\wedge g\colon A\wedge B\longrightarrow X\wedge Y$ is $\min\{n+\conn(B),m+\conn(X)\}$. Therefore the connectivity of
\[(F,\Phi)_\ast\colon V(\mathcal{C};\mathcal{M},\underline{i})\longrightarrow V(\mathcal{C}';\mathcal{M}',\underline{i})\]
is
\[\min\limits_l\{\delta(i_l)+(\sum\limits_{0\leq h\neq l\leq k}i_h)-1\}\]
for each $\underline{i}\in I[k]$.
Notice that here we are using that $F$ is a bijection on objects, to make sure to reach all the wedge components. The map
\[(F,\Phi)_\ast\colon \Omega^{\underline{i}}V(\mathcal{C};\mathcal{M},\underline{i})\longrightarrow \Omega^{\underline{i}}V(\mathcal{C}';\mathcal{M}',\underline{i})\]
is thus
\[\min\limits_l\{\delta(i_l)+(\sum\limits_{0\leq h\neq l\leq k}i_h)-1\}-|\underline{i}|=\min\limits_l\{\delta(i_l)-i_l\}-1\]
connected.
This tends to infinity on $I[k]$, and therefore it induces an equivalence on homotopy colimits by \ref{connectgoestoinfty}. Consequently
\[(F,\Phi)_\ast\colon \THR(\mathcal{C};\mathcal{M})\longrightarrow \THR(\mathcal{C}';\mathcal{M}')\]
is an equivalence.

To show that $(F,\Phi)_\ast$ is an equivalence on the fixed points, we show that it is a levelwise equivalence on the fixed points of the subdivision. For 
\[\underline{i}=(i_0,i_1,\dots,i_k,i_{k+1},i_k,\dots,i_1)\in I[2k+1]^G\]
the map
$(F,\Phi)^{G}_\ast\colon V(\mathcal{C};\mathcal{M},\underline{i})^G\longrightarrow V(\mathcal{C}';\mathcal{M}',\underline{i})^G$
is, by the description of of the fixed points of \ref{fixedptdescr},  the wedge over the tuples $(c_0,\dots,c_k)$ of the smash products of the maps
\[F_l\colon\mathcal{C}(c_{l},c_{l-1})_{i_l}\longrightarrow \mathcal{C}'(Fc_{l},Fc_{l-1})_{i_l}\] for $1\leq l\leq k$, and of the maps  
\[\begin{array}{ll}F_{k+1}^G\colon\mathcal{C}(Dc_k,c_{k})_{i_{k+1}}^G\longrightarrow \mathcal{C}'(FDc_{k},Fc_{k})_{i_{k+1}}^G\\
\Phi_0^G\colon\mathcal{M}(c_{0},Dc_{0})_{i_{0}}^G\longrightarrow \mathcal{M}'(c_{0},Dc_{0})_{i_{0}}^G\end{array}\]
Thus its connectivity is the minimum of the following three quantities. The connectivity of $F_{k+1}^G$ plus the connectivity of the sources of the other maps in the smash, which is
\[\delta^G(i_{k+1})+\lceil\frac{i_0}{2}\rceil+i_1+\dots+i_k-1\]
The connectivity of $\Phi_{0}^G$ plus the connectivity of the sources of the other maps, which is
\[\delta^G(i_{0})+i_1+\dots+i_k+\lceil\frac{i_{k+1}}{2}\rceil-1\]
And finally the minimum over $1\leq l\leq k$ of the connectivity of $F_l$ plus the connectivity of the sources of the other maps, which is
\[\delta(i_l)+\lceil\frac{i_0}{2}\rceil+\sum_{1\leq h\neq l\leq k}i_h+\lceil\frac{i_{k+1}}{2}\rceil-1\]
Now recall that there is a homeomorphism
\[(S^{\underline{i}})^G\cong S^{\lceil\frac{i_0}{2}\rceil+i_1+\dots+i_k+\lceil\frac{i_{k+1}}{2}\rceil}\]
By the Whitehead theorem \ref{eqmappingspace}, the induced map
\[(F,\Phi)_\ast\colon \Omega^{\underline{i}}V(\mathcal{C};\mathcal{M},\underline{i})^G\longrightarrow \Omega^{\underline{i}}V(\mathcal{C}';\mathcal{M}',\underline{i})^G\]
is therefore
\[\min\left\{\delta^G(i_{k+1})-\lceil\frac{i_{k+1}}{2}\rceil,\delta^G(i_{0})-\lceil\frac{i_{0}}{2}\rceil,\min\limits_{1\leq l\leq k}\{\delta(i_l)-i_l\}\right\}-1\]
This tends to infinity on $I[2k+1]^G$ by the properties of $\delta$ and $\delta^G$, and therefore it induces an equivalence on homotopy colimits by \ref{connectgoestoinfty}. By the description of the fixed points of the homotopy colimit \ref{fixedhocolim}, the induced map
\[(F,\Phi)^{G}_\ast\colon \THR_{2k+1}(\mathcal{C};\mathcal{M})^G\longrightarrow\THR_{2k+1}(\mathcal{C}';\mathcal{M}')^G\]
is an equivalence. This gives the desired equivalence on the realization.
\end{proof}


\subsection{Abelian group model for $\THR$}\label{abgpmodelsec}

Let $C$ be a category enriched in abelian groups with strict duality $D$, and $M$ a bimodule with duality on $(C,D)$. We are going to show that $\THR(HC;HM)$ is $G$-equivalent to the realization of a real topological abelian group. For an abelian group $A$ we denote
\[A_{i}=A\otimes\mathbb{Z}(S^i)=A(S^i)\]
the space of reduced configurations. 

For all $\underline{i}\in I[k]$ define a topological abelian group
\[V^{\oplus}(HC;HM,\underline{i})=\bigoplus M(c_0,c_k)_{i_0}\otimes\mathbb{Z}(C(c_1,c_0)_{i_1}\wedge\dots\wedge C(c_k,c_{k-1})_{i_k})\]
where the direct sum runs over $(k+1)$-tuples of objects $c_0,\dots,c_k$ of $C$. 
Define a functor $\mathcal{G}_{k}^{\oplus}(HC;HM)\colon I[k]\longrightarrow\Top_\ast$ on objects by
\[\mathcal{G}_{k}^{\oplus}(HC;HM)(\underline{i})=\map_\ast(S^{\underline{i}},V^\oplus(HC;HM,\underline{i}))\]
and on morphisms analogously to $\mathcal{G}_{k}(HC;HM)$.
Then define
\[\THR_{k}^{\oplus}(HC;HM)=\hocolim_{I[k]}\mathcal{G}_{k}^{\oplus}(HC;HM)\]
As for $\THR$, the spaces $\THR_{k}^{\oplus}(HC;HM)$ form a simplicial topological abelian group. We define $\THR^{\oplus}(HC;HM)$ to be its realization.

There is a levelwise $G$-action on $\THR_{k}^{\oplus}(HC;HM)$, induced by the maps \[\mathcal{G}_{k}^{\oplus}(HC;HM)(\underline{i})\longrightarrow \mathcal{G}_{k}^{\oplus}(HC;HM)(\omega\underline{i})\] which conjugate an $f\in \map_\ast(S^{\underline{i}},V^\oplus(HC;HM,\underline{i}))$ by $S^{\omega\underline{i}}\cong S^{\underline{i}}\stackrel{\omega_{i_0}\wedge\dots\wedge\omega_{i_k}}{\longrightarrow}S^{\underline{i}}$ and by the map 
\[ V^{\oplus}(HC;HM,\underline{i})\longrightarrow V^{\oplus}(HC;HM,\omega\underline{i})\]
defined by the composition
\[\xymatrix{V^\oplus(HC;HM,\underline{i})\ar@{=}[d]\\
\bigoplus M(c_0,c_k)_{i_0}\otimes\mathbb{Z}(C(c_1,c_0)_{i_1}\wedge\dots\wedge C(c_k,c_{k-1})_{i_k})\ar[d]^\cong \\
\bigoplus M(c_0,c_k)_{i_0}\otimes\mathbb{Z}(C(c_k,c_{k-1})_{i_k}\wedge\dots\wedge C(c_1,c_{0})_{i_1})\ar[d]^{\omega_{i_0}\otimes \mathbb{Z}(\omega_{i_k}\wedge\dots\wedge \omega_{i_1})}\\
\bigoplus M(c_0,c_k)_{i_0}\otimes\mathbb{Z}(C(c_k,c_{k-1})_{i_k}\wedge\dots\wedge C(c_1,c_{0})_{i_1})\ar[d]^{J\otimes \mathbb{Z}(D\wedge\dots\wedge D)}\\
\bigoplus M(Dc_{k},Dc_0)_{i_0}\otimes \mathbb{Z}(C(Dc_{k-1},Dc_{k})_{i_k}\wedge\dots\wedge C(Dc_{0},Dc_{1})_{i_{1}})\ar@{=}[d]\\
V^\oplus(HC;HM,\omega\underline{i})
}\]
The the natural maps 
\[M(c_0,c_k)_{i_0}\wedge X\longrightarrow M(c_0,c_k)_{i_0}\otimes \mathbb{Z}(X)\]
with $X=C(c_1,c_0)_{i_1}\wedge\dots\wedge C(c_k,c_{k-1})_{i_k}$,
together with the inclusion of wedge products into direct sums, give maps $V(HC;HM,\underline{i})\longrightarrow V^{\oplus}(HC;HM,\underline{i})$ respecting the structure maps and the action. This induces a map of real spaces
\[\THR_{\sbt}(HC;HM)\longrightarrow \THR_{\sbt}^{\oplus}(HC;HM)\]
This is the main result of this section.

\begin{prop}\label{abgpmodel}
The map
\[\THR(HC;HM)\longrightarrow \THR^{\oplus}(HC;HM)\]
is a $G$-equivalence.
\end{prop}

\begin{proof}
We prove that the simplicial map is a levelwise equivalence, and that it induces a levelwise equivalence on the fixed points of the subdivisions. The main point is that for every $G$-CW-complex $X$ the map $M(c,d)_{i}\wedge X\longrightarrow M(c,d)_{i}\otimes \mathbb{Z}(X)$ is equivariantly
\[(2i+\conn X,i+\min\{\conn X^G,\conn X\})\]
connected.
The proof of this result is technical, and is deferred to appendix \ref{connectvitysmashwithx}.

Non equivariantly, the map $\THR_{k}(HC;HM)\longrightarrow \THR_{k}^{\oplus}(HC;HM)$ is induced by the composite
\[\xymatrix{
\bigvee\limits_{c_0,\dots,c_k}M(c_0,c_k)_{i_0}\wedge C(c_1,c_0)_{i_1}\wedge\dots\wedge C(c_{k},c_{k-1})_{i_{k}}\ar[d]\\
\bigvee\limits_{c_0,\dots,c_k}M(c_0,c_k)_{i_0}\otimes \mathbb{Z}(C(c_1,c_0)_{i_1}\wedge\dots\wedge C(c_{k},c_{k-1})_{i_{k}})\ar[d]\\
\bigoplus\limits_{c_0,\dots,c_k} M(c_0,c_k)_{i_0}\otimes\mathbb{Z}(C(c_1,c_0)_{i_1}\wedge\dots\wedge C(c_k,c_{k-1})_{i_k})
}\]

The first map is induced by $M(c_0,c_k)_{i_0}\wedge X\longrightarrow M(c_0,c_k)_{i_0}\otimes \mathbb{Z}(X)$, for $X=C(c_1,c_0)_{i_1}\wedge\dots\wedge C(c_k,c_{k-1})_{i_k}$ and it is
\[2(i_0-1)+(i_1+\dots+i_k-1)=i_0+|\underline{i}|-3\]
connected by \ref{connectvitysmashwithx}. The second map is $(2|\underline{i}|-1)$-connected as inclusion of wedges into sums, so the composite is $(i_0+|\underline{i}|-3)$-connected. Thus \[\mathcal{G}_{k}(HC;HM)(\underline{i})\longrightarrow\mathcal{G}_{k}^{\oplus}(HC;HM)(\underline{i})\] is $(i_0-3)$-connected, which tends to infinity on $I[k]$. By \ref{connectgoestoinfty} it induces a weak equivalence on the homotopy colimits.

Given a fixed point $\underline{i}=(i_0,i_1,\dots,i_k,i_{k+1},i_k,\dots,i_1)\in I[2k+1]^G$, we need to describe the map $V(HC;HM,\underline{i})^G\longrightarrow V^{\oplus}(HC;HM,\underline{i})^G$. Proposition \ref{fixedptdescr} gives us an explicit description for the source of this map. The action of $G$ on the set $ObC^{2k+1}$ is $D\circ\omega_{2k+2}$ with fixed set $ObC^{k+1}\cong\{c_0,\dots,c_k,Dc_k,\dots,Dc_0\}$ and we may write
\[ObC^{2k+1}\cong ObC^{k+1}\coprod G\times (ObC^{2k+1})_f\]
where $(ObC^{2k+1})_f$ is a set of representatives for the orbits of the free elements of $ObC^{2k+1}$.
Then $V^{\oplus}(HC;HM,\underline{i})^G$ is homeomorphic by \ref{wedgesintoproducts} to
\[\begin{array}{ll}\bigoplus\limits_{c_0,\dots,c_k}\left(M(c_0,Dc_0)_{i_0}\otimes \mathbb{Z}(C(c_1,c_0)_{i_1}\wedge\dots\wedge C(Dc_{0},Dc_{1})_{i_{1}})\right)^G\oplus\\
\bigoplus\limits_{(ObC^{2k+1})_f} M(c_0,c_{2k+1})_{i_0}\otimes\mathbb{Z}(C(c_1,c_0)_{i_1}\wedge\dots\wedge C(c_{2k+1},c_{2k})_{i_{1}})\end{array}\]
Under this identification, the map $V(HC;HM,\underline{i})^G\longrightarrow V^{\oplus}(HC;HM,\underline{i})^G$ is the composition of inclusions
\[\xymatrix{V(HC;HM,\underline{i})^G\ar@{=}[d]\\
\bigvee \left(M(c_0,Dc_0)_{i_0}\wedge C(c_1,c_0)_{i_1}\wedge\dots\wedge C(Dc_{0},Dc_{1})_{i_{1}}\right)^G\ar[d]\\
\bigvee \left(M(c_0,Dc_0)_{i_0}\otimes\mathbb{Z}(C(c_1,c_0)_{i_1}\wedge\dots\wedge C(Dc_{0},Dc_{1})_{i_{1}})\right)^G\ar[d]\\
\bigoplus \left(M(c_0,Dc_0)_{i_0}\otimes\mathbb{Z}(C(c_1,c_0)_{i_1}\wedge\dots\wedge C(Dc_{0},Dc_{1})_{i_{1}})\right)^G\ar[d]\\
*\txt{$\bigoplus\left(M(c_0,D(c_0))_{i_0}\otimes \mathbb{Z}(C(c_1,c_0)_{i_1}\wedge\dots\wedge C(Dc_{0},Dc_{1})_{i_{1}})\right)^G\oplus$\\
$\bigoplus\limits_{(ObC^{2k+1})_f} M(c_0,c_{2k+1})_{i_0}\otimes\mathbb{Z}(C(c_1,c_0)_{i_1}\wedge\dots\wedge C(c_{2k+1},c_{2k})_{i_1})$}\ar@{=}[d]\\
V^{\oplus}(HC;HM,\underline{i})^G
}\]
where the unlabeled wedges and sums are over the objects $c_0,\dots,c_k\in C$. The first map is the map \[(M(c_0,Dc_0)_{i_0}\wedge X)^G\longrightarrow (M(c_0,Dc_0)_{i_0}\otimes \mathbb{Z}(X))^G\]
for $X=C(c_1,c_0)_{i_1}\wedge\dots\wedge C(Dc_{k},c_{k})_{i_{k+1}}\wedge\dots\wedge C(Dc_{0},Dc_{1})_{i_{1}}$. It is 
$(i_0+\min\{\conn X^G,\conn X\})$-connected by \ref{connectvitysmashwithx}, and since
\[\xymatrix@R=6pt{(C(c_1,c_0)_{i_1}\wedge\dots\wedge C(Dc_{k},c_{k})_{i_{k+1}}\wedge\dots\wedge C(Dc_{0},Dc_{1})_{i_{1}})^G \ar@{=}[d]\\ C(c_1,c_0)_{i_1}\wedge\dots\wedge C(c_k,c_{k-1})_{i_k}\wedge C(Dc_{k},c_{k})_{i_{k+1}}^G}\]
this connectivity is
\[i_0+i_1+\dots+i_{k}+\lceil\frac{i_{k+1}}{2}\rceil-1\]
The second map is an inclusion of wedges into products, and it is therefore roughly twice as connected as $(M(c_0,Dc_0)_{i_0}\otimes \mathbb{Z}(X))^G$ is. This can be written as
\[(M(c_0,Dc_0)_{i_0}\otimes \mathbb{Z}(X))^G\cong (M(c_0,Dc_0)(S^{i_0}\wedge X))^G\]
which is $(\lceil\frac{i_0}{2}\rceil+i_1+\dots+i_k+\lceil\frac{i_{k+1}}{2}\rceil-1)$-connected by \ref{connconfspace}.
Thus the second map is $(|\underline{i}|-1)$-connected. The third map is as connected as the second summand of its target, which is $(|\underline{i}|-1)$-connected. Therefore the connectivity of the composite is the minimum of these three quantities, given by
\[i_0+i_1+\dots+i_{k}+\lceil\frac{i_{k+1}}{2}\rceil-1\]
By the corollary of the Whitehead theorem \ref{eqmappingspace}, the map
\[\mathcal{G}_{2k+1}(HC;HM)(\underline{i})^G\longrightarrow \mathcal{G}_{2k+1}^{\oplus}(HC;HM)(\underline{i})^G\]
is the connectivity above minus the dimension of the fixed points of $S^{\underline{i}}$, which is
\[i_0+i_1+\dots+i_{k}+\lceil\frac{i_{k+1}}{2}\rceil-1-(\lceil\frac{i_{0}}{2}\rceil+i_1+\dots+i_{k}+\lceil\frac{i_{k+1}}{2}\rceil)=\lfloor\frac{i_{0}}{2}\rfloor-1\]
connected. This tends to infinity on $I[2k+1]^G$, and by \ref{connectgoestoinfty} it induces an equivalence on homotopy colimits
\[\xymatrix{\hocolim_{I[2k+1]^G}\mathcal{G}_{2k+1}(HC;HM)(\underline{i})^G\cong \THR_{2k+1}(HC;HM)^G\ar[d]\\ \hocolim_{I[2k+1]^G}\mathcal{G}_{2k+1}^{\oplus}(HC;HM)(\underline{i})^G\cong \THR^{\oplus}_{2k+1}(HC;HM)^G}\]
\end{proof}


\subsection{A first delooping for $\THR(\mathcal{C};\mathcal{M})$}\label{firstdeloop}

Let $(\mathcal{C},D)$ be an $Sp^\Sigma$-category with duality, and $(\mathcal{M},J)$ a bimodule with duality over $(\mathcal{C},D)$.
In this section we show that $\THR(\mathcal{C};\mathcal{M})$ is the $0$-th space of a real spectrum $\uuline{\THR}(\mathcal{C};\mathcal{M})$ (see \ref{realspec}). Then we
prove that if $\mathcal{C}$ and $\mathcal{M}$ are $G$-connected (cf. \ref{defGconn}), $\uuline{\THR}(\mathcal{C};\mathcal{M})$ is a real $\Omega$-spectrum. Later on in \ref{THRdeloopings} we will describe a second delooping by the $S^{2,1}_{\sbt}$-construction in the case of split-exact categories.

For $\underline{i}=(i_0,\dots,i_k)\in I[k]$ and a $G$-space $X$, define
\[V(\mathcal{C};\!\mathcal{M},\!X,\!\underline{i})\!=\!\bigvee \mathcal{M}(c_0,c_k)_{i_0}\wedge\mathcal{C}(c_1,c_0)_{i_1}\wedge\mathcal{C}(c_2,c_1)_{i_2}\wedge\dots\wedge\mathcal{C}(c_k,c_{k-1})_{i_k}\wedge X\]
The $G$-space $X$ has the role of a dummy variable. Define
$\THR_{k}(\mathcal{C};\mathcal{M},X)$ as the homotopy colimit of the functor
$\mathcal{G}_k(\mathcal{C};\mathcal{M},X)\colon I[k]\longrightarrow \Top_\ast$,
\[\mathcal{G}_k(\mathcal{C};\mathcal{M},X)(\underline{i})=\map_\ast(S^{\underline{i}},V(\mathcal{C};\mathcal{M},X,\underline{i}))\]
Here the $G$-action is induced by the maps
\[\overline{\omega}_k\colon V(\mathcal{C};\mathcal{M},\underline{i})\longrightarrow V(\mathcal{C};\mathcal{M},\omega\underline{i})\]
from §\ref{defTHR}, smashed with the involution on $X$. This gives a real space $\THR_{\sbt}(\mathcal{C};\mathcal{M},X)$ whose realization is denoted $\THR(\mathcal{C};\mathcal{M},X)$.

Let $Y$ be another $G$-space.
For every $x\in X$ the  map $\iota_x\colon Y\longrightarrow Y\wedge X$ that sends $y$ to $y\wedge x$ induces
\[\iota_x\colon V(\mathcal{C};\mathcal{M},Y,\underline{i})\longrightarrow V(\mathcal{C};\mathcal{M},Y\wedge X,\underline{i})\]
and hence
\[(\iota_x)_\ast\colon \THR(\mathcal{C};\mathcal{M},Y)\longrightarrow \THR(\mathcal{C};\mathcal{M},Y\wedge X)\]
with $D (\iota_x)_\ast=(\iota_{\omega x})_\ast$. This induces a $G$-map
\[\THR(\mathcal{C};\mathcal{M},Y)\wedge X\longrightarrow \THR(\mathcal{C};\mathcal{M},Y\wedge X)\]
where the source is equipped with the diagonal action.

\begin{prop}\label{easydelooping} The family $\THR(\mathcal{C};\mathcal{M},S^{2m,m})$ with the structure maps
\[\xymatrix{\THR(\mathcal{C};\mathcal{M},S^{2m,m})\wedge S^{2n,n}\ar[r]\ar[dr]_{\sigma_{m,n}}&\THR(\mathcal{C};\mathcal{M},S^{2m,m}\wedge S^{2n,n})\ar@{=}[d]\\
&\THR(\mathcal{C};\mathcal{M},S^{2(m+n),m+n})}\]
defined above is a real spectrum, denoted $\uuline{\THR}(\mathcal{C};\mathcal{M})$.
If in addition $\mathcal{C}$ and $\mathcal{M}$ are $G$-connected, then $\uuline{\THR}(\mathcal{C};\mathcal{M})$ becomes a real $\Omega$-spectrum, and $\THR(\mathcal{C};\mathcal{M})$ an infinite real loop space.
\end{prop}

\begin{proof} It easy to check that $\uuline{\THR}(\mathcal{C};\mathcal{M})$ is indeed a real spectrum.

We check that the adjoints of the structure maps are weak $G$-equivalences in the $G$-connected case.
For every $\underline{i}\in I[2k+1]^G$, there is a commutative diagram of $G$-maps
\[\resizebox{1\hsize}{!}{\xymatrix{\THR_{2k\!+\!1}(\mathcal{C};\!\mathcal{M},\!S^{2m,m})\ar[dd]&\Omega^{\underline{i}}V(\mathcal{C};\!\mathcal{M},\!S^{2m,m},\!\underline{i})\ar[d]\ar[l]\\
 & \Omega^{\underline{i}}\Omega^{2n,n}(V(\mathcal{C};\!\mathcal{M},\!S^{2m,m},\!\underline{i})\!\wedge\!S^{2n,n})\ar[d]^{\cong}\\
\Omega^{2n,n}\THR_{2k\!+\!1}(\mathcal{C};\!\mathcal{M},\!S^{2(m\!+\!n),m\!+\!n}) & \Omega^{2n,n}\Omega^{\underline{i}}(V(\mathcal{C};\!\mathcal{M},\!S^{2m,m},\!\underline{i})\!\wedge\!S^{2n,n})\ar[l]
}}\]
where the horizontal maps are the canonical maps into the homotopy colimit, and the unlabeled vertical map on the right is induced by the adjoint of the identity of $V(\mathcal{C};\mathcal{M},S^{2m,m},\underline{i})\wedge S^{2n,n}$. The proof of the $G$-approximation lemma \ref{eqapproxlemma} works for $\THR(\mathcal{C};\mathcal{M},X)$, as lemma \ref{annoyingone} applies to $\mathcal{G}(\mathcal{C};\mathcal{M},X)$. Therefore one can choose $\underline{i}$ big enough so that the horizontal maps of the diagram are as connected as we like, both non-equivariantly and on the fixed points. By the suspension theorem, the vertical map is non-equivariantly
\[2(|\underline{i}|+2m-1)+1-|\underline{i}|=|\underline{i}|+2m-1\]
connected. By the $G$-suspension theorem, the connectivity of the fixed point map is
\[\resizebox{1\hsize}{!}{$\begin{array}{ll}\min\{|\underline{i}|\!-\!1,2(\lceil\frac{i_0}{2}\rceil\!+\!i_1\!+\!\dots\!+\!i_k\!+\!\lceil\frac{i_{k+1}}{2}\rceil)\!+\!1\}+m-(\lceil\frac{i_0}{2}\rceil\!+\!i_1\!+\!\dots\!+\!i_k\!+\!\lceil\frac{i_{k+1}}{2}\rceil)\\
=\lfloor\frac{i_0}{2}\rfloor+i_1+\dots+i_k+\lfloor\frac{i_{k+1}}{2}\rfloor+m-1\end{array}$}\]
which also tends to infinity with $|\underline{i}|$. This shows that
\[\THR_{2k+1}(\mathcal{C};\mathcal{M},S^{2m,m})\longrightarrow\Omega^{2n,n}\THR_{2k+1}(\mathcal{C};\mathcal{M},S^{2(m+n),m+n})\]
is $\nu$-connected for every $\nu$, and therefore a weak $G$-equivalence.
Its realization
\[\THR(\mathcal{C};\mathcal{M},S^{2m,m})\longrightarrow|\Omega^{2n,n}\THR_{\sbt}(\mathcal{C};\mathcal{M},S^{2(m+n),m+n})|\]
is therefore a weak $G$-equivalence.
Moreover the canonical map
\[\Omega^{2n,n}\THR(\mathcal{C};\mathcal{M},S^{2(m+n),m+n})\longrightarrow|\Omega^{2n,n}\THR_{\sbt}(\mathcal{C};\mathcal{M},S^{2(m+n),m+n})|\]
is a $G$-equivalence by \cite[2.4]{wittvectors}, since both $\THR(\mathcal{C};\mathcal{M},S^{2(m+n),m+n})$ and the subdivision $sd_e\THR(\mathcal{C};\mathcal{M},S^{2(m+n),m+n})^G$ are good simplicial spaces and levelwise connected when $n>0$.
\end{proof}


\subsection{$\THR$ of equivariant products}\label{secTHRprod}

Let $X$ be a finite set with an involution $\omega\colon X\longrightarrow X$, and $\{\mathcal{C}_x\}_{x\in X}$ a family of categories enriched in $Sp^{\Sigma}$ indexed over $X$. Suppose that there are enriched functors $D_x\colon\mathcal{C}_{x}^{op}\longrightarrow \mathcal{C}_{\omega x}$ such that $D_{\omega x}\circ D_{x}^{op}=\id_{\mathcal{C}_x}$.

Also suppose that $\mathcal{M}_x$ is a bimodule over $\mathcal{C}_x$, and that $J_x$ is a natural transformation
\[J_x\colon \mathcal{M}_x(c,d)\longrightarrow \mathcal{M}_{\omega x}(D_xd,D_xc)\]
with $J_{\omega_x}\circ J_x=\id_{\mathcal{M}_x}$, and such that the diagram
\[\xymatrix{\mathcal{C}_x(d,e)\wedge\mathcal{M}_x(c,d)\ar[dd]_{l}\ar[r]^-{D_x\wedge J_x} & \mathcal{C}_{\omega x}(D_xe,D_xd)\wedge\mathcal{M}_{\omega x}(D_xd,D_xc)\ar[d]^-{\gamma}
\\ & \mathcal{M}_{\omega x}(D_xd,D_xc)\wedge\mathcal{C}_{\omega x}(D_xe,D_xd)\ar[d]^-{r}\\
\mathcal{M}_x(c,e)\ar[r]^-{J_x} & \mathcal{M}_{\omega x}(D_xe,D_xc)
}\]
commutes.

Recall that the product and the wedge of symmetric spectra are defined levelwise.

\begin{defn}
Let $\prod_{X}\mathcal{C}_{x}$ be the category enriched in $Sp^{\Sigma}$ with objects $Ob\prod_{X}\mathcal{C}_{x}=\prod_{x\in X}Ob\mathcal{C}_x$ and morphisms spectra
\[(\prod_{X}\mathcal{C}_{x})(a,b)=\prod_{x\in X}\mathcal{C}_x(a_x,b_x)\]
Composition sends $(g\wedge f)\in \prod_{X}\mathcal{C}_{x}(b,c)_i\wedge\prod_{X}\mathcal{C}_{x}(a,b)_j$ to 
\[(f\circ g)_x=f_x\circ g_x\]
The identity of an object $a$ of $\prod_{X}\mathcal{C}_{x}$ is the product $\prod_{x\in X}\id_{a_x}\colon\mathbb{S}\longrightarrow \prod_{X}\mathcal{C}_{x}(a_x,a_x)$.

Define a duality $D\colon (\prod_{X}\mathcal{C}_{x})^{op}\longrightarrow \prod_{X}\mathcal{C}_{x}$ by sending an object $a$ to
\[(Da)_x=D_{\omega x}a_{\omega x}\in Ob\mathcal{C}_x\]
and a morphism $f\in \prod_{X}\mathcal{C}_{x}(a,b)_i$ to
\[(Df)_x\colon D_{\omega x}b_{\omega x}\stackrel{D_{\omega x}f_{\omega x}}{\longrightarrow} D_{\omega x}a_{\omega x}\in \mathcal{C}_x((Db)_{x},(Da)_{x})_i\]
Similarly define $\prod_{X}\mathcal{M}_{x}$ to be the bimodule over $\prod_{X}\mathcal{C}_{x}$ defined by
\[\prod_{X}\mathcal{M}_{x}(a,b)=\prod_{x\in X}\mathcal{M}_x(a_x,b_x)\]
with bimodule structure defined pointwise like for composition in $\prod_{X}\mathcal{C}_{x}$.
It has a duality $J\colon \prod_{X}\mathcal{M}_{x}(a,b)\longrightarrow\prod_{X}\mathcal{M}_{x}(Db,Da)$ defined in level $i$ by
\[J(m)_x=J_{\omega x}(m_{\omega x})\in\mathcal{M}_x((Db)_{x},(Da)_{x})_i\]
\end{defn}

For every element $x\in X$ there is a functor $\pr_x\colon \prod_{X}\mathcal{C}_{x}\longrightarrow\mathcal{C}_x$ that projects onto the $x$-component. These functors satisfy $D_x\pr_x=\pr_{\omega x}D$. Similarly, there are natural transformations \[\pr_x\colon \prod_{X}\mathcal{M}_{x}\Rightarrow\mathcal{M}_x\circ(\pr_{x}^{op}\wedge\pr_{x})\] satisfying the relation $J_x\pr_x=\pr_{\omega x}J$. These projections induce maps
\[\pr_x\colon \THR(\prod_{X}\mathcal{C}_{x},\prod_{X}\mathcal{M}_{x})\longrightarrow\THR(\mathcal{C}_x,\mathcal{M}_x)\]
for all $x$, fitting together into a map
\[\prod_{x\in X}\pr_x\colon \THR(\prod_{X}\mathcal{C}_{x},\prod_{X}\mathcal{M}_{x})\longrightarrow\prod_{X}\THR(\mathcal{C}_x,\mathcal{M}_x)\]
It is $G$-equivariant when the target has the following action. Let $Y_{\sbt}$ and $Z_{\sbt}$ be simplicial spaces, and $f_k\colon Y_k\longrightarrow Z_k$ a family of maps such that $d_lf_k=f_{k-1}d_{k-l}$ and $s_lf_k=f_{k+1}s_{k-l}$. Then there is an induced map $|f|\colon |Y_{\sbt}|\longrightarrow |Z|_{\sbt}$ that sends $[y;(t_0,\dots,t_k)]$ to $[f_k(y);(t_k,\dots,t_0)]$.
In our situation the maps $(D_x,J_x)$ satisfy this condition, and induce on realization
\[(D_x,J_x)\colon\THR(\mathcal{C}_{x},\mathcal{M}_{x})\longrightarrow \THR(\mathcal{C}_{\omega x},\mathcal{M}_{\omega x})\]
that clearly satisfy $(D_{\omega x},J_{\omega x})\circ (D_x,J_x)=\id$. Therefore they induce an action on $\prod_{X}\THR(\mathcal{C}_x,\mathcal{M}_x)$ as defined in \ref{wedgesintoproducts}.

\begin{prop}\label{THRproducts}
Suppose that each $\mathcal{C}_x$ and $\mathcal{M}_x$ is $0$-connected, and that for every fixed point $x\in X^G$, both $(\mathcal{C}_x,D_x)$ and $(\mathcal{M}_x,J_x)$ are $G$-connected.
Then the map
\[\prod_{x\in X}\pr_x\colon \THR(\prod_{X}\mathcal{C}_{x},\prod_{X}\mathcal{M}_{x})\longrightarrow\prod_{X}\THR(\mathcal{C}_x,\mathcal{M}_x)\]
is a $G$-equivalence.
\end{prop}

\begin{rem}
There are two special cases of particular interest. The first one is when the family is constant, that is $\mathcal{C}_x=\mathcal{C}$, is a $G$-connected category with duality $D_x=D$, and $\mathcal{M}_x=\mathcal{M}$ is a $G$-connected bimodule with duality $J_x=J$. We denote in this case $\prod_{X}\mathcal{C}_{x}=\mathcal{C}^X$ and $\prod_{X}\mathcal{M}_{x}=\mathcal{M}^X$. The proposition says that the map 
\[\prod_{x\in X}\pr_x\colon \THR(\mathcal{C}^X,\mathcal{M}^X)\longrightarrow\THR(\mathcal{C},\mathcal{M})^X\]
is a $G$-equivalence. In this case the action on the target space 
\[\THR(\mathcal{C},\mathcal{M})^X= \map(X,\THR(\mathcal{C},\mathcal{M}))\] is simply conjugation action.

The other case is when $X=\{0,1\}$ is the set with two elements with trivial $G$-action. In this case $\mathcal{C}_0$ and $\mathcal{C}_1$ are $G$-connected categories with dualities $D_0$ and $D_1$, and the duality on $\mathcal{C}_0\times \mathcal{C}_1$ is diagonal. Similarly for the bimodules. Then the proposition says that
\[\THR(\mathcal{C}_0\times \mathcal{C}_1,\mathcal{M}_0\times \mathcal{M}_1)\longrightarrow \THR(\mathcal{C}_0,\mathcal{M}_0)\times \THR(\mathcal{C}_1,\mathcal{M}_1)\]
is a $G$-equivalence, with the diagonal action on the target space.
\end{rem}

The proof of \ref{THRproducts} below is a straightforward generalization of the non-equivariant proof from \cite[1.6.15]{ringfctrs}: we build a diagram of $G$-maps
\begin{equation}\label{diagprod}\xymatrix{\THR(\prod_{X}\mathcal{C}_x,\prod_{X}\mathcal{M}_x)\ar[r]&\prod_{X}\THR(\mathcal{C}_x,\mathcal{M}_x)\\
\THR(\bigvee_X\mathcal{C}_x,\bigvee_X\mathcal{M}_x)\ar[u]\ar[r]&\THR(\coprod_X\mathcal{C}_x,\coprod_X\mathcal{M}_x)\ar[u]
}\end{equation}
and we show that the other three maps are $G$-equivalences.
The vertical maps essentially because the inclusion of wedges into products is highly connected. The lower map is going to admit a $G$-homotopy inverse. We proceed to define the missing objects and maps of the diagram.

\begin{defn}
We define $\bigvee_X\mathcal{C}_x$ to be the category (without identities) enriched in $Sp^{\Sigma}$ with objects $Ob\bigvee_X\mathcal{C}_x=Ob\prod_{X}\mathcal{C}_x$ and with morphisms spectra
\[(\bigvee_X\mathcal{C}_x)(a,b)=\bigvee_{x\in X}\mathcal{C}_x(a_x,b_x)\]
Composition sends $(g,y)\wedge (f,x)\in (\bigvee_X\mathcal{C}_x(b,c))_i\wedge (\bigvee_X\mathcal{C}_x(a,b))_j$ to
\[(f,x)\circ(g,y)=\left\{\begin{array}{ll}(f\circ g,x)& \mbox{ if }x=y\\
\ast &\mbox{ else }
\end{array}\right.\]

The duality $D\colon (\bigvee_X\mathcal{C}_x)^{op}\longrightarrow \bigvee_X\mathcal{C}_x$ sends an object $a$ to
\[(Da)_x=D_{\omega_x}a_{\omega x}\]
and a morphism $(f,x)\in \bigvee_X\mathcal{C}_x(a,b)_i$ to
\[(D_xf\in\mathcal{C}_{\omega x}(D_xb_x,D_xa_x)_i,\omega x)\in\bigvee_{x\in X}\mathcal{C}_x(D_{\omega x}b_{\omega x},D_{\omega x}a_{\omega x})_i\]
Similarly define $\bigvee_X\mathcal{M}_x$ to be the bimodule over $\bigvee_X\mathcal{C}_x$ given by
\[\bigvee_X\mathcal{M}_x(a,b)=\bigvee_{x\in X}\mathcal{M}_x(a_x,b_x)\]
with bimodule structure and duality defined by similar formulas as composition and duality on $\bigvee_X\mathcal{C}_x$.
\end{defn}

\begin{rem}
Even though $\bigvee_X\mathcal{C}_x$ does not have identities, one can define $\THR(\bigvee_X\mathcal{C}_x,\bigvee_X\mathcal{M}_x)$ as the realization of a real pre-simplicial space, that is a real simplicial space without degeneracies.
\end{rem}

The inclusion of sums into products defines an enriched functor (i.e. preserving compositions) $\bigvee_X\mathcal{C}_x\longrightarrow \prod_X\mathcal{C}_x$ given by the identity on objects, commuting with the dualities. Similarly, there is a map of bimodules $\bigvee_X\mathcal{M}_x\longrightarrow \prod_X\mathcal{M}_x$ commuting with the dualities. This induces an equivariant map
\[\THR(\bigvee_X\mathcal{C}_x,\bigvee_X\mathcal{M}_x)\longrightarrow \THR(\prod_X\mathcal{C}_x,\prod_X\mathcal{M}_x)\]
which sits as the left vertical map in diagram \eqref{diagprod}.

\begin{defn}
Define $\coprod_X\mathcal{C}_x$ to be the category enriched in $Sp^{\Sigma}$ with objects $Ob\coprod_X\mathcal{C}_x=\coprod_{x\in X}Ob\mathcal{C}_x$ and with morphisms spectra
\[(\coprod_X\mathcal{C}_x)((c,x),(d,y))=\left\{\begin{array}{lll}\mathcal{C}_x(c,d)&\mbox{ if }x=y\\
\ast &\mbox{else}
\end{array}\right.\]
with composition defined componentwise.
The duality $D\colon (\coprod_X\mathcal{C}_x)^{op}\longrightarrow\coprod_X\mathcal{C}_x$ sends an object $(c,x)$ to
\[D(c,x)=(D_xc,\omega x)\]
and a morphism $f\in (\coprod_X\mathcal{C}_x)((c,x),(d,x))_i=\mathcal{C}_x(c,d)_i$ to
\[D_xf\in\mathcal{C}_{\omega x}(D_xd,D_xc)_i=(\coprod_X\mathcal{C}_x)((D_xd,\omega x),(D_xc,\omega x))_i\]

One defines $\coprod_X\mathcal{M}_x$ as a bimodule with duality over $\coprod_X\mathcal{C}_x$ by similar formulas.
\end{defn}

There are no obvious interesting functors between $\coprod_X\mathcal{C}_x$ and $\bigvee_X\mathcal{C}_x$. The crucial point is that by definition of the morphisms of $\coprod_X\mathcal{C}_x$ there is a canonical homeomorphism
\[V(\coprod_X\mathcal{C}_x;\coprod_X\mathcal{M}_x,\underline{i})\cong \bigvee_{X}V(\mathcal{C}_x;\mathcal{M}_x,\underline{i})\]
for all $\underline{i}\in I[k]$ compatible with the face structure. When $\underline{i}\in I[k]^G$, the homeomorphism is equivariant. The action on the right hand side is induced by the maps $(D_x,J_x)\colon V(\mathcal{C}_x;\mathcal{M}_x,\underline{i})\longrightarrow V(\mathcal{C}_{\omega x};\mathcal{M}_{\omega x},\underline{i})$.

The inclusion
\[\bigvee_{X}V(\mathcal{C}_x;\mathcal{M}_x,\underline{i})\longrightarrow \prod_X V(\mathcal{C}_x;\mathcal{M}_x,\underline{i})\]
induces an equivariant map
\[\THR(\coprod_X\mathcal{C};\coprod_X\mathcal{M})\longrightarrow \prod_X\THR(\mathcal{C};\mathcal{M})\]
which is the right-hand vertical map of diagram \eqref{diagprod}.
It remains to define the bottom map of diagram \eqref{diagprod} \[\THR(\bigvee_X\mathcal{C}_x;\bigvee_X\mathcal{M}_x)\longrightarrow \THR(\coprod_X\mathcal{C}_x;\coprod_X\mathcal{M}_x)\]
It is induced by the maps
\[V(\bigvee_X\mathcal{C}_x;\bigvee_X\mathcal{M}_x,\underline{i})\longrightarrow \bigvee_{ X}V(\mathcal{C}_x;\mathcal{M}_x,\underline{i})\]
that send an element $((a^0,\dots,a^k),(m,x_0)\wedge (f^1,x_1)\wedge\dots\wedge (f^k,x_k))$, with $a^l\in \prod_XOb\mathcal{C}_x$, $m\in \mathcal{C}_{x_0}(a^{0}_{x_0},a^{k}_{x_0})_{i_0}$ and $f^l\in \mathcal{C}_{x_l}(a^{l}_{x_l},a^{l-1}_{x_{l}})_{i_l}$ to
\[((a^{0}_{x_0},\dots,a^{k}_{x_0}),(m,x_0)\wedge (f^1,x_0)\wedge\dots\wedge (f^k,x_0))\]
in the $x_0$-component of $\bigvee_{X}V(\mathcal{C}_x;\mathcal{M}_x,\underline{i})$ if $x_0=x_1=\dots=x_k$, and to the basepoint otherwise.

The proof of \ref{THRproducts} is broken up into the three lemmas below, where we prove that the vertical maps and the bottom map in diagram \eqref{diagprod} are $G$-equivalences.

\begin{lemma}\label{powerandsums}
The map $\THR(\bigvee_X\mathcal{C}_x,\bigvee_X\mathcal{M}_x)\longrightarrow \THR(\prod_X\mathcal{C}_x,\prod_X\mathcal{M}_x)$ is a $G$-equivalence.
\end{lemma}

\begin{proof}
We prove that the map $(\bigvee_X\mathcal{C}_x,\bigvee_X\mathcal{M}_x)\longrightarrow (\prod_X\mathcal{C}_x,\prod_X\mathcal{M}_x)$ is a $G$-stable equivalence. The result then follows from \ref{stableGeq}, since $\bigvee_X\mathcal{C}_x$ and $\bigvee_X\mathcal{M}_x$ are clearly $G$-connected.

The functor $\bigvee_X\mathcal{C}_x\longrightarrow \prod_X\mathcal{C}_x$ is the identity on objects, and therefore in particular a bijection. For two objects $a,b$, the map
\[\bigvee_X\mathcal{C}_x(a,b)_i\longrightarrow\prod_X\mathcal{C}_x(a,b)_i\]
is $(2i-1)$-connected as inclusion of wedge sums into products of $(i-1)$-connected spaces. The function $\delta(i)=2i-1$ satisfies the condition defining a stable $G$-equivalence. 
The map 
\[(\bigvee_X\mathcal{C}_x(Da,a)_{i})^G\longrightarrow (\prod_X\mathcal{C}_x(Da,a)_{i})^G\]
is by \ref{wedgesintoproducts}
\[\min\{i-1,2(\lceil\frac{i}{2}\rceil-1)+1\}=i-1\]
connected. This gives a function $\delta^G(i)$ satisfying the condition for a stable $G$-equivalence on the fixed points.
\end{proof}

\begin{lemma}
The map $\THR(\coprod_X\mathcal{C}_x,\coprod_X\mathcal{M}_x)\longrightarrow \prod_X\THR(\mathcal{C}_x,\mathcal{M}_x)$ is a $G$-equivalence.
\end{lemma}

\begin{proof}
By the $G$-approximation lemma, given any $n$ there is an $\underline{i}\in I[2k+1]^G$ such that the two vertical maps in the following commutative diagram are $(n,n)$-connected
\[\xymatrix{
\Omega^{\underline{i}}\bigvee_XV(\mathcal{C}_x;\underline{i})\ar[d]\ar[r]& \Omega^{\underline{i}}\prod_XV(\mathcal{C}_x;\underline{i})\ar[r]^{\cong}&\prod_X\Omega^{\underline{i}}V(\mathcal{C}_x;\underline{i}) \ar[d]
\\
\THR_{2k+1}(\coprod_X\mathcal{C}_x)\ar[rr] & &\prod_X\THR_{2k+1}(\mathcal{C}_x)
}\]
Thus it is enough to show that we can choose $\underline{i}$ so that the connectivity of
\[\Omega^{\underline{i}}\bigvee_XV(\mathcal{C}_x;\underline{i})\longrightarrow \Omega^{\underline{i}}\prod_XV(\mathcal{C}_x;\underline{i})\]
is as big as we like, both non-equivariantly and on the fixed points.
By \ref{wedgesintoproducts}, the map
\[\bigvee_XV(\mathcal{C}_x;\underline{i})\longrightarrow \prod_XV(\mathcal{C}_x;\underline{i})\]
is at least $(2|\underline{i}|-1,|\underline{i}|-1)$, since for $\underline{i}\in I[2k+1]^G$ the space $V(\mathcal{C}_x;\underline{i})$ is $(|\underline{i}|-1,\lceil\frac{i_0}{2}\rceil+\lceil\frac{|\underline{i}|-i_0}{2}\rceil-1)$-connected by \ref{fixedptdescr}.
Therefore when looped down by $\Omega^{\underline{i}}$, its connectivity becomes
\[(|\underline{i}|-1,\lfloor\frac{i_0}{2}\rfloor+\lfloor\frac{|\underline{i}|-i_0}{2}\rfloor-1)\]
which can be made arbitrarily big with the choice of $\underline{i}$.
\end{proof}

\begin{lemma}\label{messyhomotopy} The map
\[\THR(\bigvee_X\mathcal{C}_x,\bigvee_X\mathcal{M}_x)\longrightarrow\THR(\coprod_X\mathcal{C}_x,\coprod_X\mathcal{M}_x)\]
is a $G$-equivalence.
\end{lemma}

\begin{proof}
The proof will produce a $G$-homotopy inverse of the stated map. To this end we work on the subdivision of $\THR_{\sbt}$ which, we remember, in degree $k$ is $\THR_{2k+1}$. We choose an object $o=\{o_x\}$ in $\bigvee_X\mathcal{C}_x$ which is fixed by $D$ (if no such object exists the fixed set of both source and target in \ref{messyhomotopy} are contractible, and we refer to the non-equivariant proof of \cite[1.6.15]{ringfctrs}). For $\underline{i}\in I[2k+1]$, define a map
\[\bigvee_XV(\mathcal{C}_x;\mathcal{M}_x,\underline{i})\longrightarrow V(\bigvee_X\mathcal{C}_x;\bigvee_X\mathcal{M}_x,\underline{i})\]
by sending $((c_0,\dots,c_{2k+1}),m\wedge f^1\wedge\dots\wedge f^{2k+1})\in V(\mathcal{C}_x;\mathcal{M}_x,\underline{i})$ to
\[((\gamma(c_0),\dots,\gamma(c_{2k+1})),(m,x)\wedge (f^1,x)\wedge \dots\wedge (f^{2k+1},x))\]
where $\gamma(c)$ is the object of $\bigvee_X\mathcal{C}_x$ with $y$-component
\[\gamma(c)_y=\left\{\begin{array}{ll}c&\mbox{ if }y=x\\
o_y&\mbox{ if }y\neq x
\end{array}\right.\]
Since $o$ is fixed by the action, this map is equivariant when $\underline{i}\in I[2k+1]^G$. It is a section for the map in the statement, compatible with the structure maps. The other composite $V(\bigvee_X\mathcal{C}_x;\bigvee_X\mathcal{M}_x,\underline{i})\longrightarrow V(\bigvee_X\mathcal{C}_x;\bigvee_X\mathcal{M}_x,\underline{i})$ sends
\[((a^0,\dots,a^{2k+1}),(m,x_0)\wedge (f^1,x_1)\wedge\dots\wedge (f^{2k+1},x_{2k+1}))\]
to
\[\resizebox{1\hsize}{!}{$\left\{\begin{array}{ll}
(\gamma(a_{x_0}^0),\dots,\gamma(a_{x_{0}}^{2k\!+\!1}),(m,x_0)\wedge (f^1,x_0)\wedge\dots\wedge (f^{2k\!+\!1},x_0))&\mbox{ if }x_0\!=\!\dots\!=\!x_{2k\!+\!1}\\
\ast&\mbox{ else}
\end{array}\right.$}\]
Here the notation $(f^l,x_l)$ indicates that $f^l$ is a morphism of $\mathcal{C}_{x_l}$.
Define a map \[H\colon V(\bigvee_X\mathcal{C}_x;\bigvee_X\mathcal{M}_x,\underline{i})\times \Delta^1[k]\longrightarrow V(\bigvee_X\mathcal{C}_x;\bigvee_X\mathcal{M}_x,\underline{i})\] as follows. For $\sigma=(0^b1^{k+1-b})\in \Delta[1]_k$ with $b\neq 0$ the map $H_\sigma$ sends an element $((a_0,\dots,a_{2k+1}),(m,x_0)\wedge (f^1,x_1)\wedge\dots\wedge (f^{2k+1},x_{2k+1}))$ to
\[\resizebox{1\hsize}{!}{$\left\{\begin{array}{ll}
(\underline{c}^\sigma\!,\!(m,\!x_0)\!\wedge\!(f^1,\!x_1)\!\wedge\!\dots\!\wedge\!(f^{2k\!+\!1},\!x_{2k\!+\!1}))&
\mbox{if }x_0\!=\!\dots\!=\!x_{b}\!=\!x_{2k\!+\!2\!-\!b}\!=\!\dots\!=\!x_{2k\!+\!1}\\
\ast& \mbox{otherwise}
\end{array}\right.$}\]
where $\underline{c}^\sigma\in (Ob\prod_{X}C_x)^{2k+1}$ is the object
\[\underline{c}^\sigma=(\gamma(a_{x_0}^0),\dots,\gamma(a^{b}_{x_0}),a^{b+1},\dots,a^{2k+1-b},\gamma(a^{2k+2-b}_{x_0}),\dots,\gamma(a^{2k+1}_{x_0}))\]
For $b=0$ define $H_\sigma=\id$.
The map $H$ preserve the structure maps, and each $H_\sigma$ is equivariant if $\underline{i}\in I[2k+1]^G$.
When $b=k+1$ the map $H_\sigma$ checks all the indicies $x_0=\dots=x_{k+1}=x_{k+1}=\dots=x_{2k+1}$ and therefore $H(-,1)$ is the composite above. This induces a simplicial $G$-homotopy between the composite and the identity.
\end{proof}


\subsection{$\THR$ of Wall antistructures and Morita invariance}\label{moritasec}

We give a simpler definition of $\THR$ for a Wall antistructure, and we show that this new construction is $G$-equivalent to the definition of §\ref{THRrings}.

\begin{defn}[\cite{wall}] A \textbf{Wall antistructure} is a triple $(A,w,\epsilon)$ of a ring $A$, a unit $\epsilon\in A^{\!\times}$ and a ring map $w\colon A^{op}\longrightarrow A$ such that $w(a)=\epsilon a\epsilon^{-1}$ and $w(\epsilon)=\epsilon^{-1}$.
\end{defn}
A Wall antistructure $(A,w,\epsilon)$ is the same data as an $Ab$-enriched category with one object and non-strict additive duality. Indeed, define a category with non-strict duality $(A,w,\epsilon)$, where $A$ is the category with one object, $w\colon A^{op}\longrightarrow A$ is the duality functor, and the natural isomorphism $\id\Rightarrow w^2$ is defined at the unique object by the isomorphism $\epsilon$ of $A$.

Let $M$ be an $A$-bimodule. Define a bimodule $M\colon A^{op}\otimes A\longrightarrow Ab$ over the category $A$ as the functor that sends the unique object to $M$, and a morphism $a\otimes b$ to the map
\[M(a\otimes b)\colon M\longrightarrow M\]
that sends $m$ to $b\cdot m\cdot a$. A non-strict duality $h\colon M\Rightarrow M\circ w_\gamma$ (cf. \ref{defnonstrictdualmod}) is precisely the data of an additive map $h\colon M\longrightarrow M$ satisfying the conditions
\[\begin{array}{lll}
h(a\cdot m)=h(m)\cdot w(a)\\
h(m\cdot a)=w(a)\cdot h(m)\\
h^{2}(m)=\epsilon\cdot m\cdot\epsilon^{-1}
\end{array}\]
We recall that $\gamma$ denotes the twist isomorphism and \[w_\gamma=(w^{op}\otimes w)\circ\gamma\colon A^{op}\otimes A\longrightarrow A^{op}\otimes A\]
We call a map $h$ that satisfies the properties above a Wall-twisting.
As seen in §\ref{THRrings} this induces a bimodule with strict duality $(\mathcal{D}M,\mathcal{D}h)$ on the category with strict duality $\mathcal{D}A$ (cf. §\ref{secdefscriptd}). We describe this explicitly. The objects of $\mathcal{D}A$ are units $u\in A^{\!\times}$, and the duality sends $u$ to $w(u)\cdot\epsilon$. The bimodule $\mathcal{D}M$ sends a pair of units $(u,v)\in\mathcal{D}A^{op}\otimes\mathcal{D}A$ to the abelian group $M$. The strict duality sends $m\in M=\mathcal{D}M(u,v)$ to
\[\mathcal{D}h(m)=u^{-1}\cdot h(m)\cdot v\]
This is isomorphic to the bimodule $M^h\colon A^{op}\otimes A\longrightarrow Ab$ that sends $(u,v)$ to the abelian group of pairs \[M^h(u,v)=\{(m,m')\in M\oplus M|u\cdot m'=h(m)\cdot v\}\]
with duality that sends $(m,m')\in M^h(u,v)$ to $(m',m)\in M^h(w(v)\epsilon,w(u)\epsilon)$ (cf. \ref{defTHRforanti}). One can define
\[\THR((A,w,\epsilon);(M,h)):=\THR(H\mathcal{D}A,HM^h)\]
We compare this construction with $\THR$ of the antistructure associated to $(A,w,\epsilon)$ as defined in §\ref{THRrings}.

Recall from \ref{wallanti} that a Wall antistructure $(A,w,\epsilon)$ also induces an antistructure $(A,A,\alpha)$ in the sense of §\ref{ringduality} with $A$ as underlying ring, $A\otimes A$-module structure on $A$ defined by
\[l\cdot a\otimes b=w(b)\cdot l\cdot a\]
and involution $\alpha\colon A\longrightarrow A$ defined by $\alpha(l)=w(l)\cdot\epsilon$. This defines in its turn a category with non-strict duality $(\mathcal{P}_A,D_A,\eta^{w\epsilon})$ on the category of finitely generated projective modules $\mathcal{P}_A$. We recall that the duality functor $D_A$ is
\[D_AP=\hom_A(P,A_s)\]
where $A_s$ is the module structure on $A$ defined by $l\cdot a=w(a)\cdot l$. The abelian group $D_AP$ has the module structure defined by pointwise right multiplication. The isomorphism $\eta^{w\epsilon}\colon P\longrightarrow D^{2}_AP$ is given by $\eta^{{w\epsilon}}(a)(\lambda)=w(\lambda(a))\cdot\epsilon$.

A Wall-twisting $h\colon M\longrightarrow M$ defines an $M$-twisting for the antistructure $(A,A,\alpha=w(-)\epsilon)$ (cf. \ref{defmaptwist}) by the composite
\[\xymatrix{A_t\otimes_AM\cong M\ar[rr]^-{h(-)\cdot\epsilon}&&M\cong A_t\otimes_AM}\]
In its turn this defines a non-strict duality $h\epsilon$ on the bimodule 
\[M=\hom_A(-,-\otimes_AM)\colon\mathcal{P}^{op}_A\otimes\mathcal{P}_A\longrightarrow Ab\]
as seen in \ref{Mtwistnonstrictdual}. Strictifying the duality leads to the bimodule with strict duality
\[M^{h\epsilon}\colon \mathcal{D}\mathcal{P}_{A}^{op}\otimes \mathcal{D}\mathcal{P}_{A}\longrightarrow Ab\]
of \ref{defTHRforanti}, defined at an object $(\varphi,\varphi')\in \mathcal{D}\mathcal{P}_{A}^{op}\otimes \mathcal{D}\mathcal{P}_{A}$ as the pairs of module maps $f\colon P\longrightarrow P'\otimes_A M$, $g\colon Q'\longrightarrow Q\otimes_A M$ such that
\[\xymatrix{Q\otimes_A M\ar[d]_{\phi\otimes M}&Q'\ar[l]_-{g}\ar[d]^{\phi'}\\
D_A(P)\otimes_AM&D_AP'\ar[l]^-{\widehat{h\epsilon}(f)}
}\]
commutes, where $\widehat{h\epsilon}$ is the duality of \ref{naturalityJK}. The following is the main result of this section.

\begin{prop}\label{moritawallanti}
There is a natural $G$-equivalence
\[\THR(H\mathcal{D}A,HM^h)\longrightarrow \THR(H\mathcal{D}\mathcal{P}_{A},HM^{h\epsilon})\]
\end{prop}
In particular, suppose that $\epsilon\in A^{\!\times}$ is in the center of $A$. Then the Wall antistructure is an anti-involution $\alpha\colon A^{op}\longrightarrow A$ and it defines a strict duality on the category $A$. Similarly the Wall-twisting $h$ gives a strict duality on the bimodule $M\colon A^{op}\otimes A\longrightarrow Ab$. Composing the map of the proposition with the $\THR$-equivalence $A\rightarrow\mathcal{D}A$ of \ref{scriptdandthr} we obtain the following.
\begin{cor}
Let $\alpha\colon A^{op}\longrightarrow A$ be an anti-involution on the ring $A$, and $h\colon M\rightarrow M$ a Wall-twisting. Then there is a natural $G$-equivalence
\[\THR(HA,HM)\longrightarrow \THR(H\mathcal{D}\mathcal{P}_{A},HM^{h})\]
\end{cor}

Let us define the map of proposition \ref{moritawallanti}.
There is a morphism of categories with duality (cf. §\ref{duality})
\[(F,\xi)\colon (A,w,\epsilon)\longrightarrow (\mathcal{P}_A,D_A,\eta^{w\epsilon})\]
with underlying functor $F\colon A\longrightarrow \mathcal{P}_A$ that sends the unique object to the module $A_t$ given by right multiplication of $A$. A morphism $a\in A$ is sent to left multiplication
\[F(a)=a\cdot(-)\colon A_t\longrightarrow A_t\]
The natural isomorphism $\xi\colon Fw\Rightarrow D_AF$ is defined at the unique object by the isomorphism
\[\xi\colon A_t\longrightarrow D_A(A_t)=\hom_A(A_t,A_s)\]
that sends $1$ to $w$. The pair $(F,\xi)$ induces a functor
\[\mathcal{D}(F,\xi)\colon\mathcal{D}A\longrightarrow \mathcal{D}\mathcal{P}_A\]
that commutes strictly with the duality (cf. \ref{canddc}). Explicitly, it sends an object $u\in Ob\mathcal{D}A=A^{\!\times}$ to the object $(A_t,A_t,\xi\circ F(u))$ of $\mathcal{D}\mathcal{P}_A$, where
\[\xi\circ F(u)\colon A_t\longrightarrow D_A(A_t)=\hom_A(A_t,A_s)\]
is the isomorphism the sends $a$ to $w(-)\cdot ua$.
There is also a map of bimodules $\Phi\colon M\Rightarrow M\circ(F^{op}\otimes F)$ defined at the unique object by the isomorphism
\[\Phi\colon M\longrightarrow \hom_A(A_t,A_t\otimes_A M)\]
that sends $m$ to $\Phi(m)(a)=1\otimes m\cdot a$.
This induces a morphism of bimodules \[\mathcal{D}\Phi\colon M^h\Rightarrow M^{h\epsilon}\circ(\mathcal{D}(F,\xi)^{op}\otimes \mathcal{D}(F,\xi))\] defined at an object $(u,v)\in Ob\mathcal{D}A^{op}\otimes \mathcal{D}A=A^{\!\times}\times A^{\!\times}$ by
\[\mathcal{D}\Phi(m,m')=(\Phi(m),\Phi(m'))\]
It is well defined since
\[\xymatrix{A_t\otimes_AM\ar[d]_-{u\cdot (-)\otimes M}&&A_t\ar[ll]_-{\Phi(m')}\ar[d]^-{v\cdot (-)}\\
A_t\otimes_AM\ar[d]_-{w\cdot(-)\otimes M}&&A_t\ar@{-->}[ll]_-{\Phi(h(m))}\ar[d]^-{w\cdot (-)}\\
D_A(A_t)\otimes_AM&&D_AA_t\ar[ll]^-{\widehat{h\epsilon}(\Phi(m))}
}\]
commutes. The morphism of bimodules with duality
\[(\mathcal{D}(F,\xi),\mathcal{D}\Phi)\colon (\mathcal{D}A,M^h)\longrightarrow (\mathcal{D}\mathcal{P}_A,M^{h\epsilon})\]
induces the map of the statement of \ref{moritawallanti}. We prove that this map is a $G$-equivalence by following the structure of the classical proof of \cite[IV-2.5.14]{Dundasbook}.

\begin{proof}[Proof of \ref{moritawallanti}]
We factor the map induced in $\THR$, and we prove that each of the maps of the factorization is an equivalence. Let $M_n(A)$ be the ring of $n\times n$ matricies on $A$. There is a Wall antistructure $(M_n(A),w(-)^T,\underline{\epsilon})$, where $w(N)^T$ denotes the matrix obtained by applying $w$ to the entries of the transposed matrix $N^T$, and $\underline{\epsilon}$ is the diagonal matrix with entries $(\epsilon,\dots,\epsilon)$. The abelian group $M_n(M)$ defines a bimodule on $M_n(A)$, and \[h(-)^T\colon M_n(M)\longrightarrow M_n(M)\]
defines a Wall-twisting on $M$. There is a map
\[(A,M)\longrightarrow (M_n(A),M_n(M))\]
that includes an element $a$ of $A$ as the diagonal matrix $(a,1,\dots,1)$, and similarly on the bimodule. It commutes strictly with the dualities, and therefore it induces a map
\[(\mathcal{D}A,M^h)\longrightarrow (\mathcal{D}M_n(A),M_n(M)^{h^T})\]
Now let $\mathcal{F}_A\subset\mathcal{P}_A$ be the category of finitely generated free right $A$-modules, and $\mathcal{F}^{n}_A$ the full subcategory of $\mathcal{F}_A$ of free modules of rank less than or equal to $n$. Clearly the duality on $\mathcal{P}_A$ restricts to both $\mathcal{F}_A$ and $\mathcal{F}^{n}_A$. There is a functor $G\colon M_n(A)\longrightarrow \mathcal{F}^{n}_A$ that sends the unique object to $A_{t}^n$ and a matrix $N$ to the map
\[N\cdot(-)\colon A_{t}^n\longrightarrow A_{t}^n\]
The isomorphism
\[\zeta\colon A_{t}^n\longrightarrow D_A(A_{t}^n)=\hom_A(A_{t}^n,A_s)\]
that sends $\underline{a}=(a_1,\dots,a_n)$ to $\zeta(\underline{a})(\underline{b})=(w(b_1)a_1+\dots+w(b_n)a_n)$
defines a morphism of categories with duality \[(G,\zeta)\colon (M_n(A),w(-)^T,\epsilon)\longrightarrow (\mathcal{F}^{n}_A,D_A,\eta^{w\epsilon})\] inducing a functor
\[\mathcal{D}(G,\zeta)\colon\mathcal{D}M_n(A)\longrightarrow \mathcal{D}\mathcal{F}^{n}_A\]
that commutes strictly with the dualities.
Doing similar constructions on bimodules we obtain a diagram
\[\xymatrix{\THR(H\mathcal{D}A;HM^h)\ar[dd]_-{(\mathcal{D}(F,\xi),\mathcal{D}\Phi)}\ar[rr]^-{\mbox{Morita}}_-{\simeq \ref{morita}}&&\colim_n \THR(H\mathcal{D}M_n(A);HM_n(M)^{h^T})\ar[d]^{\mbox{cofinality}}_{\simeq\ref{cofinality}}\\
 && \colim_n \THR(H\mathcal{D}\mathcal{F}^{n}_A;HM^{h\epsilon})\ar[d]_{\simeq}\\
\THR(H\mathcal{D}\mathcal{P}_A;HM^{h\epsilon})&&\THR(H\mathcal{D}\mathcal{F}_A;HM^{h\epsilon})\ar[ll]^{\mbox{cofinality}}_{\simeq \ref{cofinality}}
}\]
The labeled maps are equivalences by the referred results, proved below.
The maps in the colimit for $M_n(A)$ are induced by the inclusions $M_n(A)\longrightarrow M_{n+1}(A)$ that add a $1$ in the bottom right corner of the matrix and zero elsewhere. If we want the maps $M_n(A)\longrightarrow \mathcal{F}^{n}_A$ to commute with the stabilization maps, we need to define the stabilization maps for $\mathcal{F}^{n}_A$ by sending a module $P\in \mathcal{F}^{n}_A$ to $P\oplus A\in \mathcal{F}^{n+1}_A$, and extending the morphisms by the identity on $A$ (it is not the standard inclusion $\mathcal{F}^{n}_A\longrightarrow \mathcal{F}^{n+1}_A $). 

The right bottom vertical map of the diagram is an equivalence because the colimit commutes with the realization, with the homotopy colimit defining $\THR$, and with the loop space (cf. \cite[A-1.5.5]{Dundasbook}). Finally, the canonical map
\[\colim_n V(H\mathcal{D}\mathcal{F}^{n}_A;HM^{h\epsilon},\underline{i})\cong V(H\mathcal{D}\mathcal{F}_A;HM^{h\epsilon},\underline{i})\]
is  a homeomorphism, commuting with the duality.

We show that the diagram commutes up to $G$-homotopy, by proving that for any given $n\geq 1$ the diagram
\[\xymatrix{\THR(H\mathcal{D}A;\!HM^{h})\ar[dr]_-{\mathcal{D}(F,\xi)}\ar[r]&\THR(H\mathcal{D}M_n(A);\!HM_n(M)^{h^T})\ar[d]^{\mathcal{D}(G,\zeta)}\\
&\THR(H\mathcal{D}\mathcal{P}_A;\!HM^{h\epsilon})}\]
commutes up to $G$-homotopy. The upper composite is induced by the morphism of categories with duality
\[(F_n,\xi_n)\colon (A,w,\epsilon)\longrightarrow (\mathcal{P}_A,D_A,\eta^{h\epsilon})\] 
where $F_n$ sends the object to $A_{t}^n$ and a morphism $a$ to the diagonal matrix
\[a\oplus\id\colon A_{t}^n\longrightarrow A_{t}^n\]
and $\xi_n\colon A_{t}^n\longrightarrow D_AA_{t}^n$ is the isomorphism $\zeta$ above. We construct a $G$-homotopy by means of a natural transformation 
\[S\colon \mathcal{D}(F,\xi)\Rightarrow \mathcal{D}(F_n,\xi_n)\] It is defined at an object $u\in Ob\mathcal{D}A=A^{\!\times}$ by the morphism of $\mathcal{D}\mathcal{P}_A$ given by the pair
\[(i\colon A_t\longrightarrow A_{t}^n,p\colon A_{t}^n\longrightarrow A_t)\]
where $i$ denotes the inclusion in the first coordinate and $p$ the projection off it. This is a well defined morphism since
\[\xymatrix{A_t\ar[d]_{\xi\circ (u\cdot(-))}&A_{t}^n\ar[l]_-{p}\ar[d]^{\xi_n\circ(u\oplus\id\cdot(-))}\\
D_AA_t&D_AA_{t}^n\ar[l]^{D(i)}
}\]
commutes.
The dual map $D(S)$ is $(p,i)$, and therefore $D(S)\circ S=\id$ and $S\circ D(S)\circ F_n(a)=F_n(a)$. This shows that $S$ is a $G$-natural transformation in the sense of \ref{Gnattransf}, and therefore by \ref{THRpreservesGequiv} it induces a simplicial $G$-homotopy
\[sd_e\THR_{\sbt}(H\mathcal{D}A;HM^h)\times\Delta[1]\longrightarrow sd_e\THR_{\sbt}(H\mathcal{D}\mathcal{P}_A;HM^{h\epsilon})\]
between the subdivisions of $\mathcal{D}(F_n,\xi_n)$ and $\mathcal{D}(F,\xi)$.
\end{proof}

We want to show that the inclusion $\mathcal{D}A\longrightarrow \mathcal{D}M_n(A)$ induces a $G$-equivalence in $\THR$. This is true in greater generality.
Let $(C,D,\eta)$ be an $Ab$-enriched category with additive duality. Define a category $M_n(C)$ for every integer $n\geq 1$ with objects $ObC^n$ and morphisms from $\underline{c}=(c_1,\dots,c_n)$ to $\underline{d}=(d_1,\dots,d_n)$ the abelian group
\[M_n(C)(\underline{c},\underline{d})=\bigoplus\limits_{1\leq i,j\leq n}C(c_i,d_j)\]
Composition of $[f]=\{f_{ij}\}\colon \underline{c}\longrightarrow \underline{d}$ and $[g]=\{g_{ij}\}\colon \underline{d}\longrightarrow \underline{e}$ is defined by
\[([g]\circ[f])_{ij}=\sum_{k=1}^ng_{kj}\circ f_{ik}\]
$M_n(C)$ inherits the structure of a category with duality $(M_n(C),D,\eta)$ with $D$ defined on objects by
\[D(\underline{c})_i=Dc_i\]
and on morphism by sending  $[f]\colon \underline{c}\longrightarrow \underline{d}$ to
\[(D[f])_{ij}=([Df]^T)_{ij}=Df_{ji}\colon Dd_i\longrightarrow Dc_j\]
The isomorphism $\eta_{\underline{c}}\colon \underline{c}\longrightarrow D^2\underline{c}$ is the diagonal matrix with diagonal entries $(\eta_{c_1},\dots ,\eta_{c_n})$. If $C$ has an object $o$ and an isomorphism $\nu\colon o\longrightarrow D(o)$ such that $D\nu\circ\eta=\nu$, there is a morphism of categories with duality 
\[(\iota_o,\xi)\colon C\longrightarrow M_n(C)\]
that sends $c$ to $(c,o,\dots,o)$ and $f\colon c\longrightarrow d$ to the diagonal matrix with entries $(f,\id_o,\dots,\id_o)$. The isomorphism $\xi$ is defined at $c\in C$ by the diagonal matrix
\[\xi_c\colon \iota_o Dc=(Dc,o,\dots,o)\longrightarrow(Dc,Do,\dots,Do)=D\iota_oc\]
with entries $(\id,\nu,\dots,\nu)$.
A similar construction gives a bimodule with non-strict duality $M_n(M)$ over $M_n(C)$ and a map $M\Rightarrow M_n(M)\circ (\iota^{op}_o\otimes\iota_o)$ preserving the dualities.

\begin{prop}[Morita invariance]\label{morita}
The induced map \[\mathcal{D}(\iota_o,\xi)_\ast\colon\THR(H\mathcal{D}C,H\mathcal{D}M)\longrightarrow\THR(H\mathcal{D}M_n(C);H\mathcal{D}M_n(M))\]
is a $G$-equivalence for every $n\geq 1$.
\end{prop}

\begin{proof}
The functor $\mathcal{D}(\iota_o,\xi)$ factors as
\[\xymatrix{\mathcal{D}C\ar[d]_{\iota_{\nu}}\ar[r]^-{\mathcal{D}(\iota_o,\xi)}&\mathcal{D}M_n(C)\\
M_n(\mathcal{D}C)\ar[ur]_-{F}^-{\simeq}
}\]
where the vertical map is the inclusion of the category $\mathcal{D}C$ for the fixed object $\nu\colon o\longrightarrow Do$, which induces an equivalence in $\THR$ by lemma \ref{matriciesforstrict} below. The diagonal functor is the following equivalence of categories. An object of $M_n(\mathcal{D}C)$ is a $n$-tuple $\underline{\varphi}=\{\varphi_i=(c_i,d_i,\phi_i)\}$ of objects of $\mathcal{D}C$. The functor $F$ sends $\underline{\varphi}$ to the object $(\underline{c},\underline{d},F(\underline{\phi}))$ of $\mathcal{D}M_n(C)$ with $\underline{c}=(c_1,\dots,c_n)$, $\underline{d}=(d_1,\dots,d_n)$ and \[F(\underline{\phi})\colon \underline{d}\longrightarrow D\underline{c}\] defined as the diagonal matrix $[\underline{\phi}]$ with entries $(\phi_1,\dots,\phi_n)$.
A morphism in $M_n(\mathcal{D}C)$ from $\underline{\varphi}$ to $\underline{\varphi}'$ is a family of pairs $\{(a_{ij},b_{ij})\}$ such that every $(a_{ij},b_{ij})$ is a morphism from $\underline{\varphi}_i$ to $\underline{\varphi}_j'$ in $\mathcal{D}C$. This is the same as a pair of matricies $[a]\colon \underline{c}\longrightarrow \underline{c}'$ and $[b]^T\colon \underline{d}'\longrightarrow \underline{d}$ such that $D[a]^T[\underline{\phi}']=[\underline{\phi}][b]^T$. Define $F$ on morphisms by
\[F([a],[b])=([a],[b]^T)\]
The functor $F$ is clearly full and faithful. Let us see that it is also essentially surjective. An object of $\mathcal{D}M_n(C)$ is a triple $(\underline{c},\underline{d},N)$ where $N\colon \underline{d}\longrightarrow D\underline{c}$ is an invertible matrix. This is isomorphic to $F(\underline{c},D\underline{c},\id)$ via the isomorphism $(\id,N^{-1})$. Indeed, the diagram
\[\xymatrix{\underline{d}\ar[d]_-{N}&&D\underline{c}\ar[ll]_-{N^{-1}}\ar[d]^-{\id}\\
D\underline{c}&&D\underline{c}\ar[ll]^-{D(\id)^T=\id}
}\]
commutes.
Thus $F$ is an equivalence of categories that commutes strictly with the dualities. By \ref{inversenonstrict} it induces an equivalence on $\THR$.
\end{proof}

\begin{lemma}\label{matriciesforstrict}
Let $(C,D)$ be an $Ab$-enriched category with strict additive duality, $M$ a bimodule with strict duality on $C$ and $o\in C$ an object fixed by the duality. The inclusion $\iota_o\colon C\longrightarrow M_n(C)$ induces a $G$-equivalence
\[\THR(HC;HM)\longrightarrow\THR(HM_n(C);HM_n(M))\]
for all $n\geq 1$.
\end{lemma}

\begin{proof}
Define $M_{n}^\vee(HC)$ as the $Sp^{\Sigma}$-enriched category with objects $ObC^n$ and morphism spectra
\[M_{n}^\vee(HC)(\underline{c},\underline{d})=\bigvee_{1\leq i,j\leq n}HC(c_i,d_j)\]
Composition is induced by the maps
\[(\bigvee_{1\leq i,j\leq n}HC(d_i,e_j)_p)\wedge(\bigvee_{1\leq l,h\leq n}HC(c_l,d_h)_q)\longrightarrow \bigvee_{1\leq l,h\leq n}HC(c_l,e_j)_{p+q}\]
that send $f\wedge g\in HC(d_i,e_j)_p\wedge HC(c_h,d_l)_q$ to $f\circ g\in HC(c_h,e_j)_{p+q}$ if $l=i$ and to the basepoint otherwise. It has a duality that sends $\underline{c}$ to $(D\underline{c})_i=Dc_i$, and a morphism $f$ in the $(i,j)$ wedge component of $M_{n}^\vee(HC)(\underline{c},\underline{d})$ to $Df$ in the $(j,i)$ wedge component of $M_{n}^\vee(HC)(D\underline{d},D\underline{c})$. The functor $\iota_o$ factors as
\[\xymatrix{HC\ar[dr]_-{\iota_o}\ar[r]^-{s_o}&M_{n}^\vee(HC)\ar[d]^{\simeq}\\
&HM_{n}(C)
}\]
where the vertical map is induced by the inclusion of wedges into direct sums. It is a stable $G$-equivalence by \ref{wedgesintoproducts} and it induces an equivalence in $\THR$ by \ref{stableGeq}. The horizontal map $s_o$ sends an object $c$ to 
$(c,o,\dots,o)$ and it is the inclusion in the $(1,1)$-wedge summand on morphisms.
We show that $s_o$ induces an equivalence in $\THR$. Consider the maps
\[V(M_{n}^\vee(HC);M_{n}^\vee(HM),\underline{i})\longrightarrow V(HC;HM,\underline{i})\]
that sends an element
\[(m,(l_0,h_0))\wedge(f_1,(l_1,h_1))\wedge\dots\wedge (f_k,(l_k,h_k))\]
into $m\wedge f_1\wedge\dots\wedge f_k$ if $l_0=h_0=l_1=h_1=\dots=l_k=h_k$ and to the base point otherwise. Here the notation $(f,(l,h))$ indicates that $f$ belongs to the $(l,h)$-wedge component of the morphism space.
This induces a $G$-map
\[\tr_k\colon\THR_k(M_{n}^\vee(HA);M_{n}^\vee(HM))\longrightarrow \THR_k(HA;HM)\]
that is clearly a retraction for $s_o$. A construction analogous to the proof of \ref{messyhomotopy} defines a $G$-homotopy on the subdivisions between $s_o\circ Tr$ and the identity. For the non-equivariant construction one can see  \cite[1.6.18]{ringfctrs}.
\end{proof}

\begin{defn}
A full subcategory $B\subset C$ is called \textbf{cofinal} if for every object $c\in C$ there is an object $r(c)\in B$ and maps in $C$
\[\xymatrix{c\ar@<.9ex>[r]^-{i_c}&r(c)\ar@<.9ex>[l]^-{p_c}}\]
with $p_c\circ i_c=\id_c$.

Let $(C,D)$ be a category with strict duality and $B\subset C$ a full subcategory closed under the duality. The category $B$ is \textbf{$G$-cofinal in $C$} if it is cofinal, and if for every object $c\in C$ with $c=Dc$, there is a $r(c)\in B$ satisfying $Dr(c)=r(c)$ and an inclusion $i_c\colon c\longrightarrow r(c)$ with $D(i_c)\circ i_c=\id_c$.
\end{defn}

\begin{ex}
If $C$ is the category of finitely generated projective $A$-modules $\mathcal{P}_A$, we may take $B$ to be the subcategory 
 category of finitely generated free modules $\mathcal{F}_A$. Then $\mathcal{F}_A$ is cofinal in $\mathcal{P}_A$ and $\mathcal{D}\mathcal{F}_A$ is $G$-cofinal in $\mathcal{D}\mathcal{P}_A$. Indeed, let $\varphi=(P,P,\phi)$ be an object of $\mathcal{D}\mathcal{P}_A$ fixed by the duality and choose an $A$-module $Q$ so that $P\oplus Q\cong A^n$. Then one can take $r(\varphi)$ to be the object
\[r(\varphi)=(P\oplus Q\oplus DQ\oplus DP \ ,\ P\oplus Q\oplus DQ\oplus DP,r(\phi))\]
of $\mathcal{D}\mathcal{F}_A$, where $r(\phi)$ is the isomorphism
\[\xymatrix{P\oplus Q\oplus DQ\oplus DP\ar[drrr]_-{r(\phi)}\ar[rrr]^-{\phi\oplus\left(\begin{smallmatrix}0& \eta\\1&0\end{smallmatrix}\right)\oplus D\phi^{-1}}&&& DP\oplus DQ\oplus D^2Q\oplus D^2P\ar[d]^{\cong}\\
&&& D(P\oplus Q\oplus DQ\oplus DP)
}\]
The map $i\colon \varphi\rightarrow r(\varphi)$ is the pair $(i_1,p_1)$ of the inclusion and projection of the $P$ summand of $P\oplus Q\oplus DQ\oplus DP$. Then
\[D(i_1,p_1)\circ(i_1,p_1)=(p_1,i_1)\circ(i_1,p_1)=(p_1\circ i_1,p_1\circ i_1)=\id\]

Similarly, the inclusion of $M_n(A)$ in $\mathcal{F}^{n}_A$ is cofinal, and the inclusion of $\mathcal{D}M_n(A)$ in $\mathcal{D}\mathcal{F}^{n}_A$ is $G$-cofinal.
\end{ex}

\begin{prop}[Cofinality]\label{cofinality}
The inclusion of a $G$-cofinal subcategory with duality $B\subset C$ defines a $G$-equivalence
\[\THR(HB,H\iota^\ast M)\longrightarrow\THR(HC,H M) \]
where $\iota^\ast M$ denotes the restriction of the bimodule $M$ to the subcategory $B$.
\end{prop}

\begin{proof}
We choose a duality preserving retraction $r\colon C\longrightarrow B$ for the inclusion functor $\iota\colon B\longrightarrow C$ and natural transformations $i\colon\id\Rightarrow r\iota$ and $p\colon r\iota\Rightarrow\id$, as follows. If $b\in B$, take $r(b)=b$ and define $i_b=p_b=\id_b$. If $Dc=c$, we choose a self dual $r(c)\in B$, an inclusion $i_c\colon c\longrightarrow r(c)$ and define $p_c=Di_c$. We have left to define $(r(c),i_c,p_c)$ for all $c\in C\backslash B$ with $Dc\neq c$. Choose a representative of each orbit under $D$, and define $(r(c),i_c,p_c)$ arbitrarily for these representatives. Then set $r(Dc)=Dr(c)$, $i_{Dc}=Dp_c$ and $p_{Dc}=Di_c$.
The retraction  $r\colon C\longrightarrow B$ is defined on objects by sending $c$ to $r(c)$ as notation suggests. A morphism $f\colon c\longrightarrow c'$ is sent to the composite
\[\xymatrix{ r(c)\ar@{-->}[d]_-{r(f)}\ar[r]^-{p_c}&c\ar[d]^{f}\\
r(c')&c'\ar[l]_-{i_c}
}\]
This defines a pre-functor: $r$ preserves compositions but not identities. Similarly there is a map of bimodules $M\Rightarrow i^\ast M\circ (r^{op}\otimes r)$ defined at a pair $(c,d)$ by
\[M(p_c,i_d)\colon M(c,d)\longrightarrow M(rc,rd)\]
The induced map $r\colon\THR_{\sbt}(HC,HM)\longrightarrow \THR_{\sbt}(HB,H\iota^\ast M)$
is a map of pre-simplicial spaces (i.e. it preserves the face structure) that commutes strictly with the involutions. It therefore induces a $G$-map on the realization (recall that we are using the thick realization).
By construction $r$ satisfies $r\circ \iota=\id$. The inclusions $i_c$ define a natural transformation $i\colon \id\Rightarrow \iota\circ r$. By construction $Di_{Dc}=p_c$, and therefore this natural transformation satisfies $Di_{Dc}\circ i_c=\id_c$ and $i_d\circ Di_{Dc}\circ r (f)= r(f)$. Thus $i$ is a $G$-natural transformation in the sense of \ref{Gnattransf}, and by \ref{THRpreservesGequiv} it induces a pre-simplicial $G$-homotopy
\[sd_e\THR_{\sbt}(HC;HM)\times\Delta[1]\longrightarrow sd_e\THR_{\sbt}(HC;HM)\]
between the identity and $\iota\circ r$.
\end{proof}


\subsection{$\THR$ of triangular and diagonal matricies}

Suppose that $C$ is a category enriched in abelian groups, $D$ a strict additive duality on $C$, and $(M,J)$ a bimodule with duality over $(C,D)$. Also, let $(X,\leq)$ be a finite poset with involution $\omega$ reversing the order. By this we mean that the involution $\omega$ extends to a strict duality on the associated category $\omega\colon (X,\leq)^{op}\longrightarrow (X,\leq)$.

\begin{defn}
The \textbf{category of triangular $X$-matricies on $HC$} is the category $T^XHC$ enriched in $Sp^{\Sigma}$ with objects $ObT^XHC=Ob(HC)^X$ and morphisms spectra
\[T^XHC(a,b)=\bigoplus_{x\leq y\in X\times X}HC(a_x,b_y)\]
Composition sends $f'\wedge f$ to
\[(f\circ f')_{x\leq y}=\sum_{x\leq j\leq y}f_{j\leq y}\circ f_{x\leq j}'\]
This uses that $HC(a_x,b_y)_i=C(a_x,b_y)\otimes\mathbb{Z}(S^i)$ is an abelian group.
There is a duality $D\colon (T^XHC)^{op}\longrightarrow T^XHC$ that sends an object $a$ to
\[(Da)_x=Da_{\omega x}\]
and a morphism $f\in T^X HC(a,b)_i$ to
\[(Df)_{x\leq y}=D(f_{\omega y\leq \omega x}\colon a_{\omega y}\longrightarrow b_{\omega x})\]
Define $T^XHM$ to be the bimodule over $T^XHC$
\[T^XHM(a,b)=\bigoplus_{x\leq y\in X\times X}HM(a_x,b_y)\]
with bimodule structure and duality defined in a similar way as for $T^XHC$.
\end{defn}

Recall that for an $Sp^\Sigma$-category with duality $\mathcal{C}$, we write $\prod_X\mathcal{C}=\mathcal{C}^X$, so that $H(C^X)=(HC^X)$.
There is an equivariant functor  $\tr_X \colon T^XHC\longrightarrow HC^X$ that is the identity on objects, and sends a morphism $\{f_{x\leq y}\}\in T^XHC(a,b)_i$ to its diagonal
\[\tr_X(f)_x=f_{x\leq x}\]
This functor admits an equivariant section $m_X\colon HC^X\longrightarrow T^XHC$ that is the identity on objects, and sends $f\in HC(a,b)_i$ to
\[m(f)_{x\leq y}=\left\{\begin{array}{ll}f_x &\mbox{x=y}\\
0&\mbox{ else}
\end{array}\right.\]
Similarly there are maps of bimodules $T^XHM\longrightarrow HM^X$ and $HM^X\longrightarrow T^XHM$. These induce equivariant maps between  $\THR(HC^X,HM^X)$ and $\THR(T^XHC;T^XHM)$.
The goal of this section is to prove the following, that generalizes \cite[1.6.20]{ringfctrs} to the equivariant setting.

\begin{prop}\label{triangdiagmat}
The map \[\tr_X\colon \THR(T^XHC;T^XHM)\longrightarrow\THR(HC^X,HM^X)\] is a $G$-homotopy equivalence with homotopy inverse $m_X$.
\end{prop}

\begin{rem} One can define a category enriched in abelian groups $T^XC$ with  objects $ObC^X$ and morphisms
\[T^XC(a,b)=\bigoplus_{x\leq y\in X\times X}C(a_x,b_y)\]
with similar composition and duality. The canonical isomorphism
\[(\bigoplus_{x\leq y\in X\times X}C(a_x,b_y))\otimes \mathbb{Z}(S^i)\cong (\bigoplus_{x\leq y\in X\times X}C(a_x,b_y)\otimes \mathbb{Z}(S^i))\]
gives an isomorphism $T^XHC\cong HT^XC$ respecting the dualities. One can also define similarly maps  between $C^X$ and $T^XC$, inducing maps between $H(C^X)$ and $HT^XC$, corresponding to $m_X$ and $\tr_X$ under the canonical isomorphisms $H(C^X)\cong (HC)^X$ and $HT^XC\cong T^XHC$.
\end{rem}

\begin{rem} For a general category $\mathcal{C}$ enriched in $Sp^{\Sigma}$, our construction $T^X$ does not make sense since we use the group structure on the morphism spectra when we define composition. One usually defines  $T^X\mathcal{C}$ by the spectra
\[T^X\mathcal{C}(a,b)=\prod_{y\in X}\bigvee_{x\leq y}\mathcal{C}(a_x,b_y)\]
In the case of $HC$, the inclusion of wedges into products gives a stable equivalence between these two constructions. There is however no obvious way to define a duality on $T^X\mathcal{C}$.
\end{rem}

The proof of \ref{triangdiagmat} is very close in spirit to the proof of \ref{THRproducts} and follows the non-equivariant proof of \cite[1.6.20]{ringfctrs}. In particular, we want to exchange products with wedges. For this purpose we define the following category, for a general $(\mathcal{C},D)$ enriched in $Sp^{\Sigma}$.

\begin{defn}
The category $T_{\vee}^X\mathcal{C}$ has objects $ObT_{\vee}^X\mathcal{C}=Ob\mathcal{C}^X$ and morphisms spectra
\[T_{\vee}^X\mathcal{C}(a,b)=\bigvee_{x\leq y\in X\times X}\mathcal{C}(a_x,b_y)\]
Composition sends $(f',x'\leq y')\wedge (f,x\leq y)$ to
\[(f,x\leq y)\circ (f',x'\leq y')=\left\{\begin{array}{ll}(f\circ f',x'\leq y)& \mbox{ if }y'=x\\
\ast &\mbox{ else }
\end{array}\right.\]
There is a duality $D\colon (T_{\vee}^X\mathcal{C})^{op}\longrightarrow T_{\vee}^X\mathcal{C}$ that sends an object $a$ to
\[(Da)_x=Da_{\omega x}\]
and a morphism $(f,x\leq y)\in T_{\vee}^X\mathcal{C}(a,b)_i$ to
\[(Df,\omega y\leq\omega x)\in\mathcal{C}(Db_{y},Da_{x})_i\subset\bigvee_{x'\leq y'} \mathcal{C}(Db_{\omega y'},Da_{\omega x'})_i\]
If $\mathcal{M}$ is a bimodule on $\mathcal{C}$ with duality $J$, define $T_{\vee}^X\mathcal{M}$ to be
\[T^X\mathcal{M}(a,b)=\bigvee_{x\leq y\in X\times X}\mathcal{M}(a_x,b_y)\]
with bimodule structure and duality defined similarly.
\end{defn}

There is a functor $\bigvee_X\mathcal{C}\longrightarrow T_{\vee}^X\mathcal{C}$ commuting with the dualities, given by the identity on objects, and sending $(f,x)\in \bigvee_X\mathcal{C}(a,b)_i$ to $(f,x\leq x)$.
Proposition \ref{triangdiagmat} is a direct corollary of the following.

\begin{prop}
There is a commutative diagram of equivariant maps
\[\xymatrix{\THR(HC^X;HM^X)\ar[r]^-{m_X}&\THR(T^XHC;T^XHM)\\
\THR(\bigvee_XHC;\bigvee_XHM)\ar[u]\ar[r]&\THR(T_{\vee}^XHC;T_{\vee}^XHM)\ar[u]
}\]
where the vertical maps and the bottom map are $G$-equivalences.
\end{prop}

\begin{proof}
We proved in \ref{powerandsums} that the left vertical map is an equivalence. By \ref{wedgesintoproducts}, the right vertical map is a $G$-equivalence since the inclusion of wedges into products defines a stable $G$-equivalence.

We prove that the bottom map is an equivalence for a general $\mathcal{C}$, by defining a $G$-homotopy inverse for the levelwise maps
\[d\colon \THR_{2k+1}(\bigvee_X\mathcal{C};\bigvee_X\mathcal{M})\longrightarrow\THR_{2k+1}(T_{\vee}^X\mathcal{C};T_{\vee}^X\mathcal{M})\]
 We prove this by defining for every $\underline{i}\in I[2k+1]$ a map
\[\tr\colon V(T_{\vee}^X\mathcal{C};T_{\vee}^X\mathcal{M},\underline{i})\longrightarrow V(\bigvee_X\mathcal{C};\bigvee_X\mathcal{M},\underline{i})\]
compatible with the structure maps. It sends
\[((a^0,\dots,a^{2k+1}),(m,x_0\leq y_0)\wedge(f^1,x_1\leq y_1)\wedge\dots\wedge (f^{2k+1},x_{2k+1}\leq y_{2k+1}))\]
where $a^l\in Ob\mathcal{C}^X$, $m\in\mathcal{M}(a^{0}_{x_0},a^{2k+1}_{y_0})$ and $f^l\in \mathcal{C}(a^{l}_{x_l},a^{l-1}_{y_l})$ to
\[\resizebox{1\hsize}{!}{$\left\{\begin{array}{ll}((a^0\!\!,\!\dots\!,\!a^{2k\!+\!1}),\!(m,\!x_0)\!\wedge\!(f^1\!,\!x_0)\!\wedge\!\dots\!\wedge\!(f^{2k\!+\!1}\!,\!x_{0}))& \mbox{ if }x_0\!=\!y_0\!=\!x_1\!=\!\dots\!=\!y_{2k\!+\!1}\\
\ast & \mbox{  else}
\end{array}\right.$}\]
Notice that we do not just compare $x_l$ with $y_l$, but we look at all the indicies together. The composite $\tr\circ d\colon V(\bigvee_X\mathcal{C};\bigvee_X\mathcal{M};\underline{i})\longrightarrow V(\bigvee_X\mathcal{C};\bigvee_X\mathcal{M};\underline{i})$ sends
\[((a^0,\dots,a^{2k+1}),(m,x_0)\wedge(f^1,x_1)\wedge\dots\wedge (f^{2k+1},x_{2k+1}))\]
to itself if $x_0=\dots=x_{2k+1}$, and to the basepoint otherwise. A $G$-homotopy between this map and the identity, is induced by the maps \[H\colon V(\bigvee_X\mathcal{C};\bigvee_X\mathcal{M};\underline{i})\times\Delta[1]_k\longrightarrow V(\bigvee_X\mathcal{C};\bigvee_X\mathcal{M};\underline{i})\] defined for an $\sigma=(0^b,1^{k+1-b})\in \Delta[1]_k$ by sending an element as above to itself if
\[x_0=\dots=x_{b}=x_{2k+2-b}=\dots=x_{2k+1}\]
and to the basepoint otherwise.

The other composite $d\circ \tr$ sends an element \[((a^0,\dots,a^{2k+1}),(m,x_0\leq y_0)\wedge(f^1,x_1\leq y_1)\wedge\dots\wedge (f^{2k+1},x_{2k+1}\leq y_{2k+1}))\] to itself if $x_0=y_0=x_1=\dots=y_{2k+1}$ and to the basepoint otherwise. A $G$-homotopy is defined at $\sigma=(0^b,1^{k+1-b})\in \Delta[1]_k$ by sending such an element to itself if
\[\xymatrix@=3pt{ & x_0\ar@{=}[dr]&x_1\ar@{=}[dr]&x_2\ar@{=}[dr]&\dots\ar@{=}[dr]&x_{b\!-\!1}\ar@{=}[dr]&x_{b}&&x_{2k\!+\!2\!-\!b}\ar@{=}[dr]&x_{2k\!+\!3\!-\!b}\ar@{=}[dr]& \dots \ar@{=}[dr]&x_{2k}\ar@{=}[dr]&x_{2k\!+\!1}\ar@{=}[r]&y_0\\
x_{2k\!+\!1}\ar@{=}[r]&y_0&y_1&y_2&\dots &y_{b\!-\!1}&y_{b}&&y_{2k\!+\!2\!-\!b}&y_{2k\!+\!3\!-\!b}&\dots&y_{2k}&y_{2k\!+\!1}
}\]
and to the basepoint otherwise. Notice the missing equality sign in the middle of the expression. The symmetry of this relation shows that the map is equivariant for every fixed $\sigma$, and one can check that it preserves the structure maps. It therefore defines a simplicial map on the subdivision
\[H\colon \THR_{2k+1}(T_{\vee}^X\mathcal{C};T_{\vee}^X\mathcal{M})\times \Delta[1]_k\longrightarrow\THR_{2k+1}(T_{\vee}^X\mathcal{C};T_{\vee}^X\mathcal{M}) \]
When $b=0$ the table above is no condition, and therefore $H(-,0)=\id$. When $b=k+1$, the table gives that
\[x_0\leq y_0=x_{2k+1}\leq y_{2k+1}\]and also
\[x_0=y_1\geq x_1=y_2\geq\dots\geq x_{2k}=y_{2k+1}\]
Therefore $x_0=y_{2k+1}$, and this forces all the indicies to be equal since $x_l\leq y_l$. This shows that $H(-,1)$ is $d\circ\tr$.
\end{proof}


\subsection{$\THR$ of the $S^{1,1}_{\cdot}$-construction and $S^{2,1}_{\cdot}$-construction}\label{THRforS21}

Let $(C,D,\mathcal{E})$ be an exact category with strict duality, and $M\colon C^{op}\otimes C\longrightarrow Ab$ a bimodule over $C$ with duality $J\colon M\Rightarrow M\circ D_\gamma$. To simplify the notation we drop the Eilenberg-MacLane construction, and we denote
\[\THR(C;M)=\THR(HC;HM)\]
We are going to extend $(M,J)$ to a bimodule with duality
\[M_k\colon S^{2,1}_kC^{op}\otimes S^{2,1}_kC\longrightarrow Ab\]
for every $k\geq 0$, and then express $\THR(S^{2,1}_kC;M_k)$ in terms of $\THR(C;M)$. We will do a similar construction for $S^{1,1}_kC$.

Given $X,Y\in S^{2,1}_kC$ we define $M_k(X,Y)$ as the subgroup
\[M_k(X,Y)\subset\bigoplus_{\theta\in Cat([2],[k])}\!\!\!\!\!\!M(X_\theta,Y_\theta)\]
defined by the following condition. For each $a\colon \rho\longrightarrow\theta$ in $Cat([2],[k])$ we have $X(a)\colon X_\rho\longrightarrow X_\theta$ and $Y(a)\colon Y_\rho\longrightarrow Y_\theta$ and hence
\[M(X_\theta,Y_\theta)\stackrel{X(a)^\ast}{\longrightarrow}M(X_\rho,Y_\theta)\stackrel{Y(a)_\ast}{\longleftarrow}M(X_\rho,Y_\rho)\]
The requirement for $\{m_\theta\}_{\theta\in Cat([2],[k])}\in M_k(X,Y)$ is that
\[X(a)^\ast(m_\theta)=Y(a)_\ast(m_\rho)\]
for all $a\colon \rho\longrightarrow\theta$ in $Cat([2],[k])$.
A morphism $(\phi\colon X'\longrightarrow X)\otimes(\psi\colon Y\longrightarrow Y')$ in $(S^{2,1}_kC)^{op}\otimes S^{2,1}_kC$ is sent to the restriction to $M_k(X,Y)$ of the map
\[\xymatrix{\bigoplus\limits_{\theta\in Cat([2],[k])}\!\!\!\!\!M(X_\theta,Y_\theta)\ar[rr]^-{\bigoplus_\theta M(\phi_\theta,\psi_{\theta})}&& \!\!\!\bigoplus\limits_{\theta\in Cat([2],[k])}\!\!\!\!\!M(X_{\theta}',Y_{\theta}')}\]
The bimodule $M_k$ carries a duality $J\colon M_k\Rightarrow M_k\circ D_\gamma$, given by restricting to $M_k(X,Y)$ the map
\[\xymatrix{\bigoplus\limits_{\theta\in Cat([2],[k])}\!\!\!\!\!M(X_\theta,Y_\theta)\ar[rr]^-{\bigoplus_{\theta} J}&& \!\!\!\bigoplus\limits_{\theta\in Cat([2],[k])}\!\!\!\!\!M(DY_{\overline{\theta}},DX_{\overline{\theta}})}\]

\begin{ex}\label{moduleS21PA}
Let $A$ be a ring, $M$ an $A$-bimodule, and $J$ an $M$-twisting of an antistructure $(A,L,\alpha)$. Recall from §\ref{THRrings} that this data defines a bimodule with duality $M^J$ on $\mathcal{D}\mathcal{P}_A$. The induced bimodule $M^{J}_k$ on $S^{2,1}_k\mathcal{D}\mathcal{P}_A$ associates to two diagrams of triples $\varphi=(X,Y,\phi)$ and $\varphi'=(X',Y',\phi')$ the abelian group $M_{k}^J(\varphi,\varphi')$ of pairs of natural transformations of diagrams $(f\colon X\longrightarrow X'\otimes_A M,g\colon Y'\longrightarrow Y\otimes_A M)$ such that
\[\xymatrix{Y_\theta\otimes_A M\ar[r]^-{\phi_\theta\otimes M}& D_L(X_\theta)\otimes_A M\\
Y'_\theta\ar[r]_{\phi'_\theta}\ar[u]^{g_\theta} & D_L(X_\theta')\ar[u]_{\widehat{J}(f_\theta)}
}\]
commutes for all $\theta$. Indeed, the condition defining $M^{J}_k$ as a subset of the direct sum $\bigoplus M^J(X_\theta,Y_\theta)$, is exactly naturality for the families $\{f_\theta\}$ and $\{g_\theta\}$.
The induced duality on $M^{J}_k$ exchanges $f$ and $g$. When $\varphi=\varphi'$ we recover the abelian group
\[M^{J}_k(\varphi,\varphi)=\hom_A(\varphi,\varphi\otimes_AM)\]
of theorem \ref{mainKR}.
\end{ex}

A completely analogous construction defines from $(M,J)$ a bimodule with duality on  $S^{1,1}_kC$, still denoted $M_{k}$.

We denote $S^{1,1}_{\sbt}$ and $S^{2,1}_{\sbt}$ the simplicial spheres $S^{1}_{\sbt}$ and $S^{2}_{\sbt}$ with action induced by conjugation action on $Cat([k],[n])$, for $n=1,2$. Recall that this action sends an element $\sigma\colon [k]\longrightarrow [n]$ to
\[\overline{\sigma}(l)= n-\sigma(k-l)\]
If we denote an element of $Cat([k],[n])$ by a sequence $(i_0\leq\dots\leq i_k)$ of integers $0\leq i_l\leq n$, the involution is
\[(n-i_k\leq\dots\leq n-i_0)\]
For a pointed $G$-set $X$ and a pointed $G$-space $Y$ we denote by $\map_{\ast}(X,Y)$ the space of $G$-maps with $G$ acting by conjugation.

\begin{prop}\label{sumdiagsTHR} Suppose that $(C,D,\mathcal{E})$ is a split-exact category with strict duality, and $(M,J)$ a bimodule with duality over $(C,D)$. There are $G$-equivalences
\[\THR(S^{1,1}_kC;M_k)\longrightarrow \map_\ast(S^{1,1}_{k},\THR(C;M))\]
and
\[\THR(S^{2,1}_kC;M_k)\longrightarrow \map_\ast(S^{2,1}_{k},\THR(C;M))\]
for every $k\geq 0$.
\end{prop}

Before proving this result, we discuss an interesting consequence.
\begin{cor}\label{faceshomotopic}
There is a $G$-homotopy between the two equivariant maps
\[d_1\simeq_G(d_0+d_2)\colon\THR^\oplus(S^{1,1}_2C;M_2)\longrightarrow \THR^\oplus(C;M)\]
where $d_i$ denote the face functors for $S^{1,1}_{\sbt} C$ and $\THR^\oplus$ is the abelian group model from §\ref{abgpmodelsec}.
\end{cor}

\begin{proof}
The equivalence
\[\THR(S^{1,1}_2C;\!M_2)\longrightarrow \map_\ast(S^{1,1}_2\!,\THR(C;\!M))\!=\!\THR(C;\!M)\!\times\!\THR(C;\!M)\]
of \ref{sumdiagsTHR}
is given by the map $(d_2,d_0)$ as explained below. The action on the target swaps the two components and applies the involution of $\THR(C;M)$.
By \ref{abgpmodel}, the map
\[(d_2,d_0)\colon \THR^\oplus(S^{1,1}_2C;M_2)\longrightarrow\THR^\oplus(C;M)\times \THR^\oplus(C;M)\]
is a $G$-equivalence as well.
The map $s_2\pr_1+s_0\pr_2$ is an equivariant section for $(d_2,d_0)$, where $s_i$ denotes the degeneracy functors. Indeed
\[\begin{array}{ll}(d_2,d_0)\circ(s_2\pr_1+s_0\pr_2)&=(d_2s_2\pr_1+d_2s_0\pr_2,d_0s_2\pr_1+d_0s_0\pr_2)\\
&=(\pr_1+0,0+\pr_2)=\id\end{array}\]
Since $(d_2,d_0)$ is an equivalence, $s_2\pr_1+s_0\pr_2$ must be a $G$-homotopy inverse. 
Therefore the maps of the statement are $G$-homotopic if and only if the composites
\[\xymatrix{\THR^\oplus(C;M)\times \THR^\oplus(C;M)\ar[r]^-{\begin{smallmatrix}s_2\pr_1\\+\\s_0\pr_2\end{smallmatrix}}_-{\simeq}&\THR^\oplus(S^{1,1}_2C;M_2)\ar@<1ex>[r]^-{d_1}\ar@<-1ex>[r]_-{d_0+d_2}&\THR^\oplus(C;M)
}\]
are $G$-homotopic. These are in fact equal, since
\[\begin{array}{ll}
(d_0+d_2)\circ (s_2\pr_1+s_0\pr_2)&=d_0s_2\pr_1+d_0s_0\pr_2+d_2s_2\pr_1+d_2s_0\pr_2\\
&=\pr_2+\pr_1
\end{array}\]
and the other composite is
\[d_1\circ (s_2\pr_1+s_0\pr_2)=d_1s_2\pr_1+d_1s_0\pr_2=\pr_1+\pr_2\]
\end{proof}

We start the proof of \ref{sumdiagsTHR} with the statement for for $S^{1,1}_kC$. Let
\[d\colon S^{1,1}_kC\longrightarrow C^{S^{1,1}_k\backslash\ast}=\map_\ast(S^{1,1}_k,C)\]
be the functor that sends a diagram $X\colon Cat([1],[k])\longrightarrow C$ to its diagonal, that is to
\[d(X)(\sigma=(0^i1^{k+1-i}))=X_{k-i<k-i+1}\]
Similarly, there is a natural transformation $M_k\Rightarrow M^{S^{1,1}_k\backslash\ast}\circ(d^{op}\otimes d)$ given by the restriction to $M_k(X,Y)$ of the projection onto the diagonal
\[\bigoplus_{i<j}M(X_{i<j},Y_{i<j})\longrightarrow \bigoplus_{i}M(X_{i<i+1},Y_{i<i+1})\]
This induces on $\THR$ the map
\[\THR(S^{1,1}_kC;M_k)\longrightarrow \THR(C^{S^{1,1}_k\backslash\ast};M^{S^{1,1}_k\backslash\ast})\simeq \THR(C;M)^{S^{1,1}_k\backslash\ast}\]
from the proposition, after composing with the equivalence of \ref{THRproducts}.
There is a linear order on $S^{1,1}_k\backslash\ast$ that is reversed by the duality. It is defined by declaring \[\sigma=(0^i1^{k+1-i})\leq\sigma'=(0^j1^{k+1-j})\colon [k]\longrightarrow [1]\] if $j\leq i$.
The proposition \ref{sumdiagsTHR} for $S^{1,1}_kC$ is a direct corollary of the following statement, since by \ref{inversenonstrict} $\THR$ sends equivalences of categories to $G$-equivalences.

\begin{prop}\label{matriciesandsdot}
The map $d\colon (S^{1,1}_kC,M_k)\longrightarrow (C^{S^{1,1}_k\backslash\ast},M^{S^{1,1}_{k}\backslash\ast})$ extends to a commutative diagram of morphisms of bimodules with duality
\[\xymatrix{ (S^{1,1}_kC,M_k)\ar[r]^-{d}&(C^{S^{1,1}_k\backslash\ast},M^{S^{1,1}_k\backslash\ast})\\ 
(T^{S^{1,1}_k\backslash\ast}C,T^{S^{1,1}_k\backslash\ast}M)\ar[u]^{(F,\Phi)}\ar[ur]_{\tr}
}\]
where $\tr$ is the $\THR$-equivalence of \ref{triangdiagmat}. Moreover if the category $C$ is split-exact,
the functor $F$ is an equivalence of categories with duality, and $\Phi$ an isomorphism of bimodules with duality.
\end{prop}
In order to define the functor $F$ of the proposition we need a remark on the choice of sums in an additive category $C$.
\begin{rem}\label{sumondc}
A sum for two objects $c$ and $d$ of $C$ is an object $c\oplus d$ of $C$ with maps
\[\xymatrix{c\ar@<.8ex>[r]^-{i_c}&c\oplus d\ar@<.8ex>[l]^-{p_c}\ar@<.8ex>[r]^-{p_d}&d\ar@<.8ex>[l]^-{i_d}}\]
satisfying $p_ci_c=\id_c$, $p_di_d=\id_d$ and $i_cp_c+i_dp_d=\id_{c\oplus d}$. Since $C$ is additive, a sum exists for every pair of objects and it is both a product an a coproduct for $c$ and $d$ (cf. \cite[VIII-2.2]{maclane}).
Since $D$ is an additive functor, it preserves sums, i.e. if $c\oplus d$ is a sum for $(c,d)$, the object $D(c\oplus d)$ is a sum for $(Dd,Dc)$. This means that for a functorial choice of sums $\oplus\colon C\times C\longrightarrow C$, there is canonical natural isomorphism $\kappa_{c,d}\colon Dd\oplus Dc\longrightarrow D(c\oplus d)$. Moreover by the universal property of sums it follows that the diagram
\[\xymatrix{c\oplus d\ar[r]^-{\eta_c\oplus\eta_d}\ar[d]_{\eta_{c\oplus d}}&D^2c\oplus D^2d\ar[d]^{\kappa_{Dd,Dc}}\\
D^2(c\oplus d)\ar[r]_-{D(\kappa_{c,d})}&D(Dd\oplus Dc)
}\]
commutes. This is exactly saying that the pair $(\oplus,\kappa)\colon C\times C\longrightarrow C$ is a morphism of categories with duality, where $C\times C$ has the duality that sends $(c,d)$ to $(Dd,Dc)$. Therefore the induced functor
\[\oplus\colon \mathcal{D}C\times\mathcal{D}C\cong \mathcal{D}(C\times C)\longrightarrow \mathcal{D}C\]
commutes strictly with the duality, i.e. it satisfies $D(\phi\oplus \psi)=D\psi\oplus D\phi$. Moreover this is a choice of sums for the category $\mathcal{D}C$. Since the inclusion $C\longrightarrow\mathcal{D}C$ induces a $G$-equivalence in $\THR$ by \ref{scriptdandthr}, by replacing $C$ with $\mathcal{D}C$ we can assume that there is a choice of sum on $C$ satisfying $D(c\oplus d)=Dd\oplus Dc$.

Notice that a similar argument with the duality $D\times D$ on $C\times C$ defines a sum on $\mathcal{D}C$ with the property $D(\phi\oplus \psi)=D\phi\oplus D\psi$.
\end{rem}

\begin{proof}[Proof of \ref{matriciesandsdot}]
Let us denote $T^{1,1}_kC:=T^{S^{1,1}_k\backslash\ast}C$ the category of triangular matricies.
Recall that an object of $T^{1,1}_kC$ is family of objects $a=\{a_{\sigma}\}$ of $C$, one for each $\sigma\colon [k]\longrightarrow [1]$ in $S^{1,1}_k\backslash\ast$. A map from $a$ to $b$ in $T^{1,1}_kC$ is a family of maps $f=\{f_{\sigma\leq \sigma'}\}$ of $C$, one for each pair of elements $\sigma\leq \sigma'$ of $S^{1,1}_k\backslash\ast$, with
\[f_{\sigma\leq \sigma'}\colon a_\sigma\longrightarrow b_{\sigma'}\]
Define the functor $F\colon T^{1,1}_kC\longrightarrow S^{1,1}_kC$ by sending an object $a$ to the diagram
\[F(a)_{i<j}=\bigoplus_{\sigma\in\phi(ij)} a_{\sigma}\]
where $\phi(ij)$ is the ordered set $\phi(ij)=\{(0^{i+1}1^{k-i})\leq\dots\leq(0^j1^{k+1-j})\}$ and $\oplus$
is a functorial choice of sums like in \ref{sumondc}.
The maps of $F(a)$ are inclusions and projections of the direct summands. A morphism $\{f_{\sigma\leq \sigma'}\}$ is sent to the natural transformation $F(f)$ defined at $i<j$ by the map
\[F(f)_{i<j}\colon \bigoplus_{\sigma\in\phi(ij)} a_{\sigma}\longrightarrow \bigoplus_{\sigma\in\phi(ij)} b_{\sigma}\]
having $(\sigma,\sigma')$-component $f_{\sigma\leq \sigma'}$ if $\sigma\leq \sigma'$, and zero otherwise.

We show that $F$ is an equivalence of categories when $C$ is split-exact. Choosing splittings for the exact sequences in a diagram $X\in S^{1,1}_kC$ gives an isomorphism between $F d(X)$ and $X$, showing that $F$ is essentially surjective. For every $a,b\in T^{1,1}_kC$ the map
\[F\colon T^{1,1}_kC(a,b)\longrightarrow S^{1,1}_kC(F(a),F(b))\]
is clearly injective. To see that it is surjective, we need to show that if $f\colon F(a)\longrightarrow F(b)$ is a natural transformation, the matrix
\[f_{i<j}\colon \bigoplus_{\sigma\in\phi(ij)} a_{\sigma}\longrightarrow \bigoplus_{\sigma\in\phi(ij)} b_{\sigma}\]
has zero $(\sigma,\sigma')$-component if $\sigma>\sigma'$, that is if $\sigma=(0^l1^{k+1-l})$ and $\sigma'=(0^h1^{k+1-h})$ with $l<h$. 
By naturality of $f$ the diagram
\[\xymatrix{a_\sigma=F(a)_{l-1<l}\ar[r]\ar[d]_{f_{l-1<l}}&F(a)_{i<j}\ar[r]\ar[d]_{f_{i<j}}&a_{\sigma'}=F(a)_{h-1<h}\ar[d]^{f_{j<k}}\\
b_\sigma=F(b)_{l-1<l}\ar[r]&F(b)_{i<j}\ar[r]&b_{\sigma'}=F(b)_{h-1<h}
}\]
commutes, where the horizontal maps are inclusion and projections of direct summands. The $(\sigma,\sigma')$-component of $f_{i<j}$ is the composite from the top left corner to the bottom right corner, which is zero.

The functor $F$ commutes strictly with the dualities, since by \ref{sumondc}
\[F(Da)_{i<j}=\bigoplus_{\sigma\in\phi(ij)} Da_{\overline{\sigma}}= D(\!\!\!\bigoplus_{\sigma\in\phi(k-j,j-i)}\!\!\! a_{\sigma})=(DF(a))_{i<j}\]

The map of bimodules $\Phi\colon T^{1,1}_kM\Rightarrow M_k\circ(F^{op}\otimes F)$ is defined as follows. Since $M$ is an additive functor, there is a canonical natural isomorphism
\[\bigoplus_{i<j}M(F(a)_{i<j},F(b)_{i<j})\cong \bigoplus_{i<j}\bigoplus_{l,h=i}^{j-1}M(a_{l+1},b_{h+1})\]
and $M_k(F(a),F(b))$ is by definition a subgroup of the left hand side.
Using a similar argument as for the surjectivity of $F$ on morphism, one can show that an element $\underline{m}$ belongs to $M_k(F(a),F(b))$ if and only if its components in the right hand side of the isomorphism above satisfy
\[m_{i<j,l,h}=m_{i'<j',l,h}\]
whenever $i,i'\leq l,h\leq j-1,j'-1$, and $m_{i<j,l,h}=0$ whenever $l>h$. That is, the isomorphism above restricts to an isomorphism
\[\xymatrix{M_k(F(a),F(b))\subset\bigoplus_{i<j}M(F(a)_{i<j},F(b)_{i<j})\ar@<10ex>[d]^{\cong}\ar@<-14ex>[d]_{\Phi}^{\cong}
\\ \bigoplus_{l\leq h}M(a_l,b_h)\subset\bigoplus_{i<j}\bigoplus_{l,h=i}^{j-1}M(a_{l+1},b_{h+1})
}\]
and we remind that $T^{1,1}_kM$ is defined by
\[T^{1,1}_kM(a,b)=\bigoplus_{l\leq h}M(a_l,b_h)\]
By naturality, the isomorphism $\Phi$ preserves the dualities.
\end{proof} 

The idea for the proof of \ref{sumdiagsTHR} for $S^{2,1}_kC$ is completely similar. 
The functor \[d\colon S^{2,1}_kC\longrightarrow C^{S^{2,1}_k\backslash\ast}\] is not quite a projection this time, but it has a similar function. Given a diagram $X\in S^{2,1}_kC$ and a surjective $\sigma\colon [k]\longrightarrow [2]$, the $\sigma$-component of $d(X)$ is defined by
\[d(X)_{\sigma}=\ker(X_{i-1<i-1+j<i+j}\longrightarrow X_{i\leq i-1+j<i+j})\]
where $\sigma=(0^i1^j2^{k+1-i-j})$. Notice that a kernel for this map exists, since it is the last map of the $4$-terms exact sequence of $X$ associated to the element $(i\!-\!1\!<\!i\!<\!i\!-\!1\!+\!j\!<\!i\!+\!j)$ of $Cat([3],[k])$.
The functor $d$ sends a morphism to its restrictions on the kernels.
One can do a similar construction on bimodules. On $\THR$, this gives an equivariant  map
\[\THR(S^{2,1}_kC;M_k)\longrightarrow \THR(C^{S^{2,1}_k\backslash\ast};M^{S^{2,1}_{k}\backslash\ast})\simeq \THR(C;M)^{S^{2,1}_{k}\backslash\ast}\]
after composing with the equivalence of \ref{THRproducts}.
As previously, we prove that this is an equivalence by extending it to the category of triangular matricies on $S^{2,1}_{k}\backslash\ast$. The set $S^{2,1}_{k}\backslash\ast$ can always be displayed as illustrated here for $k=5$
\[\xymatrix{000012\ar[r]&000112\ar[r]\ar[d]&001112\ar[r]\ar[d]&011112\ar[d]\\
& 000122\ar[r]&001122\ar[r]\ar[d]&011122\ar[d]\\
& & 001222\ar[r]&011222\ar[d]\\
& & &012222
}\]
This gives $S^{2,1}_{k}\backslash\ast$ a partial order reversed by the duality, since the involution flips the diagram along the "bottom left to top right" diagonal.
Proposition \ref{sumdiagsTHR} for $S^{1,1}_kC$ is a direct consequence the following statement, since by \ref{inversenonstrict} $\THR$ sends equivalences of categories to $G$-equivalences.

\begin{prop}
The functor $d\colon (S^{2,1}_kC,M_k)\longrightarrow (C^{S^{2,1}_{k}\backslash\ast},M^{S^{2,1}_{k}\backslash\ast})$ extends to a commutative diagram of morphisms of bimodules with duality
\[\xymatrix{ (S^{2,1}_kC,M_k)\ar[r]^-{d}&(C^{S^{2,1}_{k}\backslash\ast},M^{S^{2,1}_{k}\backslash\ast})\\ 
(T^{S^{2,1}_{k}\backslash\ast}C\ar[u]^{(F,\Phi)},T^{S^{2,1}_{k}\backslash\ast}M)\ar[ur]_{\tr}
}\]
where $\tr$ is the $\THR$-equivalence of \ref{triangdiagmat}. Moreover if the category $C$ is split-exact,
the functor $F$ is an equivalence of categories with duality, and $\Phi$ an isomorphism of bimodules with duality.
\end{prop}

\begin{proof} Let us denote $T^{2,1}_kC:=T^{S^{2,1}_{k}\backslash\ast}C$ the category of triangular matricies.
The functor $F\colon T^{2,1}_kC\longrightarrow S_{k}^{2,1}C$ is defined as follows. For an injective $\theta\colon [2]\longrightarrow [k]$, denote 
\[r(\theta)=\{\rho\colon [k]\longrightarrow [2]|\rho\circ\theta=\id_{[2]}\}\]
the set of retractions of $\theta$. On objects, $F$ sends $a=\{a_\rho\}\in T^{2,1}_kC$ to
\[F(a)_{\theta}=\bigoplus_{\rho\in r(\theta)}a_{\rho}\]
On morphisms, a matrix $\{f_{\rho\leq\rho'}\}\colon a\longrightarrow b$ is sent to the natural transformation given at a $\theta\colon [2]\longrightarrow [k]$ by the map
\[F(f)_\theta\colon \bigoplus_{\rho\in r(\theta)}a_{\rho}\longrightarrow \bigoplus_{\rho\in r(\theta)}b_{\rho}\]
having $(\rho,\rho')$-component $f_{\rho\leq\rho'}$ if $\rho\leq\rho'$, and zero otherwise.
Notice that the functor $F$ for $S^{1,1}_kC$ can also be defined in terms of retractions by similar formulas. It commutes strictly with the duality, since by \ref{sumondc}
\[\begin{array}{ll}F(Da)_{\theta}&=\bigoplus_{\rho\in r(\theta)} Da_{\omega\rho}=D(\bigoplus_{\rho'\in r(\omega\theta)} a_{\rho'})\\
&=D(F(a)_{\omega\theta})=(DF(a))_{\theta}\end{array}\]
The map of bimodules $\Phi$ is defined in an analogous way as in \ref{matriciesandsdot}.

The functor $F$ is an equivalence of categories with duality when $C$ is split-exact. We sketch the argument from \cite{IbLars}. Choosing splittings for the exact sequences in a diagram $X\in S_{k}^{2,1}C$ gives an isomorphism $X\cong F d(X)$, showing that $F$ is essentially surjective. $F$ is clearly faithful, and to show that it is full one needs to show by diagram chase that morphisms $F(a)\longrightarrow F(b)$ in $S_{k}^{2,1}C$ have zero $(\rho,\rho')$-components if $\rho'>\rho$.
A similar argument shows that the bimodules map $\Phi$ is an isomorphism.
\end{proof}


\subsection{Equivariant delooping of $\THR$ by the $S^{2,1}_{\cdot}$-construction}

We saw in §\ref{firstdeloop} how to deloop equivariantly $\THR(\mathcal{C};\!\mathcal{M})$ if $\mathcal{C}$ and $\mathcal{M}$ are $G$-connected. In this section we give another way of delooping $\THR$ in the split-exact case, by means of the $S^{2,1}_{\sbt}$-construction.

Let $(C,D,\mathcal{E})$ be a split-exact category with strict duality $D$, and let $M\colon C^{op}\otimes C\longrightarrow Ab$ be a bimodule with duality $J$. Recall that we denote
\[\THR(C;M)=\THR(HC;HM)\]
for simplicity.
We saw in \ref{THRforS21} how this structure induces a bimodule with duality $M_k$ on $S^{2,1}_kC$ for all natural number $k$. We now vary the $k$-direction to get an extra real direction. More precisely, let $\tau^\ast\colon S^{2,1}_{k} C\longrightarrow S^{2,1}_{l} C$ be the functor from the simplicial structure of $S^{2,1}_{\sbt} C$, associated to a $\tau\colon [l]\longrightarrow [k]$. There is a natural transformation $\tau^\ast\colon M_{k}\Rightarrow M_{l}\circ((\tau^\ast)^{op}\otimes \tau^\ast)$ given by restriction to $M_k(X,Y)$ of the map
\[\bigoplus\limits_{\theta\in Cat([2],[k])}\!\!M(X_{\theta},Y_{\theta})\longrightarrow\bigoplus\limits_{\rho\in Cat([2],[l])}\!\!M(X_{\tau\rho},Y_{\tau\rho})\]
that projects off the components corresponding to $\theta$'s that are not of the form  $\tau\rho$. This construction defines maps $\tau^\ast\colon (S^{2,1}_{k} C,M_k)\longrightarrow (S^{2,1}_{l} C,M_{l})$, and therefore maps
\[\tau^\ast\colon \THH_{\sbt}(S^{2,1}_k C;M_k)\longrightarrow \THH_{\sbt}(S^{2,1}_{l} C;M_{l})\]
defining a real space structure on
\[\THR_{\sbt}(S^{2,1}_{\sbt} C;M)=\{[k]\longmapsto\THR(S^{2,1}_k C;M_k)\}\]
Its realization is the $G$-space denoted $\THR(S^{2,1}_{\sbt} C,M)$ (we do not denote the simplicial direction in the coefficients).

Recall that for a $G$-space $X$, we denote $\Omega^{2,1}X=\map_\ast(S^{2,1},X)$ the pointed mapping space with conjugation action, where the $2$-sphere $S^{2,1}=\mathbb{C}^+$ carries the $G$-action induced by complex conjugation.
The goal of this section is to prove the following.

\begin{theorem}\label{THRdeloopings}
For a split-exact category with strict duality $(C,D,\mathcal{E})$ with bimodule with duality $M$, the map
\[\THR(C;M)\longrightarrow \Omega^{2,1}\THR(S^{2,1}_{\sbt} C;M)\]
induced by the adjoint of the projection map 
\[\THR(C;M)\times \Delta^2=\THR(S^{2,1}_2C;M_2)\times \Delta^2\longrightarrow\THR(S^{2,1}_{\sbt} C;M)\]
is a $G$-equivalence.
\end{theorem}

This theorem shows that the $G$-spaces $\{\THR((S^{2,1}_{\sbt})^{(m)}C;M)\}_{m\geq 0}$ with the structure maps defined in the statement is a real $\Omega$-spectrum, defining an infinite real loop space structure on $\THR(C;M)$ (cf. \ref{realspec}).

We give a non-equivariant proof first. It is going to give an overview of the strategy for the proof on the fixed points. 
Notice that in a completely analogous way one can define $\THR(S^{1,1}_{\sbt} C;M)$.
\begin{prop}\label{noneqdeloopings}
The map
\[\THH(C;M)\longrightarrow \Omega^{2}\THH(S^{2,1}_{\sbt} C;M)\]
is a (non-equivariant) equivalence.
\end{prop}

\begin{proof}
We prove this in two steps. First, for the standard $S_{\sbt}$-construction there is an equivalence
\[\THH(C;M)\longrightarrow\Omega\THH(S_{\sbt} C;M)\]
when $C$ is split-exact.
Here's the argument from \cite[2.1.3]{ringfctrs}. Consider the simplicial category $PS_{\sbt} C$, given in degree $k$ by
\[(PS_{\sbt} C)_k=S_{k+1} C\]
The simplicial structure is defined by forgetting the first face and degeneracy from $S_{\sbt} C$ in every simplicial level. In this situation, the forgotten degeneracy gives a simplicial homotopy equivalence $\ast=S_0C\longrightarrow PS_{\sbt} C$. By applying $\THH$ to each functor of the homotopy, one gets a simplicial contraction $\THH(PS_{\sbt} C,M)\simeq\ast$.
Moreover for every $k$, there is a sequence of functors
\[C\longrightarrow (PS_{\sbt} C)_k\stackrel{d_{0}}{\longrightarrow} S_kC\]
where $d_{0}$ is the face we removed. Applying $\THH$, we get a commutative diagram
\[\xymatrix{\THH(C;M)\ar@{=}[d]\ar[r]&\THH((PS_{\sbt} C)_{k};M_{k+1})\ar[r]^{d_0}\ar[d]_{\simeq}&\THH(S_k C;M_k)\ar[d]_\simeq\\
\THH(C;M)\ar[r]&\THH(C;M)^{k+1}\ar[r]&\THH(C;M)^k
}\]
where the vertical maps are the equivalences from \ref{sumdiagsTHR}, and the bottom map is the trivial fibration that projects off the last component of the product. Therefore the sequence of simplicial spaces
\[\THH(C;M)\longrightarrow \THH_{\sbt}(PS_{\sbt} C;M)\stackrel{d_0}{\longrightarrow} \THH_{\sbt}(S_{\sbt} C;M)\]
is levelwise a fiber sequence. Lemma \ref{piastbusiness} below says that the conditions of the Bousfield-Friedlander theorem \cite[4.9]{goersjardine} (cf. \ref{bous-fried} below) are satisfied by the map $d_0$, and therefore the realization
\[\THH(C;M)\longrightarrow \THH(PS_{\sbt} C;M)\stackrel{d_0}{\longrightarrow} \THH(S_{\sbt} C;M)\]
is a fiber sequence as well. Since the total space is canonically contractible, the map $\THH(C;M)\longrightarrow\Omega\THH(S_{\sbt} C;M)$ is an equivalence.

Second, we prove that there is a (non-equivariant) equivalence \[\THH(S_{\sbt} C;M)\stackrel{\simeq}{\longrightarrow}\Omega\THH(S_{\sbt}^{2,1}C;M)\] Consider, for every $k$, the sequence
\[S_kC\longrightarrow P(S^{2,1}_{\sbt} C)_{k}\stackrel{d_{0}}{\longrightarrow} S^{2,1}_kC\]
where the first map is the inclusion as the bottom face, that sends a diagram $X\colon Cat([1],[k])\longrightarrow C$ to $\widetilde{X}\colon Cat([2],[k+1])\longrightarrow C$ defined by
\[\widetilde{X}(i< j< l)=\left\{\begin{array}{ll}X(j-1<l-1)&\mbox{ if }i=0 \\
0 & \mbox{ else }\end{array}\right.\]
This map is simplicial, and applying $\THH$ it defines a commutative diagram
\[\xymatrix{\THH(S_kC;M_k)\ar[d]_\simeq\ar[r]&\THH((PS^{2,1}_{\sbt})_{k} C;M_{k+1})\ar[r]^{d_0}\ar[d]_{\simeq}&\THH(S^{2,1}_k C;M_k)\ar[d]_\simeq\\
\THH(C;M)^{S^{1}_k\backslash \ast}\ar[r]&\THH(C;M)^{S^{2,1}_{k+1}\backslash \ast}\ar[r]&\THH(C;M)^{S^{2,1}_k\backslash \ast}
}\]
where the vertical maps are the equivalences from \ref{sumdiagsTHR}. The bottom row is a fibration for every $k$, since the projection map is just precomposition with the inclusion $S^{2,1}_k\longrightarrow S^{2,1}_{k+1}$ induced by $\sigma^{0}\colon[k+1]\longrightarrow [k]$.
Lemma \ref{piastbusiness} below applies exactly in the same way to the realization
\[d_0\colon\THH(PS^{2,1}_{\sbt} C;M)\longrightarrow\THH(S^{2,1}_{\sbt} C;M)\]
by using the abelian group model $\THH^\oplus$ of \ref{abgpmodel}, and since $S^{2,1}_0C=0$.
Thus Bousfield-Friedlander theorem \ref{bous-fried} applies to show that the realization
\[\THH(S_{\sbt} C;M)\longrightarrow \THH(PS^{2,1}_{\sbt} C;M)\stackrel{d_0}{\longrightarrow}\THH(S^{2,1}_{\sbt} C;M)\]
is still fiber sequence. The total space is contractible by the same argument as before, and therefore $\THH(S_{\sbt} C;M)\simeq\Omega\THH(S_{\sbt}^{2,1}C;M)$.

Combining the two results we obtain an equivalence
\[\THH(C;M)\stackrel{\simeq}{\longrightarrow}\Omega\THH(S_{\sbt} C;M)\stackrel{\simeq}{\longrightarrow}\Omega^2\THH(S^{2,1}_{\sbt} C;M)\]
which is homotopic to the map of the statement.
\end{proof}

The next lemma has been used in the proof above, and variations of it will be used throughout the section. It is central in most proofs, since it allows to commute homotopy pullbacks with realizations of simplicial spaces, by means of \ref{bous-fried} below.

\begin{defn}[{\cite[IV.4]{goersjardine}}]
Let $X_{\sbt\,\sbt}$ be a bisimplicial set such that for all $k\geq 0$ the simplicial set $X_{k,\,\sbt}$ is fibrant. For all $n\geq 1$ define $\pi_n X$ to be the simplicial set
\[\pi_nX=\{[k]\longmapsto \coprod\limits_{x\in X_{k,0}}\pi_n (X_{k,\,\sbt},x)\}\]
We say that $X_{\sbt\,\sbt}$ \textbf{satisfies the $\pi_\ast$-Kan condition} if for all $n\geq 1$ the canonical map
\[\pi_n X\longrightarrow X_{\sbt,0}\]
is a Kan-fibration. A general bisimplicial set $Y_{\sbt\,\sbt}$ satisfies the $\pi_\ast$-Kan condition if there is a levelwise weakly equivalent, levelwise fibrant bisimplicial set $X_{\sbt\,\sbt}$ that satisfies the $\pi_\ast$-Kan condition.
\end{defn}

\begin{lemma}\label{piastbusiness}
The map
\[d_0\colon\THH(PS_{\sbt} C;M)\longrightarrow \THH(S_{\sbt} C;M)\]
is the realization of a map of bisimplicial sets $Y_{\sbt\,\sbt}\longrightarrow X_{\sbt\,\sbt}$ such that
$X_{\sbt\,\sbt}$ and $Y_{\sbt\,\sbt}$ satisfy the $\pi_\ast$-Kan condition, and
\[\pi_0X_{k,\,\sbt}=\pi_0Y_{k,\,\sbt}=\ast\]
for all $k$.
\end{lemma}

\begin{proof}
The space $\THH(S_{\sbt} C;M)$ is homeomorphic to the realization of the $3$-simplicial set
\[([k],[p],[q])\mapsto\hocolim_{I[k]}\Omega^{\underline{i}}V_q(HS_p C;HM,\underline{i})\]
where $V_q(HC;HM,\underline{i})$ is defined in the same way as $V(HC;HM,\underline{i})$ by replacing the topological spaces $C(c,d)(S^i)$ by their $q$-simplicies $C(c,d)(S_{q}^i)$, so that
\[|[q]\mapsto V_q(HC;HM,\underline{i})|\cong V(HC;HM,\underline{i})\]
Taking the diagonal in the first and last simplicial directions we obtain a bisimplicial set 
\[X_{\sbt\,\sbt}=\{[k][p]\mapsto\hocolim_{I[k]}\Omega^{\underline{i}}V_k(HS_p C;HM,\underline{i})\}\]
Now define another bisimplicial set by $Y_{k,\,\sbt}=PX_{k,\,\sbt}$. The $d_0$ face map of $Y_{k,\,\sbt}$ gives a map of bi-simplicial sets
\[d_0\colon Y\longrightarrow X\]
which realizes to our map $\THH(PS_{\sbt} C;M)\longrightarrow \THH(S_{\sbt} C;M)$.

The $\pi_\ast$-Kan condition is invariant under weak equivalence (cf. \cite[IV,4.2-(1)]{goersjardine}) so we can replace $X$ and $Y$ by their abelian group models
\[X^{\oplus}_{\sbt\,\sbt}=\{[k][p]\mapsto\hocolim_{I[k]}\Omega^{\underline{i}}V^{\oplus}_k(HS_p C;HM,\underline{i})\}\ \ , \ \ Y^{\oplus}_{k,\,\sbt}=PX^{\oplus}_{k,\,\sbt}\]
Indeed by \ref{abgpmodel} the canonical maps $X_{k,\,\sbt}\longrightarrow X^{\oplus}_{k,\,\sbt}$ are equivalences for all $k$. In order to prove that $X^{\oplus}$ and $Y^{\oplus}$ satisfy the $\pi_\ast$-Kan condition it is enough, by \cite[IV,4.2-(2)]{goersjardine} to show that $X^{\oplus}$ and $Y^{\oplus}$ are levelwise fibrant and levelwise connected. They are levelwise fibrant since they are bisimplicial abelian groups. The simplicial set $Y^{\oplus}_{k,\,\sbt}$ is contractible for every fixed $k$, as
\[Y^{\oplus}_{k,\,\sbt}=PX^{\oplus}_{k,\,\sbt}\simeq X^{\oplus}_{k,0}=\ast\]
and $X^{\oplus}_{k,0}=\ast$, so that the simplicial sets $X^{\oplus}_{k,\,\sbt}$ are reduced and therefore connected.
\end{proof}

\begin{prop}[Bousfield-Friedlander, {\cite[4.9]{goersjardine}}]\label{bous-fried} Let
\[\xymatrix{Z\ar[r]\ar[d]&X\ar[d]^p\\
W\ar[r]&Y
}\]
be a square of bisimplicial sets, which is levelwise homotopy cartesian. If $X$ and $Y$ satisfies the $\pi_\ast$-Kan condition and the map $\pi_0p\colon \pi_0X\longrightarrow \pi_0Y$ is a Kan-fibration, the realization of the square above is homotopy cartesian.
\end{prop}

\begin{rem}
One could (naively) try to prove \ref{THRdeloopings} by making an equivariant version of the proof of \ref{noneqdeloopings}: define actions on the fiber sequences
\[\THH(C;M)\longrightarrow \THH(PS_{\sbt} C;M)\stackrel{d_0}{\longrightarrow} \THH(S_{\sbt} C;M)\]
and
\[\THH(S_{\sbt} C;M)\longrightarrow \THH(PS^{2,1}_{\sbt} C;M)\stackrel{d_0}{\longrightarrow}\THH(S^{2,1}_{\sbt} C;M)\]
so that the restriction to the fixed points stay fiber sequences, and so that the total spaces are $G$-contractible. But such an argument would only give deloopings with trivial action on the circle. What fails is that the path construction $PX$ on a real space $X$ is no longer a real space since the last face map corresponds under the action to the first one, which as been removed. An attempt to fix this is to first subdivide the fiber sequences, in order to obtain simplicial actions, and apply the $P$-construction afterwards. This will give equivariant fiber sequences, but it will in general change the homotopy type of the fiber. The suggested construction works out for the second sequence, giving a $G$-equivalence
\[\THR(S^{1,1}_{\sbt} C;M)\stackrel{\simeq}{\longrightarrow} \Omega\THR(S^{2,1}_{\sbt} C;M)\]
However, it is not obvious at all how to define a $G$-map $\THR(S^{1,1}_{\sbt}C;M)\longrightarrow \THR(Psd_eS^{2,1}_{\sbt} C;M)$. We will show that the canonical map
\[\THR(C;M)\stackrel{\simeq}{\longrightarrow}\Omega^{1,1}\THR(S^{1,1}_{\sbt} C;M)\] is also a $G$-equivalence, but by a different line of reasoning.
\end{rem}

\begin{defn}
A sequence of $G$-spaces and equivariant maps $F\longrightarrow Y\longrightarrow X$ is called a \textbf{$G$-fiber sequence} if it is a fiber sequence, and its restriction to the fixed points $F^G\longrightarrow Y^G\longrightarrow X^G$ is also a fiber sequence.
\end{defn}

Recall the subdivision functor from \ref{subdivision} that takes a real space to a simplicial $G$-space. This construction works just as well for real categories. Given a simplicial category $C_{\sbt}$, we denote $sd_e C_{\sbt}$ the simplicial category with $k$-simplicies
\[(sd_e C_{\sbt})_{k}=C_{2k+1}\]
and faces and degeneracies defined in degree $k$ by $\overline{d}_l=d_ld_{2k+1-l}$ and $\overline{s}_l=s_ls_{2k+1-l}$. If $C_{\sbt}$ is a real category (cf. \ref{defrealcat}), the induced levelwise duality on $sd_e C_{\sbt}$ is simplicial. Moreover there is a natural $G$-equivariant homeomorphism $|sd_e C_{\sbt}|\cong |C_{\sbt}|$.

\begin{prop}
There is a $G$-fiber sequence
\[\THR(sd_esd_eS^{1,1}_{\sbt} C;M)\longrightarrow\THR(Psd_eS^{2,1}_{\sbt} C;M)\stackrel{\overline{d}_0}{\longrightarrow} \THR(sd_eS^{2,1}_{\sbt} C;M)\]
with $G$-contractible total space $\THR(Psd_eS^{2,1}_{\sbt} C;M)$. In particular this gives a $G$-equivalence \[\THR(S^{1,1}_{\sbt} C;M)\stackrel{\simeq}{\longrightarrow} \Omega\THR(S^{2,1}_{\sbt} C;M)\]
\end{prop}

\begin{proof} We remove the bimodules from the notation for simplicity. We are going to define a sequence of functors
\[sd_esd_eS^{1,1}_{\sbt} C\longrightarrow Psd_eS^{2,1}_{\sbt} C\stackrel{\overline{d}_0}{\longrightarrow} sd_eS^{2,1}_{\sbt} C\]
where the simplicial map $\overline{d}_0$ is given by the zero face of $sd_eS^{2,1}_{\sbt} C$, the one we removed in the definition of $P sd_eS^{2,1}_{\sbt} C$. Explicitly, in degree $k$ it is given by
\[\xymatrix{\overline{d}_0\colon S^{2,1}_{2k+3} C\ar[rr]^-{d_0d_{2k+3}}&& S^{2,1}_{2k+1} C}\]
Notice that since the duality on $sd_eS^{2,1}_{\sbt} C$ is simplicial, the simplicial category $P sd_eS^{2,1}_{\sbt} C$ also has a simplicial duality, and the  map $\overline{d}_0$ is equivariant. A similar construction gives the map on the bimodules.

Before defining the map $sd_esd_eS^{1,1}_{\sbt} C\longrightarrow Psd_eS^{2,1}_{\sbt} C$,
we want to identify the levelwise homotopy fiber of $\THR(S^{2,1}_{2k+3} C)\stackrel{d_0d_{2k+3}}{\longrightarrow}  \THR(S^{2,1}_{2k+1} C)$ in order to explain why this strange double subdivision comes out.  The $G$-equivalences of \ref{sumdiagsTHR} give a commutative diagram of $G$-maps
\[\xymatrix{\THR(S^{2,1}_{2k+3} C)\ar[d]^{\simeq}\ar[r]^{d_0d_{2k+3}}&  \THR(S^{2,1}_{2k+1} C)\ar[d]^{\simeq}\\
\map_\ast(S^{2,1}_{2k+3},\THR(C))\ar[r]_{d_0d_{2k+3}}&\map_\ast(S^{2,1}_{2k+1},\THR(C))
}\]
Where the bottom map is precomposition by the inclusion $S^{2,1}_{2k+1}\longrightarrow S^{2,1}_{2k+3}$ induced by $\sigma^{2k+1}\sigma^0\colon[2k+3]\longrightarrow [2k+1]$. Notice that in particular the bottom map of the diagram is a $G$-equivariant fibration. Let us describe its fiber. In general for a pointed $G$-space $X$, the fiber of the map
\[\map_\ast(S^{2,1}_{2k+3},X)\stackrel{d_0d_{2k+3}}{\longrightarrow}\map_\ast(S^{2,1}_{2k+1},X)\]
over the basepoint is the space $F_k$ of maps $f\colon S^{2,1}_{2k+3}\longrightarrow X$ whose values at a
\[\theta=(0^i1^{2k+4-i-l}2^l)\colon [2k+3]\longrightarrow [2]\]
are $f(\theta)=\ast$ if $i>1$ or $l>1$. The map $\kappa\colon \map_\ast(S^{1,1}_{4k+3},X)\longrightarrow F_k$ that send $g$ to
\[\kappa(g)(0^i1^j2^{2k+4-i-j})=\left\{\begin{array}{lll}g(0^j1^{4k+4-j})&\mbox{ if } i=1\\
g(0^{4k+4-j}1^j)&\mbox{ if } 2k+4-i-j=1\\
\ast&\mbox{ otherwise }
\end{array}\right.\]
is an equivariant homeomorphism. Therefore we have a commutative diagram
\[\xymatrix{&\THR(S^{2,1}_{2k+3} C)\ar[d]^{\simeq}\ar[r]^{d_0d_{2k+3}}&  \THR(S^{2,1}_{2k+1} C)\ar[d]^{\simeq}\\
\THR(C^{S^{1,1}_{4k+3}\backslash\ast})\ar[d]^{\simeq}\ar[r]&\THR(C^{S^{2,1}_{2k+3}\backslash\ast})\ar[d]^{\simeq}\ar[r]_{d_0d_{2k+3}}&\THR(C^{S^{2,1}_{2k+3}\backslash\ast})\ar[d]^{\simeq}\\
\map_\ast\!(S^{1,1}_{4k\!+\!3},\!\THR(C)\!)\!\ar[r]&\map_\ast\!(S^{2,1}_{2k\!+\!3},\!\THR(C)\!)\!\ar[r]_{d_0d_{2k\!+\!3}}&\map_\ast\!(S^{2,1}_{2k\!+\!1},\!\THR(C)\!)\!
}\]
where the bottom row is a $G$-fibration, and the map $\THR(C^{S^{1,1}_{4k+3}\backslash\ast})\longrightarrow\THR(C^{S^{2,1}_{2k+3}\backslash\ast})$ is $\THR$ of the functor
\[C^{S^{1,1}_{4k+3}\backslash\ast}\longrightarrow C^{S^{2,1}_{2k+3}\backslash\ast}\]
defined in a completely analogous way to the function $\kappa$ above. By \ref{sumdiagsTHR}, the obvious model for
the levelwise homotopy fiber of $\THR(S^{2,1}_{2k+3} C)\stackrel{d_0d_{2k+3}}{\longrightarrow}  \THR(S^{2,1}_{2k+1} C)$ is $\THR(S^{1,1}_{4k+3}C)$. The point is that it should define a simplicial space in the $k$-direction, and map simplicially into $\THR(Psd_eS^{2,1}_{\sbt} C)$. To give it a simplicial structure, notice that $4k+3=2(2k+1)+1$, meaning that $S^{1,1}_{4k+3} C$ is the degree $k$ of the double subdivision
\[S^{1,1}_{4k+3} C=(sd_esd_e S^{1,1}_{\sbt}C)_k\]
It remains to define a simplicial functor $sd_esd_e S^{1,1}_{\sbt}C\longrightarrow Psd_eS^{2,1}_{\sbt} C$.
It is defined levelwise by the functor $I_k\colon S^{1,1}_{4k+3} C\longrightarrow S^{2,1}_{2k+3} C$ that sends a diagram $X\colon Cat([1],[4k+3])\longrightarrow C$ to
\[I_k(X)_{i<j<l}=\left\{\begin{array}{llll}X_{j-1<l-1}&\mbox{ if }i=0,l\neq 2k+3\\
X_{i+2k+1<j+2k+1}&\mbox{ if }i\neq 0,l= 2k+3\\
X_{j-1<j+2k+1} &\mbox{ if }i= 0,l= 2k+3\\
0&\mbox{ otherwise }
\end{array}\right.\]
For an illustration of this functor see \ref{picIk} below.
Les us check that $I_k(X)$ is well defined. For any $\psi=(i\!<\!j\!<\!l\!<\!h)\in Cat([3],[2k+3])$, we need to show that the sequence
\[I_k(X)_{i<j<l}\longrightarrow I_k(X)_{i<j<h}\longrightarrow I_k(X)_{i<l<h}\longrightarrow I_k(X)_{j<l<h}\]
is exact. If both $i$ and $h$ are non-zero, the sequence is $0$, which is exact. If $i=0$ and $h\neq 2k+3$, the sequence is
\[X_{j-1<l-1}\longrightarrow X_{j-1<h-1}\longrightarrow X_{l-1<h-1}\longrightarrow 0\]
which is exact since it is the $3$-terms exact sequence for $X\in S^{1,1}_{4k+3}C$ associated to $\theta=(j\!-\!1\!<\!l\!-\!1\!<\!h\!-\!1)\in Cat([2],[4k+3])$. A similar argument shows that the sequence is exact for $i\neq 0$ and $h=2k+3$. The last case, the interesting one, is when $i= 0$ and $h=2k+3$. In this case the middle map of the sequence factors as
\[\xymatrix{X_{j-1<l-1}\ar[r]&X_{j-1<j+2k+1}\ar[d]\ar[r]&X_{l-1<l+2k+1}\ar[r]&X_{j+2k+1<l+2k+1}\\
&X_{l-1<j+2k+1}\ar[ur]
}\]
The sequence is therefore exact since the two sequences \[X_{j-1<l-1}\longrightarrow X_{j-1<j+2k+1}\longrightarrow X_{l-1<j+2k+1}\] and \[X_{l-1<j+2k+1}\longrightarrow X_{l-1<l+2k+1}\longrightarrow X_{j+2k+1<l+2k+1}\] are the sequences for $X$ associated respectively to $(j\!-\!1\!<\!l\!-\!1\!<\!j\!+\!2k\!+\!1)$ and $(l\!-\!1\!<\!j\!+\!2k\!+\!1\!<\!l\!+\!2k\!+\!1)$ in $Cat([2],[4k+3])$, and are therefore exact.
One can easily see that $I_k$ commutes with the dualities, and more painfully that it is a simplicial functor. Moreover the diagram
\[\xymatrix{S^{1,1}_{4k+3}C\ar[r]^{I_k}\ar[d]^d&S^{2,1}_{2k+3}C\ar[d]^d\\
C^{S^{1,1}_{4k+3}\backslash\ast}\ar[r]^{\kappa}& C^{S^{2,1}_{2k+3}\backslash\ast}
}\]
commutes up to equivariant natural isomorphism.
This construction gives a commutative diagram
\[\resizebox{1\hsize}{!}{\xymatrix{\THR(S^{1,1}_{4k\!+\!3}C)\ar[r]\ar[d]^{\simeq}&\THR(S^{2,1}_{2k\!+\!3} C)\ar[d]^{\simeq}\ar[r]^{d_0d_{2k\!+\!3}}&  \THR(S^{2,1}_{2k\!+\!1} C)\ar[d]^{\simeq}\\
\map_\ast\!(S^{1,1}_{4k\!+\!3},\!\THR(C)\!)\ar[r]&\map_\ast\!(S^{2,1}_{2k\!+\!3},\!\THR(C)\!)\ar[r]_{d_0d_{2k\!+\!3}}&\map_\ast\!(S^{2,1}_{2k\!+\!1},\!\THR(C)\!)
}}\]
Since the bottom row is a fibration when restricted to the fixed points, the sequence of simplicial spaces
\[\THR_{\sbt}(sd_esd_eS^{1,1}_{\sbt}C)^G\longrightarrow\THR_{\sbt}(Psd_eS^{2,1}_{\sbt} C)^G\longrightarrow\THR_{\sbt}(sd_eS^{2,1}_{\sbt} C)^G\]
is a levelwise fiber sequence. Lemma \ref{piastbusiness} holds for the realization of this map as well, since the abelian group model \ref{abgpmodel} for $\THR^G$ gives a levelwise fibrant bisimplicial set, and since $(sd_eS^{2,1}_{\sbt})_0 C=S^{2,1}_{1}C=0$.
Therefore Bousfield-Friedlander \ref{bous-fried} applies and it shows that the realization
\[\THR(sd_esd_eS^{1,1}_{\sbt}C)^G\longrightarrow\THR(Psd_eS^{2,1}_{\sbt} C)^G\longrightarrow\THR(sd_eS^{2,1}_{\sbt} C)^G\]
is a fiber sequence. The standard simplicial contraction $Psd_eS^{2,1}_{\sbt} C\simeq\ast$ respects the $G$-action, and therefore it induces a $G$-contraction of the total space $\THR(Psd_eS^{2,1}_{\sbt} C)$. Thus the fixed points $\THR(Psd_eS^{2,1}_{\sbt} C)^G$ are contractible, and the contraction gives an equivalence
\[\THR(sd_esd_eS^{1,1}_{\sbt}C)^G\stackrel{\simeq}{\longrightarrow}\Omega(\THR(sd_eS^{2,1}_{\sbt} C)^G)=\Omega(\THR(sd_eS^{2,1}_{\sbt} C))^G\]
where the first and the last space are respectively canonically homeomorphic to $\THR(S^{1,1}_{\sbt}C)^G$ and $\Omega(\THR(S^{2,1}_{\sbt} C))^G$.
\end{proof}

\begin{ex}\label{picIk} Here's a hopefully illuminating illustration of the functor $I_k\colon (sd_esd_eS^{1,1} C)_{k}\longrightarrow (Psd_eS^{2,1} C)_{k}$, for $k=1$. The functor $I_1$ remembers only the highlighted parts of the diagram $X\in S^{1,1}_7$
\begin{center}
\includegraphics[scale=0.75]{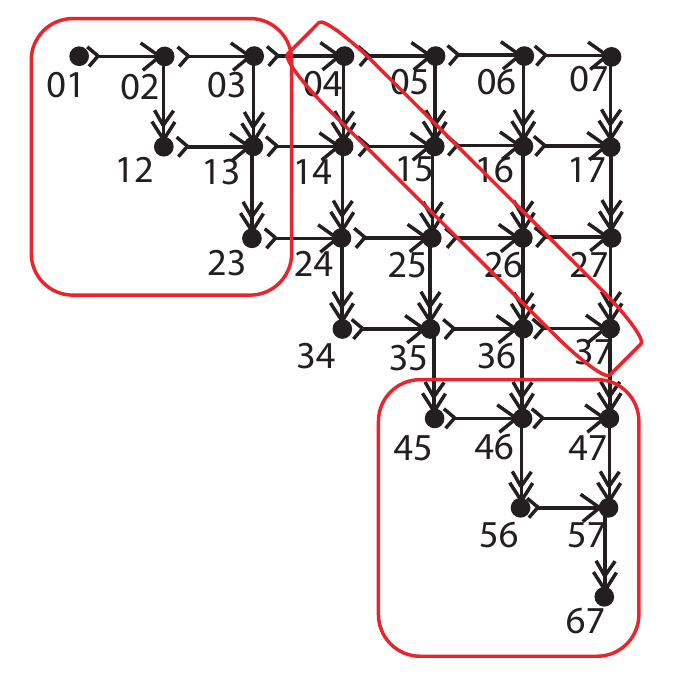}
\end{center}
Then it embeds the two triangles in the bottom and top face of the tetrahedron, and the diagonal stripe is placed on the edge formed by the intersection of the two faces
\vspace{-.3cm}
\begin{center}
\includegraphics[scale=0.7]{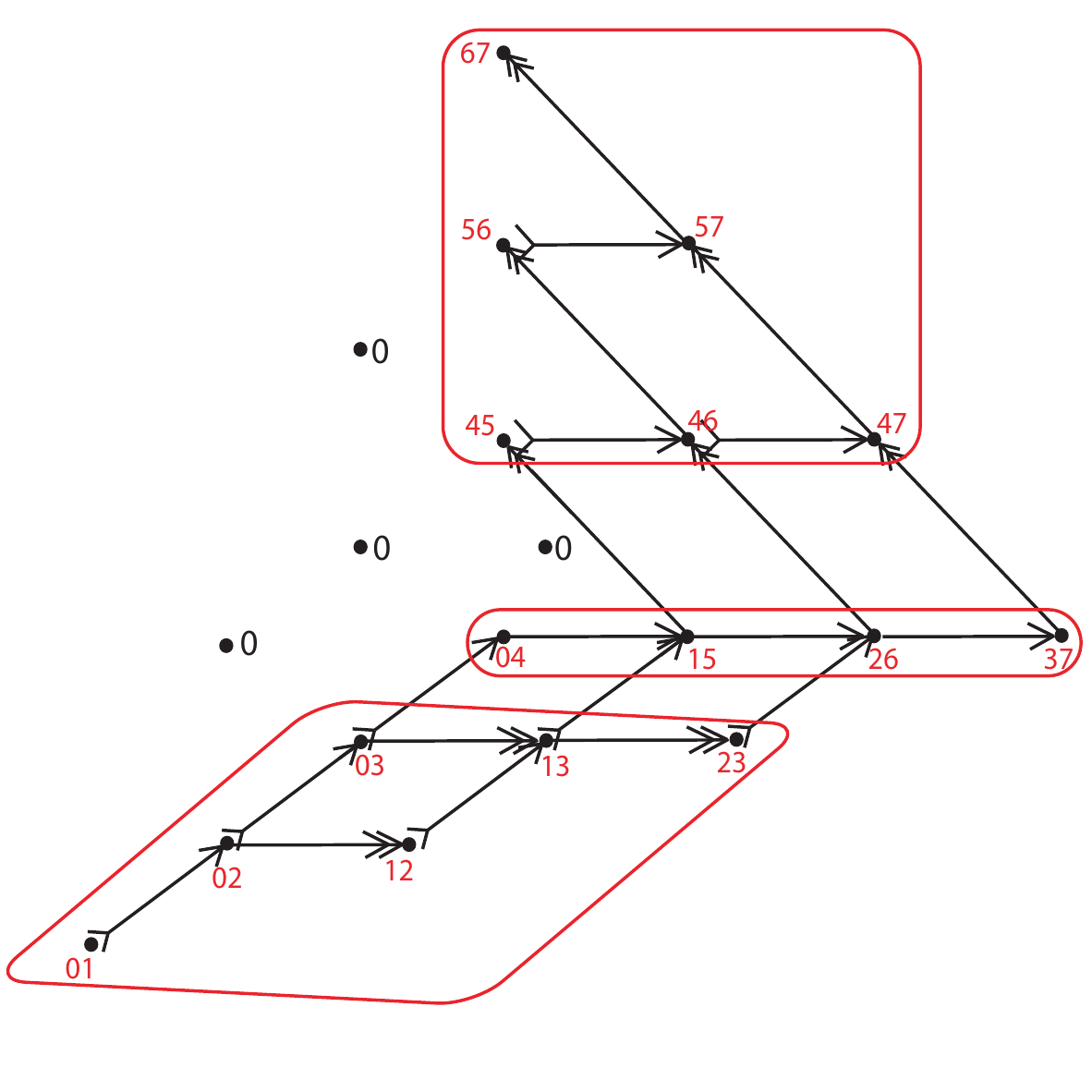}
\end{center}
It is clear from the picture that $I_1$ is equivariant, and maybe even that $I_k$ respects the simplicial structure.
\end{ex}

Theorem \ref{THRdeloopings} follows from the previous proposition and the next one.

\begin{prop}\label{THRS11}
The map $\THR(C;M)\longrightarrow\Omega^{1,1}\THR(S^{1,1}_{\sbt} C;M)$ induced by the adjoint of the projection map
\[\THR(C;M)\times\Delta^1=\THR(S^{1,1}_1C;M_1)\times\Delta^1\longrightarrow\THR(S^{1,1}_{\sbt} C;M)\]
is a $G$-equivalence.
\end{prop}

The strategy for the proof of this proposition is similar to the one given in the appendix \ref{Gbarconstructions} for proving that $M\longrightarrow\Omega^{1,1}M(S^{1,1}_{\sbt})$ is a $G$-equivalence for a group-like monoid $M$ with anti-involution. 

We factor the map of the statement through an intermediate $G$-space $\THR(P_+P_-S^{1,1}_{\sbt} C;M)$, and we prove that the two maps are $G$-equivalences.
Consider the simplicial categories $P_+S^{1,1}_{\sbt} C$ and $P_-S^{1,1}_{\sbt} C$, both defined in degree $k$ as $S^{1,1}_{k+1} C$, but with simplicial structure defined by forgetting the $0$-th face and degeneracy for $P_+S^{1,1}_{\sbt} C$, and the last ones for $P_-S^{1,1}_{\sbt} C$ (in our old notation $P=P_+$). The removed faces define simplicial functors $d_0\colon P_+S^{1,1}_{\sbt} C\longrightarrow S^{1,1}_{\sbt} C$ and $d_{L}\colon P_-S^{1,1}_{\sbt} C\longrightarrow S^{1,1}_{\sbt} C$.

\begin{rem} The duality on $S^{1,1}_{\sbt} C$ does not induce a real structure on $P_+S^{1,1}_{\sbt} C$ nor on $P_-S^{1,1}_{\sbt} C$. However, it induces mutually inverse functors \[D_+\colon P_+S^{1,1}_{\sbt} C\longrightarrow P_-S^{1,1}_{\sbt} C  \ \ \mbox{  and  }  \ \ D_-\colon P_-S^{1,1}_{\sbt} C\longrightarrow P_+S^{1,1}_{\sbt} C\] reversing the order of faces and degeneracies.
\end{rem}

We combine these two constructions into a simplicial category $P_+P_-S^{1,1}_{\sbt} C$, the double shift of $S^{1,1}_{\sbt} C$. Its $k$-simplicies are
\[P_+P_-S^{1,1}_k C=S^{1,1}_{k+2} C\]
and the simplicial structure forgets both the first and the last faces and degeneracies from $S^{1,1}_{\sbt} C$. Since the two faces we removed were interchanged, and the same for the degeneracies, the real structure of $S^{1,1}_{\sbt} C$ induces a real structure on $P_+P_-S^{1,1}_k C$.

By considering $C$ a constant simplicial category, there is a duality preserving simplicial functor
\[\iota\colon C\stackrel{\iota_0}{\longrightarrow}S^{1,1}_3C=(sd_eP_+P_-S^{1,1}_{\sbt} C)_0\stackrel{s}{\longrightarrow} sd_eP_+P_-S^{1,1}_{\sbt} C\]
where $s$ denotes the degeneracy functor and $\iota_0$ sends an object $c$ to
\[\xymatrix{
0\ar[r]&c\ar@{=}[r]\ar@{=}[d]&c\ar@{=}[d]\\
&c\ar@{=}[r]&c\ar[d]\\
&&0
}\]

\begin{rem}
One could think of $S_{\sbt} C$ as a Bar construction $BG=B(\ast;G;\ast)$ on $C$ solving the coherence problem of non-strict associativity of the sum of $C$. Then $P_+S_{\sbt} C$ and $P_-S_{\sbt} C$ should be thought of as the one sided Bar constructions $B(G;G;\ast)$ and $B(\ast;G;G)$ respectively. With this picture in mind, the double construction $P_+P_-S^{1,1}_{\sbt} C$ corresponds to a two sided Bar construction $B(G;G;G)$. The map $\iota$ above is a subdivided version of the inclusion $G\longrightarrow B(G;G;G)$, given in degree zero by the inclusion in the first summand, to make the action fit. The next lemma is analogous to the standard result that $G\longrightarrow B(G;G;G)$ is a homotopy equivalence (see \ref{Gbarconstructions} for the actual proof for monoids with action).
\end{rem}

\begin{lemma}
The inclusion $\iota\colon C\longrightarrow sd_eP_+P_-S^{1,1}_{\sbt} C$ induces a simplicial $G$-homotopy equivalence
\[\THR(C;M)\longrightarrow\THR(sd_eP_+P_-S^{1,1}_{\sbt} C;M)\]
\end{lemma}

\begin{proof}
In general $P_+$ and $P_-$ of a simplicial category are both simplicially homotopy equivalent to the $0$-simplicies of the category, via respectively
\[d_{0}^{\,\sbt+1}\colon P_+C_{\sbt}\longrightarrow C_0\ \ \mbox{and} \ \ d_{L}^{\,\sbt+1}\colon P_-C_{\sbt}\longrightarrow C_0\] 
where $d_{L}$ denotes the last face in every simplicial degree.
In our situation $P_+P_-S^{1,1}_{\sbt} C\simeq (P_-S^{1,1}_{\sbt} C)_0=S^{1,1}_1 C=C$. To check that the equivalence respects the $G$-action we write down the maps explicitly.

We define a simplicial retraction for $\iota$.
Denote by $\overline{d}_l$ and $\overline{s}_l$ the faces and degeneracies of $sd_eS^{1,1}_{\sbt} C$.
The functor
\[r=\overline{d}_{0}^{k+1}\colon (sd_eP_-P_+S^{1,1}_{\sbt} C)_k=(P_+sd_eS^{1,1}_{\sbt} C)_k\longrightarrow (sd_eS^{1,1}_{\sbt} C)_0=C\]
defines an equivariant retraction for $\iota$.
For every $\sigma=(0^b1^{k+1-b})\in\Delta[1]_k$ define a functor \[H_\sigma\colon sd_eP_-P_+S^{1,1}_kC\longrightarrow sd_eP_-P_+S^{1,1}_kC\] by the composite
\[H_\sigma=\overline{s}_{0}^{b}\overline{d}_{0}^{b}\]
This defines a simplicial homotopy between the identity and $\iota\circ r$. Since the action on the subdivision $sd_eS^{1,1}_{\sbt} C$ is simplicial, this is a homotopy of duality preserving functors. Applying $\THR$ to each of these functors one gets the desired simplicial homotopy, showing that $\iota$ and $r$ induce mutually inverse $G$-homotopy equivalences.
\end{proof}

We compare \[\THR(sd_eP_{\!+\!}P_{\!-\!}S^{1,1}_{\sbt}C;\!M)=|sd_e\THR(P_{\!+\!}P_{\!-\!}S^{1,1}_{\sbt}C;\!M)|\cong\THR(P_{\!+\!}P_{\!-\!}S^{1,1}_{\sbt}C;\!M)\]
 with the loop space $\Omega^{1,1}\THR(S^{1,1}_{\sbt} C;M)$ by finding a new model for the loop space.
Recall that the levelwise action on $\THR(S^{1,1}_{\sbt} C;M)$ induces a commutative diagram
\[\xymatrix{\THR(P_+S^{1,1}_{\sbt} C;M)\ar@<1ex>[r]^{D_+}\ar[d]_{d_0}&\THR(P_-S^{1,1}_{\sbt} C;M)\ar@<1ex>[l]^{D_-}\ar[d]^{d_{L}}\\
\THR(S^{1,1}_{\sbt} C;M)\ar[r]^{D}&\THR(S^{1,1}_{\sbt} C;M)
}\]
where $d_L$ denotes the iterated last face.
Since the functors $D_+$ and $D_-$ reverse the simplicial structure one needs to reverse the order of the simplex coordinate in order to have well defined maps on the realization.
To simplify the notation write $X_k=\THR(S^{1,1}_k C;M)$, so that 
\[|P_+X_{\sbt}|=P_+\THR(S^{1,1}_{\sbt} C;M)=\THR(P_+S^{1,1}_{\sbt} C;M)\] and similarly for $P_-$.
In this general situation, the homotopy pullback $|P_+X_{\sbt}|\stackrel{h}{\times}_{|X_{\sbt}|}|P_-X_{\sbt}|$ of the diagram
\[\xymatrix{ &|P_+X_{\sbt}|\ar[d]^{d_0}\\
|P_-X_{\sbt}|\ar[r]_{d_{L}}& |X_{\sbt}|
}\]
admits a $G$-action (cf. appendix \ref{Gbarconstructions}). The action sends a triple $(y,\gamma,y')$ of a $y\in |P_+X_{\sbt}|$, a $y'\in |P_-X_{\sbt}|$ and a path $\gamma\colon I\longrightarrow |X_{\sbt}|$ to
\[(D_-y',D\overline{\gamma},D_+y)\]
where $D\overline{\gamma}$ denotes the path obtained by applying the duality pointwise to the backwards path.
Notice that $P_+X_{\sbt}$ is equivalent as a simplicial space to $X_0=\THR(\ast)=\ast$. Therefore by lemma \ref{modelloop11} below, there is a $G$-homotopy equivalence
\[\Omega^{1,1}\THR(S^{1,1}_{\sbt} C;M)\stackrel{\simeq}{\longrightarrow}|P_+X_{\sbt}|\stackrel{h}{\times}_{|X_{\sbt}|}|P_-X_{\sbt}|\]
Hence it suffices to compare $\THR(P_+P_-S^{1,1}_{\sbt} C;M)$ with $|P_+X_{\sbt}|\stackrel{h}{\times}_{|X_{\sbt}|}|P_-X_{\sbt}|$.
Applying $\THR$ to the commutative diagram of simplicial functors
\[\xymatrix{P_+P_-S^{1,1}_{\sbt} C\ar[r]^-{d_{L}}\ar[d]^{d_0}&P_+S^{1,1}_{\sbt} C\ar[d]^{d_0}\\
P_-S^{1,1}_{\sbt} C\ar[r]_-{d_{L}}&S^{1,1}_{\sbt} C
}\]
one gets an equivariant map $\THR(P_+P_-S^{1,1}_{\sbt} C;M)\longrightarrow |P_+X_{\sbt}|\stackrel{h}{\times}_{|X_{\sbt}|}|P_-X_{\sbt}|$ (factoring through the actual pullback).
Proposition \ref{THRS11} will therefore follow from the next lemma.

\begin{lemma}
The natural map
\[\THR(P_+P_-S^{1,1}_{\sbt} C;M)\longrightarrow|P_+X_{\sbt}|\stackrel{h}{\times}_{|X_{\sbt}|}|P_-X_{\sbt}|\]
is a $G$-equivalence.
\end{lemma}

\begin{proof}
First we prove that at for each $k$ the space $\THR(P_+P_-S^{1,1}_k C;M)$ is equivalent to the levelwise homotopy pullback $X_{k+1}\stackrel{h}{\times}_{X_k}X_{k+1}$, where we remember the notation $X_k=\THR(S^{1,1}_kC;M)$. Then we will use the fibration properties of $\THR$ to show that this homotopy pullback commutes (equivariantly!) with the realization.

First non-equivariantly, by \ref{sumdiagsTHR}, there is an equivalence between the following two diagrams
\[\xymatrix{X_{k+1}\ar[d]^{\simeq}\ar[r]^{d_0}&X_k\ar[d]^{\simeq}&X_{k+1}\ar[d]^{\simeq}\ar[l]_{d_{k+1}}\\
\THR(C;M)^{k+1}\ar[r]^{d_0}&\THR(C;M)^{k}&\THR(C;M)^{k+1}\ar[l]_{d_{k+1}}
}\]
where the maps on the bottom row are the trivial fibrations that project off the first and the last component, respectively. Therefore the homotopy limits of the two horizontal lines are naturally equivalent. Since the bottom maps are trivial fibrations, the homotopy pullback is equivalent to the pullback, which is $\THR(C;M)^{k+2}$. All together there is a commutative diagram
\[\xymatrix{\THR(P_+P_-S^{1,1}_k C;M)\ar[dr]_{\simeq}\ar[r]& X_{k+1}\stackrel{h}{\times}_{X_k}X_{k+1}\ar[d]^{\simeq}\\
& \THR(C;M)^{k+2}
}\]
where the diagonal map is once more the equivalence of \ref{sumdiagsTHR}. The key point for showing that this is an equivariant equivalence is to use the homeomorphism
\[(X_{k+1}\stackrel{h}{\times}_{X_k}X_{k+1})^G\cong X_{k+1}\stackrel{h}{\times}_{X_k}X_{k}^G\]
of \ref{fixedptshopb}. Again by \ref{sumdiagsTHR} there is an equivalence of diagrams
\[\xymatrix{X_{k+1}\ar[d]^{\simeq}\ar[r]^{d_0}&X_k\ar[d]^{\simeq}&X_{k}^G\ar[d]^{\simeq}\ar[l]\\
\THR(C;M)^{k+1}\ar[r]^{d_0}&\THR(C;M)^{k}&(\THR(C;M)^{k})^G\ar[l]
}\]
and hence an equivalence from $(X_{k+1}\stackrel{h}{\times}_{X_k}X_{k+1})^G$ to the homotopy pullback of the bottom line, which is equivalent to the actual pullback since $d_0$ is the (trivial) fibration. Separating the cases where $k$ is even or odd, one can easily see that there is a homeomorphism
\[\THR(C;M)\times (THR(C;M)^k)^G\cong (THR(C;M)^{k+2})^G\]
This gives a commutative diagram
\[\xymatrix{\THR(P_+P_-S^{1,1}_k C;M)^G\ar[dr]_{\simeq}\ar[r]& (X_{k+1}\stackrel{h}{\times}_{X_k}X_{k+1})^G\ar[d]^{\simeq}\\
& (\THR(C;M)^{k+2})^G
}\]
where the diagonal map is the equivalence of \ref{sumdiagsTHR}.

It remains to show that the realization of the levelwise pullback is equivalent to the pullback of the realization, that is that the canonical map
\[|X_{k+1}\stackrel{h}{\times}_{X_k}X_{k+1}|\longrightarrow |P_+X_{\sbt}|\stackrel{h}{\times}_{|X_{\sbt}|}|P_-X_{\sbt}|\]
is a $G$-equivalence.
Non-equivariantly, it follows by the Bousfield-Friedlander theorem \ref{bous-fried} that applies because of \ref{piastbusiness}. For the fixed points, there are homeomorphisms
\[\resizebox{1\hsize}{!}{$|X_{k+1}\!\stackrel{h}{\times}_{X_k}\!X_{k+1}|^G\!\cong\!|sd_e(P_+X_{\sbt}\!\stackrel{h}{\times}_{X_{\sbt}}\!P_-X_{\sbt})_{k}^G|\!\cong\!|sd_e(P_+X_{\sbt})_{k}\!\stackrel{h}{\times}_{sd_e(X_{\sbt})_k}\!sd_e(X_{\sbt})_{k}^G|$}\]
and by lemma \ref{piastbusiness}, Bousfield-Friedlander applies.
 This gives an equivalence
\[|sd_e(P_+X_{\sbt})_{k}\stackrel{h}{\times}_{sd_e(X_{\sbt})_k}sd_e(X_{\sbt})_{k}^G|\simeq|sd_eP_+X_{\sbt}|\stackrel{h}{\times}_{|sd_eX_{\sbt}|}|sd_eX_{\sbt}^G|\]
and the right hand side is homeomorphic to
\[|sd_eP_+X_{\sbt}|\stackrel{h}{\times}_{|sd_eX_{\sbt}|}|sd_eX_{\sbt}^G|\cong |P_+X_{\sbt}|\stackrel{h}{\times}_{|X_{\sbt}|}|X_{\sbt}|^G\cong (|P_+X_{\sbt}|\stackrel{h}{\times}_{|X_{\sbt}|}|P_-X_{\sbt}|)^G\]
again by \ref{fixedptshopb}.
\end{proof}


\subsection{Stable $0$-simplicies of $\THR$}
We prove that $\THR$ of an exact category is $G$-equivalent to a stabilization of its zero simplicies, an equivariant version of \cite[2.2.3]{ringfctrs}. Our proof relies heavily on the non-equivariant proof, and might hopefully be improved.

Let $(C,D,\mathcal{E})$ be an exact category with strict duality, and $M\colon C^{op}\otimes C\longrightarrow Ab$ a bimodule with duality $J$. Recall that for simplicity we denote $\THR(C;M):=\THR(HC;HM)$.

\begin{defn}
An abelian group $M$ is called \textbf{uniquely $2$-divisible} if the group homomorphism
\[M\stackrel{\Delta}{\longrightarrow}M\oplus M\stackrel{+}{\longrightarrow} M\]
sending $x$ to $x+x$ is an isomorphism, i.e. if $M$ is a $\mathbb{Z}[1/2]$-module.

A bimodule $M\colon C^{op}\otimes C\longrightarrow Ab$ is called uniquely $2$-divisible if it takes values in uniquely $2$-divisible groups.
\end{defn}

Our main example of a divisible bimodule is the following. Suppose that $M$ is a bimodule over a ring $A$, and that $2\in A$ is invertible. Then the bimodule
\[M(-,-)=\hom_A(-,-\otimes_A M)\colon \mathcal{P}_{A}^{op}\otimes \mathcal{P}_A\longrightarrow Ab\]
is uniquely $2$-divisible. The inverse for the homomorphism above is given by multiplying pointwise by the inverse of $2$. Similarly, if $J$ is an $M$-twisting of an antistructure $(A,L,\alpha)$, the bimodule
\[M^J\colon \mathcal{D}\mathcal{P}_{A}^{op}\otimes \mathcal{D}\mathcal{P}_A\longrightarrow Ab\]
of §\ref{THRrings} is uniquely $2$-divisible.

\begin{prop}\label{stable0simpl}
Let $M\colon C^{op}\otimes C\longrightarrow Ab$ be a uniquely $2$-divisible bimodule with duality over $(C,D,\mathcal{E})$. Then the inclusion of the $0$-simplicies induces a $G$-equivalence
\[\colim_m\Omega^{2m,m}\THR_0((S^{2,1}_{\sbt})^{(m)}C;M)\longrightarrow \colim_m\Omega^{2m,m}\THR((S^{2,1}_{\sbt})^{(m)}C;M)\]
where the stabilization maps are adjoint to
\[X(C;M)\times \Delta^2=X(S^{2,1}_2C;M)\times \Delta^2\longrightarrow |X(S^{2,1}_{\sbt} C;M)|\]
for $X=\THR_0,\THR$.
\end{prop}

\begin{rem}\label{descr0simpl}
If $(C,\mathcal{E})$ is split-exact then by \ref{THRdeloopings} the target of the map of the statement is $G$-equivalent to $\THR(C;M)$. The proposition above gives a $G$-equivalence
\[\THR(C;M)\simeq \colim_m\Omega^{2m,m}|\bigoplus\limits_{X\in Ob (S^{2,1}_{\sbt})^{(m)} C}M(X,X)|\]
where direct sum has the involution induced by the maps $J\colon M(X,X)\longrightarrow M(DX,DX)$ as defined in \ref{wedgesintoproducts}. 
Indeed, there is a diagram of natural $G$-equivalences
\[\resizebox{1\hsize}{!}{\xymatrix{\colim\limits_m\!\Omega^{2m,m}|\!\!\!\!\!\bigoplus\limits_{X\in Ob (S^{2,1}_{\sbt})^{(m)} C}\!\!\!\!\!M(X,X)|\ar[d]^{\simeq}_{\ref{Gdeloopingsformonoid}}\\
\colim\limits_m\!\Omega^{2m,m}\THR_{0}^\oplus((S^{2,1}_{\sbt})^{(m)}C;\!M)\ar[r]^-{\simeq}_-{\ref{stable0simpl}}&\colim\limits_m\!\Omega^{2m,m}\THR^\oplus((S^{2,1}_{\sbt})^{(m)}C;\!M)\\
\colim\limits_m\!\Omega^{2m,m}\THR_{0}((S^{2,1}_{\sbt})^{(m)}C;\!M)\ar[u]_-{\simeq}^-{\ref{abgpmodel}}\ar[r]^-{\simeq}_-{\ref{stable0simpl}}&\colim\limits_m\!\Omega^{2m,m}\THR((S^{2,1}_{\sbt})^{(m)}C;\!M)\ar[u]_-{\simeq}^-{\ref{abgpmodel}}\\
& \THR(C;\!M)\ar[u]_-{\simeq}^-{\ref{THRdeloopings}}
}}\]
\end{rem}

We want to use our assumption of divisibility to pass from the non-equivariant statement to the $G$-fixed points. In \cite[V-1.2.1]{Dundasbook}, the authors prove the non-equivariant statement for the $S_{\sbt}$-construction. We give a quick review of their proof to make sure it applies when we replace the $S_{\sbt}$-construction with the $S^{2,1}_{\sbt}$-construction.

Recall from §\ref{abgpmodelsec} that $\THH_{k}^{\oplus}(C;M)$ is the homotopy colimit of the loop space of
\[V^{\oplus}(C;M,\underline{i})=\bigoplus\limits_{c_0,\dots,c_k} M(c_0,c_k)_{i_0}\otimes\bigotimes_{l=1}^k\mathbb{Z}(C(c_{l},c_{l-1})_{i_l})\]
where we denote $A_i=A(S^i)$ for an abelian group $A$.
Define $HH_{k}^{\oplus}(C;M)=V^{\oplus}(C;M,\underline{0})$, that is
\[HH_{k}^{\oplus}(C;M)=\bigoplus\limits_{c_0,\dots,c_k} M(c_0,c_k)\otimes\mathbb{Z}(C(c_1,c_0))\otimes\dots\otimes\mathbb{Z}(C(c_k,c_{k-1}))\]
The simplicial structure on $\THH_{\sbt}^{\oplus}(C;M)$ restricts to a simplicial structure on $HH_{\sbt}^{\oplus}(C;M)$, and the inclusion $V^{\oplus}(C;M,\underline{0})\longrightarrow V^{\oplus}(C;M,\underline{i})$ induces a commutative diagram
\[\xymatrix{HH_{\sbt}^{\oplus}(C;M)\ar[r]^-{\simeq}&\THH_{\sbt}^{\oplus}(C;M)\\
HH_{0}^{\oplus}(C;M)\ar[u]^s\ar[r]_-{\simeq}&\THH_{0}^{\oplus}(C;M)\ar[u]_s
}\]
where the zero simplicies are constant simplicial spaces.
The bottom horizontal map is an equivalence, as it is the direct sum of the homotopy equivalences $M(c,c)\longrightarrow \Omega^i M(c,c)(S^i)$. The top horizontal map is an equivalence by \cite[IV-2.4.1]{Dundasbook}. The non-equivariant version of \ref{stable0simpl} then follows if we can prove that the degeneracy map induces a homotopy equivalence
\[s\colon\colim\limits_m\!\Omega^{2m,m}HH_{0}^{\oplus}((S^{2,1}_{\sbt})^{(m)} C;\!M)\longrightarrow \colim\limits_m\!\Omega^{2m,m}HH_{k}^{\oplus}((S^{2,1}_{\sbt})^{(m)} C;\!M)\]
for every $k$.
A homotopy inverse for this map is induced by the iterated $0$-face $d_{0}^k$. Clearly $d_{0}^k\circ s$ is the identity. In order to define a homotopy between the other composite and the identity, one defines maps
\[E,W\colon HH_{k}^{\oplus}(C;M)\longrightarrow HH_{k}^{\oplus}(S_2C;M_2)\]
such that 
\[\begin{array}{llll}d_0E=\id\\
d_2W=d_{0}^k\circ s\\
d_2E=0=d_0W\\
d_1E=d_1W
\end{array}\]
as in \cite[I-3.3.1]{Dundasbook}, where $d_i\colon S_2C\longrightarrow C$ are the face functors.
If one can show that $d_1$ is homotopic to $d_0+d_2$ as maps
\[HH_{k}^{\oplus}(S_2C;M)\longrightarrow HH_{k}^{\oplus}(C;M)\]
we would be done by noticing that
\[d_{0}^k\circ s=d_2W+d_0W\simeq d_1W=d_1E\simeq d_0E+d_2E=\id\]
The maps $d_1$ and $d_0+d_2$ become homotopic after stabilizing with the $S^{2,1}_{\sbt}$-construction, by the following result applied to $HH_{k}^{\oplus}(-)$. Its main ingredient is the additivity theorem for the $S^{2,1}_{\sbt}$-construction of \cite{IbLars}.

\begin{lemma}[cf. {\cite[2.2.2]{ringfctrs}}]\label{formalprodpres}
If $X$ is a functor from exact categories and bimodules to abelian groups such that $X(0)=0$, there is a homotopy
\[d_1\!\simeq\!d_0\!+\!d_2\colon \colim\limits_m\!\Omega^{2m}X((S^{2,1}_{\sbt})^{(m)}S_2C;\!M)\longrightarrow \colim\limits_m\!\Omega^{2m}X((S^{2,1}_{\sbt})^{(m)}C;\!M)\]
\end{lemma}

\begin{proof}
Let us denote
\[Y(-;-)=\colim_m\Omega^{2m}X((S^{2,1}_{\sbt})^{(m)}(-);-)\]
In \cite{IbLars} the authors prove an additivity theorem for the $S^{2,1}_{\sbt}$-construction, analogous to the one of \cite{Mc}. Forgetting the $G$-action, it gives a homotopy equivalence
\[(d_0,d_2)\colon Y(S_2C;M)\stackrel{\simeq}{\longrightarrow} Y(C\times C;M\times M)\]
The projections induce an equivalence
\[Y(C\times C;M\times M)\stackrel{\simeq}{\longrightarrow} Y(C;M)\times Y(C;M)\]
Indeed since $X(0)=0$, and $S^{2,1}_{\sbt}$ is $1$-reduced, the  map
\[\resizebox{1\hsize}{!}{$X((S^{2,1}_{\sbt})^{(m)}C\!\times\!(S^{2,1}_{\sbt})^{(m)}C';M\!\times\!M')\longrightarrow X((S^{2,1}_{\sbt})^{(m)}C;\!M)\times X((S^{2,1}_{\sbt})^{(m)}C';\!M')$}\]
is $4k$-connected. This follows from \cite[11.12]{mayiterated} by taking the geometric realization one simplicial direction at the time. Thus the projections induce an equivalence on the colimit
\[Y(C\times C';M\times M')\longrightarrow Y(C;M)\times Y(C';M')\]
and the proof of \ref{faceshomotopic} applies word by word to $Y$ finish the proof.
\end{proof}

\begin{proof}[Proof of \ref{stable0simpl}]
We prove that the inclusion of the zero simplicies into the realization
\[s\colon \THR^{\oplus}_0(C;M)\longrightarrow \THR^{\oplus}(C;M)\]
becomes a $G$-equivalence after stabilization.
The iterated $0$-th face map defines a map
\[d_0\colon \THR^{\oplus}(C;M)\longrightarrow \THR^{\oplus}_0(C;M)\]
This is not equivariant, since the conjugate $Dd_0D$ is the iterated last face. Notice that since the bimodule $M$ is uniquely $2$-divisible, so is the topological group $\THR^{\oplus}(C;M)$. For an element
\[x=m\wedge f_1\wedge\dots\wedge f_k\in V^{\oplus}(HC;HM,\underline{i})\]
the element $\frac{x}{2}=\frac{m}{2}\wedge f_1\wedge\dots\wedge f_k$ is the unique element such that $2\cdot \frac{x}{2}=x$. Define a natural map
\[r=\frac{d_0+Dd_0D}{2}\colon \THR^{\oplus}(C;M)\longrightarrow \THR^{\oplus}_0(C;M)\]
This is clearly a retraction for $s$, and it is equivariant since the duality is additive.
Let $H$ be the homotopy
\[\resizebox{1\hsize}{!}{\xymatrix{\colim\limits_m\!\Omega^{2m,m}\THR^\oplus((S^{2,1}_{\sbt})^{(m)}C;\!M)\ar[rr]^{s\circ d_0\underset{H}{\simeq}\id}&& \colim\limits_m\!\Omega^{2m,m}\THR^\oplus((S^{2,1}_{\sbt})^{(m)}C;\!M)}}\]
induced under the homotopy equivalences
\[\xymatrix{HH^{\oplus}(C;M)\ar[r]^-{\simeq}\ar@/_1pc/[d]_{d_0}&\THH^{\oplus}(C;M)\ar@/^1pc/[d]^{d_0}\\
HH_{0}^{\oplus}(C;M)\ar[u]_s\ar[r]_-{\simeq}&\THH_{0}^{\oplus}(C;M)\ar[u]^s
}\]
by the homotopy between the stabilization of $s\circ d_0$ and the identity for $HH^{\oplus}(C;M)$ constructed above.
Then $\frac{H+DH(D,-)}{2}$ is a $G$-homotopy between the identity and $s\circ r$.
\end{proof}

The divisibility condition of \ref{stable0simpl} is unfortunate. I end this section by some speculations over how to remove the condition.
The simplicial abelian group $HH_{\sbt}^{\oplus}(C;M)$ has an obvious real structure, and the inclusion 
\[HH_{\sbt}^{\oplus}(C;M)\longrightarrow \THR_{\sbt}^\oplus(C;M)\]
is equivariant. A first step would be to generalize the proof of \cite[IV-2.4.1]{Dundasbook} to show that the inclusion induces a $G$-equivalence on realizations. Assuming this, it is enough to prove the statement for $HH^{\oplus}$. There is a way to produce an equivariant retraction for the degeneracy map
\[s\colon HH^{\oplus}_0(C;M)\longrightarrow HH^{\oplus}_k(C;M)\]
even when $M$ is not uniquely $2$-divisible. Suppose that there is a functor
\[\kappa\colon S^{2,1}_3C\longrightarrow C\]
associating to an exact sequence $X=(a\stackrel{i}{\rightarrow}b\stackrel{m}{\rightarrow}c\stackrel{p}{\rightarrow}d)$ an object $\kappa(X)$ that is both a kernel for $p$ and a cokernel for $i$, such that $\kappa(DX)=D\kappa(X)$. In particular this is possible for $C=\mathcal{D}\mathcal{P}_A$ for any ring $A$, since the functor
\[\kappa\colon S^{2,1}_3\mathcal{P}_A\longrightarrow \mathcal{P}_A\]
defined by $\kappa(a\stackrel{i}{\rightarrow}b\stackrel{m}{\rightarrow}c\stackrel{p}{\rightarrow}d)=\ker p$, together with the natural isomorphism
\[\xi_{X}\colon\kappa DX=\ker D(i)\cong D \coker i= D b/\ker m\stackrel{(D m)^{-1}}{\rightarrow} D Im(m)=D \ker p\]
is a morphism of categories with duality, and therefore it induces a strictly duality preserving functor on the $\mathcal{D}$ construction. If we further assume that any map in $C$ admits a kernel, we can define a map
\[R\colon HH^{\oplus}_k(C;M)\longrightarrow HH^{\oplus}_0(S^{2,1}_3C;M)\]
by sending $c_k\stackrel{m}{\leftarrow} c_0\stackrel{f_1}{\leftarrow}c_1\leftarrow\dots\leftarrow c_{k-1}\stackrel{f_k}{\leftarrow}c_k$ to
\[\xymatrix{\ker(f_1\dots f_k)\ar[d]^i & \ker(f_1\dots f_k)\ar[l]_0\ar[d]^i\\
c_k\ar[d]^{f_1\dots f_k} &c_k\ar[d]^{f_1\dots f_k}\ar[l]_{m\cdot f_1\dots f_k}\\
c_0\ar[d]^p&c_0\ar[l]^{f_1\dots f_k\cdot m}\ar[d]^p\\
\coker (f_1\dots f_k)& \coker(f_1\dots f_k)\ar[l]_0
}\]
where $i$ and $p$ denote respectively the inclusion and the projection.
The map $r:=\kappa\circ R$ is a retraction for $s$, since $R\circ s$ is the diagram
\[\xymatrix{0\ar[d] & 0\ar[l]\ar[d]\\
c\ar[d]^{\id} &c\ar[d]^{\id}\ar[l]_{m}\\
c\ar[d]&c\ar[l]^{m}\ar[d]\\
0& 0\ar[l]
}\]
Moreover $r$ is equivariant since $\kappa(DX)=D\kappa(X)$.
One should think of $r$ as an average between the iterated $0$-face of $HH_{\sbt}$ and the iterated last face. Indeed,
a variation of the argument of \ref{formalprodpres} above, using the equivariant additivity theorem of \cite{IbLars}, shows that after stabilization there is a homotopy $H$ between the maps
\[d_1-d_0\underset{H}{\simeq} d_2-d_3\colon HH^{\oplus}_q(S^{2,1}_3C;M)\longrightarrow HH^{\oplus}_q(C;M)\]
which is reversed by the duality, in the sense that $DH(D(-),t)=H(-,1-t)$. In particular, the middle point $H(-,1/2)$ is an equivariant map. Also notice that
\[(d_1-d_0)\circ  R=d_{0}^k \ \ \ \mbox{and} \ \ \  (d_2-d_3)\circ R=Dd_{0}^kD\]
and the middle map $H(R,1/2)$ is in fact $r$.
We can display all these maps and homotopies in the following picture, taking place in the mapping space of self maps of the stabilization of $HH^{\oplus}_k(C;M)$.
\[\xymatrix{
&\stackrel{s\circ d_{0}^k}{\bullet}\ar[rr]^{\overline{W}}\ar[dd]_<<<<<<{H(R,-)}&& \stackrel{d_1W=d_1E}{\bullet}  \\
\ar@{--}[rrrrr]_<<<<<<<<<<<<<<<<<<{s\circ r}^>>>>>>>>>{\id}& \bullet & && \bullet \ar[ul]_{\overline{E}}\ar[dl]^-{D\overline{E}D} & \\
& \underset{s\circ Dd_{0}^kD}{\bullet}\ar[rr]_-{D\overline{W}D} && \underset{d_1DWD=d_1DED}{\bullet}
}\]
where $\overline{E}$ and $\overline{W}$ denote the homotopies induced by the maps \[E,W\colon HH^{\oplus}_k(C;M)\longrightarrow HH^{\oplus}_k(S_2C;M)\] from \cite[I-3.3.1]{Dundasbook} mentioned above.
Conjugation with the duality flips the picture along the dashed line. In particular we have two homotopies between $s\circ r$ and the identity.
In order to obtain a $G$-homotopy, one should try to take an average of these homotopies. One could try to fill in the picture, by defining a map
\[K\colon HH^{\oplus}_k(C;M)\longrightarrow HH^{\oplus}_k(S^{2,1}_3S^{1,1}_2C;M)\]
with the appropriate faces and symmetry properties. Composing with $\kappa$ it would give the desired $G$-homotopy, as an average between the concatenation of $\overline{W}$ and $\overline{E}$ and the concatenation of $D\overline{W}D$ and $D\overline{E}D$.

\newpage


\section{Real trace maps $\KR\longrightarrow \THR$}\label{trace}

The classical trace map is a natural weak map $K(C)\longrightarrow \THH(C)$ (or to $TC(C)$) first introduced in \cite{BHM} by Bökstedt, Hsiang and Madsen. It led to many important computations in $K$-theory of rings, some by means of the Dundas-McCarthy theorem (cf. \cite{dmthm}, \cite{dundasrelKandTC} and \cite[VII]{Dundasbook}).
A main step in the proof of the Dundas-McCarthy theorem is to show that the trace map induces an equivalence
\[\colim\limits_n\Omega^n\widetilde{K}(A\ltimes M(S^{n}))\simeq \THH(\mathcal{P}_A;M(S^{1}))\]
cf. \cite[5.3]{stableDM} and \cite[§4]{dmthm}.
This section collects the properties of $\KR$ and $\THR$ proved so far to show a $G$-analogue of the above equivalence.

In §\ref{stablekrsec} we define a stabilization $\widetilde{\KR}^S(A\ltimes M)$ for $\KR(A\ltimes M)$ in the "module direction". In §\ref{stableKRisTHRsec} we define a weak $G$-map (i.e. a zig-zag of $G$-maps whose wrong way pointing maps are weak $G$-equivalences)
\[\beta\colon \widetilde{\KR}(A\ltimes M)\longrightarrow \THR(\mathcal{D}\mathcal{P}_{A};M^J(S^{1,1}_{\sbt}))\]
inducing a $G$-equivalence between $\widetilde{\KR}^S(A\ltimes M)$ and $\THR(\mathcal{D}\mathcal{P}_{A};M^J(S^{1,1}_{\sbt}))$. In §\ref{reltracemaps} we define a real trace map $\tr\colon\KR(C)\longrightarrow \THR(C)$ and show that there is a diagram of weak $G$-maps
\[\xymatrix{\widetilde{\KR}(A\ltimes M)\ar[dr]_-{\beta}\ar[r]^-{\widetilde{\tr}}&\widetilde{\THR}(\mathcal{D}\mathcal{P}_{A\ltimes M})\ar[d]^\alpha\\
& \THR(\mathcal{D}\mathcal{P}_{A};M^J(S^{1,1}_{\sbt}))}\]
that commutes up to $G$-homotopy.


\subsection{Stable real $K$-theory}\label{stablekrsec}

Let $A$ be a ring, $M_{\sbt}$ be a simplicial $A$-bimodule (i.e. a simplicial object in the category of $A$-bimodules), and $(A,L,\alpha)$ an antistructure.

\begin{defn}
An \textbf{$M_{\sbt}$-twisting} of $(A,L,\alpha)$ is a family of $M_k$-twistings  \[J_k\colon L_t\otimes_A M_k\longrightarrow L_t\otimes_A M_k\]
that reverse the simplicial structure of $L_t\otimes M_{\sbt}$, in the sense that
\[\begin{array}{ll}(L_t\otimes d_l)J_k=J_{k-1}(L_t\otimes d_{k-l})\\
(L_t\otimes s_l)J_k=J_{k+1}(L_t\otimes s_{k-l})&0\leq l\leq k
\end{array}\]
where $d_l$ and $s_l$ denote faces and degeneracies of $M_{\sbt}$.
\end{defn}

\begin{ex}
Let $M$ be an $A$-bimodule, $J\colon L_t\otimes_A M\longrightarrow L_t\otimes_A M$ be an $M$-twisting of an antistructure  $(A,L,\alpha)$, and $X_{\sbt}$ a real set with involution $\omega$. The maps $J_k$ defined by the diagram
\[\xymatrix{L_t\otimes_A M(X_k)\ar[d]^{\cong}\ar@{-->}[r]^{J_k}&L_t\otimes_A M(X_k)\ar[d]_{\cong}\\
(L_t\otimes_A M)(X_k)\ar[r]_{J(\omega)}&(L_t\otimes_A M)(X_k)
}\]
assemble into a $M(X_{\sbt})$-twisting of $(A,L,\alpha)$.
\end{ex}

Given an $M_{\sbt}$-twisting $J_{\sbt}$ of $(A,L,\alpha)$, every $J_k$ defines an antistructure $(A\ltimes M_k,L^{J_k},\alpha^{J_k})$, and therefore a duality on $\mathcal{P}_{A\ltimes M_k}$ (cf. §\ref{squarezeroext}). The simplicial structure of $M_{\sbt}$ gives the structure of a real real category on 
\[[k],[n]\longmapsto iS^{2,1}_n\mathcal{D}\mathcal{P}_{A\ltimes M_k}\]
The structure map in the $k$-direction associated to a $\theta\colon[k]\longrightarrow[k']$ is defined to be the functor induced from 
the map \[\theta^\ast\colon (A\ltimes M_{k'},L^{J_{k'}},\alpha^{J_{k'}})\longrightarrow (A\ltimes M_k,L^{J_k},\alpha^{J_k})\] defined from the bimodule map $\theta^\ast\colon M_{k'}\longrightarrow M_k$ as in \ref{mtwistingantistructure}. In order to have strict simplicial identities one needs to use the strictly functorial model for the category of finitely generated projective modules described in \ref{functorialPA}.
\begin{defn} Given an $M_{\sbt}$-twisting $J_{\sbt}$, define
\[\KR(A\ltimes M_{\sbt})=\Omega^{2,1}|iS^{2,1}_{\sbt}\mathcal{D}\mathcal{P}_{A\ltimes M_{\sbt}}|\]
\end{defn}
The realization can be taken either one real direction at the time or diagonally.
The projections $A\ltimes M_k\longrightarrow A$ induce an equivariant map $\KR(A\ltimes M_{\sbt})\longrightarrow \KR(A)$, and we denote its homotopy fiber by $\widetilde{\KR}(A\ltimes M_{\sbt})$.

In order to define the stabilization for $\widetilde{\KR}(A\ltimes M)$ we iterate the equivariant Bar constructions in the $M$-direction as follows.
Let us denote $S^{2n,n}_{\sbt}$ the $2n$-fold smash product of simplicial sets
\[S^{2n,n}=(S^{1}_{\sbt})^{op}\wedge\dots\wedge (S^{1}_{\sbt})^{op}\wedge S^{1}_{\sbt}\wedge\dots\wedge S^{1}_{\sbt} \ \ \ \ \ \ \ \ \ \ (2n\mbox{ factors})\]
with levelwise involution $\omega$ given by reversing the order of the smash factors. 
Here the superscript ${}^{op}$ denotes the simplicial set where the order of the faces and degeneracies is reversed.
This involution gives $S^{2n,n}_{\sbt}$ the structure of a real set. For every $\sigma\in S^{2,1}_k$ denote
\[\iota_\sigma\colon S^{2n,n}_k\longrightarrow S^{2(n+1),n+1}_k\]
the inclusion
\[\iota_\sigma(x_1\wedge\dots\wedge x_{2n})=(x_1\wedge\dots\wedge x_{n}\wedge\sigma\wedge x_{n+1}\wedge\dots\wedge x_{2n})\]
that satisfies $\omega\circ\iota_\sigma=\iota_{\omega\sigma}\circ\omega$.
This induces a map of twistings (see \ref{defmaptwist})
\[M(\iota_\sigma)\colon M(S^{2n,n}_k)\longrightarrow M(S^{2(n+1),n+1}_k)\]
and thus a functor
\[(\iota_\sigma)_\ast\colon S^{2,1}_{\sbt}\mathcal{D}\mathcal{P}_{A\ltimes M(S^{2n,n}_k)}\longrightarrow S^{2,1}_{\sbt}\mathcal{D}\mathcal{P}_{A\ltimes M(S^{2(n+1),n+1}_k)}\]
that satisfies $(\iota_\sigma)_\ast D=D(\iota_{\omega\sigma})_\ast$.

\begin{defn}
Let $M$ be an $A$-bimodule, and $J$ an $M$-twisting of an antistructure $(A,L,\alpha)$. The stable real $K$-theory of $J$ is the $G$-space
\[\widetilde{\KR}^S(A\ltimes M)=\colim_{n}\Omega^{2n,n}\widetilde{\KR}(A\ltimes M(S^{2n,n}_{\sbt}))\]
with stabilization maps defined to be the adjoints of the realizations of
\[S^{2,1}_k\wedge|iS^{2,1}_{\sbt}\mathcal{D}\mathcal{P}_{A\ltimes M(S^{2n,n}_k)}|\longrightarrow |iS^{2,1}_{\sbt}\mathcal{D}\mathcal{P}_{A\ltimes M(S^{2(n+1),n+1}_k)}|\]
induced from mapping $\sigma\wedge x$ to $(\iota_\sigma)_\ast x$.
\end{defn}

In \cite{stableDM} the authors define stable $K$-theory as the colimit of $\Omega^{n+1}\widetilde{K}(A\ltimes M(S^{n}_{\sbt}))$. The missing degree shift in our definition has the consequence of having an extra $S^{1,1}_{\sbt}$-smash factor in the coefficients of $\THR$ in the main result of the next section compared to \cite[5.3]{stableDM}.

\begin{rem}
We extended $\KR$ to simplicial rings degree-wise. In the non-equivariant case, one can of course consider the category of simplicial modules over a simplicial ring $A_{\sbt}$, and take its $K$-theory. In order to consider only cellular modules (cf. \cite[§2.3]{waldhausen}). $K$-theory for simplicial modules and the degreewise $K$-theory differ in general, but in the case of $A\ltimes M$ the homotopy fibers of the projection maps induced by $A\ltimes M\longrightarrow A$ agree for the two theories. The published proof of this uses the description of $K$-theory via $B\widehat{GL}(A_{\sbt})^+$, the plus construction of the classifying space of the group of matricies invertible up to homotopy (see \cite[4.1]{Dundasdegreewise}).

Our equivariant situation is tricky. It is fairly straightforward to define what an antistructure over a simplicial ring should be, and this induces a duality on the category of all simplicial modules, given by a $\hom$-mapping space construction. However this duality does not preserve cellular modules, and it is not even clear if it does up to equivalence. Therefore there is no obvious $\KR$ theory for simplicial modules involving the $S^{2,1}_{\sbt}$-construction.
\end{rem}

\begin{rem}\label{stableKRgoodwilliediff} Fix an antistructure $(A,L,\alpha)$ and an $M$-twisting $J$, and consider real $K$-theory as a functor
\[\widetilde{\KR}(A\ltimes M(-))\colon \{\mbox{Based real sets}\}\longrightarrow G\!-\!\Top_\ast\]
from the category of based real sets to pointed $G$-spaces (or eventually real spectra). There is another functor
\[\widetilde{\KR}^S(A\ltimes M;-)\colon \{\mbox{Based real sets}\}\longrightarrow G\!-\!\Top_\ast\]
defined by
\[\widetilde{\KR}^S(A\ltimes M;X)=\colim_{n}\Omega^{2n,n}\widetilde{\KR}(A\ltimes M(S^{2n,n}_{\sbt}\wedge X))\] 
using similar stabilization maps.
This is a $G$-version of the Goodwillie differential for $\widetilde{\KR}(A\ltimes M(-))$ (cf. \cite{calculusI} and \cite{stableDM}). Stable real $K$-theory $\widetilde{\KR}^S(A\ltimes M)$ is the zero space of the "derivative real spectrum" \[\{\widetilde{\KR}^S(A\ltimes M;S^{2m,m})\}_{m\geq 0}\]
\end{rem}


\subsection{Stable real $K$-theory is $\THR$}\label{stableKRisTHRsec}

We show that there is an equivalence between stable $\KR$-theory and $\THR$, via a trace-like map. This theorem is the $G$-analogue of \cite[5.3]{stableDM}. Recall that we use the notation \[\THR(C;M):=\THR(HC;HM)\]

If  $M_{\sbt}$ is a simplicial $A$-bimodule, and $J_{\sbt}\colon L_t\otimes_A M_{\sbt}\longrightarrow L_t\otimes_A M_{\sbt}$ an $M_{\sbt}$-twisting of $(A,L,\alpha)$, the structure maps on $M_{\sbt}$ define a real structure on
\[[k]\longrightarrow \THR(\mathcal{D}\mathcal{P}_A;M_{k}^{J_k})\]
analogous to the case of $\KR(-)$ explained in the previous section. Its realization is the $G$-space $\THR(\mathcal{D}\mathcal{P}_A;M^{J}_{\sbt})$ (again, one needs to use the strictly functorial model of \ref{functorialPA} for $\mathcal{P}_A$ in order to get the simplicial identities strictly.).

\begin{theorem}\label{stableKRisTHR}
Let $A$ be a ring, $M$ be an $A$-bimodule, and $J$ an $M$-twisting of an antistructure $(A,L,\alpha)$. If $2$ is invertible in $A$, there is a weak $G$-equivalence
\[\widetilde{\KR}^S(A\ltimes M)\simeq\THR(\mathcal{D}\mathcal{P}_A;M^J(S^{1,1}_{\sbt}))\]
\end{theorem}

\begin{rem}
The assumption that $2$ is invertible in $A$ is needed only to apply \ref{stable0simpl}. If one could prove \ref{stable0simpl} without this assumption, all the results of this section would hold without this extra condition.
\end{rem}

\begin{proof}
The weak map of the statement is constructed as follow. For every $k$ there is a zigzag
\[\xymatrix{\widetilde{\KR}(A\ltimes M(S^{2n,n}_k)) &
 \widetilde{\KR}(A;M(S^{2n,n}_k\wedge S^{1,1}_{\sbt}))\ar[l]_{\simeq}^{\ref{relKR}}\ar[d]\\
 & \Omega^{2,1}\THR_0(S^{2,1}_{\sbt}\mathcal{D}\mathcal{P}_A;M^J(S^{2n,n}_k\wedge S^{1,1}_{\sbt}))\ar[d]\\
 & \Omega^{2,1}\THR(S^{2,1}_{\sbt}\mathcal{D}\mathcal{P}_A;M^J(S^{2n,n}_k\wedge S^{1,1}_{\sbt}))\\
 &\THR(\mathcal{D}\mathcal{P}_A;M^J(S^{2n,n}_k\wedge S^{1,1}_{\sbt})) \ar[u]^-{\simeq}_-{\ref{THRdeloopings}}
}\]
For the top horizontal map we used the canonical isomorphism $M(X)(S^{1,1})\cong M(X\wedge S^{1,1})$. The top vertical map is induced, under the equivalence of \ref{relKR}, by the canonical map into the homotopy colimit
\[\xymatrix{\bigvee\limits_{\varphi\in Ob S^{2,1}_{\sbt}\mathcal{D}\mathcal{P}_A}\hom_A(\varphi,\varphi\otimes_A N)\ar@{=}[d]\\
\bigvee\limits_{\varphi\in Ob S^{2,1}_{\sbt}\mathcal{D}\mathcal{P}_A}N^J(\varphi,\varphi)\ar[d]\\ \hocolim\limits_{I}\Omega^{i}\!\!\!\!\!\!\bigvee\limits_{\varphi\in Ob S^{2,1}_{\sbt}\mathcal{D}\mathcal{P}_A}\!\!\!\!N^J(\varphi,\varphi)(S^i)\ar@{=}[d]\\
\THR_0(S^{2,1}_{\sbt}\mathcal{D}\mathcal{P}_A;N^J)}\]
where $N=M(S^{2n,n}_k\wedge S^{1,1}_{\sbt})$.
The middle vertical map is just the inclusion of the zero simplicies in the realization.

Realizing in the $k$-direction and looping we obtain a zig-zag
\[\xymatrix{\Omega^{2n,n}\widetilde{\KR}(A\!\ltimes\!M(S^{2n,n}_{\sbt}))\ar@{-->}[dd] &
 \Omega^{2n,n}\widetilde{\KR}(A;\!M(S^{2n,n}_{\sbt}\!\wedge\!S^{1,1}_{\sbt}))\ar[l]_-{\simeq}^-{\ref{relKR}}\ar[d]\\
 & \Omega^{2n,n}\Omega^{2,1}\THR(S^{2,1}_{\sbt}\mathcal{D}\mathcal{P}_A;\!M^J(S^{2n,n}_{\sbt}\!\wedge\!S^{1,1}_{\sbt}))\\
\THR(\mathcal{D}\mathcal{P}_A;\!M^J(S^{1,1}_{\sbt}))\ar[r]_-{\simeq}^-{\ref{stabTHR}} & \Omega^{2n,n}\THR(\mathcal{D}\mathcal{P}_A;\!M^J(S^{2n,n}_{\sbt}\!\wedge\!S^{1,1}_{\sbt})) \ar[u]^{\simeq}_{\ref{THRdeloopings}}
}\]
where the bottom horizontal map is induced by the same functors defining the stabilization map for $\widetilde{\KR}^S$. Precisely, it is the adjoint of the realization of the map
\[S^{2n,n}_k\wedge \THR(\mathcal{D}\mathcal{P}_A;M^J(S^{1,1}_{\sbt}))\longrightarrow\THR(\mathcal{D}\mathcal{P}_A;M^J(S^{2n,n}_k\wedge S^{1,1}_{\sbt}))\]
given by sending $\sigma\wedge x$ to $(\iota_\sigma)_\ast x$, where $\iota_\sigma\colon M(S^{1,1}_{\sbt})\longrightarrow M(S^{2n,n}_k\wedge S^{1,1}_{\sbt})$ is the map of bimodules induced by the inclusion $S^0_k\longrightarrow S^{2n,n}_k$ determined by $\sigma\in S^{2n,n}_k$. It is an equivalence by \ref{stabTHR} below.

The zig-zag respects the stabilization maps and thus induces a weak map on the colimit
\[\widetilde{\KR}^S(A\ltimes M)\longrightarrow\THR(\mathcal{D}\mathcal{P}_A;M^J(S^{1,1}_{\sbt}))\]
Apply \ref{connfirsttracemap} below to the $\nu=(2n,n-1)$-connected real set $X_{\sbt}=S^{2n,n}_{\sbt}\wedge S^{1,1}_{\sbt}$ to show that this weak map is $(2n+1,n-1)$-connected. Therefore it induced an equivalence on colimits, proving \ref{stableKRisTHR}.
\end{proof}

\begin{rem} For a fixed $M$-twisting $J$, we have the functor
\[\THR(\mathcal{D}\mathcal{P}_A;M^J(-))\colon \{\mbox{Real sets}\}\longrightarrow G\!-\!\Top_\ast\]
analogous to $\widetilde{\KR}$ in \ref{stableKRgoodwilliediff}. Its Goodwillie $G$-differential is the functor defined at a real set $X_{\sbt}$ by
\[\THR^S(\mathcal{D}\mathcal{P}_A;M^J;X_{\sbt}):=\colim_{n}\Omega^{2n,n}\THR(\mathcal{D}\mathcal{P}_A;M^J(S^{2n,n}_{\sbt}\wedge X_{\sbt}))\]
The canonical map $\THR(\mathcal{D}\mathcal{P}_A;M^J(X_{\sbt}))\longrightarrow \THR^S(\mathcal{D}\mathcal{P}_A;M^J;X_{\sbt})$ is a $G$-equivalence by \ref{stabTHR} below. This is the $G$-analogous of the property of $\THH$ of being $1$-excisive, (cf. \cite{calculusII}).
 Proposition \ref{connfirsttracemap} says that the differential of reduced $\KR$-theory is $G$-equivalent the $S^{1,1}_{\sbt}$-suspension of $\THR$.
\end{rem}

\begin{prop}\label{connfirsttracemap} Let $J$ be an $M$-twisting of $(A,L,\alpha)$ with $2\in A$ invertible, and let $X_{\sbt}$ be a real set. If the $G$-space $|X_{\sbt}|$ is $(\nu_1,\nu_2)$-connected, the map 
\[\widetilde{\KR}(A;M(X_{\sbt}))\longrightarrow\Omega^{2,1}\THR(S^{2,1}_{\sbt}\mathcal{D}\mathcal{P}_A;M^J(X_{\sbt}))\]
is non-equivariantly $(2\nu_1+1)$-connected, and $\min\{\nu_1,2\nu_2+1\}$-connected on the fixed points.
\end{prop}

\begin{proof}
The map of the statement is $\Omega^{2,1}$-loop of the top row of the commutative diagram (cf. §\ref{secintrokranti} and \ref{mainKR})
\[\resizebox{1\hsize}{!}{\xymatrix{|\bigvee\limits_{S^{2,1}_{\sbt}\mathcal{D}\mathcal{P}_A}\hom_A(\varphi,\varphi\otimes_A\! M(X_{\sbt}))|\ar[r]\ar@{=}[d]&
\THR(S^{2,1}_{\sbt}\mathcal{D}\mathcal{P}_A;M^J(X_{\sbt}))\ar[d]_-{\simeq}^-{\ref{abgpmodel}}\\
|\bigvee\limits_{S^{2,1}_{\sbt}\mathcal{D}\mathcal{P}_A}\hom_A(\varphi,\varphi\otimes_A\! M(X_{\sbt}))|\ar[r]\ar[d]^-{\simeq}_-{\ref{relKR}}&
\THR^{\oplus}(S^{2,1}_{\sbt}\mathcal{D}\mathcal{P}_A;M^J(X_{\sbt}))\ar[d]_-{\simeq}^-{\ref{stable0simpl}}\\
\colim\limits_p\!\Omega^{2p,p}|\!\!\!\!\!\!\!\!\!\!\bigvee\limits_{(S^{2,1}_{\sbt})^{(p\!+\!1)}\mathcal{D}\mathcal{P}_A}\!\!\!\!\!\!\!\!\!\!\!\!\hom_A(\vartheta,\vartheta\otimes_A\!\!M(X_{\sbt}))|\ar[r]\ar@{=}[d]&
\colim\limits_p\!\Omega^{2p,p}\!\THR^{\oplus}\!((S^{2,1}_{\sbt})^{(p\!+\!1)}\mathcal{D}\mathcal{P}_A;\!M^J\!(X_{\sbt}))\\
\colim\limits_p\!\Omega^{2p,p}|\!\!\!\!\!\!\!\!\!\!\bigvee\limits_{(S^{2,1}_{\sbt})^{(p\!+\!1)}\mathcal{D}\mathcal{P}_A}\!\!\!\!\!\!\!\!\!\!\!\!\hom_A(\vartheta,\vartheta\otimes_A\!\!M(X_{\sbt}))|\ar[r]&
\colim\limits_p\!\Omega^{2p,p}|\!\!\!\!\!\!\!\!\bigoplus\limits_{(S^{2,1}_{\sbt})^{(p\!+\!1)}\mathcal{D}\mathcal{P}_A}\!\!\!\!\!\!\!\!\hom_A(\vartheta,\vartheta\otimes_A\!\!M(X_{\sbt}))|\ar[u]^-{\simeq}_-{\ref{descr0simpl}}
}}\]
Here $(S^{2,1}_{\sbt})^{(p+1)}$ denotes the diagonal of the $S^{2,1}_{\sbt}$-construction iterated $(p+1)$-times, and
the bottom map is induced by the inclusion of wedges into sums. Write
\[\kappa=(2\nu_1+1,\min\{\nu_1,2\nu_2+1\})\]
for the connectivity of the statement. By the diagram above it is enough to show that the map
\[|\!\!\!\!\bigvee\limits_{(S^{2,1}_{\sbt})^{(p+1)}\mathcal{D}\mathcal{P}_A}\!\!\!\!\hom_A(\vartheta,\vartheta\otimes_A\!M(X_{\sbt}))|\longrightarrow|\!\!\!\!\bigoplus\limits_{(S^{2,1}_{\sbt})^{(p+1)}\mathcal{D}\mathcal{P}_A}\!\!\!\!\hom_A(\vartheta,\vartheta\otimes_A\!M(X_{\sbt}))|\]
is $(\kappa+(2(p+1),p+1))$-connected.
In every simplicial degree $l$, the map of spaces
\[|\!\!\!\!\bigvee\limits_{(S^{2,1}_{l})^{(p+1)}\mathcal{D}\mathcal{P}_A}\!\!\!\!\hom_A(\vartheta,\vartheta\otimes_A\!M(X_{\sbt}))|\longrightarrow|\!\!\!\!\bigoplus\limits_{(S^{2,1}_{l})^{(p+1)}\mathcal{D}\mathcal{P}_A}\!\!\!\!\hom_A(\vartheta,\vartheta\otimes_A\!M(X_{\sbt}))|\]
is $\kappa$-connected, as the inclusion of wedges into products (cf. \ref{wedgesintoproducts} and \ref{connconfspace}). Moreover, in simplicial degree $l\leq 2p+1$ the map is a $G$-equivalence. This uses that $S^{2,1}_0C=S^{2,1}_1C=0$. Therefore in degree $l$ the map is non-equivariantly $(\kappa_1+2(p+1)-l)$-connected. Thus  its realization is $(\kappa_1+2(p+1))$-connected, since in general a map of simplicial spaces which is $(c-l)$-connected in simplicial degree $l$, induces a $c$-connected map on the realization (cf. \cite[11.12]{mayiterated}).

In order to describe the map on the fixed points, we need to subdivide. In degree $l\leq p$, the subdivision of the map is an equivalence on the fixed points, since the non-subdivided map is  a $G$-equivalence up to degree $2p+1$. Thus the subdivided map is $(\kappa_2+p+1-l)$-connected on the fixed points, in every simplicial degree $l$. This shows that the realization is $(\kappa_2+p+1)$-connected on the fixed points.
\end{proof}

\begin{prop}\label{stabTHR}
If $2\in A$ is invertible, the map
\[\THR(\mathcal{D}\mathcal{P}_A;M^J(X_{\sbt}))\longrightarrow \Omega^{2n,n}\THR(\mathcal{D}\mathcal{P}_A;M^J(S^{2n,n}_{\sbt}\wedge X_{\sbt}))\]
is a $G$-equivalence for every real set $X_{\sbt}$ and all $n\geq 0$.
\end{prop}

\begin{proof}
We use the abelian group model $\THR^\oplus$ for $\THR$. The functors $(\iota_\sigma)_\ast$ used above to define the stabilization maps for $\THR$ give stabilization maps for $\THR^\oplus$ as well, and there is a commutative diagram
\[\xymatrix{\THR(\mathcal{D}\mathcal{P}_A;M^J(X_{\sbt}))\ar[r]\ar[d]^{\simeq}_{\ref{abgpmodel}}& \Omega^{2n,n}\THR(\mathcal{D}\mathcal{P}_A;M^J(S^{2n,n}_{\sbt}\wedge X_{\sbt}))\ar[d]^{\simeq}_{\ref{abgpmodel}}\\
\THR^\oplus(\mathcal{D}\mathcal{P}_A;M^J(X_{\sbt}))\ar[r]& \Omega^{2n,n}\THR^\oplus(\mathcal{D}\mathcal{P}_A;M^J(S^{2n,n}_{\sbt}\wedge X_{\sbt}))
}\]
We show that the bottom map is an equivalence. By \ref{stable0simpl}, it is enough to show that the map 
\[\resizebox{1\hsize}{!}{$\THR_{0}^\oplus((S^{2,1}_k)^{(p)}\mathcal{D}\mathcal{P}_A;\!M^J(X_{\sbt}))\longrightarrow \Omega^{2n,n}\THR_{0}^\oplus((S^{2,1}_k)^{(p)}\mathcal{D}\mathcal{P}_A;\!M^J(S^{2n,n}_{\sbt}\!\wedge\!X_{\sbt}))$}\]
is an equivalence for all $k$ and $p$. Since loops commute equivariantly with direct sums, the $G$-equivalence $M\longrightarrow\Omega^iM(S^i)$ of \ref{Gdeloopingsformonoid} below (with appropriate $G$-action on $S^i$) gives a commutative diagram
\[\resizebox{1\hsize}{!}{\xymatrix{\THR_{0}^\oplus((S^{2,1}_k)^{(p)}\mathcal{D}\mathcal{P}_A;\!M^J(X_{\sbt}))\ar[r]& \Omega^{2n,n}\THR_{0}^\oplus((S^{2,1}_k)^{(p)}\mathcal{D}\mathcal{P}_A;\!M^J(S^{2n,n}_{\sbt}\!\!\wedge\!X_{\sbt}))\\
\!\!\!\bigoplus\limits_{(S^{2,1}_k)^{(p)}\mathcal{D}\mathcal{P}_A}\!\!\!\!\!\!\!\hom_A(\vartheta,\vartheta\otimes_A\!M)(X_{\sbt})\ar[u]^-{\simeq}_-{\ref{descr0simpl}}\ar@<1ex>[r]& \Omega^{2n,n}\!\!\!\!\!\!\bigoplus\limits_{(S^{2,1}_k)^{(p)}\mathcal{D}\mathcal{P}_A}\!\!\!\!\!\hom_A(\vartheta,\vartheta\otimes_A\!M)(S^{2n,n}_{\sbt}\!\!\wedge\!X_{\sbt})\ar[u]^-{\simeq}_-{\ref{descr0simpl}}
}}\]
(see \ref{descr0simpl}).
Commuting $\Omega^{2n,n}$ with the direct sum, the bottom map becomes the direct sum of the maps 
\[\hom_A(\vartheta,\vartheta\otimes_A M)(X_{\sbt})\longrightarrow \Omega^{2n,n}\hom_A(\vartheta,\vartheta\otimes_A M)(X_{\sbt}\wedge S^{2n,n}_{\sbt})\]
which are $G$-equivalences again by \ref{Gdeloopingsformonoid}.
\end{proof}


\subsection{Relation with the real trace map}\label{reltracemaps}

The classical trace map is a natural weak map from $K$-theory to $\THH$, first defined in \cite{BHM}. Following \cite{ringfctrs},\cite{IbLars} we define the real analogue $\KR\longrightarrow\THR$.

Let $(\mathcal{C},D,\mathcal{E})$ be an exact category with strict duality. Recall that we denote $\THR(C)=\THR(C;C)$ when we take coefficients in the $\hom$-bimodule 
\[\hom_{HC}(-,-)\colon HC^{op}\wedge HC\longrightarrow Sp^\Sigma\]
(cf. \ref{homexeil}).

\begin{defn}\label{deftrace}
The \textbf{real trace map} $\KR(C)\longrightarrow\Omega^{2,1}\THR(S^{2,1}_{\sbt}C)$ is the weak $G$-map defined by the diagram
\[\xymatrix{\KR(C)\ar@{-->}[ddrrrr]_{\tr}&&\Omega^{2,1}|ObS^{2,1}_{\sbt} C|\ar[ll]^-{\simeq}_-{\ref{inclobinisos21}}\ar[rr]&&\Omega^{2,1}\THR_0(S^{2,1}_{\sbt} C)\ar[dd]\\
\\
&&&&\Omega^{2,1}\THR(S^{2,1}_{\sbt} C)
}\] 
\end{defn}
In \ref{deftrace}, the right-pointing horizontal map is induced from the map
\[Ob S^{2,1}_{\sbt} C\longrightarrow\THR_0(S^{2,1}_{\sbt} C)\]
that sends an object $X$ of $S^{2,1}_n C$ to the element of the homotopy colimit defined by its identity
\[\id_X\in\hom_{S^{2,1}_nC}(X,X)\subset V(HS^{2,1}_nC,0)\]
If $(C,\mathcal{E})$ is split-exact, we can compose with the $G$-equivalence of \ref{THRdeloopings} to get a weak $G$-map
\[\xymatrix{ \KR(C)\ar[r]^-{\tr}&\Omega^{2,1}\THR(S^{2,1}_{\sbt} C)&\THR(C)\ar[l]^-{\simeq}
}\] 
It induces a map of real spectra upon iteration of $S^{2,1}_{\sbt}$.

Now let $C=\mathcal{D}\mathcal{P}_{A\ltimes M}$ with duality induced by an $M$-twisting $J$ of the antistructure $(A,L,\alpha)$. In the previous section we considered another equivariant weak map from $\KR$ to $\THR$, namely
\[\xymatrix{\widetilde{\KR}(A\ltimes M)\ar@{-->}@<-3ex>[d]_-{\beta}&\Omega^{2,1}|\bigvee\limits_{\varphi\in S^{2,1}_{\sbt}\mathcal{D}\mathcal{P}_A}\hom_A(\varphi,\varphi\otimes_A M(S^{1,1}_{\sbt}))|\ar[l]^-{\simeq}_-{\ref{relKR}}\ar[d]\\
\THR(\mathcal{D}\mathcal{P}_{A };M^J(S^{1,1}_{\sbt}))\ar[r]^-{\simeq}_-{\ref{THRdeloopings}} &\Omega^{2,1}\THR(S^{2,1}_{\sbt} \mathcal{D}\mathcal{P}_{A };M^J(S^{1,1}_{\sbt}))
}\]
The trace map of \ref{deftrace} induces a weak $G$-map
\[\widetilde{\tr}\colon \widetilde{\KR}(A\ltimes M)\longrightarrow \widetilde{\THR}(\mathcal{D}\mathcal{P}_{A\ltimes M})\]
on homotopy fibers. 
We shall now compare the maps $\widetilde{\tr}$ and $\beta$ (see \cite[§4]{dmthm} for the non-equivariant case).
\begin{theorem}\label{comparisontrbeta} There is a diagram of weak $G$-maps
\[\xymatrix{\widetilde{\KR}(A\ltimes M)\ar[dr]_-{\beta}\ar[r]^-{\widetilde{\tr}}&\widetilde{\THR}(\mathcal{D}\mathcal{P}_{A\ltimes M})\ar[d]^\alpha\\
& \THR(\mathcal{D}\mathcal{P}_{A};M^J(S^{1,1}_{\sbt}))}\]
that commutes in the homotopy category.
\end{theorem}
This last theorem applied degreewise to the simplicial modules $M(S^{2n,n}_{\sbt})$, \ref{connfirsttracemap} implies:

\begin{cor} Let $J$ be an $M$-twisting of an antistructure $(A,L,\alpha)$, and assume $2\in A$ invertible.
The composition
\[\widetilde{\KR}(A\!\ltimes\!M(S^{2n,n}_{\sbt}))\stackrel{\widetilde{\tr}}{\longrightarrow} \widetilde{\THR}(\mathcal{D}\mathcal{P}_{A\!\ltimes\! M(\!S^{2n,n}_{\sbt}\!)})\stackrel{\alpha}{\longrightarrow}\THR(\mathcal{D}\mathcal{P}_{A};\!M^J\!(S^{2n,n}_{\sbt}\wedge S^{1,1}_{\sbt}))\]
is $(4n+1,2n-1)$-connected.
\end{cor}

The construction of $\alpha$ is analogous to the classical construction of \cite[§4]{dmthm}. It will be convenient to work in the category of real spectra, since some of the maps are only $G$-equivalences stably. The map $\alpha$ of \ref{comparisontrbeta} is then the infinite loop map of a map $\uuline{\alpha}$ of real spectra (cf. \ref{realspec}) defined by the diagram
\[\xymatrix{\widetilde{\uuline{\THR}}(\mathcal{D}\mathcal{P}_{A\ltimes M})\ar@{-->}[ddddd]^{\uuline{\alpha}}&\ar@{}[ddddd]_{(1)}&\widetilde{\uuline{\THR}}(\mathcal{D}\mathcal{P}_{A}\ltimes M^J)\ar[ll]_-{\simeq}^-{\ref{generalfish}}\\
&&\widetilde{\uuline{\THR}}(H\mathcal{D}\mathcal{P}_{A}\vee HM^J)\ar[u]^{\simeq}_{\ref{stableGeq}}\\
&&\bigvee_{a=1}^\infty\uuline{\THR}^{(a)}(\mathcal{D}\mathcal{P}_{A};M^J)\ar[u]^{\simeq_s}_{\ref{htpyfibwedgesthr}}\ar[d]^{\pr^{(1)}}\\
&&\uuline{\THR}^{(1)}(\mathcal{D}\mathcal{P}_{A};M^J)\\
&&S^{1,1}_+\wedge \uuline{\THR}(\mathcal{D}\mathcal{P}_{A};M^J)\ar[u]_{\ref{s11smashthr}}^{\simeq_s}\ar[d]^{\pr}\\
\uuline{\THR}(\mathcal{D}\mathcal{P}_{A};M^J(S^{1,1}_{\sbt}))&&S^{1,1}\wedge \uuline{\THR}(\mathcal{D}\mathcal{P}_{A};M^J)\ar[ll]_-{\simeq_s}^-{\ref{lastmapdiag}}
}\]
The symbol $\simeq_s$ indicates a stable $G$-equivalence of real spectra, $\simeq$ a levelwise $G$-equivalence.

The terms and maps in the diagram will be defined below. They make sense in the generality of a category 
$(C,D)$ enriched in Abelian groups with strict duality, and a bimodule $M\colon C^{op}\otimes C\longrightarrow Ab$ with duality $J$ (cf. \ref{dualitymodabelian}).

\begin{defn}
The category $C\ltimes M$ has the same objects as $C$, and morphism groups
\[(C\ltimes M)(c,c')=C(c,c')\oplus M(c,c')\]
Composition of $(f,m)\colon c\longrightarrow c'$ and $(f',m')\colon c'\longrightarrow c''$ is defined by
\[(f',m')\circ(f,m)=(f\circ f,f'\cdot m+m'\cdot f)\]
The duality $(D,J)\colon (C\ltimes M)^{op}\longrightarrow C\ltimes M$ is $D$ on objects and $D\oplus J$ on morphisms.
\end{defn}
Projection off the $M$-component gives a functor $C\ltimes M\longrightarrow C$, and $\widetilde{\THR}(C\ltimes M)$ is the homotopy fiber of the induced map.

\begin{ex}\label{generalfish}
For $C=\mathcal{D}\mathcal{P}_A$ and $M^J$ the bimodule associated to an $M$-twisting $J$ of $(A,L,\alpha)$, this construction gives an equality
\[(\mathcal{D}\mathcal{P}_A)\ltimes M^J=\mathcal{D}(\mathcal{P}(A;M))^0\]
(cf. §\ref{THRrings} and \ref{model1})
where we remind that $\mathcal{D}(\mathcal{P}(A;M))^0$ is the full subcategory of the category $\mathcal{D}\mathcal{P}(A;M)$ on objects of the form $(\phi,0)$.\\
We used this category as a model for $\mathcal{D}\mathcal{P}_{A\ltimes M}$, when we defined an equivalence of categories with duality
\[(\mathcal{D}\mathcal{P}_A)\ltimes M^J=\mathcal{D}(\mathcal{P}(A;M))^0\stackrel{\simeq}{\longrightarrow}\mathcal{D}\mathcal{P}_{A\ltimes M}\]
in \ref{inclfullsubcat} and \ref{modelPAM}. This induces a $G$-homotopy equivalence
\[\THR((\mathcal{D}\mathcal{P}_A)\ltimes M^J)\stackrel{\simeq}{\longrightarrow}\THR(\mathcal{D}\mathcal{P}_{A\ltimes M})\]
since $\THR$ preserves equivalences (cf. \ref{THRpreservesGequiv}).
More generally, we can apply the $S^{2,1}_{\sbt}$-construction. By \ref{noneqeq} and \ref{inclfullsubcat} we get a $G$-equivalence
\[\THR((S^{2,1}_k\mathcal{D}\mathcal{P}_A)\ltimes M^J)\stackrel{\simeq}{\longrightarrow}\THR(S^{2,1}_k\mathcal{D}\mathcal{P}_{A\ltimes M})\]
for every $k$.
The construction clearly extends to the deloopings of $\THR$ of §\ref{firstdeloop}, giving a levelwise equivalence of real spectra
\[\uuline{\THR}((\mathcal{D}\mathcal{P}_A)\ltimes M^J)\stackrel{\simeq}{\longrightarrow}\uuline{\THR}(\mathcal{D}\mathcal{P}_{A\ltimes M})\]
This is an equivalence over $\uuline{\THR}(\mathcal{D}\mathcal{P}_{A})$, and its restriction to the homotopy fiber is the
top horizontal map in diagram $(1)$.
\end{ex}

Let $\mathcal{C}$ be an $Sp^\Sigma$-category with duality and $\mathcal{M}$ a bimodule with duality over it (cf. §\ref{defTHR}).

\begin{defn}
Define $\mathcal{C}\vee \mathcal{M}$ as the $Sp^\Sigma$-category with the same object as $\mathcal{C}$, and morphism spectra
\[(\mathcal{C}\vee \mathcal{M})(c,c')=\mathcal{C}(c,c')\vee \mathcal{M}(c,c')\]
Composition is defined by the composite
\[\xymatrix{(\mathcal{C}(c',c'')\vee \mathcal{M}(c',c''))\wedge(\mathcal{C}(c,c')\vee \mathcal{M}(c,c'))\ar[d]^-{\cong}\\
 *\txt{$(\mathcal{C}(c',c'')\wedge \mathcal{C}(c,c'))\vee (\mathcal{C}(c',c'')\wedge \mathcal{M}(c,c'))\vee$\\
 $(\mathcal{M}(c',c'')\wedge \mathcal{C}(c,c'))\vee(\mathcal{M}(c',c'')\wedge \mathcal{M}(c,c'))$}\ar[d]^-{\circ\vee l\vee r\vee\ast}\\
 \mathcal{C}(c,c'')\vee \mathcal{M}(c,c'')\vee \mathcal{M}(c,c'')\ar[d]^-{\id\vee\mbox{fold}}\\
\mathcal{C}(c,c'')\vee \mathcal{M}(c,c'')
}\]
where $l$ and $r$ denote respectively left and right action of $\mathcal{C}$ on the bimodule $\mathcal{M}$.
The identity of an object $c$ is its identity $\mathbb{S}\longrightarrow\mathcal{C}(c,c)$ in $\mathcal{C}$ followed by the inclusion in the wedge. There is a duality on $\mathcal{C}\vee \mathcal{M}$ defined by $D$ on objects and by $D\vee J$ on morphisms.
\end{defn}

The inclusion of wedges into products defines a diagram of duality preserving functors
\[\xymatrix{HC\vee HM\ar[r]\ar[dr]& H(C\ltimes M)\ar[d]\\
&HC
}\]
The horizontal functor is a stable $G$-equivalence by \ref{wedgesintoproducts}, and it induces a $G$-equivalence
\[\widetilde{\THR}(HC\vee HM)\stackrel{\simeq}{\longrightarrow} \widetilde{\THR}(C\ltimes M)\]
The same construction on the real spectra gives the levelwise equivalence
\[\widetilde{\uuline{\THR}}(HC\vee HM)\stackrel{\simeq}{\longrightarrow} \widetilde{\uuline{\THR}}(C\ltimes M)\]
appearing as the top right vertical map in diagram $(1)$.

We next define $\uuline{\THR}^{(a)}(\mathcal{C};\mathcal{M})$.
We remind from §\ref{thhsec} that $V(\mathcal{C}\vee \mathcal{M},\underline{i}):=V(\mathcal{C}\vee \mathcal{M};\mathcal{C}\vee \mathcal{M},\underline{i})$ is a wedge sum indexed over $k$-tuples of objects of $\mathcal{C}\vee \mathcal{M}$. Hence the $(c_0,\dots,c_k)$-wedge component of $V(\mathcal{C}\vee \mathcal{M},\underline{i})$ is a wedge sum of spaces of the form
\[\mathcal{F}^0(c_0,c_k)_{i_0}\wedge \mathcal{F}^1(c_1,c_0)_{i_1}\wedge\dots\wedge \mathcal{F}^k(c_{k},c_{k-1})_{i_k}\]
where $\mathcal{F}^l$ is either $\mathcal{C}$ or $\mathcal{M}$.
For every integer $a\geq 0$ let $V^{(a)}(\mathcal{C};\mathcal{M},\underline{i})$ be the subspace of $V(\mathcal{C}\vee \mathcal{M},\underline{i})$ of wedge summands with exactly $a$ smash factors $\mathcal{M}$. For example
\[\begin{array}{llll}V^{(1)}(\mathcal{C};\!\mathcal{M},\!\underline{i})=\!\!\!\!\bigvee\limits_{c_0,\dots,c_k}&\!\!\!\!\!\!\!\!(\mathcal{M}(c_0,c_k)_{i_0}\wedge\mathcal{C}(c_1,c_0)_{i_1}\wedge\dots\wedge\mathcal{C}(c_k,c_{k-1})_{i_k})\\
&\!\!\!\!\!\!\!\!\!\!\!\vee(\mathcal{C}(c_0,c_k)_{i_0}\wedge\mathcal{M}(c_1,c_0)_{i_1}\!\wedge\!\mathcal{C}(c_2,c_1)_{i_2}\!\wedge\!\dots\!\wedge\!\mathcal{C}(c_k,c_{k-1})_{i_k})\\
&\ \ \ \ \ \ \ \ \ \ \ \ \ \ \ \ \ \ \ \ \ \ \ \ \ \ \ \ \  \vdots\\
&\!\!\!\!\!\!\!\!\!\!\!\vee(\mathcal{C}(c_0,c_k)_{i_0}\wedge\dots\wedge\mathcal{C}(c_{k-1},c_{k-2})_{i_{k-1}}\wedge\mathcal{M}(c_k,c_{k-1})_{i_k})
\end{array}\]
Define the functor $\mathcal{G}^{(a)}_k(\mathcal{C};\mathcal{M})\colon I[k]\longrightarrow \Top_\ast$ by
\[\underline{i}\longmapsto\mathcal{G}^{(a)}_k(\mathcal{C};\mathcal{M})(\underline{i})=\map_\ast(S^{\underline{i}},V^{(a)}(\mathcal{C};\mathcal{M},\underline{i}))\]
and set
\[\THR^{(a)}_k(\mathcal{C};\mathcal{M})=\hocolim_{I[k+1]}\mathcal{G}^{(a)}_k(\mathcal{C};\mathcal{M})\]
The $G$-action, the cyclic action and the simplicial structure on $\THR_{\sbt}(\mathcal{C}\vee\mathcal{M})$ all restrict to define a dihedral space structure on the simplicial space $\THR^{(a)}_{\sbt}(\mathcal{C};\mathcal{M})$. The $G$-space $\THR^{(a)}(\mathcal{C};\mathcal{M})$ is its geometric realization. 
The same construction for the spaces $V^{(a)}(\mathcal{C};\mathcal{M},\underline{i})\wedge S^{2m,m}$ leads to a real spectrum $\uuline{\THR}^{(a)}(\mathcal{C};\mathcal{M})$.
The inclusion of $V^{(a)}(\mathcal{C};\mathcal{M},\underline{i})$ in $V(\mathcal{C}\vee\mathcal{M},\underline{i})$ induces a map of real spectra
\[\uuline{\THR}^{(a)}(\mathcal{C};\mathcal{M})\longrightarrow \uuline{\THR}(\mathcal{C}\vee\mathcal{M})\]
for every $a\geq 0$. The map $\pr^{(1)}$ in diagram $(1)$ is the projection onto the $a=1$ wedge component.

\begin{prop}\label{htpyfibwedgesthr}
If $\mathcal{C}$ and $\mathcal{M}$ are $G$-connected, the canonical map
\[\bigvee_{a=0}^\infty\uuline{\THR}^{(a)}(\mathcal{C};\mathcal{M})\stackrel{\simeq_s}{\longrightarrow} \uuline{\THR}(\mathcal{C}\vee\mathcal{M})\]
is a stable $G$-equivalence of real spectra, and induces a stable $G$-equivalence
\[\bigvee_{a=1}^\infty\uuline{\THR}^{(a)}(\mathcal{C};\mathcal{M})\stackrel{\simeq_s}{\longrightarrow}\widetilde{\uuline{\THR}}(\mathcal{C}\vee\mathcal{M})\]
\end{prop}

\begin{proof}
Since wedges commute with realizations, it is enough to prove that
\[\bigvee_{a=0}^\infty\uuline{\THR}^{(a)}_{2k+1}(\mathcal{C};\mathcal{M})\longrightarrow \uuline{\THR}_{2k+1}(\mathcal{C}\vee\mathcal{M})\]
is a stable $G$-equivalence for all $k$. For every $\underline{i}\in I[2k+1]$ and $m\geq 0$ there is a commutative diagram
\[\xymatrix{\bigvee_{a=0}^\infty\THR^{(a)}_{2k+1}(\mathcal{C};\mathcal{M},S^{2m,m})\ar[dd]&\bigvee_{a=0}^\infty\Omega^{\underline{i}}(V^{(a)}(\mathcal{C};\mathcal{M},\underline{i})\wedge S^{2m,m})\ar[l]\ar[d]\\
& \Omega^{\underline{i}}\bigvee_{a=0}^\infty (V^{(a)}(\mathcal{C};\mathcal{M},\underline{i})\wedge S^{2m,m})\ar[d]^{\cong}\\
\THR_{2k+1}(\mathcal{C}\vee\mathcal{M},S^{2m,m})&\Omega^{\underline{i}}(V(\mathcal{C}\vee\mathcal{M},\underline{i})\wedge S^{2m,m})\ar[l]
}\]
By the $G$-approximation lemma, one can choose $\underline{i}$ big enough so that the horizontal maps are as connected as we like. The top right vertical map is $(4m-1,2m-1)$-connected by \ref{wedgesandmaps}.
On realizations, the map
\[\bigvee_{a=0}^\infty\THR^{(a)}(\mathcal{C};\mathcal{M},S^{2m,m})\longrightarrow \THR(\mathcal{C}\vee\mathcal{M},S^{2m,m})\]
is also $(4m+1,2m-1)$-connected, so that
\[\Omega^{2m,m}\bigvee_{a=0}^\infty\THR^{(a)}(\mathcal{C};\mathcal{M},S^{2m,m})\longrightarrow \Omega^{2m,m}\THR(\mathcal{C}\vee\mathcal{M},S^{2m,m})\]
is $(2m+1,m-1)$-connected. This tends to infinity with $m$.
\end{proof}

We now come to the stable equivalence
\[S^{1,1}_+\wedge\uuline{\THR}(\mathcal{C};\mathcal{M})\longrightarrow\uuline{\THR}^{(1)}(\mathcal{C};\mathcal{M})\]
For $\underline{i}\in I[k]$ consider the homeomorphism
\[\phi\colon (S^{1,1}_k)_+\wedge V(\mathcal{C};\mathcal{M},\underline{i})\stackrel{\cong}{\longrightarrow}V^{(1)}(\mathcal{C};\mathcal{M},\underline{i})\]
that sends $(\sigma, m\wedge f_1\wedge\dots\wedge f_k)$ to the unique cyclic permutation
\[f_{k+1-\sigma}\wedge\dots\wedge f_{k}\wedge m\wedge f_1\wedge\dots\wedge f_{k-\sigma} \]
that has $m$ in position $\sigma$, upon using the bijection $S^{1,1}_k\cong\{0,\dots,k\}$ that counts the number of $1$-values of $\sigma\colon[k]\longrightarrow[1]$ (remember that the $\sigma\equiv 0$ and $\sigma\equiv 1$ are identified).
It induces a diagram of simplicial maps
\[\xymatrix{(S^{1,1}_k)_+\!\wedge\!\THR_k(\mathcal{C};\!\mathcal{M},\!S^{2m,m})\ar[dd]\ar[r]^-{\cong}& \hocolim\limits_{I[k]}(S^{1,1}_k)_+\!\wedge\!\Omega^{\underline{i}}(V(\mathcal{C};\!\mathcal{M},\!\underline{i})\!\wedge\! S^{2m,m})\ar[dd]\\
\ar@{}[r]|-{(2)}&
\\
\THR^{(1)}_k(\mathcal{C};\!\mathcal{M},\!S^{2m,m})&\hocolim\limits_{I[k]}\Omega^{\underline{i}}((S^{1,1}_k)_+\!\wedge\! V(\mathcal{C};\!\mathcal{M},\!\underline{i})\!\wedge\!S^{2m,m})\ar[l]^-{\cong}_-{\phi}
}\]
\begin{prop}\label{s11smashthr}
If $\mathcal{C}$ and $\mathcal{M}$ are $G$-connected,
the realization of the above map
\[S^{1,1}_+\wedge\uuline{\THR}(\mathcal{C};\mathcal{M})\stackrel{\simeq_s}{\longrightarrow}\uuline{\THR}^{(1)}(\mathcal{C};\mathcal{M})\]
is a stable $G$-equivalence of real spectra.
\end{prop}
\begin{proof}
The right-hand vertical map of diagram $(2)$ above is $(4m+1,2m-1)$-connected by \ref{wedgesandmaps}. Therefore its realization is also $(4m+1,2m-1)$-connected, so that 
\[\Omega^{2m,m}(S^{1,1}_+\wedge\THR(\mathcal{C};\mathcal{M},S^{2m,m}))\longrightarrow\Omega^{2m,m}\THR^{(1)}(\mathcal{C};\mathcal{M},S^{2m,m})\]
is $(2m+1,m-1)$-connected, which tends to infinity with $m$.
\end{proof}
There is one last map to discuss in diagram $(1)$, namely
\[S^{1,1}\wedge \uuline{\THR}(C;M)\longrightarrow\uuline{\THR}(C;M(S^{1,1}_{\sbt}))\]
It is the realization of the simplicial map defined in degree $k$ by
\[\xymatrix{S^{1,1}_k\!\wedge\!\THR_k(C;\!M,\!S^{2m,m})\ar[dd]\ar[r]^-{\cong} &\hocolim\limits_{I[k]}S^{1,1}_k\!\wedge\!\Omega^{\underline{i}}(V(HC;\!HM,\!\underline{i})\!\wedge\!S^{2m,m})\ar[dd]^{\simeq_s}\\
\ar@{}[r]|-{(3)}&\\
\THR_k(C;\!M(S^{1,1}_k),\!S^{2m,m})&\hocolim\limits_{I[k]}\Omega^{\underline{i}}(S^{1,1}_k\!\wedge\!V(HC;\!HM,\!\underline{i})\!\wedge\!S^{2m,m})\ar[l]\\
}\]
The right vertical map is a stable equivalence by the proof of \ref{s11smashthr}.
The bottom horizontal map is induced by the canonical maps
\[M(c_0,c_k)(S^{i_0})\wedge S^{1,1}_k\longrightarrow M(c_0,c_k)(S^{i_0}\wedge S^{1,1}_k)\]

\begin{prop}\label{lastmapdiag} 
The map 
\[S^{1,1}\wedge \uuline{\THR}(C;M)\stackrel{\simeq_s}{\longrightarrow}\uuline{\THR}(C;M(S^{1,1}_{\sbt}))\]
is a stable $G$-equivalence of real spectra.
\end{prop}
\begin{proof}
We show that the subdivision of the bottom map in diagram $(3)$ is a levelwise $G$-equivalence.
The canonical map
\[M(c_0,c_k)(S^{i_0})\wedge S^{1,1}_{2k+1}\longrightarrow M(c_0,c_k)(S^{i_0}\wedge S^{1,1}_{2k+1})\]
is $(2i_0-1,i_0-1)$-connected
by an easy case of \ref{connectvitysmashwithx}. 
Thus the map 
\[\Omega^{\underline{i}}(S^{1,1}_k\wedge V(HC;HM,\underline{i})\wedge S^{2m,m})\longrightarrow \Omega^{\underline{i}}(V(HC;HM(S^{1,1}_k),\underline{i})\wedge S^{2m,m})\]
is
\[(i_0+2m-1,\lfloor\frac{i_{0}}{2}\rfloor-2+m-1)\]
connected. This tends to infinity on $I[2k+1]$, and therefore it is an equivalence on the homotopy colimit for all $m$ and $k$ by \ref{connectgoestoinfty}.
\end{proof}

\begin{proof}[Proof of \ref{comparisontrbeta}] To show that $\alpha\circ \widetilde{\tr}=\beta$ in the homotopy category,
we build a diagram of $G$-maps below, that commutes up to $G$-homotopy. The dashed maps are yet to be defined. Recall that
\[\widetilde{S}^{2,1}_{\sbt}(A;M)=|\bigvee\limits_{\varphi\in S^{2,1}_{\sbt}\mathcal{D}\mathcal{P}_A}\hom_A(\varphi,\varphi\otimes_A M)|\]
so that by definition $\widetilde{\KR}(A;M)=\Omega^{2,1}\widetilde{S}^{2,1}_{\sbt}(A;M)$.
\newpage
\phantom{a}
\vspace{-3.4cm}
\[\rotatebox{270}{\scalebox{.8}[.82]{\xymatrix{
&&&&&&
\\
\widetilde{\KR}(A\!\ltimes\!\!M)&
\Omega^{2,1}\widetilde{|ObS^{2,1}_{\sbt}\mathcal{D}\mathcal{P}_{A\ltimes\! M}|}\ar@{-->}[dddl]^{\widetilde{\gamma}}\ar[l]^-{\simeq}_-{\ref{inclobinisos21}}\ar[r]^-{\widetilde{\tr}}\ar@{}@<-6ex>[dddl]|{(I)}
&
\Omega^{2,1}\widetilde{\THR}(S^{2,1}_{\sbt}\mathcal{D}\mathcal{P}_{A\ltimes\! M})\ar@{}@<-20ex>[ddd]|{(II)}\ar `u[u] `[urrrr]^-{\alpha} `[dddddddddrrrr] [dddddddddrrr]&
\ar@{}[ddd]|{(III)}&
\Omega^{2,1}\widetilde{\THR}(HS^{2,1}_{\sbt}\mathcal{D}\mathcal{P}_{A}\!\vee\! HM)\ar[ll]^-{\simeq}_-{\ref{generalfish}}&
\widetilde{\THR}(H\mathcal{D}\mathcal{P}_{A}\!\vee\! HM)\ar[l]^-{\simeq}_-{\ref{THRdeloopings}}\ar@/^2pc/[ddd]^{pr}\\
\\
\\
\widetilde{\KR}(A;\!M(S^{1,1}_{\sbt}))\ar@{-->}[rr]^-{s}\ar[dddrr]_-{=}\ar[uuu]_-{\simeq}^-{\ref{relKR}}&&
\widetilde{\KR}(A;\!M(S^{1,1}_{\sbt\,+}))\ar[ddd]^{\pr}\ar@{-->}[uuu]_{\widetilde{\Phi}}&
\Omega^{2,1}(S^{1,1}_+\!\wedge\widetilde{S}^{2,1}_{\sbt}(A;\!M))\ar[l]_-{\ref{connectvitysmashwithx}}\ar[r]^-{\id\wedge\beta}\ar[ddd]^{\pr}&
\Omega^{2,1}(S^{1,1}_+\!\wedge\THR(S^{2,1}_{\sbt}\mathcal{D}\mathcal{P}_{A};\!M))\ar@{-->}[uuu]_{\tau}\ar[ddd]^{\pr}&
S^{1,1}_+\!\wedge\THR(\mathcal{D}\mathcal{P}_{A};\!M)\ar[l]\ar@{-->}[uuu]_{\tau}\ar[ddd]^{\pr}\\
\\
\\
&&\widetilde{\KR}(A;\!M(S^{1,1}_{\sbt}))\ar[dddrr]_-{\beta} &
\Omega^{2,1}(S^{1,1}\!\wedge\widetilde{S}^{2,1}_{\sbt}(A;\!M))\ar[l]_-{\ref{connectvitysmashwithx}}\ar[r]^-{\id\wedge\beta}&
\Omega^{2,1}(S^{1,1}\!\wedge\THR(S^{2,1}_{\sbt}\mathcal{D}\mathcal{P}_{A};\!M))\ar[ddd]_{\simeq_s}^{\ref{lastmapdiag}}
& S^{1,1}\!\wedge\THR(\mathcal{D}\mathcal{P}_{A};\!M)\ar[l]\ar[ddd]_{\simeq_s}^{\ref{lastmapdiag}}\\
\\
\\
&&&& \Omega^{2,1}\THR(S^{2,1}_{\sbt}\mathcal{D}\mathcal{P}_{A};\!M(S^{1,1}_{\sbt})) &
\THR(\mathcal{D}\mathcal{P}_{A};\!M(S^{1,1}_{\sbt}))\ar[l]_-{\simeq}^-{\ref{THRdeloopings}}&
}}}\]
\newpage

The weak map $\alpha$ is the composition of the outer arrows of the diagram from $\Omega^{2,1}\widetilde{\THR}(S^{2,1}_{\sbt}\mathcal{D}\mathcal{P}_{A\ltimes M})\simeq \widetilde{\THR}(\mathcal{D}\mathcal{P}_{A\ltimes M})$ to $\THR(\mathcal{D}\mathcal{P}_{A};M(S^{1,1}_{\sbt}))$.
The unlabeled squares commute by naturality of our constructions. We define the dashed maps and show that the squares $(I)$, $(II)$ and $(III)$ commute up to $G$-homotopy.

The map $s\colon \widetilde{\KR}(A; M(S^{1,1}_{\sbt}))\longrightarrow
\widetilde{\KR}(A; M(S^{1,1}_{\sbt\,+}))$ appearing at the top left of the diagram is induced by the canonical section \[  M(S^{1,1}_{\sbt})\cong \mathcal{N}_{\sbt} M\longrightarrow \mathcal{N}^{cy}_{\sbt} M\cong M(S^{1,1}_{\sbt\,+})\] of the projection $M(S^{1,1}_{\sbt\,+})\longrightarrow M(S^{1,1}_{\sbt})$. It sends an element $(m_1,\dots,m_k)$ to $(-\sum_im_i,m_1,\dots,m_k)$.

The map $\widetilde{\gamma}$ in the top triangle is the restriction (to the homotopy fiber) of the realization of a simplicial functor
\[\gamma_n\colon ObS^{2,1}_n\mathcal{D}\mathcal{P}_{A\ltimes M}\longrightarrow \coprod\limits_{\varphi\in ObS^{2,1}_n\mathcal{D}\mathcal{P}_{A}}\hom_A(\varphi,\varphi\otimes_A M)=S^{2,1}_n(A;M)\]
where $ObS^{2,1}_n\mathcal{D}\mathcal{P}_{A\ltimes M}$ is considered as a discrete category (i.e. only identity morphisms).
It sends an object $(H,K,\nu)$ to the object $(H\!\otimes_{A\!\ltimes\!M}\!A,K\!\otimes_{A\!\ltimes\!M}\!A,\gamma_n(\nu))$, where $\gamma_n(\nu)$ is the composite
\[\resizebox{1\hsize}{!}{\xymatrix{K_\theta\otimes_{A\ltimes\!M}\!A\ar[drr]_{\gamma_n(\nu)_\theta}\ar[r]^-{\nu_\theta\otimes A}&D_{L^J}(H_\theta)\otimes_{A\ltimes\!M}\!A\ar[r]^-{\cong}& \hom_A(H_\theta\otimes_{A\ltimes\!M}\!A,(L^{J}_t\otimes_{A\ltimes \!M}\!A)_s)\ar[d]^\cong\\
&&\hom_A(H_\theta\otimes_{A\ltimes M}\!A,L_s)}}\]
where the unlabeled isomorphisms are the canonical ones.

Let us prove that triangle $(I)$ commutes up to $G$-homotopy. It is the equivalence $\Psi$ of \ref{mainKR} that induces the vertical map of $(I)$.
Consider the diagram
\[\xymatrix{iS^{2,1}_{\sbt}\mathcal{D}\mathcal{P}_{A\ltimes M}\ar[r]_-{F'}^-{\simeq}&i\mathcal{D}T^{2,1}_{\sbt}(A;M)\ar@/^-1.3pc/[l]_-{\mathcal{D}\overline{F}}\\
ObS^{2,1}_{\sbt}\mathcal{D}\mathcal{P}_{A\ltimes M}\ar@{^{(}->}[u]_-{\iota'}\ar[r]_-{\gamma}&S^{2,1}_{\sbt}(A;M)\ar@{^{(}->}[u]_-{\iota}
}\]
where $F'$ is an inverse of $\mathcal{D}\overline{F}$ construct in the proof of \ref{noneqeq}, $\iota$ is the inclusion functor of §\ref{model}, and $\Psi=\mathcal{D}\overline{F}\circ\iota$. Since $\mathcal{D}\overline{F}$ and $F'$ are mutually inverse equivalences of categories with duality, there is a $G$-natural isomorphism $\kappa\colon \mathcal{D}\overline{F}\circ F'\Rightarrow \id$. Restricting to the discrete category of objects, it is gives $G$-natural isomorphism
\[\kappa\colon \Psi\circ\gamma=\mathcal{D}\overline{F}\circ\iota\circ\gamma =\mathcal{D}\overline{F}\circ F'\circ \iota'\Rightarrow \iota'\]
inducing a $G$-homotopy between $\widetilde{\Psi}\circ\widetilde{\gamma}$ and the inclusion of the objects of $S^{2,1}_{\sbt}\mathcal{D}\mathcal{P}_{A\ltimes M}$.

The map $\widetilde{\Phi}$ in the diagram is the restriction to the homotopy fiber of the map induced by the simplicial map
\[\Phi_k\colon \coprod\limits_{ObS^{2,1}_{\sbt}\mathcal{D}\mathcal{P}_{A}}\hom_A(\varphi,\varphi\otimes_A M((S^{1,1}_k)_+))\longrightarrow \THR_k(S^{2,1}_{\sbt}\mathcal{D}\mathcal{P}_{A\ltimes M})\]
that sends $(h_0,\dots,h_k)\in \hom_A(\varphi,\varphi\otimes_A M)^{\oplus(k+1)}$ to the element of the homotopy colimit defined by
\[\left(\begin{array}{ll}1&0\\h_0&1\end{array}\right)\wedge\dots\wedge \left(\begin{array}{ll}1&0\\h_k&1\end{array}\right)\]
in the $(\Psi(\varphi),\dots,\Psi(\varphi))$-wedge component of $V(HS^{2,1}_{\sbt}\mathcal{D}\mathcal{P}_{A\ltimes M},0)$. Each $h_i$ is actually a pair $(f_i,g_i)$ with $f_i\colon X\longrightarrow X\otimes_AM$ and $g_i\colon Y\longrightarrow Y\otimes_A M$, and
\[\left(\begin{array}{ll}1&0\\h_i&1\end{array}\right)=\Psi(f_i,g_i)=\left(\left(\begin{array}{ll}1&0\\f_i&1\end{array}\right),\left(\begin{array}{ll}1&0\\g_i&1\end{array}\right)\right)\]
With this definition square $(II)$ commutes up to $G$-homotopy. Indeed, an object $\nu\in S^{2,1}_{\sbt}\mathcal{D}\mathcal{P}_{A\ltimes M}$ is sent by $\Phi s \gamma$ to
\[\Phi s (0_{\gamma(\nu)})=\Phi (0_{\gamma(\nu)})=\id_{\Psi\gamma(\nu)}\wedge\dots\wedge\id_{\Psi\gamma(\nu)}\]
The trace map sends $\nu$ to
\[\tr(\nu)=\id_{\nu}\wedge\dots\wedge\id_{\nu}\]
The $G$-natural isomorphism $\Psi\circ\gamma\Rightarrow \iota'$ described above induces a simplicial $G$-homotopy between $\Phi s G$ and $\tr$, defined in a completely analogous way to the simplicial homotopy of \ref{THRpreservesGequiv}.

The map $\tau\colon S^{1,1}_+\wedge\THR(C;M)\longrightarrow \widetilde{\THR}(HC\vee HM)$ is the composite
\[\xymatrix{S^{1,1}_+\wedge\THR(C;M)\ar[ddr]_{\tau}\ar[r]^-{\simeq_s}_-{\ref{s11smashthr}}& \THR^{(1)}(C;M)\ar[d]\\
&\bigvee_{a\geq 1}\THR^{(a)}(C;M)\ar[d]^-{\simeq_s}_{\ref{htpyfibwedgesthr}}\\
& \widetilde{\THR}(HC\vee HM)}\]
where the unlabeled map is the inclusion in the $a=1$ wedge summand. It is easy to check that square $(III)$ commutes. Indeed, an element $\sigma\wedge h\in (S^{1,1}_k)_+\wedge\hom_A(\varphi,\varphi\otimes_A M)$ is sent by both composites to the element of $\THR_k(S^{2,1}_{\sbt}\mathcal{D}\mathcal{P}_{A\ltimes M})$ represented by
\[\id_{\Psi(\varphi)}\wedge\dots\wedge\id_{\Psi(\varphi)}\wedge\left(\begin{array}{ll}1&0\\h&1\end{array}\right)
\wedge\id_{\Psi(\varphi)}\wedge\dots\wedge \id_{\Psi(\varphi)}\]
where the matrix is in the smash factor in position $\sigma$.
\end{proof}

\newpage


\section{Appendix}

\subsection{Quillen modules over antistructures and $M$-twistings}\label{appendixdualities}

In \cite{quillenmod} Quillen defines a category of bimodules over an object $X$ of an arbitrary category $C$. Here
we show that in the case of the category of antistructures $C$ (cf. §\ref{ringduality}), the category of bimodules over $(A,L,\alpha)$ is equivalent to the category of twistings of $(A,L,\alpha)$ of §\ref{squarezeroext}.

\begin{defn}[\cite{quillenmod}] Let $C$ be a category with finite limits and $X$ an object of $C$. The \textbf{category of bimodules over $X$} is the category
\[\modu(X)=(C/X)^{ab}\]
of abelian group objects in the over category $C/X$ (cf. \cite[III-§6]{maclane}).
\end{defn}

Let $\mathcal{T}_{(A,L,\alpha)}$ be the category of twistings of an antistructure $(A,L,\alpha)$. An object is a pair $(M,J)$ of an $A$-bimodule $M$ and an $M$-twisting $J\colon L_t\otimes_AM\longrightarrow L_t\otimes_AM$ (cf. \ref{defmaptwist}). A morphism from $(M,J)$ to $(M',J')$ is a map of bimodules $f\colon M\longrightarrow M'$ such that
\[\xymatrix{L_t\otimes_AM\ar[d]_{J}\ar[r]^-{L_t\otimes f}&L_t\otimes_AM'\ar[d]^{J'}\\
L_t\otimes_AM\ar[r]_-{L_t\otimes f}&L_t\otimes_AM'
}\]
commutes.

\begin{theorem}\label{quillenmodstory}
Let $(A,L,\alpha)$ be an antistructure. There is an equivalence of categories
\[\Xi\colon \mathcal{T}_{(A,L,\alpha)}\longrightarrow \modu(A,L,\alpha)\]
\end{theorem}

Explicitly, a module over $X\in C$ is a quadruple $(\theta,\eta,e,(-)^{-1})$ of a map $\theta\colon Y\longrightarrow X$ in $C$ and maps $\eta\colon Y\times_XY\longrightarrow Y$, $e\colon X\longrightarrow Y$ and $(-)^{-1}\colon Y\longrightarrow Y$ over $X$ satisfying conditions expressing associativity and commutativity of $\eta$, that $e$ is a unit for $\eta$ and that $(-)^{-1}$ is inversion for the addition $\eta$.

The functor $\Xi$ at a twisting $(M,J)$ is defined as the following quadruple $(\theta,\eta,e,(-)^{-1})$. A map of antistructures is a pair of maps, one defined between the underlying rings and one between the underlying modules, satisfying certain conditions (cf. \ref{mapofantistruct}). The map of antistructures $\theta$ is the projection
\[\theta\colon (A\ltimes M,L^J,\alpha^J)\longrightarrow (A,L,\alpha)\]
both on the ring and on the module.
We remind that $(A\ltimes M,L^J,\alpha^J)$ is the antistructure defined in §\ref{squarezeroext}, with $L^J=L_t\oplus L_t\otimes_AM$ and $\alpha^J=(\alpha, J)$. In order to define the map $\eta$ we notice that the pull-back of antistructures is componentwise, in particular
\[(A\ltimes M,L^J,\alpha^J)\times_{(A,L,\alpha)}(A\ltimes M,L^J,\alpha^J)\!=\!(A\ltimes M\times _AA\ltimes M,L^J\times_LL^J,\alpha^J\times\alpha^J)\]
The group map $\eta$ is the map of antistructures
\[\eta=(\eta_1,\eta_2)\colon (A\ltimes M\times _AA\ltimes M,L^J\times_LL^J,\alpha^J\times\alpha^J)\longrightarrow (A\ltimes M,L^J,\alpha^J)\]
defined by $\eta_1(a,m,a,m')=(a,m+m')$ and $\eta_2(l,l'\otimes m,l,l''\otimes m')=(l,l'\otimes m+l''\otimes m')$.
The unit map
\[e\colon (A,L,\alpha)\longrightarrow (A\ltimes M,L^J,\alpha^J)\]
is the zero section both on $A$ and on $L$. Inversion
\[(-)^{-1}\colon (A\ltimes M,L^J,\alpha^J)\longrightarrow (A\ltimes M,L^J,\alpha^J)\]
sends $(a,m)$ to $(a,-m)$ on the ring and $(l,l'\otimes m)$ to $(l,-l'\otimes m)$ on the module.
On morphisms, $\Xi$ sends a map of twistings $f\colon M\longrightarrow M'$ to the morphism of $(A,L,\alpha)$-bimodules defined by the map of antistructures
\[(A\ltimes f,L_t\oplus L_t\otimes f)\colon (A\ltimes M,L^J,\alpha^J)\longrightarrow (A\ltimes M',L^{J'},\alpha^{J'})\]

We show below that $\Xi$ is fully faithful and essentially surjective. For essential surjectivity we need to know more about the antistructures with underlying module $L_t\oplus L_t\otimes_AM$.

\begin{defn}\label{deffulltwist}
A \textbf{full twisting of an antistructure $(A,L,\alpha)$} is a pair $(J,K)$ of additive maps $J\colon L_t\otimes_A M\longrightarrow L_t\otimes_A M$ and $K\colon L\longrightarrow L_t\otimes_A M$ satisfying
\begin{enumerate}
\item $J(l\otimes m\cdot a\otimes a')=J(l\otimes m)\cdot a'\otimes a$
\item $J\circ J=\id$
\item $K(l\cdot a\otimes a')=K(l\cdot a\otimes 1)\cdot a'\otimes 1+K(l\cdot 1\otimes a')\cdot 1\otimes a-K(l)\cdot a'\otimes a$
\item $K\circ\alpha=-J\circ K$
\end{enumerate}
We say that $(J,K)$ is \textbf{regular} if $K$ defines a module map $K\colon L_s\longrightarrow (L_t\otimes_A M)_t$, that is if
\[K(l\cdot 1\otimes a)=K(l)\cdot a\otimes 1\]
for all $l\in L$ and $a\in A$.
\end{defn}

\begin{rem}
If $(J,K)$ is regular, the identity $K=-JK\alpha$ implies that $K(l\cdot a\otimes 1)=K(l)\cdot 1\otimes a$ as well. Notice also that an $M$-twisting is a full twisting with $K=0$. 
\end{rem}

We generalize the construction of the antistructure $(A\ltimes M,L^{J},\alpha^J)$ of §\ref{squarezeroext}. From a full twisting $(J,K)$, define another $A\ltimes M$-module structure on $F(L_t)=L_t\oplus L_t\otimes_AM$ by defining $(l,l'\otimes m)\cdot 1\otimes (a,n)$ to be
\[(l\cdot 1\otimes a,l'\otimes m\cdot 1\otimes a+J(\alpha(l)\otimes n)+K(\alpha(l)\cdot a\otimes 1)-K(\alpha(l))\cdot 1\otimes a)\]
This module structure together with the standard structure of $F(L_t)$ define a $A\ltimes M\otimes A\ltimes M$-module denoted $L^{J,K}$. Define a map $\alpha^{J,K}\colon L^{J,K}\longrightarrow L^{J,K}$ by
\[\alpha^{J,K}(l,l'\otimes m)=(\alpha(l),J(l'\otimes m)+K(l))\]
Notice that if the full twisting is regular, $L^{J,K}=L^J$, but $\alpha^{J,K}$ does not need to agree with $\alpha^J$.

\begin{lemma}\label{classantiFLt}
The triple $(A\ltimes M,L^{J,K},\alpha^{J,K})$ is an antistructure covering $(A,L,\alpha)$. Moreover every antistructure $(A\ltimes M,L^M,\alpha^M)$ over $(A,L,\alpha)$ with underlying target module $L_{t}^M=L_t\oplus L_t\otimes_AM$ is of the form
\[(A\ltimes M,L^M,\alpha^M)=(A\ltimes M,L^{J,K},\alpha^{J,K})\]
for some full twisting $(J,K)$.
\end{lemma}

\begin{proof}
The proof that $(A\ltimes M,L^{J,K},\alpha^{J,K})$ is an antistructure is essentially the same as the special case $K=0$ presented in \ref{mtwistingantistructure}. It uses the fact that the map
\[\xi_P\colon F(\hom(P,L_s))\longrightarrow \hom(F(P),L^{J,K}_s)\]
that sends $(\lambda,\delta\otimes n)$ to the module map
\[\xi(\lambda,\delta\otimes n)(p,p'\otimes m)=(\lambda(p),\delta(p)\otimes n+K(\alpha(\lambda(p)))+J(\alpha(\lambda(p'))\otimes m))\]
is a bijection (not a map of modules unless $(J,K)$ is regular, cf. \ref{lemmaisoequivariant}).

Now suppose that $(A\ltimes M,L^M,\alpha^M)$ is an antistructure over $(A,L,\alpha)$ with $L_{t}^M=F(L_t)=L_t\oplus L_t\otimes_AM$.
Let us identify the second $A\ltimes M$-module structure on $L^M$. Multiplication by $(a,n)\in A\ltimes M$ in the second structure is a map of right modules
\[(-)\cdot 1\otimes (a,n)\colon F(L_t)\longrightarrow F(L_t)\]
By \ref{modelPAM}, such a map is of the form
\[(l,l'\otimes m)\cdot 1\otimes (a,n)=(\phi_{a,n}(l),\phi_{a,n}(l')\otimes m+f_{a,n}(l))\]
for some maps of right $A$-modules $\phi_{a,n}\colon L_t\longrightarrow L_t$ and $f_{a,n}\colon L_t\longrightarrow L_t\otimes_A M$. The fact that the second structure is indeed a module structure gives some compatibility identities between these maps, and from these one can see that the maps $f_{0,n}$ fit together into a map $J\colon L_t\otimes_A M\longrightarrow L_t\otimes_A M$ defined by
\[J(l\otimes m)=f_{0,m}(\alpha(l))\]
The map $\alpha^M$ must be of the form
\[\alpha^M(l,l'\otimes m)=(\alpha(l),J(l'\otimes m)+K(l))\]
for some map $K\colon L\longrightarrow L_t\otimes_A M$. Again by the module identities, one can check that the pair $(J,K)$ defines a full twisting.
\end{proof}

\begin{ex}\label{antiinvolutioncase2} Recall from \ref{antiinvolutioncase} that an anti-involution $\overline{\alpha}$ on $A\ltimes M$ covering $\alpha\colon A^{op}\longrightarrow A$ is of the form
\[\overline{\alpha}(a,m)=(\alpha(a),r(a)+j(m))\]
for maps $j$ and $r$ satisfying relations very similar to \ref{deffulltwist}. Indeed, the antistructure $(A\ltimes M,A\ltimes M,\overline{\alpha})$ is the same as the one associated to the full twisting $(J,K)$ given by $(j,r)$ under the canonical isomorphism $L_t\otimes_A M=A\otimes_A M\cong M$.

Because of the properties of $r$, the full twisting given by $(j,r)$ is regular if and only if $r=0$.
\end{ex}

\begin{proof}[Proof of \ref{quillenmodstory}]
We show that $\Xi$ is fully faithful and essentially surjective. Faithfulness is clear.

To see that $\Xi$ is full, let 
\[\Xi(M,J)\longrightarrow \Xi(M',J')\]
be a map of abelian group objects. This is a map of antistructures
\[(g,G)\colon (A\ltimes M,L^J,\alpha^J)\longrightarrow (A\ltimes M',L^{J'},\alpha^{J'})\]
over $(A,L,\alpha)$ satisfying certain compatibility conditions with the group structure. We need to show that $(g,G)=\Xi(f)$ for some map of bimodules $f\colon M\longrightarrow M'$. The map $g\colon A\ltimes M\longrightarrow A\ltimes M'$ is a ring map over $A$, and so of the form
\[g(a,m)=(a,f(m)+r(a))\]
for a map of bimodules $f\colon M\longrightarrow M'$ and a derivation $r\colon A\longrightarrow M'$. The map $(g,G)$ preserves the unit maps
\[e\colon (A,L,\alpha)\longrightarrow (A\ltimes M,L^J,\alpha^J) \ \ \ , \ \ e'\colon (A,L,\alpha)\longrightarrow (A\ltimes M,L^{J'},\alpha^{J'})\]
defined by the zero section, that is $(g,G)\circ e=e'$. In particular the map $g$ satisfies $g(a,0)=(a,0)$, and so $r$ must be zero. This shows that $g=A\ltimes f$. Let us see that $G=L_t\oplus L_t\otimes f$.
The map $G$ is an additive map over $L$ also commuting with the zero section, and so of the form
\[G(l,l'\otimes m)=(l,F(l'\otimes m))\]
for some additive map $F\colon L_t\otimes_AM\longrightarrow L_t\otimes_AM'$. Since $(g,G)$ is a map of antistructures it preserves the target module structure, i.e.
\[G((l,l'\otimes m)\cdot (a,n)\otimes 1)=G(l,l'\otimes m)\cdot(a,f(n))\otimes 1\]
This forces $F$ to satisfy
\[F(l'\otimes m\cdot a+l\otimes n)=F(l'\otimes m)\cdot a+l\otimes f(n)\]
and setting $a=0$ we obtain $F=L_t\otimes f$.

To see that $\Xi$ is essentially surjective we need to prove that any abelian group object \[(\theta\colon (B,K,\beta)\longrightarrow (A,L,\alpha),\eta,e,(-)^{-1})\] in the category of antistructures over $(A,L,\alpha)$ is isomorphic to $\Xi(M,J)$ for some $A$-bimodule $M$ and $M$-twisting $J$. Since the map $e$ is a map over $(A,L,\alpha)$ the diagram
\[\xymatrix{(A,L,\alpha)\ar[dr]_{\id}\ar[r]^-{e}&(B,K,\beta)\ar[d]^{\theta}\\
&(A,L,\alpha)
}\]
commutes, that is $e$ is a section for $\theta$. In particular the underlying ring map $e_1\colon A\longrightarrow B$ is a section for the ring map $\theta_1\colon B\longrightarrow A$.
Define $M:=\ker\theta_1$ with $A$-bimodule structure via the section $e_1$, that is $a\cdot m\cdot a'=e_1(a)\cdot m\cdot e_1(a')$. There is a bijection
\[\phi\colon A\ltimes M \longrightarrow B\]
defined by $\phi(a,m)=e_1(a)+m$. It is a ring map if and only if the product of any two elements $m,m'$ in the kernel $M$ is zero. To show this, notice that for any $m\in M$ the pairs $(m,0)$ and $(0,m)$ are in the pullback $B\times_A B$. Since $e_1$ is the unit for $\eta_1\colon B\times_A B\longrightarrow B$ it satisfies
\[\eta_1(m,0)=\eta_1(0,m)=\eta_1(e_1(0),m)=m\]
Since $\eta_1$ is a ring map we have
\[m\cdot m'=\eta_1(m,0)\cdot \eta_1(0,m')=\eta_1((m,0)\cdot (0,m'))=\eta_1(0)=0\]
and therefore $\phi$ is a ring isomorphism.

The underlying map of modules $e_2\colon L\longrightarrow K$
induces an isomorphism of $B$-modules
\[\widetilde{e}_2\colon L_t\otimes_AB\longrightarrow K_t\]
(cf. \ref{mapofantistruct}) that induces an isomorphism of $A\ltimes M$-modules
\[L_t\oplus L_t\otimes_AM\cong L_t\otimes_AA\ltimes M\stackrel{L_t\otimes\phi}{\longrightarrow}L_t\otimes_AB\stackrel{\widetilde{e}_2}{\longrightarrow}K_t\]
By lemma \ref{classantiFLt} above there is a full twisting $(J,K)$ such that the map
\[(\phi,\widetilde{e}_2\circ (L_t\otimes\phi))\colon (A\ltimes M,L^{J,K},\alpha^{J,K})\longrightarrow (B,K,\beta)\]
is an isomorphism of antistructures. A calculation shows that
\[(\phi,\widetilde{e}_2\circ (L_t\otimes\phi))\circ(s,S)=(e_1,e_2)\]
where $(s,S)\colon (A,L,\alpha)\longrightarrow (A\ltimes M,L^{J,K},\alpha^{J,K})$ is the zero section.
Therefore $(s,S)$ is also a map of antistructures, and as such $S\circ \alpha=\alpha^{J,K}\circ S$. This forces $K=0$, and therefore we defined an isomorphism of antistructures
\[(\phi,\widetilde{e}_2\circ (L_t\otimes\phi))\colon (A\ltimes M,L^{J},\alpha^{J})\longrightarrow (B,K,\beta)\]
It remains to show that $(\phi,\widetilde{e}_2\circ (L_t\otimes\phi))$ defines a map of abelian group objects
\[\Xi(M,J)\longrightarrow (\theta,\eta,e,(-)^{-1})\]
The identity $(\phi,\widetilde{e}_2\circ (L_t\otimes\phi))\circ(s,S)=e$ above shows that the unit is preserved, since the zero section $(s,S)$ is the unit for $\Xi(M,J)$. The ring map $\phi$ preserves the group structure $\eta_1$, in the sense that the diagram
\[\xymatrix{ A\ltimes M\times_AA\ltimes M\ar[r]^-{\phi\times_A\phi}\ar[d]&B\times_AB\ar[d]^{\eta_1}\\
A\ltimes M\ar[r]_-{\phi}&B
}\]
commutes. Here the left vertical map is the underlying ring map for the group multiplication map of $\Xi(M,J)$.  Indeed the top composite is
\[\eta_1(\phi(a,m),\phi(a,m'))=\eta_1(e_1(a),e_1(a))+\eta_1(m,0)+\eta_1(0,m')=e_1(a)+m+m'\]
which is equal to the bottom composite $\phi(a,m+m')$.
A similar proof shows that the group law is preserved for the maps of modules.
\end{proof}

We finish the section with a classification of all the dualities on $\mathcal{P}_{A\ltimes M}$ covering a given $(A,L,\alpha)$ and a discussion on the equivalence of categories $F\colon \mathcal{P}(A;M)\longrightarrow \mathcal{P}_{A\ltimes M}$.
We remind from §\ref{model1} that the category $\mathcal{P}(A;M)$ has objects $Ob\mathcal{P}_A$, and the morphism set from $P$ to $Q$ is
\[\hom_A(P,Q)\oplus\hom_A(P,Q\otimes_AM)\]
The equivalence $F$ is defined by $F(P)=P\oplus P\otimes_AM$.

\begin{prop}\label{classanti}
Let $(A\ltimes M,L^M,\alpha^M)$ be an antistructure admitting a map of antistructures
\[\theta\colon(A\ltimes M,L^M,\alpha^M)\longrightarrow (A,L,\alpha)\]
with the projection $A\ltimes M\longrightarrow M$ as underlying ring map. Then there is a full twisting $(J,K)$ and an isomorphism of antistructures over $(A,L,\alpha)$ between $(A\ltimes M,L^M,\alpha^M)$ and $(A\ltimes M,L^{J,K},\alpha^{J,K})$.
\end{prop}

\begin{proof}
If $(\pr,G)\colon (A\ltimes M,L^M,\alpha^M)\longrightarrow (A,L,\alpha)$
is a map of antistructures, the map
\[\widetilde{G}\colon L^{M}_t\otimes_{A\ltimes M} A\longrightarrow  L_t\]
given by $\widetilde{G}(l\otimes a)=G(l)\cdot a\otimes 1$ is an isomorphism of right $A$-modules (cf. \ref{mapofantistruct}).
This induces an isomorphism 
\[F(\widetilde{G}^{-1})\colon F(L_t)\stackrel{\cong}{\longrightarrow} F(L^{M}_t\otimes_{A\ltimes M} A)\]
with right-hand side isomorphic (non-canonically) to $L^{M}_t$ (cf. proof of \ref{modelPAM}). Therefore there is a unique $A\ltimes M\otimes A\ltimes M$-module structure on $F(L_t)$ and a unique $\beta\colon F(L_t)\longrightarrow F(L_t)$ such that the isomorphism $L_{t}^M\cong F(L_t)$ gives an isomorphism of antistructures
\[(A\ltimes M,L^M,\alpha^M)\cong (A\ltimes M,F(L_t),\beta)\]
By lemma \ref{classantiFLt} above the antistructure $(A\ltimes M,F(L_t),\beta)$ must be of the form $(A\ltimes M,L^{J,K},\alpha^{J,K})$ for some full twisting $(J,K)$.
\end{proof}

The next follows directly from \ref{dualitiesandequivalences}.

\begin{prop}
Given a full twisting $(J,K)$ of $(A,L,\alpha)$, there is a duality on $\mathcal{P}(A;M)$ covering $(A,L,\alpha)$ such that the functor $F\colon \mathcal{P}(A;M)\longrightarrow\mathcal{P}_{A\ltimes M}$ is an equivalence of categories with duality.
\end{prop}

In the regular case, we are able to describe the duality on $\mathcal{P}(A;M)$ corresponding to $(J,K)$ explicitly.

\begin{prop} Let $(J,K)$ be a regular $M$-twisting of $(A,L,\alpha)$. The functor $D_{J}$ from \ref{defs21(A;M)}, together with the natural isomorphism $\eta$ defined by the pair
\[\eta_P=(\eta_{P}^\alpha,\eta_{P}^M)\in\hom_A(P,D_{L}^2(P))\oplus\hom_A(P,D_{L}^2(P)\otimes_A M)\]
with $\eta_{P}^M(p)=\mu^{-1}(\lambda\mapsto -K(\lambda(p)))$ is a duality structure on $\mathcal{P}(A;M)$. 

The functor $F\colon \mathcal{P}(A;M)\longrightarrow\mathcal{P}_{A\ltimes M}$ together with the natural isomorphism \[\xi_P\colon F(\hom(P,L_s))\longrightarrow \hom(F(P),L^{J,K}_s)\]
defined by
\[\xi(\lambda,\delta\otimes n)(p,p'\otimes m)=(\lambda(p),\delta(p)\otimes n+K(\alpha(\lambda(p)))+J(\alpha(\lambda(p'))\otimes m))\]
is an equivalence of categories with duality
\[(F,\xi)\colon(\mathcal{P}(A;M),D_J,\eta)\longrightarrow (\mathcal{P}_{A\ltimes M},D_{L^{J,K}},\eta)\]
over $(\mathcal{P}_{A},D_{L},\eta^\alpha)$.
\end{prop}
Here $\mu\colon \hom_A(P,D_{L}^2(P))\otimes_A M\longrightarrow\hom_A(P,D_{L}^2(P)\otimes_A M)$ is the canonical isomorphism.
Notice that since $(J,K)$ is regular, $\lambda\mapsto -K(\lambda(p))$ is indeed an element of $\hom(D_L(P),(L_t\otimes_A M)_s)$.

\begin{proof}
Same proof of \ref{lemmaisoequivariant}.
\end{proof}

\begin{rem}
In this more general setting, the action on the category $\mathcal{D}\mathcal{P}(A;M)$ does not restrict to the subcategory $\mathcal{D}\mathcal{P}(A;M)^0$ on objects of the form $(\phi,0)$, since
\[D_J(\phi,0)=(D_L(\phi),0)\circ(\eta^\alpha,\eta^M)=(D_L(\phi),(D_L(\phi)\otimes M)\circ\eta^M)\]
and $\eta^M$ is not necessarily zero if $K\neq 0$ (cf. §\ref{model2}).
Therefore there is no obvious duality on the category $S^{2,1}_{\sbt}(A;M)$ for a non-zero $K$, and no analogue for theorem \ref{mainKR}.
\end{rem}

What stops us from explicitly computing the duality on $\mathcal{P}(A;M)$ associated to a non-regular $(J,K)$, is the fact that the bijections $\xi_P$ are not maps of modules. However, we know by \ref{modelPAM} that there exist an isomorphism
\[\hom_{A\ltimes M}(F(P),L^{J,K}_s)\cong F(\hom_{A\ltimes M}(F(P),L^{J,K}_s)\otimes_{A\ltimes M}A)\]
and the right hand side is canonically isomorphic to $F(\hom_A(P,L_s))$. One could try to write down these isomorphisms by choosing projective basis for the modules, but I do not think that the result would not be very satisfying.

One could alternatively proceed backwards. Given a duality $(D,\eta)$ on $\mathcal{P}(A;M)$ covering $(A,L,\alpha)$, there is an induced duality on $\mathcal{P}_{A\ltimes M}$. This duality is associated to an antistructure by \ref{classdualpa}, and this antistructure only depends on what the duality does to $A\ltimes M$. For the module $A\ltimes M$, all the choices of isomorphism are canonical, and one can explicitly show that this antistructure is induced by the following full twisting. The functor $D\colon\mathcal{P}(A;M)^{op}\longrightarrow \mathcal{P}(A;M)$ on a morphism $(\phi,f)\colon P\longrightarrow Q$ is of the form
\[D(\phi,f)=(D_L(\phi),H_{P,Q}(\phi)+I_{P,Q}(f))\]
for natural maps
\[H_{P,Q}\colon \hom(P,Q)\longrightarrow \hom(D_L(Q),D_L(P)\otimes_A M)\]
and
\[I_{P,Q}\colon \hom(P,Q\otimes_A M)\longrightarrow \hom(D_L(Q),D_L(P)\otimes_A M)\]
Denote $f_m\colon A\longrightarrow A\otimes_A M$ the map of right $A$-modules
\[f_m(a)=1\otimes m\cdot a\]
Now define $\overline{J}\colon D_L(A)\otimes_A M\longrightarrow D_L(A)\otimes_A M$ by
\[\overline{J}(\lambda\otimes m)=I_{A,A}(f_m)(\lambda)\]
The $J$-map of the full twisting is then
\[J\colon L_t\otimes_A M\cong D_L(A)\otimes_A M\stackrel{\overline{J}}{\longrightarrow }D_L(A)\otimes_A M\cong L_t\otimes_A M\]
For the map $K$, consider the second component of the natural isomorphism $\eta$ at $A$. This is a module map $\eta^M\colon A\longrightarrow D_{L}^2(A)\otimes_A M\cong D_L(L_t)\otimes_A M$, and $K\colon L_t\longrightarrow L_t\otimes_A M$ is defined by
\[K(l)=\mu(\eta^M(1))(l)+\overline{H}(l)\]
where $\overline{H}$ corresponds to $H_{A,D_L(A)}$ under the canonical isomorphisms
\[\xymatrix{{L_t\cong D_L(A)\cong\hom_A(A,D_L(A))}\ar@<-12ex>[d]^-{\overline{H}}\ar[rr]^-{H_{A,D_L(A)}}&&\hom_A(D_{L}^2(A),D_L(A)\otimes_A M)\ar[d]^{\eta_\alpha}\\
{L_t\otimes_A M\cong D_L(A)\otimes_A M}&&\hom_A(A,D_L(A)\otimes_A M)\ar[ll]_-{\cong}
}\]
and $\mu\colon D_L(L_t)\otimes_AM\longrightarrow\hom_A(L_t,(L_t\otimes_A M)_s)$ is as usual the canonical isomorphism.


\subsection{$G$-Bar constructions (joint work with Kristian Moi)}\label{Gbarconstructions}

This section is inspired in part by conversations with Nisan Stiennon.
Recall that $S^{1,1}_{\sbt}$ denotes the simplicial circle $S^{1}_{\sbt}=\Delta[1]/\partial\Delta[1]$ with levelwise $G=\mathbb{Z}/2$-action induced from the involution $\omega_k$ on  $Cat([k],[1])$ sending $\sigma\colon [k]\longrightarrow [1]$ to
\[(\omega\sigma)(j)= 1-\sigma(k-j)\]
Writing $\sigma=(0^j1^{k+1-j})$, the involution sends $\sigma$ to $(0^{k+1-j}1^j)$.
Given a topological monoid $M$, consider the simplicial topological monoid $M(S^{1,1}_{\sbt})$ defined degreewise by taking reduced configurations of points of $S^{1,1}_{\sbt}$ labeled by $M$, that is
\[M(S^{1,1}_{\sbt})_k=M(S^{1,1}_k)=(\bigoplus_{\sigma\in S^{1,1}_k}M\cdot\sigma)/M\cdot\ast\]
Faces and degeneracies are defined by linear extension of faces and degeneracies of $S^{1,1}_{\sbt}$. Notice that this is isomorphic to the standard Bar construction on $M$. Now suppose that $M$ has an anti-involution, that is a monoid map $D\colon M^{op}\longrightarrow M$ with $D^{2}=\id$. This induces a degreewise $G$-action $D\colon M(S^{1,1}_k)^{op}\longrightarrow M(S^{1,1}_k)$ by sending $\sum_im_i\sigma_i$ to
\[D(\sum_im_i\sigma_i)=\sum_iD(m_i)\omega\sigma_{i}\]
This action is not simplicial, but it defines a real space structure on $M(S^{1,1}_{\sbt})$ (cf. \ref{defrealob}). Its realization is the $G$-space denoted $M(S^{1,1})$. The one-simplicies of $M(S^{1,1}_{\sbt})$ are canonically $G$-homeomorphic to $M$, and the adjoint of the projection map
\[M\times\Delta^1\longrightarrow M(S^{1,1})\]
induces a $G$-map $M\longrightarrow \Omega^{1,1}M(S^{1,1})$. Here $\Omega^{1,1}=\map_\ast(|S^{1,1}_{\sbt}|,-)$ denotes the pointed mapping space with conjugation action. The next result has also been proved independently, using somewhat different methods, by Nisan Stiennon \cite{Nisan}.

\begin{prop}\label{S11deloopformonoids}
Let $M$ be a monoid with anti-involution, and suppose that $\pi_0(M)$ is a group. Then the map $M\longrightarrow \Omega^{1,1}M(S^{1,1})$ is a $G$-equivalence.
\end{prop}

\begin{rem}
Kristian Moi proved that for any monoid with anti-involution $M$, the canonical map $M^G\longrightarrow (\Omega^{1,1}M(S^{1,1}))^G$ induces an equivalence in homology after localizing the action of $\pi_0M$ on $H_\ast(M^G)$.
\end{rem}

The proof of the proposition is given later in the section. Now let $S^{1,0}_{\sbt}$ be the simplicial circle with trivial $G$-action, and suppose that the topological monoid with anti-involution $M$ is abelian. In this case the fixed points of the action $M^G$ is a submonoid of $M$, and one can define $M(S^{1,0}_{\sbt})$ as the same simplicial space as $M(S^{1,1}_{\sbt})$, but
with levelwise action defined by 
\[D(\sum_im_i\sigma_i)=\sum_iD(m_i)\sigma_{i}\]
This action is now simplicial, and the adjoint of the projection map
\[M\times\Delta^1\longrightarrow M(S^{1,0})\]
induces an equivariant map $M\longrightarrow \Omega M(S^{1,1})$ where the loop space now carries the pointwise action (that is, the conjugation action on the mapping space for the trivial action on the circle).

\begin{prop}\label{S10deloopformonoids}
Let $M$ be an abelian topological monoid with additive involution, and suppose that both $\pi_0(M)$ and $\pi_0(M^G)$ are groups. Then the canonical map \[M\longrightarrow \Omega M(S^{1,0})\] is a $G$-equivalence.
\end{prop}
The proof of this last proposition can be easily reduced to the classical non-equivariant proof (see e.g. \cite[7.6]{egbgfib}), which we recall.

\begin{proof}
Consider the simplicial space $P M(S^{1,0}_{\sbt})$ define in degree $k$ as 
\[PM(S^{1,0}_{\sbt})_k=M(S^{1,0}_{k+1})\]
with simplicial structure obtained by forgetting $d_0$ and $s_0$ in every simplicial degree from $M(S^{1,0}_{\sbt})$.
The forgotten face $d_0$ gives a simplicial map $d_0\colon PM(S^{1,0}_{\sbt})\longrightarrow M(S^{1,0}_{\sbt})$. Levelwise it is the trivial fibration with fiber $M$. It is a known fact that since $M$ is group-like, a levelwise fibration realizes to a fiber sequence
\[M\longrightarrow PM(S^{1,0})\longrightarrow M(S^{1,0})\]
This can be shown using the criterion \ref{gplikeqfib} at the end of this section.
Moreover the space  $PM(S^{1,0}_{\sbt})$ is simplicially homotopy equivalent to $M(S^{1,0}_0)=\ast$, and therefore the canonical map $M\longrightarrow \Omega M(S^{1,0})$ is an equivalence.

We now prove the statement on the fixed points.
The restriction of the map $d_0$ levelwise on the fixed points, gives a sequence 
\[M^G\longrightarrow PM^G(S^{1,0})\longrightarrow M^G(S^{1,0})\]
which is again a fiber sequence since $M^G$ is group-like.
This gives an equivalence
\[M^G\stackrel{\simeq}{\longrightarrow}\Omega M^G(S^{1,0})\cong \Omega M(S^{1,0})^G\cong(\Omega M(S^{1,0}))^G \]
where the first homeomorphism comes from the fact that the action on $S^{1,0}_{\sbt}$ is trivial.
\end{proof}

Since $G$ has only two irreducible real representations, one can decompose any representation sphere $S^V$ as a smash product of $S^{1,0}$ and $S^{1,1}$. An easy induction argument using the fact that $M(S^{1,i}_{\sbt}\wedge S^{1,j}_{\sbt})\cong M(S^{1,i})( S^{1,j}_{\sbt})$ for $i,j=0,1$ gives the following.

\begin{cor}\label{Gdeloopingsformonoid}
Let $M$ be an abelian topological monoid with additive involution, and suppose that both $\pi_0(M)$ and $\pi_0(M^G)$ are groups. Then the canonical map
\[M\longrightarrow\Omega^VM(S^V)\]
is a $G$-equivalence.
\end{cor}

We now prove proposition \ref{S11deloopformonoids}. As seen in the proof above, a standard technique to show that something is a loop space, is to build a fiber sequence with contractible total space. Here we develop a framework that allows us to show that a $G$-space is $\Omega^{1,1}$ of something, and later we will apply it to prove the proposition.

Let $X$ be a pointed $G$-space with action given by an involution $D\colon X\longrightarrow X$, and let
\[\xymatrix{Y_+\ar[d]_{p_+}\ar[r]^{D_+}_{\cong}&Y_-\ar[d]^{p_-}\\
X\ar[r]^D&X
}\]
be a commutative diagram of pointed spaces. Denote $D_-$ the inverse of the homeomorphism $D_+$. Since $D^2=\id$, it follows that $Dp_-=p_+D_-$.
In this situation, the homotopy pullback $Y_+\stackrel{h}{\times}_XY_-$ of the diagram
\[\xymatrix{& Y_+\ar[d]^{p_+}\\
Y_-\ar[r]_{p_-}& X
}\]
admits $G$-action, defined as follows. An element $(y,\gamma,y')\in Y_+\stackrel{h}{\times}_XY_-$ is a triple $(y,\gamma,y')\in Y_+\times X^I\times Y_-$ such that $\gamma(0)=p_+(y)$ and $\gamma(1)=p_-(y')$. This triple is sent by the action to
\[(D_-(y'),D\overline{\gamma},D_+(y))\]
where $D\overline{\gamma}$ denotes the path obtained by applying the involution $D$ pointwise to the backwards path $\overline{\gamma}(t)=\gamma(1-t)$.

One can describe the fixed points of the homotopy pullback again as a homotopy pullback.
\begin{lemma}\label{fixedptshopb}
There is a natural homeomorphism between $(Y_+\stackrel{h}{\times}_XY_-)^G$ and the homotopy pullback $Y_+\stackrel{h}{\times}_XX^G$ of the diagram
\[\xymatrix{& Y_+\ar[d]^{p_+}\\
X^G\ar[r]& X
}\]
where the bottom map is the inclusion of the fixed points.
\end{lemma}

\begin{proof}
An element of $(Y_+\stackrel{h}{\times}_XY_-)^G$ is a triple of the form $(y,\gamma,D_+y)$, where the path $\gamma$ satisfies
\[\gamma=D\overline{\gamma}\]
This means that $\gamma$ is determined by its restriction to $[0,1/2]$, since $\gamma(1-t)=D\gamma(t)$. Moreover $\gamma(1/2)$ is a fixed point since
\[D\gamma(1/2)=\overline{\gamma}(1/2)=\gamma(1-1/2)=\gamma(1/2)\]
An element of $Y_+\stackrel{h}{\times}_XX^G$ is a path $\delta\colon I\longrightarrow X$ and a $y\in Y_+$ such that $p_+(y)=\delta(0)$ and $\delta(1)\in X^G$. Thus we can map $(y,\gamma,D_+y)$ to the pair $(y,\gamma|_{[0,1/2]})$. This defines a map
\[(Y_+\stackrel{h}{\times}_XY_-)^G\longrightarrow Y_+\stackrel{h}{\times}_XX^G\]
with inverse sending $(y,\delta)$ to the triple $(y,\gamma,D_+y)$ with $\gamma$ defined by
\[\gamma(t)=\left\{\begin{array}{ll}\delta(t)&\mbox{ if }t\in[0,1/2]\\
D\delta(2-2t)&\mbox{ if }t\in[1/2,1]
\end{array}\right.\]
\end{proof}

There is an equivariant inclusion map $\iota\colon \Omega^{1,1}X\longrightarrow Y_+\stackrel{h}{\times}_XY_-$ that sends a loop $\tau$ of $X$ to the triple $(\ast_+,\tau,\ast_-)$, where $\ast_+$ and $\ast_-$ denote respectively the base points of $Y_+$ and $Y_-$.

\begin{lemma}\label{modelloop11}
A contraction of $Y_+$ induces a $G$-homotopy inverse for the inclusion 
\[\iota\colon \Omega^{1,1}X\longrightarrow Y_+\stackrel{h}{\times}_XY_-\]
\end{lemma}

\begin{proof}
Suppose that there is a contraction $H_+\colon Y_+\times I\longrightarrow Y_+$ from the identity to $\ast_+$. Define a contraction $H_-\colon Y_-\times I\longrightarrow Y_-$ from $\ast_-$ to the identity by
\[H_-(y',t)=D_+H_+(D_-(y'),1-t)\]
This gives a an equivariant map $r\colon Y_+\stackrel{h}{\times}_XY_-\longrightarrow \Omega^{1,1}X$ defined by
\[r(y,\gamma,y')=(p_+\circ H_+(y,-))\ast\gamma\ast(p_-\circ H_-(y',-))\]
where $\ast$ denotes concatenation of paths. The composite $r\circ \iota$ is the identity. The other composite sends $(y,\gamma,y')$ to
\[(i\circ r)(y,\gamma,y')=(\ast_+,(p_+\circ H_+(y,-))\ast\gamma\ast(p_-\circ H_-(y',-)),\ast_-)\]
A $G$-homotopy between this map an the identity is defined by sending $(y,\gamma,y',t)$ to
\[(p_+(H_+(y,t)),(p_+\circ H_+(y,-))|_{[t,1]}\ast\gamma\ast(p_-\circ H_-(y',-))|_{[0,1-t]},p_-(H_-(y,t)))\]
\end{proof}

Let $X_{\sbt}$ be a real space with contractible $0$-simplicies. Define $P_+X_{\sbt}$ to be the simplicial space with $k$-simplicies $X_{k+1}$, and with simplicial structure given by forgetting $d_0$ and $s_0$ from $X_{\sbt}$ in every simplicial degree. Similarly, denote $P_-X_{\sbt}$ the simplicial space with the same simplicies, but where the simplicial structure is given by forgetting the last face and degeneracy of $X_{\sbt}$. The forgotten faces give simplicial maps 
\[\xymatrix{&P_+X_{\sbt}\ar[d]^{d_0}\\
P_-X_{\sbt}\ar[r]_{d_{L}}&X_{\sbt}}\]
The levelwise involution $D$ on $X_{\sbt}$ does not give a real structure on $P_+X_{\sbt}$, nor on $P_-X_{\sbt}$, but it induces a map $D_+\colon P_+X_{\sbt}\longrightarrow P_-X_{\sbt}$ that reverses the order of faces and degeneracies. It induces a homeomorphism on realizations \[D_+\colon |P_+X_{\sbt}|\longrightarrow |P_-X_{\sbt}|\]
defined by sending $[x;t_0,\dots,t_n]$ to $[Dx;t_n,\dots,t_0]$.
Setting $X=|X_{\sbt}|$, $Y_+=|P_+X_{\sbt}|$, $Y_-=|P_-X_{\sbt}|$, $p_+=d_0$ and $p_-=d_{L}$ we are in the situation described above. Moreover the forgotten degeneracy defines a simplicial homotopy equivalence $s\colon \ast\simeq X_0\longrightarrow P_+X_{\sbt}$ and therefore a contraction of $Y_+$.

When $X_{\sbt}=M(S^{1,1}_{\sbt})$, the lemma gives a model for $\Omega^{1,1}M(S^{1,1})$.
We prove \ref{S11deloopformonoids} by defining a $G$-equivalence 
\[M\longrightarrow |P_+X|\stackrel{h}{\times}_{|X|}|P_-X|\]
We define it through the the two-sided Bar construction $P_+P_-M(S^{1,1}_{\sbt})$. The levelwise action of $M(S^{1,1}_{\sbt})$ induces a real structure on $P_+P_-M(S^{1,1}_{\sbt})$, and considering $M$ as a constant real space there is an equivariant inclusion
\[M\longrightarrow sd_eP_+P_-M(S^{1,1}_{\sbt})\]
that in simplicial degree $k$ sends $m$ to $m\cdot(0^{k+2}1^{k+2})\in M(S^{1,1}_{2k+3})$.
Notice also that $P_+P_-M(S^{1,1}_{\sbt})$ is by definition the pullback $P_+X\times_{X} P_-X$, and therefore our construction gives a $G$-map 
\[M\longrightarrow |sd_eP_+P_-M(S^{1,1}_{\sbt})|\cong|P_+P_-M(S^{1,1}_{\sbt})|\longrightarrow |P_+X|\stackrel{h}{\times}_{|X|}|P_-X|\]
where the last map is the canonical inclusion from the pullback into the homotopy pullback.
The proof of \ref{S11deloopformonoids} follows by the two lemmas below.

\begin{lemma}
The inclusion $M\longrightarrow sd_eP_+P_-M(S^{1,1}_{\sbt})$ is a simplicial $G$-homotopy equivalence.
\end{lemma}

\begin{proof}
It is well known that the inclusion into the two sided-Bar construction is a homotopy equivalence, but since we want to be careful with the $G$-action we write down the maps.
Define a retraction 
$r\colon M(S^{1,1}_{2k+3})\longrightarrow M$ by projecting onto the $(0^{k+2}1^{k+2})$-component.
Denote $\overline{d}_l$ and $\overline{s}_l$ the faces an degeneracies of $sd_eM(S^{1,1}_{\sbt})$, and define a simplicial map $H\colon sd_eP_+P_-M(S^{1,1}_{\sbt})\times\Delta[1]\longrightarrow sd_eP_+P_-M(S^{1,1}_{\sbt})$ in degree $k$ by
\[H(-,\sigma)=\overline{s}_{0}^{b}\overline{d}_{0}^{b}\]
where $\sigma=(0^b,1^{k+1-b})$. Notice that, in degree $k$, the retraction followed by the inclusion is the map $s_{0}^{k+1}d_{0}^{k+1}$, and therefore this is a homotopy from the retraction to the identity. Since the action on $sd_eM(S^{1,1}_{\sbt})$ is simplicial, the action commutes with the faces and degeneracies defining $H$, and therefore $H$ is a $G$-homotopy.
\end{proof}

\begin{lemma}
The canonical map $|P_+P_-M(S^{1,1}_{\sbt})|\longrightarrow |P_+X|\stackrel{h}{\times}_{|X|}|P_-X|$ is a $G$-equivalence.
\end{lemma}

\begin{proof}
Levelwise, the space $P_+P_-M(S^{1,1}_{\sbt})_k=M(S^{1,1}_{k+2})$ is the pullback \[(P_+X)_k{\times}_{X_k}(P_-X)_k=M(S^{1,1}_{k+1})\times_{M(S^{1,1}_{k})}M(S^{1,1}_{k+1})\]
since $M(S^{1,1}_{k})\cong M^k$. Moreover the action corresponds to the action on the pullback given by $D(y,y')=(Dy',Dy)$. The fixed points of the action are again described by a pullback \[(P_+P_-M(S^{1,1}_{\sbt}))_{k}^G\cong M(S^{1,1}_{k+1})\times_{M(S^{1,1}_{k})}(M(S^{1,1}_{k}))^G\]
Since realization commutes with pullbacks one gets a $G$-homeomorphism
\[|P_+P_-M(S^{1,1}_{\sbt})|\cong|P_+X|{\times}_{|X|}|P_-X|\] where again the action on the right hand side is defined by $D(y,y')=(Dy',Dy)$.
Therefore one just needs to show that the canonical map \[|P_+X|{\times}_{|X|}|P_-X|\longrightarrow |P_+X|\stackrel{h}{\times}_{|X|}|P_-X|\] is a $G$-equivalence. By the descriptions of the fixed points above, we need to show that the two diagrams
\[\xymatrix{|P_+X|{\times}_{|X|}|P_-X|\ar[d]\ar[r]&|P_+X|\ar[d]^{d_0}\\
|P_-X|\ar[r]_{d_{L}}&|X|} \ \ \ \
\xymatrix{|P_+X|{\times}_{|X|}|X|^G\ar[d]\ar[r]&|P_+X|\ar[d]^{d_0}\\
|X|^G\ar[r]&|X|}
\]
are homotopy cartesian. This is the case if all the vertical maps are quasi-fibrations. Both maps $d_0$ are levelwise trivial fibrations, and since fibrations pull back along any map, levelwise the vertical maps are all fibrations. In order to show that the realizations are quasi fibrations one can use the criterion \ref{gplikeqfib} below, which holds for all the vertical maps if we assume that $M$ is group-like.
\end{proof}

\begin{lemma}\label{gplikeqfib}
Let $f_{\sbt}\colon Y_{\sbt}\longrightarrow X_{\sbt}$ be a map of simplicial spaces such that for all $k$ and $0\leq l\leq k$ the diagrams
\[\xymatrix{Y_k\ar[r]^-{d_l}\ar[d]_{f_k}& Y_{k-1}\ar[d]^{f_{k-1}}\\
X_k\ar[r]^-{d_l}&X_{k-1}
}\]
are homotopy cartesian. Then the diagram
\[\xymatrix{Y_0\ar[r]\ar[d]_{f_0}& |Y_{\sbt}|\ar[d]^{|f|}\\
X_0\ar[r]^{}&|X_{\sbt}|
}\]
is also homotopy cartesian.
\end{lemma}


\subsection{$G$-connectivity of $A(S^j)\wedge |X|\longrightarrow A(S^j\wedge X)$}

Let $X_{\sbt}$ be a based real set with involution $\overline{(\cdot)}$, and $A$ an abelian group with additive involution $D$. Recall that $A(X)$ is defined as the realization of the real set with $k$-simplicies
\[A(X_k)=\bigoplus_{x\in X_k} A\cdot x/A\cdot\ast\]
and action induced by sending $a\cdot x$ to $Da\cdot\overline{x}$, and that $S^j$ denotes the smash of $j$ simplicial circles $S^1$ with involution given by reversing the order of the smash factors. Since realization preserves products there is a natural homeomorphism
\[A(S^j)\wedge |X|\cong|[k]\longmapsto A(S_{k}^j)\wedge X_k|\]

\begin{prop}\label{connectvitysmashwithx}
The map
\[A(S^j)\wedge |X|\longrightarrow A(S^j\wedge X)\]
induced by the map that sends $(\sum m_i\cdot\sigma_i)\wedge x$ to $\sum m_i\cdot(\sigma_i\wedge x)$
is non-equivariantly $(2j+\conn X)$-connected, and its restriction to the fixed points is
\[j+\min\{\conn X,\conn X^G\}\]
connected (cf. §\ref{secghtpytheory}).
\end{prop}
We need two lemmas before starting the proof.

\begin{lemma}\label{connconfspace}
For any simplicial based $G$-set $X$, the space $A(X)^G$ is
\[\min\{\conn(X),\conn(X^G)\}\]
connected.
\end{lemma}

\begin{proof}
Recall that if $Y$ is a fibrant based simplicial set, the $k$-th homotopy group of $|Y|$ is trivial if and only if for every $y\in Y_k$ with $d_ly=\ast$, there is a $z\in Y_{k+1}$ such that $d_0z=y$ and $d_{l>0}z=\ast$. Now, $A(X)^G$ is the realization of the simplicial set
\[[k]\longmapsto A(X_k)^G\]
which is fibrant since it is a simplicial abelian group. 
Moreover, by replacing $X$ with the $G$-homotopy equivalent simplicial $G$-set $\sin_{\sbt} |X|$, we can assume $X$ and $X^G$ to be fibrant as well.

An element $\xi\in A(X_k)^G$ is a finite sum of the form
\[\xi=\sum_{i}a_i\cdot x_i+\sum_{j}b_j\cdot y_j+\sum_{i}Da_i\cdot\overline{x}_i\]
where the $x_i$ belong to $X_k\backslash X_{k}^G$, the $y_j$ belong to $X_{k}^G$, the $a_i$ to $A$ and the $b_j$ to $A^G$. Moreover if $\xi$ is such that $d_l\xi=\ast$ we must also have $d_lx_i=\ast$ and $d_ly_j=\ast$.

For $k\leq \min\{\conn(X),\conn(X^G)\}$, we can choose elements $z_i\in X_{k+1}$ such that $d_0z_i=x_i$ and $d_{l>0}z_i=\ast$, and elements $w_j\in X_{k+1}^G$ such that $d_0w_j=y_j$ and $d_{l>0}w_j=\ast$. The linear combination
\[\psi=\sum_{i}a_i\cdot z_i+\sum_{j}b_j\cdot w_j+\sum_{i}Da_i\cdot \overline{z}_i\in A(X_{k+1})^G\]
satisfies $d_0\psi=\xi$ and $d_{l>0}\psi=\ast$, showing that $\pi_kA(X)^G$ is trivial.
\end{proof}

Recall that a $G$-fiber sequence is a fiber sequence of $G$-maps whose restriction to the fixed points is also a fiber sequence. A real abelian group is a simplicial abelian group with levelwise additive involution that reverses the simplicial structure (cf. \ref{defrealob}). 

\begin{lemma}\label{fibseqinclskel}
Suppose that $Y$ is a real subset of a real set $X$, and $N$ is a real abelian group. Suppose moreover that both $N$ and $N^G$ are connected. Then the maps of real sets $Y\longrightarrow X\longrightarrow X/Y$ induce a $G$-fiber sequence
\[N(Y)\longrightarrow N(X)\longrightarrow N(X/Y)\]
on the realization.
\end{lemma}

\begin{proof}
Let us first show that for every $k$, both $N(X_k)\longrightarrow N(X_k/Y_k)$ and $N(X_k)^G\longrightarrow N(X_k/Y_k)^G$ are fibrations. We start with the non-equivariant map. This is the map
\[\bigoplus_{x\in X_k\backslash\ast}\!\!N\cdot x\longrightarrow \!\!\!\!\!\bigoplus_{[x]\in (X_k/Y_k)\backslash\ast}\!\!\!\!\!\!\!N\cdot [x]\cong \!\!\!\bigoplus_{x\in X_k\backslash Y_k}\!\!\!N\cdot x\]
that projects off the summands generated by elements of $Y_k$. It is clearly a fibration. Let us describe this map on the fixed points. Quite in general, there is a (non simplicial) homeomorphism
\[N[X_k\backslash X_{k}^G]\oplus N^G(X_{k}^G)\cong N(X_k)^G\]
that sends $(\sum n_i\cdot x_i,\sum m_j\cdot w_j)$ to
\[\sum n_i\cdot x_i+\sum m_j\cdot w_j+\sum D n_i\cdot \overline{x}_i\]
Here $N[-]$ denotes the unreduced configuration space.
Under this identification, the quotient map is the map
\[(\!\!\!\bigoplus_{x\in X_k\backslash X_{k}^G}\!\!\!\!\!N\cdot x)\oplus\!\!\!\bigoplus_{x\in X_{k}^G\backslash\ast}\!\!\!\!N^G\cdot x\longrightarrow (\!\!\!\!\!\!\!\bigoplus_{x\in X_k\backslash Y_k\backslash X_{k}^G}\!\!\!\!\!\!\!\!N\cdot x)\oplus \!\!\!\!\!\!\bigoplus_{x\in (X_{k}\backslash Y_k)^G}\!\!\!\!\!\!\!N^G\cdot x\]
that projects off the $Y_k$-summands. It is clearly a fibration.

In particular we showed that for every $k$ the sequence
\[N(Y_k)\longrightarrow N(X_k)\longrightarrow N(X_k/Y_k)\]
is a $G$-fiber sequence. We use Bousfield-Friedlander theorem \ref{bous-fried} to show that its realization is also a $G$-fiber sequence. Clearly our map is the realization of a map of bisimplicial $G$-abelian groups, namely the projection map from
\[[p][k]\longmapsto \{N_{2p+1}(X_{2k+1})\}\]
to
\[[p][k]\longmapsto \{N_{2p+1}(X_{2k+1}/Y_{2k+1})\}\]
By \cite[IV,4.2-(2)]{goersjardine}, the conditions for the Bousfield-Friedlander theorem \ref{bous-fried} apply if the simplicial sets $sd_e N_{\sbt}(X_{k}/Y_{k})$, $sd_e N_{\sbt}(X_{k}/Y_{k})^G$, $sd_e N_{\sbt}(X_{k})$ and $sd_e N_{\sbt}(X_{k})^G$ are fibrant and connected. They are fibrant since they are simplicial abelian groups. They are levelwise connected since both $N$ and $N^G$ are connected (for the fixed points use the splitting $ N_{\sbt}(X_{k})\cong N[X_k\backslash X_{k}^G]\oplus N^G(X_{k}^G)$).
\end{proof}

We prove \ref{connectvitysmashwithx} by induction on the cells of $|X|$, and therefore we need to prove it for spheres first. Suppose that $X$ is a wedge of spheres $X=\bigvee_J S^{n+1}$ where $J$ is a $G$-set. That is, $G$ acts trivially on the sphere and permutes the wedge components. This action is defined precisely in \ref{wedgesintoproducts}.

\begin{prop}\label{shpereinduction}
In the case of $X=\bigvee_J S^{n+1}$, the map 
\[\tau_{S^{n+1}}\colon A(S^j)\wedge \bigvee_J S^{n+1}\longrightarrow A(S^j\wedge \bigvee_J S^{n+1})\]
is $(2j+n,j+n)$-connected.
\end{prop}

\begin{proof} Notice that the canonical homeomorphism
\[A(S^j\wedge \bigvee_J S^{n+1})\cong\bigoplus_JA(S^j\wedge S^{n+1})\]
is equivariant (for the action on the direct sum of \ref{wedgesintoproducts}). 
Rewrite $\tau_{S^{n+1}}$ as
\[A(S^j)\!\wedge\!\bigvee\limits_J S^{n+1}\!\cong\!\bigvee\limits_J (A(S^j)\!\wedge\!S^{n+1})\longrightarrow \bigoplus\limits_J A(S^j\!\wedge\!S^{n+1})\!\cong\!A(S^j\!\wedge\!\bigvee\limits_J S^{n+1})\]
The middle map factors as
\[\xymatrix{\bigvee_J (A(S^j)\wedge S^{n+1})\ar[r]^f\ar[dr]&\bigvee_J A(S^j\wedge S^{n+1})\ar[d]^i\\
&\bigoplus_J A(S^j\wedge S^{n+1})
}\]
The connectivity of $i$ is non-equivariantly $2(j+n)+1$ as inclusion of wedges into sums. Its connectivity on fixed points is, by \ref{wedgesintoproducts},
\[\min\{2(j+n)+1,2(\frac{j}{2}+n)+1\}=j+2n+1\]
since $A(S^j\wedge S^{n+1})$ is $(j+n,\frac{j}{2}+n)$-connected by \ref{linear0connected}.
Therefore it is enough to show that the connectivity of $f$ is at least $(2j+n,j+n)$.

Non-equivariantly, $f$ is just a wedge of maps $A(S^j)\wedge S^{n+1}\longrightarrow A(S^j\wedge S^{n+1})$, and on the fixed points it is the wedge
\[\bigvee_{J^G} A(S^j)^G\wedge S^{n+1}\longrightarrow\bigvee_{J^G} A(S^j\wedge S^{n+1})^G\]
Therefore $f$ is as connected as the map $\tau\colon A(S^j)\wedge S^{n+1}\longrightarrow A(S^j\wedge S^{n+1})$. To see how much $\tau$ is connected, consider the commutative diagram
\[\xymatrix{\Omega^{n+1}(A(S^j)\wedge S^{n+1})\ar[r]^-{\Omega^{n+1}\tau}&\Omega^{n+1}A(S^j\wedge S^{n+1})\cong\Omega^{n+1}A(S^{j})(S^{n+1})\\
A(S^j)\ar[u]\ar[ur]_{\simeq}
}\]
The diagonal map is an equivalence by \ref{Gdeloopingsformonoid}, where the topological monoid with involution is now $A(S^{j})$.  The vertical map is the adjoint of the identity, which is $(2j-1,j-1)$-connected by the $G$-suspension theorem \ref{Gsuspthm}. Thus $\tau$ is $(2j+n,j+n)$-connected.
\end{proof}

\begin{proof}[Proof of \ref{connectvitysmashwithx}]
We make a proof by induction on the cells of $|X|$.

Suppose that $|X|$ is zero dimensional, that is a  pointed $G$-set. In this case $\tau$ is the inclusion of wedges into sums
\[\tau\colon A(S^j)\wedge |X|\cong\bigvee_{|X|}A(S^j)\longrightarrow
\bigoplus_{|X|}A(S^j)\cong A(S^j\wedge X)\]
Non equivariantly this is $2(j-1)+1=(2j-1)$-connected ($X$ is in this case $(-1)$-connected). Its restriction to the fixed  points $\tau^G$ has connectivity
\[\min\{2(\frac{j}{2}-1)+1,j-1\}=j-1\]
by \ref{wedgesintoproducts}.

For the induction step, suppose that we proved the statement for all $Y$ of dimension less than or equal to $n$, and take $X$ of dimension $n+1$.

Let us start by calculating the non-equivariant connectivity of \[\tau\colon A(S^j)\wedge |X|\longrightarrow A(S^j\wedge X)\]
We need to separate different cases.
\begin{itemize}
\item $n<\conn X$: Since $X$ is of dimension $n+1$, this forces $X$ to be contractible, and in this case $\tau$ is trivially an equivalence since both its source and its target are contractible.
\item $\conn X< n$: The inclusion of the (simplicial) $n$-skeleton $X^{(n)}\longrightarrow X^{(n+1)}=X$ induces a diagram of maps
\[\xymatrix{A(S^j)\wedge |X^{(n)}|\ar[r]\ar[d]^-{\tau^{(n)}}&A(S^j)\wedge |X|\ar[r]\ar[d]^-{\tau}&A(S^j)\wedge \bigvee S^{n+1}\ar[d]^-{\tau_{S^{n+1}}}\\
A(S^j\wedge X^{(n)})\ar[r]&A(S^j\wedge X)\ar[r]&A(S^j\wedge \bigvee S^{n+1})
}\]
The bottom row is a fiber sequence by the lemma \ref{fibseqinclskel}, by setting $N=A(S^j)$.
The top row is a cofiber sequence, as smash of a cofiber sequence with a space. By the Blakers-Massey theorem (cf. \cite[2.3]{calculusII}), the map from $A(S^j)\wedge |X^{(n)}|$ to the homotopy fiber of the projection map $A(S^j)\wedge |X|\longrightarrow A(S^j)\wedge \bigvee S^{n+1}$ is
\[1-2+(j+\conn X^{(n)}+1)+(j+n)\geq 2j+n-1\geq 2j+\conn X\]
connected. By a five-lemma argument for the long exact sequences in homotopy groups, $\tau$ is $(2j+\conn X)$-connected as long as  both $\tau_{S^{n+1}}$ and $\tau^{(n)}$ are at least $(2j+\conn X)$-connected.
The map $\tau_{S^{n+1}}$ is $(2j+n)$-connected by the sphere case \ref{shpereinduction}, which is greater than $2j+\conn X$. By hypothesis of induction, the connectivity of $\tau^{(n)}$ is $2j+\conn X^{(n)}$. Since the inclusion of the $n$-skeleton $X^{(n)}\longrightarrow X$ is $n$-connected and $\conn X< n$, we must have $\conn X^{(n)}\geq\conn X$, and therefore $\tau^{(n)}$ is at least $(2j+\conn X)$-connected.
\item $n=\conn X$: In this case 
$X$ is homotopy equivalent to a wedge of spheres (see e.g. \cite[4.16]{hatcher}). Therefore $\tau$ is $(2j+n)$-connected by \ref{shpereinduction}.
\end{itemize} 
We now study the connectivity of the restriction to the fixed points \[\tau^G\colon A(S^j)^G\wedge |X|^G\longrightarrow A(S^j\wedge X)^G\] with a similar technique. By replacing $X$ with its subdivision, we can assume that the action is simplicial, so that $|X|^G\cong |X^G|$ is a subcomplex of $|X|$. Denote
\[c=\min\{\conn X^G,\conn X\}\]
and we separate the problem in the following cases.
\begin{itemize}
\item $n<c$: Both $X$ and $X^G$ are at most of dimension $n+1$, and therefore contractible. The map $\tau^G$ is an equivalence since both its source and its target are contractible (cf. \ref{connconfspace}).
\item $c<n$: The inclusion of the $n$-skeleton gives a diagram
\[\xymatrix{A(S^j)^G\wedge |X^{(n)}|^G\ar[r]\ar[d]^-{{\tau^{(n)}}^G}&A(S^j)^G\wedge X^G\ar[r]\ar[d]^-{\tau^G}&A(S^j)^G\wedge (\bigvee S^{n+1})^G\ar[d]^-{\tau_{S^{n+1}}^G}\\
A(S^j\wedge X^{(n)})^G\ar[r]&A(S^j\wedge X)^G\ar[r]&A(S^j\wedge \bigvee S^{n+1})^G
}\]
with a fiber sequence as bottom row by \ref{fibseqinclskel}. Notice that since ${X^{(n)}}^G={X^G}^{(n)}$, the top row is a cofiber sequence. By Blakers-Massey, the map from $A(S^j)^G\wedge |X^{(n)}|^G$ to the homotopy fiber of the projection map $A(S^j)^G\wedge |X|^G\longrightarrow A(S^j)^G\wedge (\bigvee S^{n+1})^G$ is
\[\lceil\frac{j}{2}\rceil+\conn {X^{(n)}}^G+\lceil\frac{j}{2}\rceil+n\geq j+c\]
connected.
By the five lemma, $\tau^G$ is $(j+c)$-connected if both $\tau_{S^{n+1}}^G$ and ${\tau^{(n)}}^G$ are at least $(j+c)$-connected.
The map $\tau_{S^{n+1}}^G$ is $(j+n)$-connected by \ref{shpereinduction}, which is bigger than $j+c$.
By hypothesis of induction, the connectivity of  ${\tau^{(n)}}^G$ is 
\[j+\min\{\conn{X^{(n)}}^G,\conn X^{(n)}\}\]
Since the inclusions of the zero skeleta $X^{(n)}\longrightarrow X$ and ${X^{(n)}}^{G}={X^{G}}^{(n)}\longrightarrow X^G$ are $n$-connected, and $n<c$, we must have $c\leq \min\{\conn{X^{(n)}}^G,\conn X^{(n)}\}$. Therefore ${\tau^{(n)}}^G$ is at least $(j+c)$-connected.
\item $n<\conn X^G$: This forces $X^G$ to be contractible since it is at most of dimension $n+1$. In this case $\tau^G$ has contractible source, and its target has connectivity
\[\min\{j+\conn X,\lceil\frac{j}{2}\rceil+\conn X^G\}=j+\conn X\]
by \ref{connconfspace}.
\item $c=n,n\geq \conn X^G$: This condition implies $n=\conn X^G\leq\conn X$. In particular $X^G$ is homotopy equivalent to a wedge of $(n+1)$-spheres. Consider $X^G$ as a trivial simplicial $G$-set, and factor $\tau^G$ as
\[A(S^j)^G\wedge |X|^G=(A(S^j)\wedge |X|^G)^G{\longrightarrow} A(S^j\wedge X^G)^G\stackrel{A(\iota)^G}{\longrightarrow} A(S^j\wedge X)^G\]
The first map is the map $\tau$ for the trivial $G$-space $X^G$, which is $(j+n)$-connected by \ref{shpereinduction} since $X^G$ is a wedge of spheres. The second map is induced by the $G$-map $\iota\colon S^j\wedge X^G\longrightarrow S^j\wedge X$ by functoriality of $A(-)$. It remains to show that $A(\iota)^G$ is $(j+c)$-connected. The inclusion $\iota\colon S^j\wedge X^G\longrightarrow S^j\wedge X$ induces a fiber sequence
\[A(S^j\wedge X^G)^G\stackrel{A(\iota)^G}{\longrightarrow} A(S^j\wedge X)^G\longrightarrow A(S^j\wedge X/X^G)^G\]
by \ref{fibseqinclskel}. Therefore $A(\iota)^G$ is as connected as $A(S^j\wedge X/X^G)^G$. By \ref{connconfspace} this has connectivity
\[\min\{j+\conn X/X^G,\lceil\frac{j}{2}\rceil+\conn X^G/X^G\}=j+\conn X/X^G=j+c\]
\end{itemize}
\end{proof}

\newpage

\bibliographystyle{amsalpha}
\bibliography{biblio}

\end{document}